\documentclass[12pt,english]{article}
\usepackage[latin9]{inputenc}
\usepackage{geometry}
\usepackage{verbatim}
\usepackage{mathrsfs}
\usepackage{amsmath}
\usepackage{amsthm}
\usepackage{amssymb}
\usepackage{graphicx}
\usepackage{wasysym}
\usepackage{esint}
\usepackage{xargs}[2008/03/08]
\usepackage[intoc]{nomencl}
\makenomenclature

\makeatletter

\DeclareTextSymbolDefault{\textquotedbl}{T1}
\newcommand{\lyxadded}[3]{#3}
\newcommand{\lyxdeleted}[3]{}

\numberwithin{equation}{section}
\theoremstyle{plain}
\newtheorem{thm}{\protect\theoremname}[section]
\theoremstyle{definition}
\newtheorem{problem}[thm]{\protect\problemname}
\theoremstyle{remark}
\newtheorem{rem}[thm]{\protect\remarkname}
\theoremstyle{plain}
\newtheorem{lem}[thm]{\protect\lemmaname}
\theoremstyle{remark}
\newtheorem*{rem*}{\protect\remarkname}
\theoremstyle{definition}
\newtheorem{defn}[thm]{\protect\definitionname}
\theoremstyle{plain}
\newtheorem{assumption}[thm]{\protect\assumptionname}
\theoremstyle{definition}
\newtheorem{example}[thm]{\protect\examplename}
\theoremstyle{plain}
\newtheorem{cor}[thm]{\protect\corollaryname}
\theoremstyle{remark}
\newtheorem{notation}[thm]{\protect\notationname}
\theoremstyle{plain}
\newtheorem{criterion}[thm]{\protect\criterionname}
\theoremstyle{plain}
\newtheorem{conjecture}[thm]{\protect\conjecturename}

\usepackage{fancyhdr}

\makeatother

\usepackage{babel}

\providecommand{\conjecturename}{Conjecture}
\providecommand{\assumptionname}{Assumption}
\providecommand{\corollaryname}{Corollary}
\providecommand{\criterionname}{Criterion}
\providecommand{\definitionname}{Definition}
\providecommand{\examplename}{Example}
\providecommand{\lemmaname}{Lemma}
\providecommand{\notationname}{Notation}
\providecommand{\problemname}{Problem}
\providecommand{\remarkname}{Remark}
\providecommand{\theoremname}{Theorem}

\begin{document}
\global\long\def\divery{\mathrm{div}_{y}}%
\global\long\def\diver{\mathrm{div}}%
\global\long\def\diverx{\mathrm{div}_{x}}%

\global\long\def\nablay{\nabla_{y}}%
\global\long\def\nablax{\nabla_{x}}%

\global\long\def\support{\mathrm{supp}}%

\global\long\def\epi{\mathrm{epi}}%
\global\long\def\short{\mathrm{short}}%
\global\long\def\dom{\mathrm{dom}}%
\global\long\def\Pois{\mathrm{Pois}}%
\global\long\def\mat{\mathrm{mat}}%
\global\long\def\REM{\mathrm{rem}}%
\global\long\def\NEW{\mathrm{new}}%
\global\long\def\pois{\mathrm{pois}}%
\global\long\def\Length{\mathrm{Length}}%

\global\long\def\rmD{\mathrm{D}}%
\global\long\def\rmM{\mathrm{M}}%
\global\long\def\rmP{\mathrm{P}}%
\global\long\def\rmR{\mathrm{R}}%
\global\long\def\rmT{\mathrm{T}}%

\global\long\def\id{\mathrm{id}}%

\global\long\def\eps{\varepsilon}%

\global\long\def\borelB{\mathcal{B}}%
\global\long\def\lebesgueL{\mathcal{L}}%
\global\long\def\hausdorffH{\mathcal{H}}%

\global\long\def\ue{u^{\eps}}%
\global\long\def\ve{v^{\eps}}%

\global\long\def\vel{\boldsymbol{\upsilon}}%
\global\long\def\v{\upsilon}%

\global\long\def\bB{\boldsymbol{\mathrm{B}}}%
\global\long\def\bE{\boldsymbol{\mathrm{E}}}%
\global\long\def\bQ{\boldsymbol{\mathrm{Q}}}%
\global\long\def\bT{\boldsymbol{\mathrm{T}}}%
\global\long\def\bU{\boldsymbol{\mathrm{U}}}%
\global\long\def\bP{\boldsymbol{\mathrm{P}}}%
\global\long\def\bY{\boldsymbol{\mathrm{Y}}}%
\global\long\def\boldnu{\boldsymbol{\nu}}%

\global\long\def\be{\boldsymbol{\mathrm{e}}}%
\global\long\def\bj{\boldsymbol{\mathrm{j}}}%

\global\long\def\A{\mathbb{A}}%
\global\long\def\B{\mathbb{B}}%
\global\long\def\D{\mathbb{D}}%
\global\long\def\E{\mathbb{E}}%
\global\long\def\G{\mathbb{G}}%
\global\long\def\I{\mathbb{I}}%
\global\long\def\M{\mathbb{M}}%
\global\long\def\N{\mathbb{N}}%
\global\long\def\O{\mathbb{O}}%
\global\long\def\P{\mathbb{P}}%
\global\long\def\Q{\mathbb{Q}}%
\global\long\def\R{\mathbb{R}}%
\global\long\def\S{\mathbb{S}}%
\global\long\def\T{\mathbb{T}}%
\global\long\def\U{\mathbb{U}}%
\global\long\def\X{\mathbb{X}}%
\global\long\def\Y{\mathbb{Y}}%
\global\long\def\Z{\mathbb{Z}}%

\global\long\def\Apaths{\A\!\X}%

\global\long\def\closedsets{\mathfrak{F}}%

\global\long\def\sB{\mathscr{B}}%
\global\long\def\sF{\mathscr{F}}%
\global\long\def\sI{\mathscr{I}}%
\global\long\def\sT{\mathscr{T}}%
\global\long\def\ttopology#1{\sT_{#1}}%

\global\long\def\fA{\mathfrak{A}}%
\global\long\def\fB{\mathfrak{B}}%
\global\long\def\fC{\mathfrak{C}}%
\global\long\def\fI{\mathfrak{I}}%
\global\long\def\fM{\mathfrak{M}}%
\global\long\def\fm{\mathfrak{m}}%
\global\long\def\fP{\mathfrak{P}}%
\global\long\def\fp{\mathfrak{p}}%
\global\long\def\fr{\mathfrak{r}}%
\global\long\def\fs{\mathfrak{s}}%
\global\long\def\ft{\mathfrak{t}}%
\global\long\def\fY{\mathfrak{Y}}%
\global\long\def\fa{\mathfrak{a}}%
\global\long\def\fb{\mathfrak{b}}%
\global\long\def\fd{\mathfrak{d}}%

\global\long\def\sfR{\mathsf{R}}%
\global\long\def\sfS{\mathsf{S}}%

\global\long\def\Rd{\mathbb{R}^{d}}%
\global\long\def\Rdd{\mathbb{R}^{d\times d}}%
\global\long\def\Zd{\mathbb{Z}^{d}}%
\global\long\def\I{\mathbb{I}}%

\global\long\def\d{\mathrm{d}}%
\global\long\def\v{\upsilon}%

\global\long\def\weakto{\rightharpoonup}%

\global\long\def\tsto{\stackrel{2s}{\to}}%
\global\long\def\tsweakto{\stackrel{2s}{\weakto}}%

\global\long\def\simple{\mathrm{simple}}%
\global\long\def\fl{\mathrm{flat}}%
\global\long\def\loc{\mathrm{loc}}%
\global\long\def\pot{\mathrm{pot}}%
\global\long\def\sol{\mathrm{sol}}%
\global\long\def\argmin{\mathrm{argmin}}%
\global\long\def\hull{\mathrm{hull}}%

\global\long\def\dist{\mathrm{dist}}%
\global\long\def\conv{\mathrm{conv}}%
\global\long\def\co{\mathfrak{co}}%

\global\long\def\boldY{\boldsymbol{\mathrm{Y}}}%

\global\long\def\mugammapalm{\mu_{\Gamma,\mathcal{P}}}%
\global\long\def\mupalm{\mu_{\mathcal{P}}}%
\global\long\def\nupalm{\nu_{\mathcal{P}}}%

\global\long\def\cA{\mathcal{A}}%
\global\long\def\cB{\mathcal{B}}%
\global\long\def\cC{\mathcal{C}}%
\global\long\def\cE{\mathcal{E}}%
\global\long\def\cF{\mathcal{F}}%
\global\long\def\cG{\mathcal{G}}%
\global\long\def\cH{\mathcal{H}}%
\global\long\def\cI{\mathcal{I}}%
\global\long\def\cL{\mathcal{L}}%
\global\long\def\cM{\mathcal{M}}%
\global\long\def\cO{\mathcal{O}}%
\global\long\def\cP{\mathcal{P}}%
\global\long\def\cQ{\mathcal{Q}}%
\global\long\def\cR{\mathcal{R}}%
\global\long\def\cS{\mathcal{S}}%
\global\long\def\cT{\mathcal{T}}%
\global\long\def\cU{\mathcal{U}}%
\global\long\def\cV{\mathcal{V}}%
\global\long\def\cX{\mathcal{X}}%
\global\long\def\cY{\mathcal{Y}}%

\global\long\def\sE{\mathscr{E}}%
\global\long\def\sI{\mathscr{I}}%

\global\long\def\mapA{\mathcal{A}}%

\global\long\def\muomega{\mu_{\omega}}%
\global\long\def\muomegaeps{\mu_{\omega}^{\eps}}%
\global\long\def\mugammaomega{\mu_{\Gamma(\omega)}}%
\global\long\def\mugammaomegaeps{\mu_{\Gamma(\omega)}^{\eps}}%
\global\long\def\Ball#1#2{\mathbb{B}_{#1}{\left(#2\right)}}%
\newcommandx\Balldim[3][usedefault, addprefix=\global, 1=]{\mathbb{B}_{#2}^{#1}{\left(#3\right)}}%

\newcommandx\Balldimclosed[3][usedefault, addprefix=\global, 1=]{\overline{\mathbb{B}}_{#2}^{#1}{\left(#3\right)}}%
\global\long\def\cone{\mathbb{C}}%

\global\long\def\norm#1{\left\Vert #1\right\Vert }%

\global\long\def\scp#1#2{\left\langle #1,#2\right\rangle }%

\global\long\def\LOM{L^{2}(\Omega;\Rdd)}%
\global\long\def\LOMns{L_{n,s}^{2}(\Omega)}%
\global\long\def\LOMn{L_{n}^{2}(\Omega)}%
\global\long\def\e{\mathrm{e}}%
\global\long\def\diam{\mathrm{diam}}%

\global\long\def\of#1{{\left(#1\right)}}%

%\begin{frontmatter}

\title{Stochastic homogenization on randomly perforated domains}

\author{Martin Heida}
\maketitle
%% article and other classes: abstract after maketitle
\begin{abstract}
We study the existence of uniformly bounded extension and trace operators
for $W^{1,p}$-functions on randomly perforated domains, where the
geometry is assumed to be stationary ergodic. Such extension and trace
operators are important for compactness in stochastic homogenization.
In contrast to former approaches and results, we use very weak assumptions
on the geometry which we call local $(\delta,M)$-regularity, isotropic
cone mixing and bounded average connectivity. The first concept measures
local Lipschitz regularity of the domain while the second measures
the mesoscopic distribution of void space. The third is the most tricky
part and measures the \textquotedbl mesoscopic\textquotedbl{} connectivity
of the geometry.

In contrast to former approaches we do not require a minimal distance
between the inclusions and we allow for globally unbounded Lipschitz
constants and percolating holes. We will illustrate our method by
applying it to the Boolean model based on a Poisson point process
and to a Delaunay pipe process.

We finally introduce suitable Sobolev spaces on $\Rd$ and $\Omega$
in order to construct a stochastic two-scale convergence method and
apply the resulting theory to the homogenization of a $p$-Laplace
problem on a randomly perforated domain.
\end{abstract}
%\author[label1]{Ond\v{r}ej Sou\v{c}ek\corref{cor1}}
%\ead{Ondrej.Soucek@mff.cuni.cz}
%\author[label2]{Martin Heida}
%\author[label1]{Josef M\'{a}lek}
%\address[label1]{Charles University, Faculty of Mathematics and Physics, Mathematical Institute, Sokolovsk\'{a} 83, 186 75, Praha 8, Czech Republic}
%\address[label2]{Weierstrass Institute
%for Applied Analysis and Stochastics
%Mohrenstrasse 39
%D-10117 Berlin, Germany}

%\end{frontmatter}


\tableofcontents{}

\section{Introduction}

In 1979 Papanicolaou and Varadhan \cite{papanicolaou1979boundary}
and Kozlov \cite{kozlov1979averaging} for the first time independently
introduced concepts for the averaging of random elliptic operators.
At that time, the periodic homogenization theory had already advanced
to some extend (as can be seen in the book \cite{papanicolau1978asymptotic}
that had appeared one year before) dealing also with non-uniformly
elliptic operators \cite{marcellini1978homogenization} and domains
with periodic holes \cite{cioranescu1979homogenization}.

Even though the works \cite{kozlov1979averaging,papanicolaou1979boundary}
clearly guide the way to a stochastic homogenization theory, this
theory advanced quite slowly over the past 4 decades. Compared to
the stochastic case, periodic homogenization developed very strong
with methods that are now well developed and broadly used. The most
popular methods today seem to be the two-scale convergence method
by Allaire and Nguetseng \cite{allaire1992homogenization,Nguetseng1989}
in 1989/1992 and the periodic unfolding method \cite{cioranescu2002periodic}
by Cioranescu, Damlamian and Griso in 2002. Both methods are conceptually
related to asymptotic expansion and very intuitive to handle. It is
interesting to observe that the stochastic counterpart, the stochastic
two-scale convergence, was developed only in 2006 by Zhikov and Piatnitsky
\cite{Zhikov2006}, with the stochastic unfolding developed only recently
in \cite{neukammvarga2018stochastic,heida2019lambda}.

\lyxadded{heida}{Fri Jun 05 16:36:23 2020}{A further work by Bourgeat,
Mikelic and }Wright\lyxadded{heida}{Fri Jun 05 16:36:23 2020}{ \cite{Bourgeat1994}
introduced two-scale convergence in the mean. This }sense\lyxadded{heida}{Fri Jun 05 16:36:23 2020}{
of two-scale convergence is indeed a special case of the stochastic
unfolding, which can only be applied in an averaged }sense\lyxadded{heida}{Fri Jun 05 16:36:23 2020}{,
too. This leads us to a fundamental difference between the periodic
and the stochastic homogenization. In stochastic homogenization we
distinguish between quenched convergence, i.e. for almost every realization
one can prove homogenization, and homogenization in the mean, which
means that homogenization takes place in expectation.}

\lyxadded{heida}{Fri Jun 05 16:36:23 2020}{In particular in nonlinear
}non-convex\lyxadded{heida}{Fri Jun 05 16:36:23 2020}{ problems (that
is: we cannot rely on weak convergence methods) the quenched convergence
is of uttermost importance, as this }sense\lyxadded{heida}{Fri Jun 05 16:36:23 2020}{
of convergence allows to use - for each fixed $\omega$ - compactness
in the spaces $H^{1}(\bQ)$. On the other hand, convergence in the
mean deals with convergence in $L^{2}(\Omega;H^{1}(\bQ))$, which
goes in hand with a loss of compactness.}

\lyxadded{heida}{Fri Jun 05 16:36:23 2020}{The results presented below
are meant for application in }quenched\lyxadded{heida}{Fri Jun 05 16:36:23 2020}{
convergence. The estimates for the extension and trace operators which
are derived strongly depends on the realization of the geometry -
thus on $\omega$. Nevertheless, if the geometry is stationary, a
corresponding estimate can be achieved for almost every $\omega$.}

\subsubsection*{The Problem}

The discrepancy in the speed of progress between periodic and stochastic
homogenization is due to technical problems that arise from the randomness
of parameters. In this work we will consider uniform extension operators
for randomly perforated \lyxadded{heida}{Sat Jun 06 10:51:40 2020}{stationary
domains. We use stationarity (see Def. \ref{def:stationary}) as this
is the standard way to cope with the lack of periodicity. }Let us
first have a look at a typical application to illustrate the need
of the extension operators that we construct below.

Let $\bP(\omega)\subset\Rd$ be a stationary random open set and let
$\eps>0$ be the smallness parameter and let $\tilde{\bP}\of{\omega}$
be a connected component of $\bP\of{\omega}$. For a bounded open
domain, we consider $\bQ_{\tilde{\bP}}^{\eps}(\omega):=\bQ\cap\eps\tilde{\bP}(\omega)$
and $\Gamma^{\eps}(\omega):=\bQ\cap\eps\partial\tilde{\bP}(\omega)$
with outer normal $\nu_{\Gamma^{\eps}(\omega)}$. We study the following
problem in Section \ref{sec:Homogenization-of-p-Laplace}:
\begin{align}
-\diver\left(a\left|\nabla\ue\right|^{p-2}\nabla\ue\right) & =g &  & \mbox{on }\bQ_{\tilde{\bP}}^{\eps}(\omega)\,,\nonumber \\
u & =0 &  & \mbox{on }\partial\bQ\,,\label{eq:system-eps-p-Laplace}\\
\left|\nabla\ue\right|^{p-2}\nabla\ue\cdot\nu_{\Gamma^{\eps}(\omega)} & =f(\ue) &  & \mbox{on }\Gamma^{\eps}(\omega)\,.\nonumber 
\end{align}
Note that for simplicity of illustration, the only randomness that
we consider in this problem is due to $\bP(\omega)$, i.e. we assume
$a\equiv const$.

Problem (\ref{eq:system-eps-p-Laplace}) can be recast into a variational
problem, i.e. solutions of (\ref{eq:system-eps-p-Laplace}) are local
minimizers of the energy functional 
\[
\cE_{\eps,\omega}(u)=\int_{\bQ_{\tilde{\bP}}^{\eps}(\omega)}\left(\frac{1}{p}\left|\nabla u\right|^{p}-gu\right)+\int_{\Gamma^{\eps}(\omega)}\int_{0}^{u}F(s)\d s\,,
\]
where $F$ is convex with $\partial F=f$. This problem will be treated
in Theorem \ref{thm:Final-homogenization-Theorem} and the final Remark
\ref{rem:final-remark}.

One way to prove homogenization of (\ref{eq:system-eps-p-Laplace})
is to prove $\Gamma$-convergence of $\cE_{\eps,\omega}$. Conceptually,
this implies convergence of the minimizers $\ue$ to a minimizer of
the limit functional. However, the minimizers are elements of $W^{1,p}(\bQ_{\tilde{\bP}}^{\eps})$
and since this space changes with $\eps$, we lack compactness in
order to pass to the limit in the nonlinearity. The canonical path
to circumvent this issue in \emph{periodic }homogenization is via
uniformly bounded extension operators $\cU_{\eps}:\,W^{1,p}(\bQ_{\tilde{\bP}}^{\eps})\to W^{1,p}(\bQ)$,
see \cite{hopker2014diss,hopker2014note}, combined with uniformly
bounded trace operators, see \cite{gahn2016homogenization,guillen2015quasistatic}.

The first proof for the existence of periodic extension operators
was due to Cioranescu and Paulin \cite{cioranescu1979homogenization}
in 1979, while the proof in its full generality was provided only
recently by H\"opker and B\"ohm \cite{hopker2014note} and H\"pker \cite{hopker2016extension}.
In this work we will generalize parts of the results of \cite{hopker2016extension}
to a stochastic setting. A modified version of the original proof
of \cite{hopker2016extension} is provided in Section \ref{sec:Periodic-extension-theorem}.
It relies on three ingredients: the local Lipschitz regularity of
the surface, the periodicity of the geometry and the connectedness.
Local Lipschitz regularity together with periodicity imply global
Lipschitz regularity of the surface. In particular, one can construct
a local extension operator on every cell $z+(-\delta,1+\delta)^{d}$,
$z\in\Zd$ which might then be glued together using a periodic partition
of unity of $\Rd$. The connectedness of the geometry assures that
the difference of the average of a function $u$ on two different
cells $z_{1}$ and $z_{2}$ can be computed from the gradient along
a path connecting the two cells and being fully comprised in $z_{1}+(-1,2)^{d}$.

In the stochastic case the proof of existence of suitable extension
operators is much more involved and not every geometry will eventually
allow us to be successful. In fact, we will not be able - in general
- to even provide extension operators $\cU_{\eps}:\,W^{1,p}\left(\bQ_{\tilde{\bP}}^{\eps}(\omega)\right)\to W^{1,p}(\bQ)$
but rather obtain $\cU_{\eps}:\,W^{1,p}\left(\bQ_{\tilde{\bP}}^{\eps}(\omega)\right)\to W^{1,r}(\bQ)$,
where $r<p$ depends (among others) on the dimension and on the distribution
of the Lipschitz constant of $\partial\tilde{\bP}(\omega)$. This
is due to the presence of arbitrarily ``bad'' local behavior of
the geometry.

The theory developed below also allows to provide estimates on the
trace operator 
\[
\cT_{\omega}:\,C^{1}\of{\overline{\bP}(\omega)}\to C\of{\partial\bP(\omega)}
\]
when seen as an operator $\cT_{\omega}:\,W_{\loc}^{1,p}\of{\bP(\omega)}\to L_{\loc}^{r}\of{\partial\bP(\omega)}$,
where again $1\leq r<p$ in general.

We summarize the above discussion in the following.
\begin{problem}
\label{prob:Problem-1}Find (computationally or rigorously) verifiable
conditions on stationary random geometries that allow to prove existence
of extension operators 
\[
\cU_{\eps}:\,W_{0,\partial\bQ}^{1,p}\left(\bQ_{\tilde{\bP}}^{\eps}(\omega)\right)\to W^{1,r}(\bQ)\quad\text{s.t.}\quad\norm{\nabla\cU_{\eps}u}_{L^{r}(\bQ)}\leq C\norm{\nabla u}_{L^{p}\left(\bQ_{\tilde{\bP}}^{\eps}(\omega)\right)}\,,
\]
where $r\geq1$ and $C>0$ are independent of $\eps$ and where 
\[
W_{0,\partial\bQ}^{1,p}\left(\bQ_{\tilde{\bP}}^{\eps}(\omega)\right)=\left\{ u\in W^{1,p}\left(\bQ_{\tilde{\bP}}^{\eps}(\omega)\right)\,:\;u|_{\partial\bQ}\equiv0\right\} \,.
\]
\end{problem}

\begin{problem}
\label{prob:Problem-2}Find (computationally or rigorously) verifiable
conditions on stationary random geometries that allow to prove an
estimate
\[
\eps\norm{\cT_{\eps}u}_{L^{r}(\bQ\cap\eps\partial\bP)}^{r}\leq C\left(\norm u_{L^{p}\of{\bQ\cap\eps\bP(\omega)}}^{r}+\eps^{r}\norm{\nabla u}_{L^{p}\of{\bQ\cap\eps\bP(\omega)}}^{r}\right)\,,
\]
where $r\geq1$ and $C>0$ are independent of $\eps$.
\end{problem}

Let us mention at this place existing results in literature. In recent
years, Guillen and Kim \cite{guillen2015quasistatic} have proved
existence of uniformly bounded extension operators $\cU_{\eps}:\,W^{1,p}\left(\bQ_{\tilde{\bP}}^{\eps}(\omega)\right)\to W^{1,p}(\bQ)$
in the context of minimally smooth surfaces, i.e. uniformly Lipschitz
and uniformly bounded inclusions with uniform minimal distance. A
homogenization result of integral functionals on randomly perforated
domains with uniformly bounded inclusions was provided by Piat and
Piatnitsky \cite{piat2010gamma}. Concerning unbounded inclusions
and non-uniformly Lipschitz geometries, the present work seems to
be the first approach. Since Problem \ref{prob:Problem-2} is easier
to handle, we first explain our concept of microscopic regularity
in view of $\cT_{\omega}$ and then go on to extension operators.

\subsubsection*{$\left(\delta,M\right)$-Regularity and the Trace Operator}

We introduce two concepts which are suited for the current and potentially
also for further studies. The first of these two concepts is inspired
by the concept of minimal smoothness \cite{stein2016singular} and
accounts for the local regularity of $\partial\bP$. Deviating from
\cite{stein2016singular} we will call it \emph{local $\left(\delta,M\right)$-regularity}
(see Definition \ref{def:loc-del-M-reg}). Although this assumption
is very weak, its consequences concerning local coverings of $\partial\bP$
are powerful. Based on this concept, we introduce the functions $\delta$,
$\hat{\rho}$ and $\rho$ on $\partial\bP$ as well as $M_{[\eta]}$
and $M_{[\eta],\Rd}$ for $\eta\in\left\{ \delta,\hat{\rho},\rho\right\} $
in Lemmas \ref{lem:properties-delta-M-regular}, \ref{lem:rho-p-lsc},
\ref{lem:M-eta} and \ref{lem:delta-rho-M-measurable} and make the
following assumptions:
\begin{assumption}
\label{assu:Trace}Let $\bP(\omega)$ be a random open set such that
for $1\leq r<p_{0}<p$ and $\eta\in\left\{ \rho,\hat{\rho},\delta\right\} $
it holds either 
\begin{align*}
\int_{\Omega}\eta^{-\frac{1}{p_{0}-r}}\d\mugammapalm+\E\of{M_{[\frac{1}{8}\eta],\Rd}^{\left(\frac{1}{p_{0}}+1\right)\frac{p}{p-p_{0}}}} & <\infty\,,\\
\text{or }\int_{\Omega}\left(\eta M_{[\frac{1}{16}\eta],\Rd}\right)^{-\frac{1}{p-r}}\d\mugammapalm & <\infty\,.
\end{align*}
\end{assumption}

Having studied the properties of $\left(\delta,M\right)$-regular
sets in detail in Sections \ref{subsec:Microscopic-Regularity} and
\ref{subsec:Local--Regularity} it is very easy to prove the following
trace theorem (for notations we refer to Section \ref{sec:Preliminaries}
and Section \ref{subsec:Microscopic-Regularity}). Note that via a
simple rescaling, this provides a solution to Problem \ref{prob:Problem-1}.
\begin{thm}[Solution of Problem \ref{prob:Problem-1}]
\label{thm:Main-Thm-1}Let $\bP(\omega)$ be a stationary and ergodic
random open set which is almost surely locally $\left(\delta,M\right)$
regular and let Assumption \ref{assu:Trace} hold. For given $\omega$
let $\cT_{\omega}:\,C^{1}\of{\overline{\bP}(\omega)}\to C\of{\partial\bP(\omega)}$
be the trace operator. Then for almost every $\omega$ the extension
$\cT_{\omega}:\,W_{\loc}^{1,p}\of{\bP(\omega)}\to L_{\loc}^{r}\of{\partial\bP(\omega)}$
is continuous and there exists a constant $C_{\omega}>0$ s.t. it
holds for every bounded Lipschitz domain $\bQ\supset\Ball 10$ and
every $n\in\N$
\[
\norm{\cT_{\omega}u}_{L^{r}(\partial\bP\cap n\bQ)}\leq C_{\omega}\norm u_{W^{1,p}\left(\Ball{\fr}{n\bQ}\cap\bP\right)}.
\]
\end{thm}

\begin{proof}
This is a consequence of Theorem \ref{thm:uniform-trace-estimate-1},
stationarity and ergodicity and the ergodic theorem.
\end{proof}

\subsubsection*{Construction of Extension Operators}

The main results of this work is on extension operators on randomly
perforated domains. In order to construct a suitable extension operator,
we use

\paragraph*{Step 1: $\left(\delta,M\right)$-regularity}

Concerning extension results, the concept of $\left(\delta,M\right)$-regularity
suggests the naive approach to use a local open covering of $\partial\bP$
and to add the local extension operators via a partition of unity
in order to construct a global extension operator. We call this ansatz
naive since one would not chose this approach even in the periodic
setting, as it is known to lead to unbounded gradients. Nevertheless,
this ansatz is followed in Section \ref{subsec:Microscopic-Extension-dm}
for two reasons. The first reason is illustration of an important
principle: The extension operator $\cU=\tilde{\cU}+\hat{\cU}$ can
be split up into a \emph{local }part $\tilde{\cU}$, whose norm can
be estimated by local properties of $\partial\bP$, and a \emph{global
}part $\hat{\cU}$ whose norm is determined by \emph{connectivity},
an issue which has to be resolved afterwards, and corresponds to Step
2 in the proof of Theorem \ref{thm:Ext-per-connected} below (periodic
case), where one glues together the local extension operators on the
periodic cells. The second reason is that this first estimate, although
it cannot be applied globally, is very well suited for constructing
a local extension operator. \lyxadded{heida}{Sat Jun 06 10:59:19 2020}{Lemma
\ref{lem:local-delta-M-extension-estimate} hence provides estimates
of a certain extension operator which has the property that the constant
in the estimate tends to $+\infty$ as the domain grows.}

However, this first ansatz grants some insight into the structure
of the extension problem. In particular, we find the following result
which will provide a better understanding of the Sobolev spaces $W^{1,p}(\Omega)$
and $W^{1,r,p}(\Omega,\bP)$ on the probability space $\Omega$.
\begin{assumption}
\label{assu:local-extension}Let $\bP(\omega)$ be a random open set
such that Assumption \ref{assu:M-alpha-bound} hold and let $\hat{d}$
be the constant from (\ref{eq:lem:properties-local-rho-convering-4}).
\begin{enumerate}
\item Assume for $r<p$ that 
\begin{equation}
\E\of{\tilde{M}_{[\frac{1}{8}\hat{\rho}]}^{\frac{p\left(\hat{d}+1\right)}{p-r}}}+\E\of{\tilde{M}_{[\frac{1}{8}\hat{\rho}]}^{\frac{p\left(\hat{d}+\alpha\right)}{p-r}}}<\infty\label{eq:assu:local-extension-1}
\end{equation}
\item Assume for $r<p_{0}<p_{1}<p$ that (\ref{eq:assu:local-extension-1})
and either
\[
\E\of{\tilde{M}_{[\frac{1}{8}\hat{\rho}]}^{\frac{p_{1}\left(d-2\right)(p_{0}-r)}{r(p_{1}-p_{0})}}}+\E\of{\tilde{M}_{[\frac{1}{8}\hat{\rho}]}^{\frac{\alpha p_{1}p}{p-p_{1}}}}+\E\of{\rho^{1-\frac{rp_{0}}{p_{0}-r}}}<\infty
\]
or 
\[
\E\of{\tilde{M}_{[\frac{1}{8}\hat{\rho}]}^{\alpha p_{0}p}}+\E\of{\rho_{\mathrm{bulk}}^{-\frac{rp_{0}}{p_{0}-r}}}<\infty\,,
\]
where 
\[
\rho_{\mathrm{bulk}}\of x:=\inf\left\{ \rho\of{\tilde{x}}\,:\;\tilde{x}\in\partial\bP\,\text{s.t. }x\in\Ball{\frac{1}{8}\rho\of{\tilde{x}}}{\tilde{x}}\right\} \,,
\]
\end{enumerate}
\end{assumption}

\begin{thm}
\label{thm:Main-Thm-2}Let Assumption \ref{assu:local-extension}
hold and let $\tau$ be ergodic. Then for almost every $\omega$ the
extension operator $\cU:\,W_{\loc}^{1,p}\of{\bP(\omega)}\to W_{\loc}^{1,r}\of{\Rd}$
provided in (\ref{eq:def:cU-Q-2}) is well defined and for $\bQ\subset\Rd$
a bounded domain with Lipschitz boundary there exists a constant $C(\omega)$
such that for every positive $n\geq1$ and every $u\in W^{1,p}(\bP(\omega)\cap\Ball{\fr}{n\bQ})$
\begin{align*}
\frac{1}{n^{d}\left|\bQ\right|}\int_{n\bQ}\left(\left|\cU u\right|^{r}+\left|\nabla\cU u\right|^{r}\right) & \leq C(\omega)\left(\frac{1}{n^{d}\left|\bQ\right|}\int_{\bP(\omega)\cap\Ball{\fr}{n\bQ}}\left|u\right|^{p}+\left|\nabla u\right|^{p}\right)^{\frac{r}{p}}\,.
\end{align*}
\end{thm}

\begin{proof}
This follows from Lemmas \ref{lem:local-delta-M-extension-estimate},
\ref{lem:first-estim-nabla-phi-0} and \ref{lem:delta-tilde-construction-estimate}
on noting that $\nabla\phi_{0}\leq C\rho_{\mathrm{bulk}}^{-1}$.
\end{proof}
Theorem \ref{thm:Main-Thm-2}, though useful, is not satisfactory
for homogenization, as $\nabla\cU u$ is bounded by $u$ and not solely
$\nabla u$. Therefore, some more work is needed.

\paragraph*{Step 2: isotropic cone mixing}

In order to account for the issue of connectedness in a proper way
on the macroscopic level, we propose our second fundamental concept
of \emph{isotropic cone mixing geometries} (see Definition \ref{def:iso-cone-mix}),
which allow to construct a global Voronoi tessellation of $\Rd$ with
good local covering properties. This definition, though being rather
technical, can be verified rather easily using Criterion \ref{cri:stat-erg-ball}.

In short, isotropic cone mixing allows to distribute balls $B_{i}=\Ball{_{\frac{\fr}{2}}}{x_{i}}$
of a uniform minimal radius $\frac{1}{2}\fr$ within $\bP$ such that
the centers $x_{i}$ of the balls $B_{i}$ generate a Voronoi mesh
of cells $G_{i}$ with diameter $d_{i}$, distributed according to
a function $f(d)$ (see Lemma \ref{lem:Iso-cone-geo-estimate}). These
Voronoi cells in general might be of arbitrary large diameter $d_{i}$,
although they are bounded in the statistical average. Due to this
lack of a uniform bound, we call the distribution of Voronoi cells
the \emph{mesoscopic regularity }of the geometry.

\paragraph*{Step 3: gluing}

The Voronoi cells resulting from an isotropic cone mixing geometry
are well suited for the gluing of local extension operators. We will
construct the macroscopic extension operator in an analogue way to
\cite{hopker2016extension}, replacing the periodic cells by the Voronoi
cells (see Figure \ref{fig:sketch-extension}). In Theorem \ref{thm:Fine-meso-estimate}
we provide a first abstract result how the norm of the glued operator
can be estimated from the distribution of $M$, the geometry of the
Voronoi mesh and the connectivity, even though the last two properties
enter rather indirectly. To make this more clear, we note at this
points that the extension operator depends on two types of local averages:
To each Voronoi cell $G_{i}$ we take the average $\cM_{i}u$ over
$B_{i}$. Furthermore, to every local microscopic extension operator
chosen in Section \ref{sec:Extension-and-Trace-d-M} there corresponds
a local average $\tau_{j}u$ close to the boundary. We will see that
the norm of the extension operator strongly depends on the differences
$\left|\cM_{i}u-\cM_{j}u\right|$ and $\left|\cM_{j}u-\tau_{k}u\right|$.

In Theorem \ref{thm:Rough-meso-estimate} we will see that the dependence
on $\left|\cM_{i}u-\cM_{j}u\right|$ can be eliminated with the price
to increase the cost of ``unfortunate distributions'' of $G_{i}$
and of the local $\left(\delta,M\right)$ regularity. The remaining
dependence which we leave unresolved is the dependence on $\left|\cM_{j}u-\tau_{k}u\right|$.
This dependence is linked to quantitative connectedness properties
of the geometry. By this we mean more than the topological question
of connectedness. In particular, we need an estimate of the type $\left|\cM_{j}u-\tau_{k}u\right|^{r}\leq\int_{G_{i}}C(x)\left|\nabla u(x)\right|^{r}\d x$
which will finally allow us an estimate of $\sum_{j,k}\left|\cM_{j}u-\tau_{k}u\right|^{r}$
in terms of $\nabla u$. Unfortunately, the classical percolation
theory, which deals with connectedness of random geometries, is not
developed to answer this question. In this paper, we will use two
workarounds which we call ``statistically harmonic'' and ``statistically
connected''. However, further research has to be conducted. We state
our first main theorem.
\begin{thm}
\label{thm:Main-Theorem-3}Let $\bP(\omega)$ be a stationary ergodic
random open set which is almost surely $\left(\delta,M\right)$-regular
(Def. \ref{def:loc-del-M-reg}) and isotropic cone mixing for $\fr>0$
and $f(R)$ (Def. \ref{def:iso-cone-mix}) and statistically harmonic
(Def. \ref{def:statis-mixing}) and let $1\leq r<p\leq\infty$. Let
$\bQ\subset\Rd$ be bounded open with Lipschitz boundary as well as
$s\in(r,p)$ such that 
\begin{align*}
\E\of{\tilde{M}^{\frac{2pd}{p-r}}} & <+\infty\,,\\
\sum_{k=1}^{\infty}\left(k+1\right)^{d\left(\frac{2p-s}{p-s}\right)+\left(d+1\right)(2r+2)\frac{p}{p-s}+r\left(\frac{p}{p-s}-1\right)}f\of k & <+\infty\,,\\
\E\of{\sup_{R}\frac{1}{R^{d}}\int_{\Ball R0}\left(\sum_{k}P\of{d_{k}}\chi_{\fA_{4,k}}C_{k}\right)^{\frac{p}{p-s}}} & <+\infty\,.
\end{align*}
Then for almost every $\omega$ the extension operator $\cU:\,W_{\loc}^{1,p}\of{\bP(\omega)}\to W_{\loc}^{1,r}\of{\Rd}$
provided in (\ref{eq:def-global-U}) is well defined with a constant
$C(\omega)$ such that for every positive $n\geq1$
\begin{align*}
\frac{1}{n^{d}\left|\bQ\right|}\int_{n\bQ}\left|\cU u\right|^{r} & \leq C(\omega)\left(\frac{1}{n^{d}\left|\bQ\right|}\int_{\bP(\omega)\cap n\bQ}\left|u\right|^{p}\right)^{\frac{r}{p}}\\
\frac{1}{n^{d}\left|\bQ\right|}\int_{n\bQ}\left|\nabla\cU u\right|^{r} & \leq C(\omega)\left(\frac{1}{n^{d}\left|\bQ\right|}\int_{\bP(\omega)\cap n\bQ}\left|\nabla u\right|^{p}\right)^{\frac{r}{p}}\,.
\end{align*}
\end{thm}

\begin{proof}
This follows from Theorems \ref{thm:Fine-meso-estimate}, \ref{thm:Rough-meso-estimate}
and \ref{thm:final-estimate} on noting that in the general case we
have to assume $\alpha=\hat{d}=d$. Furthermore, we need Lemma \ref{lem:estim-E-fa-fb}.
\end{proof}
In practical applications, one would need to verify whether $\bP$
is statistically harmonic via numerical simulations. The problem particularly
results in the numerical evaluation of a Laplace operator.

Based on this insight, we develop an alternative approach: The connectedness
of $\bP$ is quantified by introducing directly a discrete graph on
$\bP$ and a discrete Poisson equation on this graph. The construction
of the graph and the evaluation of the Poisson equation can be done
numerically, but with the advantage that the discrete quantities are
now directly connected to the analytical theory. Additionally to the
$(\delta,M)$-regularity we have to deal with the average diameter
$d_{j}$ of the cells of a the global Voronoi tessellation and the
local stretch factor $\sfS_{j}$. We impose the following assumptions:
\begin{assumption}
\label{assu:dirichlet-extension}Let $\bP(\omega)$ be a random open
set such that Assumption \ref{assu:M-alpha-bound} hold and let $\hat{d}$
be the constant from (\ref{eq:lem:properties-local-rho-convering-4}).
Let (\ref{eq:assu:local-extension-1}) and for $r<\tilde{s}<s<p$
let either 
\[
\E\of{\tilde{M}_{[\frac{1}{8}\hat{\rho}]}^{\frac{p_{1}\left(d-2\right)(\tilde{s}-r)}{r(s-\tilde{s})}}}+\E\of{\rho^{1-\frac{\tilde{s}r}{\tilde{s}-r}}}<\infty
\]
or 
\[
\E\of{\rho_{\mathrm{bulk}}^{-\frac{sr}{s-r}}}<\infty\,.
\]
Furthermore, let $\bP(\omega)$ be almost surely isotropic cone mixing
for $\fr>0$ and $f(R)$ (Def. \ref{def:iso-cone-mix}) as well as
locally connected and let the local stretch factor (see Definition
Theorem \ref{thm:strech-main-thm} and Definition \ref{def:thm:strech-main-thm})
satisfy $\P\of{\sfS>S_{0}}\leq f_{s}(S_{0})$ such that 
\begin{align*}
\sum_{k=1}^{\infty}\left(k+1\right)^{2d+\frac{r(d-1)+drs}{s-r}}f\of k & <+\infty\,,\\
\sum_{k,N=1}^{\infty}\left[\left(N+1\right)\left(k+1\right)\right]^{d\frac{2p-s}{p-s}+\frac{s-1}{s}\frac{p}{p-s}+r\frac{s}{p-s}}\left(k+1\right)^{d\frac{p}{p-s}}f(k)f_{S}(N) & <+\infty\,.
\end{align*}
\end{assumption}

The second main theorem can be formulated as follows:
\begin{thm}
\label{thm:Main-Theorem-4}Let $\bP(\omega)$ be a stationary ergodic
random open set which is almost surely $\left(\delta,M\right)$-regular
(Def. \ref{def:loc-del-M-reg}) and isotropic cone mixing for $\fr>0$
and $f(R)$ (Def. \ref{def:iso-cone-mix}) as well as locally connected
and satisfy $\P\of{\sfS>S_{0}}\leq f_{s}(S_{0})$ such that Assumption
\ref{assu:dirichlet-extension} holds. For $1\leq r<\tilde{s}<s<p\leq\infty$
and $\bQ\subset\Rd$ a bounded domain with Lipschitz boundary. Then
for almost every $\omega$ the extension operator $\cU:\,W_{\loc}^{1,p}\of{\bP(\omega)}\to W_{\loc}^{1,r}\of{\Rd}$
provided in (\ref{eq:def-global-U}) is well defined with a constant
$C(\omega)$ such that for every positive $n\geq1$ and every $u\in W_{0,\partial\bQ}^{1,p}(\bP(\omega)\cap n\bQ)$
\begin{align}
\frac{1}{n^{d}\left|\bQ\right|}\int_{n\bQ}\left|\cU u\right|^{r} & \leq C(\omega)\left(\frac{1}{n^{d}\left|\bQ\right|}\int_{\bP(\omega)\cap n\bQ}\left|u\right|^{p}\right)^{\frac{r}{p}}\,,\label{eq:thm:Main-Theorem-4-1}\\
\frac{1}{n^{d}\left|\bQ\right|}\int_{n\bQ}\left|\nabla\cU u\right|^{r} & \leq C(\omega)\left(\frac{1}{n^{d}\left|\bQ\right|}\int_{\bP(\omega)\cap n\bQ}\left|\nabla u\right|^{p}\right)^{\frac{r}{p}}\,.\label{eq:thm:Main-Theorem-4-2}
\end{align}
\end{thm}

\begin{proof}
We combine Theorem \ref{thm:Fine-meso-estimate} with Lemmas \ref{lem:6-4}
and \ref{lem:6-5} as well as Lemmas \ref{lem:delta-tilde-construction-estimate}
and \ref{lem:estim-E-fa-fb} to obtain the first and second condition.
The remaining condition is inferred from Theorem \ref{thm:strech-main-thm}
and Lemma \ref{lem:estim-E-fa-fb}.
\end{proof}

\subsection*{Sobolev Spaces on $\Omega$}

Besides the evident benefit of the above extension and trace theorems,
let us note that these theorems are also needed for the construction
of the suitable Sobolev spaces on $\Omega$. In Section \ref{sec:Sobolev-spaces-on}
we recall some standard construction of Sobolev spaces on the probability
space $\Omega$ and provide some links between two major approaches
which seem to be hard to find in one place. We will need this summing
up in order to better illustrate the generalization to perforated
domains.

To understand our ansatz, we recall a result from \cite{heida2011extension}
that there exist $\bP\subset\Omega$ and $\Gamma\subset\Omega$ such
that for almost every $\omega\in\Omega$ $\chi_{\bP(\omega)}(x)=\chi_{\bP}(\tau_{x}\omega)$
and $\chi_{\Gamma(\omega)}(x)=\chi_{\Gamma}(\tau_{x}\omega)$, where
$\Gamma(\omega):=\partial\bP(\omega)$. The random set $\bP(\omega)$
leads to Sobolev spaces $W^{1,p}(\bP(\omega))$, e.g. by defining
$W^{1,p}(\bP(\omega)):=\left\{ \chi_{\bP(\omega)}u:\,u\in W^{1,p}(\Rd)\right\} $.
We will see that we can introduce spaces $W^{1,p}(\bP)$, but this
construction is more involved than in $\Rd$ and heavily relies on
the almost sure extension property guarantied by Theorem \ref{thm:Main-Thm-2}.
Once we have introduced the spaces $W^{1,p}(\bP)$ we can also introduce
``trace''-operators $\cT_{\Omega}:\,W^{1,p}(\bP)\to L^{r}(\Gamma)$,
where $\Gamma\subset\Omega$ with $\chi_{\Gamma(\omega)}(x)=\chi_{\Gamma}(\tau_{x}\omega)$,
and $L^{r}(\Gamma)$ is to be understood w.r.t. the Palm measure on
$\Gamma$. This construction will rely on Theorems \ref{thm:Main-Thm-2}
and \ref{thm:Main-Thm-1}. In all our results, we only provide sufficient
conditions for the existence of the respective spaces and operators.
Necessary conditions are left for future studies.

\subsection*{Discussion: Random Geometries and Applicability of the Method}

In Section \ref{sec:Two-scale-convergence} we will discuss how the
present results can be applied in the framework of the stochastic
two-scale convergence method. However, this concerns only the analytic
aspect of applicability.

The more important question is the applicability of the presented
theory from the point of view of random geometries. Of course our
result can be applied to periodic geometries and hence also to stochastic
geometries which originate from random perturbations of periodic geometries
as long as these perturbations are - in the statistical average -
``not to large''. However, it is a well justified question if the
estimates presented here are applicable also for other models.

In Section \ref{sec:Sample-Geometries} we discuss three standard
models from the theory of stochastic geometries. The first one is
the Boolean model based on a Poisson point process. Here we can show
that the micro- and mesoscopic assumptions are fulfilled, at least
in case $\bP$ is given as the union of balls. If we choose $\bP$
as the complement of the balls, we currently seem to run into difficulties.
However, this problem might be overcome using a Matern modification
of the Poisson process. We deal with such Matern modifications in
Section \ref{subsec:Matern-Process}. What remains challenging in
both settings are the proofs of statistical harmony or statistical
connectivity. However, if the Matern process strongly excludes points
that are to close to each other, the connectivity issue can be resolved.

A further class which will be discussed are a system of Delaunay pipes
based on a Matern process. In this case, even though the geometry
might locally become very irregular, all properties can be verified.
Hence, we identified at least one non-trivial, non-quasi-periodic
geometry to which our approach can be applied for sure.

The above mentioned construction of Sobolev spaces and the application
in the homogenization result of Theorem \ref{thm:Final-homogenization-Theorem}
clearly demonstrate the benefits of the new methodology.

\subsection*{Notes}

\subsubsection*{Structure of the article}

We close the introduction by providing an overview over the article
and its main contributions. In Section \ref{sec:Preliminaries} we
collect some basic concepts and inequalities from the theory of Sobolev
spaces, random geometries and discrete and continuous ergodic theory.
We furthermore establish local regularity properties for what we call
$\eta$-regular sets, as well as a related covering theorem in Section
\ref{subsec:Local--Regularity}. In Section \ref{subsec:Dynamical-Systems-on-Zd}
we will demonstrate that stationary ergodic random open sets induce
stationary processes on $\Zd$, a fact which is used later in the
construction of the mesoscopic Voronoi tessellation in Section \ref{subsec:Mesoscopic-Regularity}.

In Section \ref{sec:Periodic-extension-theorem} we provide a proof
of the periodic extension result in a simplified setting. This is
for completeness and self-containedness of the paper, in order to
make a comparison between stochastic and periodic approach easily
accessible to the reader.

In Section \ref{sec:Nonlocal-regularity} we introduce the regularity
concepts of this work. More precisely, in Section \ref{subsec:Microscopic-Regularity}
we introduce the concept of local $\left(\delta,M\right)$-regularity
and use the theory of Section \ref{subsec:Local--Regularity} in order
to establish a local covering result for $\partial\bP$, which will
allow us to infer most of our extension and trace results. In Section
\ref{subsec:Mesoscopic-Regularity} we show how isotropic cone mixing
geometries allow us to construct a stationary Voronoi tessellation
of $\Rd$ such that all related quantities like ``diameter'' of
the cells are stationary variables whose expectation can be expressed
in terms of the isotropic cone mixing function $f$. Moreover we prove
the important integration Lemma \ref{lem:estim-E-fa-fb}.

In Sections \ref{sec:Extension-and-Trace-d-M}--\ref{sec:Construction-of-Macroscopic-1}
we finally provide the aforementioned extension operators and prove
estimates for these extension operators and for the trace operator.

In Section \ref{sec:Sample-Geometries} we study some sample geometries
and in Section \ref{sec:Two-scale-convergence} we discuss the homogenization
problem.

\subsubsection*{A Remark on Notation}

This article uses concepts from partial differential equations, measure
theory, probability theory and random geometry. Additionally, we introduce
concepts which we believe have not been introduced before. This makes
it difficult to introduce readable self contained notation (the most
important aspect being symbols used with different meaning) and enforces
the use of various different mathematical fonts. Therefore, we provide
an index of notation at the end of this work. As a rough orientation,
the reader may keep the following in mind:

We use the standard notation $\N$, $\Q$, $\R$, $\Z$ for natural
($>0$), rational, real and integer numbers. $\P$ denotes a probability
measure, $\E$ the expectation. Furthermore, we use special notation
for some geometrical objects, i.e. $\T^{d}=[0,1)^{d}$ for the torus
($\T$ equipped with the topology of the torus), $\I^{d}=(0,1)^{d}$
the open interval as a subset of $\Rd$ (we often omit the index $d$),
$\B$ a ball, $\cone$ a cone and $\X$ a set of points. In the context
of finite sets $A$, we write $\#A$ for the number of elements.

Bold large symbols ($\bU$, $\bQ$, $\bP$,$\dots$) refer to open
subsets of $\Rd$ or to closed subsets with $\partial\bP=\partial\mathring{\bP}$.
The Greek letter $\Gamma$ refers to a $d-1$ dimensional manifold
(aside from the notion of $\Gamma$-convergence).

Calligraphic symbols ($\mathcal{A}$, $\cU$, $\dots$) usually refer
to operators and large Gothic symbols ($\fB,\fC,\dots$) indicate
topological spaces, except for $\fA$.

\section{\label{sec:Preliminaries}Preliminaries}

We first collect some notation and mathematical concepts which will
be frequently used throughout this paper. We first start with the
standard geometric objects, which will be labeled by bold letters.

\subsection{Fundamental Geometric Objects}

\textbf{Unit cube~~~} The torus $\T=[0,1)^{d}$ has 
the topology of the metric $d(x,y)=\min_{z\in\Zd}\left|x-y+z\right|$.
In contrast, the open interval $\I^{d}:=(0,1)^{d}$ is considered
as a subset of $\Rd$. We often omit the index $d$ if this does not
provoke confusion.\nomenclature[T]{$\T$}{$\T=[0,1)^d$ the torus (Section \ref{sec:Preliminaries})}\nomenclature[I]{$\I$}{$\I=[0,1)^d$ the torus (Section \ref{sec:Preliminaries})}

\textbf{Balls~~~} Given a metric space $\left(M,d\right)$ we denote
$\Ball rx$ \nomenclature[Ball]{$\Ball{r}{x}$}{Ball around $x$ with radius $r$ (Section \ref{sec:Preliminaries})}the
open ball around $x\in M$ with radius $r>0$. The surface of the
unit ball in $\Rd$ is $\S^{d-1}$.

\textbf{Points~~~} A sequence of points will be labeled by $\X:=\left(x_{i}\right)_{i\in\N}$.\nomenclature[X]{$\X$, $Y$}{Families of points (Section \ref{sec:Preliminaries})}

\textbf{A cone~~~} in $\Rd$ is usually labeled by $\cone$. In
particular, we define for a vector $\nu$ of unit length, $0<\alpha<\frac{\pi}{2}$
and $R>0$ the \nomenclature[Cone]{$\cone_{\nu,\alpha,R}(x)$}{Cone with apix $x$, direction $\nu$, opening angle $\alpha$ and hight $R$  (Section \ref{sec:Preliminaries})}cone
\[
\cone_{\nu,\alpha,R}(x):=\left\{ z\in\Ball Rx\,:\;z\cdot\nu>\left|z\right|\cos\alpha\right\} \quad\text{and}\quad\cone_{\nu,\alpha}(x):=\cone_{\nu,\alpha,\infty}(x)\,.
\]
\textbf{Inner and outer hull~~~} We use balls of radius $r>0$
to define for a closed set $\bP\subset\Rd$ the sets \nomenclature[P]{$\bP_{r},\bP_{-r}$}{Inner and outer hull of $\bP$ with hight $r$ (Section \ref{sec:Preliminaries})}
\begin{equation}
\begin{aligned}\bP_{r} & :=\overline{\Ball r{\bP}}:=\left\{ x\in\Rd\,:\;\dist\left(x,\bP\right)\leq r\right\} \,,\\
\bP_{-r} & :=\Rd\backslash\left[\Ball r{\Rd\setminus\bP}\right]:=\left\{ x\in\Rd\,:\;\dist\left(x,\Rd\setminus\bP\right)\geq r\right\} \,.
\end{aligned}
\label{eq:Pr}
\end{equation}
One can consider these sets as inner and outer hulls of $\bP$. The
last definition resembles a concept of ``negative distance'' of
$x\in\bP$ to $\partial\bP$ and ``positive distance'' of $x\not\in\bP$
to $\partial\bP$. For $A\subset\Rd$ we denote $\conv(A)$ \nomenclature[convex]{$\conv A$}{Convex hull of $A$ (Section \ref{sec:Preliminaries})}the
closed convex hull of $A$. 

The natural geometric measures we use in this work are the Lebesgue
measure on $\Rd$, written $\left|A\right|$ for $A\subset\Rd$, and
the $k$-dimensional Hausdorff measure, denoted by $\cH^{k}$ on $k$-dimensional
submanifolds of $\Rd$ (for $k\leq d$).

\subsection{Local Extensions and Traces}

Let $\bP\subset\Rd$ be an open set and let $p\in\partial\bP$ and
$\delta>0$ be a constant such that $\Ball{\delta}p\cap\partial\bP$
is graph of a Lipschitz function. We denote \nomenclature[Meta]{$M(p,\delta)$}{Lemma \eqref{eq:M(delta-p)}}
\begin{align}
M(p,\delta) & :=\inf\left\{ M:\,\,\exists\phi:U\subset\R^{d-1}\to\R\text{ }\right.\nonumber \\
 & \phantom{:=\inf}\left.\phi\,\text{Lipschitz, with constant }M\text{ s.t. }\Ball{\delta}p\cap\partial\bP\text{ is graph of }\phi\right\} \,.\label{eq:M(delta-p)}
\end{align}

\begin{rem}
For every $p$, the function $M(p,\cdot)$ is monotone increasing
in $\delta$.
\end{rem}

In the following, we formulate some extension and trace results. Although
it is well known how such results are proved and the proofs are standard,
we include them for completeness.
\begin{lem}[Uniform Extension for Balls]
\label{lem:uniform-extension-lemma} Let $\bP\subset\Rd$ be an open
set, $0\in\partial\bP$ and assume there exists $\delta>0$, $M>0$
and an open domain $U\subset\Ball{\delta}0\subset\R^{d-1}$ such that
$\partial\bP\cap\Ball{\delta}0$ is graph of a Lipschitz function
$\varphi:\,U\subset\R^{d-1}\to\Rd$ of the form $\varphi(\tilde{x})=(\tilde{x},\phi(\tilde{x}))$
in $\Ball{\delta}0$ with Lipschitz constant $M$ and $\varphi(0)=0$.
Writing $x=(\tilde{x},x_{d})$ and defining $\rho=\delta\sqrt{4M^{2}+2}^{-1}$
there exist an extension operator\nomenclature[Uc]{$\cU$}{local and global extension operators (Lemma \ref{lem:uniform-extension-lemma})}
\begin{equation}
\left(\cU u\right)(x)=\begin{cases}
u(x) & \mbox{ if }x_{d}<\phi(\tilde{x})\\
4u\left(\tilde{x},{-}\frac{x_{d}}{2}+\frac{3}{2}\phi(\tilde{x})\right)-3u\left(\tilde{x},{-}x_{d}+2\phi(\tilde{x})\right) & \mbox{ if }x_{d}>\phi(\tilde{x})
\end{cases}\,,\label{eq:lem:def-cU}
\end{equation}
such that for\nomenclature[Ac]{$\cA\left(0,\bP,\rho\right)$}{$\left\{ \left(\tilde{x},{-}x_{d}+2\phi(\tilde{x})\right)\,:\;\left(\tilde{x},x_{d}\right)\in\Ball{\rho}0\backslash\bP\right\} $ (Lemma \ref{lem:uniform-extension-lemma})}
\begin{equation}
\cA\left(0,\bP,\rho\right):=\left\{ \left(\tilde{x},{-}x_{d}+2\phi(\tilde{x})\right)\,:\;\left(\tilde{x},x_{d}\right)\in\Ball{\rho}0\backslash\bP\right\} \,,\label{eq:lem:uniform-extension-lemma-A-rho}
\end{equation}
and for every $p\in[1,\infty]$ the operator 
\[
\cU:\,W^{1,p}\of{\cA\left(0,\bP,\rho\right)}\to W^{1,p}(\Ball{\rho}0)\,,
\]
is continuous with 
\begin{equation}
\left\Vert \cU u\right\Vert _{L^{p}(\Ball{\rho}0\backslash\bP)}\leq7\left\Vert u\right\Vert _{L^{p}\left(\cA\left(0,\bP,\rho\right)\right)}\,,\qquad\left\Vert \nabla\cU u\right\Vert _{L^{p}(\Ball{\rho}0\backslash\bP)}\leq14M\left\Vert \nabla u\right\Vert _{L^{p}\left(\cA\left(0,\bP,\rho\right)\right)}\,.\label{eq:lem:uniform-extension-lemma-estim}
\end{equation}
\end{lem}

\begin{rem}
\label{rem:Extension}It is well known (\cite[chapter 5]{Evans2010})
that for every bounded domain $\bU\subset\Rd$ with $C^{0,1}$-boundary
there exists a continuous extension operator $\cU:\,W^{1,p}(\bU)\to W^{1,p}(\Rd)$.

\begin{comment}
\begin{rem}
The proof of Lemma \ref{lem:uniform-extension-lemma} also unveils
the following result
\end{rem}

\begin{cor}
\label{cor:uniform-extension-corollary}Let $\bP\subset\Rd$ be an
open set, $0\in\partial\bP$ and assume there exists $\delta>0$,
$M>0$ with $\rho=\delta\sqrt{4M^{2}+2}^{-1}$ and an open domain
$U\subset\Ball{\rho}0\subset\R^{d-1}$ such that $\partial\bP\cap\Ball{\rho}0$
is graph of a Lipschitz function $\varphi:\,U\subset\R^{d-1}\to\Rd$
of the form $\varphi(\tilde{x})=(\tilde{x},\phi(\tilde{x}))$ in $\Ball{\rho}0$
with Lipschitz constant $M$ and $\varphi(0)=0$. Writing $x=(\tilde{x},x_{d})$
assume that $x\in\Ball{\delta}0$ satisfies $x\in\bP$ iff $x_{d}<\phi\of{\tilde{x}}$.
Then $\cU$ given by (\ref{eq:lem:def-cU}) satisfies
\[
\cU:\,W^{1,p}\of{\Ball{\delta}0\cap\bP}\to W^{1,p}\of{\Ball{\rho}0\backslash\bP}\,.
\]
\end{cor}

\end{comment}
\end{rem}

\begin{proof}[Proof of Lemma \ref{lem:uniform-extension-lemma}]
The extended function $\varphi:\,U\times\R\to U\times\R$, $\varphi(x)=(\tilde{x},\phi(\tilde{x})+x_{d})$
is bijective with $\varphi^{-1}(x)=\left(\tilde{x},x_{d}-\phi(\tilde{x})\right)$.
In particular, both $\varphi$ and $\varphi^{-1}$ are Lipschitz continuous
with Lipschitz constant $M+1$.

W.l.o.g. we assume that 
\[
\varphi\left(U\times({-}\infty,0)\right)\cap\Ball{\delta}0=\bP\cap\Ball{\delta}0\cap\left(U\times\R\right)
\]
 implying $\varphi\left(U\times(0,\infty)\right)\cap\bP=\emptyset$.

\emph{Step 1}: We consider the extension operator $\cU_{+}:\,W^{1,p}(\R^{d-1}\times({-}\infty,0))\to W^{1,p}(\Rd)$
having the form \cite[chapter 5]{Evans2010}, \cite{adams2003sobolev}
\[
\left(\cU_{+}u\right)(x)=\begin{cases}
u(x) & \mbox{ if }x_{d}<0\\
4u\left(\tilde{x},{-}\frac{x_{d}}{2}\right)-3u\left(\tilde{x},{-}x_{d}\right) & \mbox{ if }x_{d}>0
\end{cases}\,.
\]
We make use of this operator and define 
\[
\cU u(x):=\left(\cU_{+}\left(u\circ\varphi\right)\right)\circ\varphi^{-1}(x)\,.
\]
Note that all three operators $u\mapsto u\circ\varphi$, $\cU_{+}$
and $v\mapsto v\circ\varphi^{-1}$ map $W^{1,p}$-functions to $W^{1,p}$-functions.
By the definition of $\cU_{+}$ we may explicitly calculate (\ref{eq:lem:def-cU}).
In particular, $\cU u(x)$ is well defined for $x\in\Ball{\delta}0\backslash\bP$
whenever 
\begin{equation}
\left(\tilde{x},{-}x_{d}+2\phi(\tilde{x})\right)\in\Ball{\delta}0\,.\label{eq:lem:uniform-extension-lemma-help-1}
\end{equation}
\emph{Step 2}: We seek for $\rho>0$ such that (\ref{eq:lem:uniform-extension-lemma-help-1})
is satisfied for every $x\in\Ball{\rho}0\backslash\bP$ and such that
$\cA\left(0,\bP,\rho\right)\subset\Ball{\delta}0$. For $\rho<\delta$
and $x=(\tilde{x},x_{d})\in\Ball{\rho}0$, we find with $\varphi(0)=0$
and $\left|x_{d}\right|\leq\sqrt{\rho^{2}-|\tilde{x}|^{2}}$ that
\begin{align*}
-x_{d}+2\phi(\tilde{x}) & \in\left(x_{d}-2M|\tilde{x}|\,,\,x_{d}+2M|\tilde{x}|\right)\\
 & \subset\left({-}\sqrt{\rho^{2}-|\tilde{x}|^{2}}-2M|\tilde{x}|\,,\,\sqrt{\rho^{2}-|\tilde{x}|^{2}}+2M|\tilde{x}|\right)\,.
\end{align*}
In particular, 
\[
\max_{(\tilde{x},x_{d})\in\Ball{\rho}0\backslash\bP}\left|{-}x_{d}+2\phi(\tilde{x})\right|\leq\rho\sqrt{4M^{2}+1}
\]
and (\ref{eq:lem:uniform-extension-lemma-help-1}) holds if 
\[
\left|-x_{d}+2\phi(\tilde{x})\right|^{2}+\left|\tilde{x}\right|^{2}\leq\rho^{2}\left(4M^{2}+1\right)+\rho^{2}\leq\delta^{2}\,.
\]
Hence we require $\rho=\delta\sqrt{4M^{2}+2}^{-1}$. It is now easy
to verify (\ref{eq:lem:uniform-extension-lemma-estim}) from the definition
of $\cU$ and the chain rule.
\end{proof}
\begin{lem}
\label{lem:basic-trace}Let $\bP\subset\Rd$ be an open set, $0\in\partial\bP$
and assume there exists $\delta>0$, $M>0$ and an open domain $U\subset\Ball{\delta}0\subset\R^{d-1}$
such that $\partial\bP\cap\Ball{\delta}0$ is graph of a Lipschitz
function $\varphi:\,U\subset\R^{d-1}\to\Rd$ of the form $\varphi(\tilde{x})=(\tilde{x},\phi(\tilde{x}))$
in $\Ball{\delta}0$ with Lipschitz constant $M$ and $\varphi(0)=0$.
Writing $x=(\tilde{x},x_{d})$ we consider the trace operator $\cT:\,C^{1}\left(\bP\cap\Ball{2\delta}0\right)\to C\left(\partial\bP\cap\Ball{\delta}0\right)$.
For every $p\in[1,\infty]$ and every $r<\frac{p\left(1-d\right)}{\left(p-d\right)}$
the operator $\cT$ can be continuously extended to 
\[
\cT:\,W^{1,p}\left(\bP\cap\Ball{2\delta}0\right)\to L^{r}(\partial\bP\cap\Ball{\delta}0)\,,
\]
such that 
\begin{equation}
\norm{\cT u}_{L^{r}(\partial\bP\cap\Ball{\delta}0)}\leq C_{r,p}\delta^{\frac{d\left(p-r\right)}{rp}-\frac{1}{r}}\sqrt{4M^{2}+2}^{\frac{1}{r}+1}\norm u_{W^{1,p}\left(\bP\cap\Ball{2\delta}0\right)}\,.\label{eq:lem:basic-trace}
\end{equation}
\end{lem}

\begin{proof}
We proceed similar to the proof of Lemma \ref{lem:uniform-extension-lemma}.

Step 1: Writing $B_{\delta}=\Ball{\delta}0$ together with $B_{\delta}^{-}=\left\{ x\in B_{\delta}:\,x_{d}<0\right\} $
and $\Sigma_{\delta}:=\left\{ x\in B_{\delta}:\,x_{d}=0\right\} $
we recall the standard estimate
\[
\left(\int_{\Sigma_{1}}\left|u\right|^{r}\right)^{\frac{1}{r}}\leq C_{r,p}\left(\left(\int_{B_{1}^{-}}\left|\nabla u\right|^{p}\right)^{\frac{1}{p}}+\left(\int_{B_{1}^{-}}\left|u\right|^{p}\right)^{\frac{1}{p}}\right)\,,
\]
which leads to
\[
\left(\int_{\Sigma_{\delta}}\left|u\right|^{r}\right)^{\frac{1}{r}}\leq C_{r,p}\delta^{\frac{d\left(p-r\right)}{rp}-\frac{1}{r}}\left(\left(\int_{B_{\delta}^{-}}\left|\nabla u\right|^{p}\right)^{\frac{1}{p}}+\left(\int_{B_{\delta}^{-}}\left|u\right|^{p}\right)^{\frac{1}{p}}\right)\,.
\]
Step 2: Using the transformation rule and the fact that $1\leq\left|\det D\varphi\right|\leq\sqrt{4M^{2}+2}$
we infer (\ref{eq:lem:basic-trace}) similar to Step 2 in the proof
of Lemma \ref{lem:uniform-extension-lemma}.
\begin{align*}
\left(\int_{\partial\bP\cap\Ball{\delta}0}\left|u\right|^{r}\right)^{\frac{1}{r}} & \leq\sqrt{4M^{2}+2}^{\frac{1}{r}}\left(\int_{\Sigma_{\delta}}\left|u\circ\varphi\right|^{r}\right)^{\frac{1}{r}}\\
 & \leq C_{r,p}\delta^{\frac{d\left(p-r\right)}{rp}-\frac{1}{r}}\sqrt{4M^{2}+2}^{\frac{1}{r}}\left(\left(\int_{B_{\delta}^{-}}\left|\nabla\left(u\circ\varphi\right)\right|^{p}\right)^{\frac{1}{p}}+\left(\int_{B_{\delta}^{-}}\left|u\circ\varphi\right|^{p}\right)^{\frac{1}{p}}\right)\\
 & \leq C_{r,p}\delta^{\frac{d\left(p-r\right)}{rp}-\frac{1}{r}}\sqrt{4M^{2}+2}^{\frac{1}{r}+1}\,\cdot\\
 & \qquad\cdot\left(\left(\int_{B_{\delta}^{-}}\left|\left(\nabla u\right)\circ\varphi\right|^{p}\det D\varphi\right)^{\frac{1}{p}}+\left(\int_{B_{\delta}^{-}}\left|u\circ\varphi\right|^{p}\det D\varphi\right)^{\frac{1}{p}}\right)
\end{align*}
and from this we conclude the Lemma.
\end{proof}

\subsection{Poincar\'e Inequalities}

We denote 
\[
W_{(0),r}^{1,p}(\Ball r0):=\left\{ u\in W^{1,p}(\Ball r0)\,:\;\exists x:\,B_{r}(x)\subset\Ball r0\;\vee\;\fint_{B_{r}(x)}u=0\right\} \,.
\]
Note that this is not a linear vector space.
\begin{lem}
\label{lem:Poincare-ball}For every $p\in(1,\infty)$ there exists
$C_{p}>0$ such that the following holds: Let $r<1$ and $x\in\Ball 10$
such that $B_{r}(x)\subset\Ball 10$ then for every $u\in W^{1,p}(\Ball 10)$
\begin{equation}
\left\Vert u\right\Vert _{L^{p}(\Ball 10)}^{p}\leq C_{p}\left(\left\Vert \nabla u\right\Vert _{L^{p}(\Ball 10)}^{p}+\frac{1}{r^{d}}\left\Vert u\right\Vert _{L^{p}(B_{r}(x))}^{p}\right)\,,\label{eq:lem:Poincare-ball-1}
\end{equation}
and for every $u\in W_{(0),r}^{1,p}(\Ball 10)$ it holds 
\begin{equation}
\left\Vert u\right\Vert _{L^{p}(\Ball 10)}^{p}\leq C_{p}\left(1+r^{p-d}\right)\left\Vert \nabla u\right\Vert _{L^{p}(\Ball 10)}^{p}\,.\label{eq:lem:Poincare-ball-2}
\end{equation}
\end{lem}

\begin{rem*}
In case $p\geq d$ we find that (\ref{eq:lem:Poincare-ball-2}) holds
iff $u(x)=0$ for some $x\in\Ball 10$.
\end{rem*}
\begin{proof}
In a first step, we assume $x=0$. The underlying idea of the proof
is to compare every $u(y)$, $y\in\Ball 10\backslash\Ball r0$ with
$u(rx)$. In particular, we obtain for $y\in\Ball 10\backslash\Ball r0$
that 
\[
u(y)=u(ry)+\int_{0}^{1}\nabla u\of{ry+t(1-r)y}\cdot(1-r)y\,\d t
\]
and hence by Jensen's inequality 
\[
\left|u(y)\right|^{p}\leq C\left(\int_{0}^{1}\left|\nabla u\of{ry+t(1-r)y}\right|^{p}(1-r)^{p}\left|y\right|^{p}\,\d t+\left|u(ry)\right|^{p}\right)\,.
\]
We integrate the last expression over $\Ball 10\backslash\Ball r0$
and find 
\begin{align*}
\int_{\Ball 10\backslash\Ball r0}\left|u(y)\right|^{p}\d y & \leq\int_{S^{d-1}}\int_{r}^{1}C\left(\int_{0}^{1}\left|\nabla u\of{rs\nu+t(1-r)s\nu}\right|^{p}(1-r)^{p}s^{p}\,\d t\right)s^{d-1}\d s\d\nu\\
 & \quad+\int_{\Ball 10\backslash\Ball r0}\left|u(ry)\right|^{p}\d y\\
 & \leq\int_{S^{d-1}}\int_{r}^{1}C\left(\int_{rs}^{s}\left|\nabla u\of{t\nu}\right|^{p}(1-r)^{p-1}s^{p-1}\,\d t\right)s^{d-1}\d s\\
 & \quad+\int_{\Ball 10\backslash\Ball r0}\left|u(ry)\right|^{p}\d y\\
 & \leq C\left\Vert \nabla u\right\Vert _{L^{p}(\Ball 10)}^{p}+\frac{1}{r^{d}}\left\Vert u\right\Vert _{L^{p}(\Ball r0)}^{p}\,.
\end{align*}
For general $x\in\Ball 10$, use the extension operator $\cU:\,W^{1,p}(\Ball 10)\to W^{1,p}(B_{4}(0))$
(see Remark \ref{rem:Extension}) such that $\left\Vert \cU u\right\Vert _{W^{1,p}(B_{4}(0))}\leq C\norm u_{W^{1,p}(\Ball 10)}$
and $\left\Vert \nabla\cU u\right\Vert _{W^{1,p}(B_{4}(0))}\leq C\norm{\nabla u}_{W^{1,p}(\Ball 10)}$.
Since $\Ball 10\subset B_{2}(x)\subset B_{4}(0)$ we infer 
\[
\left\Vert u\right\Vert _{L^{p}(\Ball 10)}^{p}\leq\left\Vert \cU u\right\Vert _{L^{p}(B_{2}(x))}^{p}\leq C\left(\left\Vert \nabla\cU u\right\Vert _{L^{p}(B_{2}(x))}^{p}+\frac{1}{r^{d}}\left\Vert \cU u\right\Vert _{L^{p}(B_{r}(x))}^{p}\right)\,.
\]
and hence (\ref{eq:lem:Poincare-ball-1}). Furthermore, since there
holds $\left\Vert u\right\Vert _{L^{p}(\Ball 10)}^{p}\leq C\left\Vert \nabla u\right\Vert _{L^{p}(\Ball 10)}^{p}$
for every $u\in W_{(0)}^{1,p}(\Ball 10)$, a scaling argument shows
$\left\Vert u\right\Vert _{L^{p}(\Ball r0)}^{p}\leq Cr^{p}\left\Vert \nabla u\right\Vert _{L^{p}(\Ball r0)}^{p}$
for every \\ 	$u\in W_{(0),r}^{1,p}(\Ball 10)$ and hence (\ref{eq:lem:Poincare-ball-2}).
\end{proof}
\begin{lem}
\label{lem:scaled-poincare}Let $0<r<R<\infty$ and $p\in(1,\infty)$
and $q\leq pd/(d-p)$ (if $p<d$) or $q=\infty$ (if $p\geq d$).
Then there exists $C_{p,q}$ such that for every convex set $\bP$
with polytope boundary $\partial\bP\subset\Ball R0\backslash\overline{\Ball r0}$
\begin{equation}
\norm u_{L^{q}\left(\bP\right)}^{p}\leq C_{p,q}R^{-d\left(1-\frac{p}{q}\right)}\left(\int_{\bP}\left(R^{p}\left(\frac{R}{r}\right)^{p+1}\left|\nabla u\right|^{p}+\frac{R^{d+1}}{r^{d+1}}\left|u\right|^{p}\right)\right)\,,\label{eq:rem-lem:Poincare-polytope-1}
\end{equation}
and for every $u\in W_{\left(0\right),r}^{1,p}\left(\Ball R0\right)$
\begin{equation}
\left\Vert u\right\Vert _{L^{q}(\Ball R0)}^{p}\leq\fC_{p,q}(R,r)\left\Vert \nabla u\right\Vert _{L^{p}(\Ball R0)}^{p}\,,\label{eq:rem-lem:Poincare-polytope-2}
\end{equation}
where 
\begin{equation}
\fC_{p,q}(R,r):=C_{p,q}R^{-d\left(1-\frac{p}{q}\right)+p}\left(\left(\frac{R}{r}\right)^{p+1}+\left(\frac{R}{r}\right)^{d+1}\right)\label{eq:Poincare-scaled-constant}
\end{equation}
\end{lem}

\begin{rem}
For the critical Sobolev index $q=\frac{pd}{d-p}$ we infer $d\left(1-\frac{p}{q}\right)=p$.
\end{rem}

\begin{proof}
First note that by a simple scaling argument based on the integral
transformation rule the equations (\ref{eq:lem:Poincare-ball-1})
yields for every $u\in W^{1,p}(\Ball r0)$ 
\begin{equation}
\left\Vert u\right\Vert _{L^{q}(\Ball R0)}^{p}\leq C_{p,q}R^{-d\left(1-\frac{p}{q}\right)}\left(R^{p}\left\Vert \nabla u\right\Vert _{L^{p}(\Ball R0)}^{p}+\frac{R^{d}}{r^{d}}\left\Vert u\right\Vert _{L^{p}(\Ball r0)}^{p}\right)\label{eq:rem-lem:Poincare-ball-1}
\end{equation}
and (\ref{eq:lem:Poincare-ball-2}) yields for every $u\in W_{(0),r}^{1,p}(\Ball r0)$
\begin{equation}
\left\Vert u\right\Vert _{L^{q}(\Ball R0)}^{p}\leq C_{p,q}R^{p}R^{-d\left(1-\frac{p}{q}\right)}\left(1+\left(\frac{r}{R}\right)^{p-d}\right)\left\Vert \nabla u\right\Vert _{L^{p}(\Ball R0)}^{p}\,.\label{eq:rem-lem:Poincare-ball-2}
\end{equation}
Now, for $\nu\in\S^{d-1}$ we denote $P\left(\nu\right)$ as the unique
$p\in\partial\bP\cap(0,\infty)\nu$ and for $x\in\Rd\backslash\{0\}$
we denote $\nu_{x}:=\frac{x}{\norm x}$ and consider the bijective
Lipschitz map
\[
\varphi_{P}:\,\bP\to\Ball r0\,,\qquad x\mapsto R\frac{x}{\norm{P(\nu_{x})}}\,.
\]
 Then we infer from (\ref{eq:rem-lem:Poincare-ball-1}) 
\begin{align*}
\left\Vert u\circ\tilde{\varphi}_{P}^{-1}\right\Vert _{L^{q}(\Ball R0)}^{p} & \leq CR^{-d\left(1-\frac{p}{q}\right)}\left(R^{p}\left\Vert \nabla\left(u\circ\tilde{\varphi}_{P}^{-1}\right)\right\Vert _{L^{p}(\Ball R0)}^{p}+\frac{R^{d}}{r^{d}}\left\Vert u\circ\tilde{\varphi}_{P}^{-1}\right\Vert _{L^{p}(\Ball r0)}^{p}\right)
\end{align*}
or, after transformation of integrals, 
\begin{multline*}
\left(\int_{\bP}\left|u\right|^{q}\left|\det\rmD\tilde{\varphi}_{P}\right|\right)^{\frac{p}{q}}\\
\leq CR^{-d\left(1-\frac{p}{q}\right)}\left(\int_{\bP}\left(R^{p}\left|\left(\nabla u\right)\left(\rmD\tilde{\varphi}_{P}\right)^{-1}\right|^{p}+\frac{R^{d}}{r^{d}}\chi_{\tilde{\varphi}_{P}^{-1}\Ball r0}\left|u\right|^{p}\right)\left|\det\rmD\tilde{\varphi}_{P}\right|\right)\,.
\end{multline*}
It remains to estimate the derivatives of $\varphi_{P}$. In polar
coordinates, the radial derivative is $\partial_{r}\varphi_{P}(x)=\frac{R}{\norm{P(\nu_{x})}}$,
while the tangential derivative is more complicated to calculate.
However, in case $\nu\bot\mathrm{T}_{P(\nu)}$ we obtain $\partial_{\S^{d-1}}\varphi_{P}(x)=\I_{\R^{d-1}}$,
which is by the same time the minimal absolute value for each tangential
derivative, and $\partial_{\S^{d-1}}\varphi_{P}(x)$ becomes maximal
in edges where $2\tan\alpha=r^{-1}\sqrt{R^{2}-r^{2}}$ and $\left\Vert \partial\varphi_{P}\right\Vert (x_{0})=\left\Vert \frac{R}{\left\Vert x_{0}\right\Vert }\id-\frac{Rx_{0}}{\left\Vert x_{0}\right\Vert ^{3}}\otimes x\right\Vert \leq2\frac{R}{r}$
(see Figure ...... ).Now we make use of the fact that $\tilde{\varphi}_{P}$
increases the volume locally with a rate smaller than $\left\Vert \partial\varphi_{P}\right\Vert $and
hence $\left|\det\rmD\tilde{\varphi}_{P}\right|\geq1$. On the other
hand, we have $\left|\left(\rmD\tilde{\varphi}_{P}\right)^{-1}\right|<\frac{R}{r}$
and hence (\ref{eq:rem-lem:Poincare-polytope-1}). In a similar way
we infer (\ref{eq:rem-lem:Poincare-polytope-2}) from (\ref{eq:rem-lem:Poincare-ball-2}).
\end{proof}

\subsection{Voronoi Tessellations and Delaunay Triangulation}
\begin{defn}[Voronoi Tessellation]
\label{def:Voronoi}Let $\X=\left(x_{i}\right)_{i\in\N}$ be a sequence
of points in $\Rd$ with $x_{i}\neq x_{k}$ if $i\neq k$. For each
$x\in\X$ let \nomenclature[G]{$G(x)$}{Voronoi cell with center $x$ (Definition \ref{def:Voronoi})}
\[
G\of x:=\left\{ y\in\Rd\,:\;\forall\tilde{x}\in\X\backslash\left\{ x\right\} :\,\left|x-y\right|<\left|\tilde{x}-y\right|\right\} \,.
\]
Then $\left(G\of{x_{i}}\right)_{i\in\N}$ is called the \emph{Voronoi
tessellation} of $\Rd$ w.r.t. $\X$. For each $x\in\X$ we define
$d\of x:=\diam G\of x$.
\end{defn}

We will need the following result on Voronoi tessellation of a minimal
diameter.
\begin{lem}
\label{lem:estim-diam-Voronoi-cells}Let $\fr>0$ and let $\X=\left(x_{i}\right)_{i\in\N}$
be a sequence of points in $\Rd$ with $\left|x_{i}-x_{k}\right|>2\fr$
if $i\neq k$. For $x\in\X$ let $\cI\of x:=\left\{ y\in\X\,:\;G\of y\cap\Ball{\fr}{G\of x}\not=\emptyset\right\} $.
Then 
\begin{equation}
\#\cI\of x\leq\left(\frac{4d\of x}{\fr}\right)^{d}\,.\label{eq:lem:estim-diam-Voronoi-cells}
\end{equation}
\end{lem}

\begin{proof}
Let $\X_{k}=\left\{ x_{j}\in\X\,:\;\cH^{d-1}\of{\partial G_{k}\cap\partial G_{j}}\geq0\right\} $
the neighbors of $x_{k}$ and $d_{k}:=d\of{x_{k}}$. Then all $x_{j}\in\X$
satisfy $\left|x_{k}-x_{j}\right|\leq2d_{k}$. Moreover, every $\tilde{x}\in\X$
with $\left|\text{\ensuremath{\tilde{x}}}-x_{k}\right|>4d_{k}$ has
the property that $\dist\of{\,\partial G\left(\tilde{x}\right),\,x_{k}\,}>2d_{k}>d_{k}+\fr$
and $\tilde{x}\not\in\cI_{k}$. Since every Voronoi cell contains
a ball of radius $\fr$, this implies that $\#\cI_{k}\leq\left|\Ball{4d_{k}}{x_{k}}\right|/\left|\Ball{\fr}0\right|=\left(\frac{4d_{k}}{\fr}\right)^{d}$.
\end{proof}
\begin{defn}[Delaunay Triangulation]
\label{def:delaunay}Let $\X=\left(x_{i}\right)_{i\in\N}$ be a sequence
of points in $\Rd$ with $x_{i}\neq x_{k}$ if $i\neq k$. The Delaunay
triangulation is the dual graph of the Voronoi tessellation, i.e.
we say $\D(\X):=\left\{ (x,y):\;\hausdorffH^{d-1}\of{\partial G(x)\cap\partial G(y)}\neq0\right\} $.
\end{defn}

\subsection{\label{subsec:Local--Regularity}Local $\eta$-Regularity}

\begin{figure}
 \begin{minipage}[c]{0.5\textwidth} \includegraphics{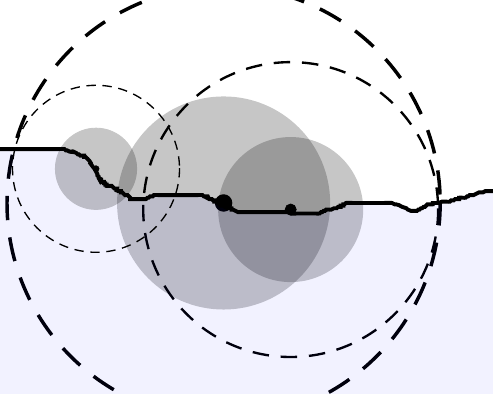}\end{minipage}\hfill   \begin{minipage}[c]{0.45\textwidth}\caption{An illustration of $\eta$-regularity. In Theorem \ref{thm:delta-M-rho-covering}
we will rely on a ``gray'' region like in this picture.}
\end{minipage}
\end{figure}

We say that a function $F:A\to\left\{ 0,1\right\} $ holds ``true''
in $a\in A$ if $F(a)=1$ and ``false'' if $F(a)=0$.
\begin{defn}[$\eta$-regularity]
\label{def:eta-regular}\nomenclature[G]{$\eta$-regular (local)}{(Definition \ref{def:eta-regular})}A
set $\bP\subset\Rd$ is called locally $\eta$-regular with $f:\,\bP\times(0,r]\to\left\{ 0,1\right\} $
and $\fr>0$ if $f(p,\cdot)$ is decreasing and
\begin{equation}
f(p,\eta)=1\quad\Rightarrow\quad\forall\,\eps\in\of{0,\frac{1}{2}}\,,\,\tilde{p}\in\Ball{\eps\eta}p\cap\bP\,,\,\tilde{\eta}\in\left(0,\,(1-\eps)\eta\right)\,:\;f\of{\tilde{p},\tilde{\eta}}=1\,.\label{eq:def:eta-regular}
\end{equation}
For $p\in\bP$ we write $\eta(p):=\sup\left\{ \eta\in(0,\fr)\,:\,f(p,\eta)=1\right\} $.
\end{defn}

\begin{lem}
\label{lem:eta-lipschitz}Let $\bP$ be a locally $\eta$-regular
set with $f$ and $\fr$ and $\eta(p)$. Then $\eta:\,\bP\to\R$ is
locally Lipschitz continuous with Lipschitz constant $4$ and for
every $\eps\in\left(0,\frac{1}{2}\right)$ and $\tilde{p}\in\Ball{\eps\eta}p\cap\bP$
it holds
\begin{equation}
\frac{1-\eps}{1-2\eps}\eta\of p>\eta\of{\tilde{p}}>\eta\of p-\left|p-\tilde{p}\right|>\left(1-\eps\right)\eta\of p\,.\label{eq:eta-lipschitz-ineq-chain}
\end{equation}
Furthermore,
\begin{equation}
\left|p-\tilde{p}\right|\leq\eps\max\left\{ \eta\of p,\eta\of{\tilde{p}}\right\} \quad\Rightarrow\quad\left|p-\tilde{p}\right|\leq\frac{\eps}{1-\eps}\min\left\{ \eta\of p,\eta\of{\tilde{p}}\right\} \label{eq:lem:eta-lipschitz-dist-estim}
\end{equation}
\end{lem}

\begin{proof}
We infer from (\ref{eq:def:eta-regular}) for every $\eps\in\left(0,\frac{1}{2}\right)$
and $\tilde{p}$ such that $\left|\tilde{p}-p\right|<\eps\eta(p)$
let $\tilde{\eta}<\eta(p)$ such that also $\left|\tilde{p}-p\right|<\eps\tilde{\eta}$.
It then holds $f\of{\tilde{p},\left(1-\eps\right)\tilde{\eta}}=1$
and hence $\eta\of{\tilde{p}}\geq\left(1-\eps\right)\tilde{\eta}$.
Taking the supremum over $\sup\left\{ \tilde{\eta}:\,\tilde{\eta}<\eta(p)\right\} $
we find $\eta\of{\tilde{p}}\geq\left(1-\eps\right)\eta(p)$ i.e. 
\begin{align*}
\eta\of{\tilde{p}} & \geq\sup_{\hat{p}}\left\{ \left(1-\eps\right)\eta\of{\hat{p}}\,:\;\left|\tilde{p}-\hat{p}\right|<\eps\eta\of{\hat{p}}\right\} \\
 & \geq\eta\of p-\left|p-\tilde{p}\right|>\left(1-\eps\right)\eta\of p
\end{align*}
 which implies $\left|\tilde{p}-p\right|<\frac{\eps}{1-\eps}\eta\of{\tilde{p}}$.
This in turn leads to $\eta\of p>\left(1-\frac{\eps}{1-\eps}\right)\eta\of{\tilde{p}}$
or 
\[
\eta\of p=\frac{1-\eps}{1-\eps}\eta\of p<\frac{1}{1-\eps}\left(\eta\of p-\left|p-\tilde{p}\right|\right)<\frac{1}{1-\eps}\eta\of{\tilde{p}}\leq\frac{1}{1-2\eps}\eta\of p\,,
\]
implying (\ref{eq:eta-lipschitz-ineq-chain}) and continuity of $\eta$.

Let $\left|p-\tilde{p}\right|=\eps\eta\of p\leq2\eps\eta{\left(\tilde{p}\right)}$,
the last inequality particularly implies also $\eta{\left(p\right)}\geq\left(1-2\eps\right)\eta{\left(\tilde{p}\right)}$.
Together with $\left|p-\tilde{p}\right|\leq2\eps\eta{\left(\tilde{p}\right)}\leq4\eps\eta{\left(p\right)}=4\left|p-\tilde{p}\right|$
we have 
\[
4\left|p-\tilde{p}\right|\geq2\eps\eta{\left(\tilde{p}\right)}\geq\eta{\left(\tilde{p}\right)}-\eta{\left(p\right)}\geq-\eps\eta{\left(p\right)}=-\left|p-\tilde{p}\right|\,.
\]
Finally, in order to prove (\ref{eq:lem:eta-lipschitz-dist-estim}),
w.l.o.g. let $\eta\of{\tilde{p}}\leq\eta\of p$. Then
\[
\left|p-\tilde{p}\right|\leq\eps\eta\of p\leq\frac{\eps}{1-\eps}\eta\of{\tilde{p}}\,.
\]
\end{proof}
We make use of the latter Lemmas in order to prove the following covering-regularity
of $\partial\bP$.
\begin{thm}
\label{thm:delta-M-rho-covering}Let $\Gamma\subset\Rd$ be a closed
set and let $\eta\of{\cdot}\in C\of{\Gamma}$ be bounded and satisfy
for every $\eps\in\left(0,\frac{1}{2}\right)$ and for $\left|p-\tilde{p}\right|<\eps\eta\of p$
\begin{equation}
\frac{1-\eps}{1-2\eps}\,\eta\of p>\eta\of{\tilde{p}}>\eta\of p-\left|p-\tilde{p}\right|>\left(1-\eps\right)\eta\of p\,.\label{eq:thm:delta-M-rho-covering-a}
\end{equation}
and define $\tilde{\eta}\of p=2^{-K}\eta\of p$, $K\geq2$. Then for
every $C\in(0,1)$ there exists a locally finite covering of $\Gamma$
with balls $\Ball{\tilde{\eta}\of{p_{k}}}{p_{k}}$ for a countable
number of points $\of{p_{k}}_{k\in\N}\subset\Gamma$ such that for
every $i\neq k$ with $\Ball{\tilde{\eta}\of{p_{i}}}{p_{i}}\cap\Ball{\tilde{\eta}\of{p_{k}}}{p_{k}}\neq\emptyset$
it holds 
\begin{equation}
\begin{aligned} & \frac{2^{K-1}-1}{2^{K-1}}\tilde{\eta}\of{p_{i}}\leq\tilde{\eta}\of{p_{k}}\leq\frac{2^{K-1}}{2^{K-1}-1}\tilde{\eta}\of{p_{i}}\\
\text{and}\quad & \frac{2^{K}-1}{2^{K-1}-1}\min{\left\{ \tilde{\eta}\of{p_{i}},\tilde{\eta}\of{p_{k}}\right\} }\geq\left|p_{i}-p_{k}\right|\geq C\max{\left\{ \tilde{\eta}\of{p_{i}},\tilde{\eta}\of{p_{k}}\right\} }
\end{aligned}
\label{eq:thm:delta-M-rho-covering}
\end{equation}
\end{thm}

\begin{proof}
W.o.l.g. assume $\tilde{\eta}<(1-\delta)$. Consider $\tilde{Q}:=\left[0,\frac{1}{n}\right]^{d}$,
let $q_{1,\dots,n^{d}}$ denote the $n^{d}$ elements of $[0,1)^{d}\cap\frac{\Q^{d}}{n}$
and let $\tilde{Q}_{z,i}=\tilde{Q}+z+q_{i}$. We set $B_{(0)}:=\emptyset$,
$\Gamma_{1}=\Gamma$, $\eta_{k}:=\left(1-\delta\right)^{k}$ and for
$k\geq1$ we construct the covering using inductively defined open
sets $B_{(k)}$ and closed set $\Gamma_{k}$ as follows:
\begin{enumerate}
\item Define $\Gamma_{k,1}=\Gamma_{k}$. For $i=1,\dots,n^{d}$ do the following:
\begin{enumerate}
\item For every $z\in\Zd$ do 
\[
\begin{aligned} & \text{if }\exists p\in\left(\eta_{k}\tilde{Q}_{z,i}\right)\cap\Gamma_{k,i},\,\tilde{\eta}\of p\in(\eta_{k},\eta_{k-1}] & \text{then set } & b_{z,i}=\Ball{\tilde{\eta}\of p}p\,,\;\X_{z,i}=\left\{ p\right\} \\
 & \text{otherwise } & \text{set } & b_{z,i}=\emptyset\,,\;\X_{z,i}=\emptyset\,.
\end{aligned}
\]
\item Define $B_{(k),i}:=\bigcup_{z\in\Zd}b_{z,i}$ and $\Gamma_{k,i+1}=\Gamma_{k}\backslash B_{(k),i}$
and $\X_{(k),i}:=\bigcup_{z\in\Zd}\X_{z,i}$.\\
Observe: $p_{1},p_{2}\in\X_{(k),i}$ implies $\left|p_{1}-p_{2}\right|>\left(1-\frac{1}{n}\right)\eta_{k}$
and $p_{3}\in\X_{(k),j}$, $j<i$ implies $p_{1}\not\in\Ball{\eta_{k}}{p_{3}}$
and hence $\left|p_{1}-p_{3}\right|>\eta_{k}$. Similar, $p_{3}\in\X_{l}$,
$l<k$, implies $\left|p_{1}-p_{3}\right|>\eta_{l}>\eta_{k}$.
\end{enumerate}
\item Define $\Gamma_{k+1}:=\Gamma_{k,2^{d}+1}$, $\X_{k}:=\bigcup_{i}\X_{(k),i}$.
\end{enumerate}
The above covering of $\Gamma$ is complete in the sense that every
$x\in\Gamma$ lies into one of the balls (by contradiction). We denote
$\X:=\bigcup_{k}\X_{k}=\left(p_{i}\right)_{i\in\N}$ the family of
centers of the above constructed covering of $\Gamma$ and find the
following properties: Let $p_{1},p_{2}\in\X$ be such that $\Ball{\tilde{\eta}\of{p_{1}}}{p_{1}}\cap\Ball{\tilde{\eta}\of{p_{2}}}{p_{2}}\neq\emptyset$.
W.l.o.g. let $\tilde{\eta}\of{p_{1}}\geq\tilde{\eta}\of{p_{2}}$.
Then the following two properties are satisfied due to (\ref{eq:thm:delta-M-rho-covering-a})
\begin{enumerate}
\item It holds $\left|p_{1}-p_{2}\right|\leq2\tilde{\eta}\of{p_{1}}\leq\frac{1}{2^{K-1}}\eta\of{p_{1}}$
and hence $\Ball{\tilde{\eta}\of{p_{2}}}{p_{2}}\subset\Ball{2^{2-K}\eta\of{p_{1}}}{p_{1}}$
and $\eta\of{p_{2}}\geq\frac{2^{K-1}-1}{2^{K-1}}\eta\of{p_{1}}$.
Furthermore $\tilde{\eta}\of{p_{1}}\geq\tilde{\eta}\of{p_{2}}\geq\frac{2^{K-1}-1}{2^{K-1}}\tilde{\eta}\of{p_{1}}$.
\item Let $k$ such that $\tilde{\eta}\of{p_{1}}\in\left(\eta_{k},\eta_{k+1}\right]$.
If also $\tilde{\eta}\of{p_{2}}\in\left(\eta_{k},\eta_{k+1}\right]$
then observation 1.(b) implies $\left|p_{1}-p_{2}\right|\geq\left(1-\frac{1}{n}\right)\eta_{k}\geq\left(1-\frac{1}{n}\right)\left(1-\delta\right)\tilde{\eta}\of{p_{1}}$.
If $\tilde{\eta}\of{p_{2}}\not\in\left[\eta_{k},\eta_{k+1}\right)$
then $\tilde{\eta}\of{p_{2}}<\eta_{k}$ and hence $p_{2}\not\in\Ball{\tilde{\eta}\of{p_{1}}}{p_{1}}$,
implying $\left|p_{1}-p_{2}\right|>\tilde{\eta}\of{p_{1}}$.
\end{enumerate}
Choosing $n$ and $\delta$ appropriately, this concludes the proof.
\end{proof}

\subsection{Dynamical Systems}
\begin{assumption}
\lyxadded{heida}{Sat Jun 06 17:27:29 2020}{\label{assu:separable}}Throughout\lyxadded{heida}{Sat Jun 06 17:27:29 2020}{
this work we assume that $(\Omega,\sF,\P)$ is a probability space
with countably generated $\sigma$-algebra $\sF$.}
\end{assumption}

\lyxadded{heida}{Sat Jun 06 17:27:29 2020}{Due to the insight in \cite{heida2011extension},
shortly }sketched\lyxadded{heida}{Sat Jun 06 17:27:29 2020}{ in the
next two subsections, after a measurable transformation the probability
space $\Omega$ can be assumed to be metric and separable, which always
ensures Assumption \ref{assu:separable}.}
\begin{defn}[Dynamical system]\label{def:A-dynamical-system}
\nomenclature[Tau]{$\tau_x$}{Dynamical system (Definitions \ref{assu:Omega-mu-tau}, \ref{def:A-dynamical-system-Zd}) with respect to $x\in\Rd$ or $x\in\Zd$}\label{Def:Omega-mu-tau}A
dynamical system on $\Omega$ is a family $(\tau_{x})_{x\in\Rd}$
of measurable bijective mappings $\tau_{x}:\Omega\mapsto\Omega$ satisfying
(i)-(iii):

\begin{enumerate}
\item [(i)]$\tau_{x}\circ\tau_{y}=\tau_{x+y}$ , $\tau_{0}=id$ (Group
property)
\item [(ii)]$\P(\tau_{-x}B)=\P(B)\quad\forall x\in\Rd,\,\,B\in\sF$ (Measure
preserving)
\item [(iii)]$A:\,\,\Rd\times\Omega\rightarrow\Omega\qquad(x,\omega)\mapsto\tau_{x}\omega$
is measurable (Measurability of evaluation)
\end{enumerate}
\end{defn}

A set $A\subset\Omega$ is almost invariant if \nomenclature[I]{$\sI$}{Invariant sets, \eqref{eq:invariant-sets}}$\P\left(\left(A\cup\tau_{x}A\right)\backslash\left(A\cap\tau_{x}A\right)\right)=0$.
The family 
\begin{equation}
\sI=\left\{ A\in\sF\,:\;\forall x\in\Rd\,\P\left(\left(A\cup\tau_{x}A\right)\backslash\left(A\cap\tau_{x}A\right)\right)=0\right\} \label{eq:invariant-sets}
\end{equation}
 of almost invariant sets is $\sigma$-algebra and \nomenclature[E]{$\E(f|\sI)$}{Expectation of $f$ wrt. the invariant sets, \eqref{eq:invariant-sets-expectation}}
\begin{equation}
\E\left(f|\sI\right)\text{denotes the expectation of }f:\,\Omega\to\R\text{ w.r.t. }\sI\,.\label{eq:invariant-sets-expectation}
\end{equation}
A concept linked to dynamical systems is the concept of stationarity.
\begin{defn}[Stationary]
\label{def:stationary}\nomenclature{stationary}{Definition \ref{def:stationary}}Let
$X$ be a measurable space and let $f:\Omega\times\Rd\to X$. Then
$f$ is called (weakly) stationary if $f(\omega,x)=f(\tau_{x}\omega,0)$
for (almost) every $x$.
\end{defn}

\begin{defn}
\label{def:convex-averaging-sequence}\nomenclature[Convex averaging sequence]{Convex averaging sequence}{(Definition \ref{def:convex-averaging-sequence}) }A
family $\left(A_{n}\right)_{n\in\N}\subset\Rd$ is called convex averaging
sequence if
\begin{enumerate}
\item [(i)]each $A_{n}$ is convex
\item [(ii)]for every $n\in\N$ holds $A_{n}\subset A_{n+1}$
\item [(iii)]there exists a sequence $r_{n}$ with $r_{n}\to\infty$ as
$n\to\infty$ such that $B_{r_{n}}(0)\subseteq A_{n}$.
\end{enumerate}
\end{defn}

We sometimes may take the following stronger assumption.
\begin{defn}
A convex averaging sequence $A_{n}$ is called regular if 
\[
\left|A_{n}\right|^{-1}\#\left\{ z\in\Zd\,:\;\left(z+\T\right)\cap\partial A_{n}\neq\emptyset\right\} \to0\,.
\]
\end{defn}

The latter condition is evidently fulfilled for sequences of cones
or balls. Convex averaging sequences are important in the context
of ergodic theorems.
\begin{thm}[Ergodic Theorem \cite{Daley1988} Theorems 10.2.II and also \cite{tempel1972ergodic}]
\label{thm:Ergodic-Theorem} \nomenclature[Ergodic Theorem]{Ergodic Theorem}{(Theorems \ref{thm:Ergodic-Theorem}, \ref{thm:Ergodic-Theorem-ran-meas}) }Let
$\left(A_{n}\right)_{n\in\N}\subset\Rd$ be a convex averaging sequence,
let $(\tau_{x})_{x\in\Rd}$ be a dynamical system on $\Omega$ with
invariant $\sigma$-algebra $\sI$ and let $f:\,\Omega\to\R$ be measurable
with $\left|\E(f)\right|<\infty$. Then for almost all $\omega\in\Omega$
\begin{equation}
\left|A_{n}\right|^{-1}\int_{A_{n}}f\of{\tau_{x}\omega}\,\d x\to\E\of{f|\sI}\,.\label{eq:ergodic convergence}
\end{equation}
\end{thm}

We observe that $\E\left(f|\sI\right)$ is of particular importance.
For the calculations in this work, we will particularly focus on the
case of trivial $\sI$. This is called ergodicity, as we will explain
in the following.
\begin{defn}[Ergodicity and Mixing]
\nomenclature[Ergodicity]{Ergodicity}{(Definition \ref{def:erg-mixing}) }\label{def:erg-mixing}\nomenclature[Mixing]{Mixing}{(Definition \ref{def:erg-mixing}) }A
dynamical system $(\tau_{x})_{x\in\Rd}$ which is given on a probability space $(\Omega,\sF,\P)$
is called \emph{mixing }if for every measurable $A,B\subset\Omega$
it holds 
\begin{equation}
\lim_{\norm x\to\infty}\P\of{A\cap\tau_{x}B}=\P(A)\,\P(B)\,.\label{eq:def-mixing}
\end{equation}
A dynamical system is called \emph{ergodic }if 
\begin{equation}
\lim_{n\to\infty}\frac{1}{\left(2n\right)^{d}}\int_{[-n,n]^{d}}\P\of{A\cap\tau_{x}B}\d x=\P(A)\,\P(B)\,.\label{eq:def-ergodic}
\end{equation}
\end{defn}

\begin{rem}
\label{rem:trivial-mixing}a) Let $\Omega=\left\{ \omega_{0}=0\right\} $
with the trivial $\sigma$-algebra and $\tau_{x}\omega_{0}=\omega_{0}$.
Then $\tau$ is evidently mixing. However, the realizations are constant
functions $f_{\omega}(x)=c$ on $\Rd$ for some constant $c$.

b) A typical ergodic system is given by $\Omega=\T$ with the Lebesgue
$\sigma$-algebra and $\P=\lebesgueL$ the Lebesgue measure. The dynamical
system is given by $\tau_{x}y:=\left(x+y\right)\!\mod\T$.

c) It is known that $(\tau_{x})_{x\in\Rd}$ is ergodic if and only
if every almost invariant set $A\in\sI$ has probability $\P(A)\in\left\{ 0,1\right\} $
(see \cite{Daley1988} Proposition 10.3.III) i.e. 
\begin{equation}
\left[\,\,\forall x\,\P\of{\left(\tau_{x}A\cup A\right)\backslash\left(\tau_{x}A\cap A\right)}=0\,\,\right]\;\Rightarrow\;\P\of A\in\left\{ 0,1\right\} \,.\label{eq:def-ergodic-2}
\end{equation}

d) It is sufficient to show (\ref{eq:def-mixing}) or (\ref{eq:def-ergodic})
for $A$ and $B$ in a ring that generates the $\sigma$-algebra $\sF$.
We refer to \cite{Daley1988}, Section 10.2, for the later results.
\end{rem}

A further useful property of ergodic dynamical systems, which we will
use below, is the following:
\begin{lem}[Ergodic times mixing is ergodic]
\label{lem:erg-and-mix-is-erg}Let $(\tilde{\Omega},\tilde{\sF},\tilde{\P})$
and $(\hat{\Omega},\hat{\sF},\hat{\P})$ be probability spaces with
dynamical systems $(\tilde{\tau}_{x})_{x\in\Rd}$ and $(\hat{\tau}_{x})_{x\in\Rd}$
respectively. Let $\Omega:=\tilde{\Omega}\times\hat{\Omega}$ be the
usual product measure space with the notation $\omega=(\tilde{\omega},\hat{\omega})\in\Omega$
for $\tilde{\omega}\in\tilde{\Omega}$ and $\hat{\omega}\in\hat{\Omega}$.
If $\tilde{\tau}$ is ergodic and $\hat{\tau}$ is mixing, then $\tau_{x}(\tilde{\omega},\hat{\omega}):=(\tilde{\tau}_{x}\tilde{\omega},\hat{\tau}_{x}\hat{\omega})$
is ergodic.
\end{lem}

\begin{proof}
Relying on Remark \ref{rem:trivial-mixing}.c) we verify (\ref{eq:def-ergodic})
by proving it for sets $A=\tilde{A}\times\hat{A}$ and $B=\tilde{B}\times\hat{B}$
which generate $\sF:=\tilde{\sF}\otimes\hat{\sF}$. We make use of
$A\cap B=\left(\tilde{A}\cap\tilde{B}\right)\times\left(\hat{A}\cap\hat{B}\right)$
and observe that 
\begin{align*}
\P\of{A\cap\tau_{x}B} & =\P\of{\left(\tilde{A}\cap\tilde{\tau}_{x}\tilde{B}\right)\times\left(\hat{A}\cap\hat{\tau}_{x}\hat{B}\right)}=\hat{\P}\of{\hat{A}\cap\hat{\tau}_{x}\hat{B}}\,\tilde{\P}\of{\tilde{A}\cap\tilde{\tau}_{x}\tilde{B}}\\
 & =\hat{\P}\of{\hat{A}\cap\hat{B}}\,\tilde{\P}\of{\tilde{A}\cap\tilde{\tau}_{x}\tilde{B}}+\left[\hat{\P}\of{\hat{A}\cap\hat{\tau}_{x}\hat{B}}-\hat{\P}\of{\hat{A}\cap\hat{B}}\right]\,\tilde{\P}\of{\tilde{A}\cap\tilde{\tau}_{x}\tilde{B}}\,.
\end{align*}
Using ergodicity, we find that
\begin{align}
\lim_{n\to\infty}\frac{1}{\left(2n\right)^{d}}\int_{[-n,n]^{d}}\hat{\P}\of{\hat{A}\cap\hat{B}}\,\tilde{\P}\of{\tilde{A}\cap\tilde{\tau}_{x}\tilde{B}}\d x & =\hat{\P}\of{\left(\hat{A}\cap\hat{B}\right)}\,\tilde{\P}\of{\tilde{A}\cap\tilde{B}}\nonumber \\
 & =\P\of{A\cap B}\,.\label{eq:lem:erg-and-mix-is-erg-help-1}
\end{align}
Since $\hat{\tau}$ is mixing, we find for every $\eps>0$ some $R>0$
such that $\norm x>R$ implies $$\left|\hat{\P}\of{\hat{A}\cap\hat{\tau}_{x}\hat{B}}-\hat{\P}\of{\hat{A}\cap\hat{B}}\right|<\eps\,.$$
For $n>R$ we find 
\begin{multline}
\frac{1}{\left(2n\right)^{d}}\int_{[-n,n]^{d}}\left|\hat{\P}\of{\hat{A}\cap\hat{\tau}_{x}\hat{B}}-\hat{\P}\of{\hat{A}\cap\hat{B}}\right|\,\tilde{\P}\of{\tilde{A}\cap\tilde{\tau}_{x}\tilde{B}}\\
\leq\frac{1}{\left(2n\right)^{d}}\int_{[-n,n]^{d}}\eps+\frac{1}{\left(2n\right)^{d}}\int_{[-R,R]^{d}}2\to\eps\quad\text{as }n\to\infty\,.\label{eq:lem:erg-and-mix-is-erg-help-2}
\end{multline}
The last two limits (\ref{eq:lem:erg-and-mix-is-erg-help-1}) and
(\ref{eq:lem:erg-and-mix-is-erg-help-2}) imply (\ref{eq:def-ergodic}).
\end{proof}
\begin{rem}
The above proof heavily relies on the mixing property of $\hat{\tau}$.
Note that for $\hat{\tau}$ being only ergodic, the statement is wrong,
as can be seen from the product of two periodic processes in $\T\times\T$
(see Remark \ref{rem:trivial-mixing}). Here, the invariant sets are
given by $I_{A}:=\left\{ \left(\left(y+x\right)\!\mod\T\,,\,x\right)\,:\;y\in A\right\} $
for arbitrary measurable $A\subset\T$.
\end{rem}

\subsection{\label{subsec:Random-measures-and}Random Measures and Palm Theory}

We recall some facts from random measure theory (see \cite{Daley1988})
which will be needed for homogenization. Let $\fM(\Rd)$ \nomenclature[M]{$\fM(\Rd)$}{Measures on $\Rd$ (Section \ref{subsec:Random-measures-and}) }denote
the space of locally bounded Borel measures on $\Rd$ (i.e. bounded
on every bounded Borel-measurable set) equipped with the Vague topology,
which is generated by the sets 
\[
\left\{ \mu\,:\;\int f\,\d\mu\in A\right\} \text{ for every open }A\subset\Rd\text{ and }f\in C_{c}\of{\Rd}\,.
\]
This topology is metrizable, complete and countably generated. However, note that it is not locally compact, which implies that the Alexandroff compactification cannot be applied. A random
measure is a measurable mapping 
\[
\mu_{\bullet}:\;\Omega\to\fM(\Rd)\,,\qquad\omega\mapsto\mu_{\omega}
\]
which is equivalent to both of the following conditions
\begin{enumerate}
\item For every bounded Borel set $A\subset\Rd$ the map $\omega\mapsto\mu_{\omega}(A)$
is measurable
\item For every $\omega\mapsto\int f\d\mu_{\omega}$ the map $\omega\mapsto\int f\,\d\mu_{\omega}$
is measurable.
\end{enumerate}
A random measure is stationary if the distribution of $\mu_{\omega}(A)$
is invariant under translations of $A$ that is $\mu_{\omega}(A)$
and $\mu_{\omega}(A+x)$ share the same distribution. From stationarity
of $\mu_{\omega}$ one concludes the existence (\cite{heida2011extension,papanicolaou1979boundary}
and references therein) of a dynamical system $\left(\tau_{x}\right)_{x\in\Rd}$
on $\Omega$ such that $\mu_{\omega}\left(A+x\right)=\mu_{\tau_{x}\omega}\left(A\right)$.
By a deep theorem due to Mecke (see \cite{Mecke1967,Daley1988}) the
measure 
\[
\mupalm(A)=\int_{\Omega}\int_{\Rd}g(s)\,\chi_{A}(\tau_{s}\omega)\,\d\muomega(s)\,\d\P(\omega)
\]
can be defined on $\Omega$ for every positive $g\in L^{1}(\Rd)$
with compact support. $\mupalm$ is independent from $g$ and in case
$\mu_{\omega}=\lebesgueL$ we find $\mupalm=\P$. Furthermore, for
every $\borelB(\Rd)\times\borelB(\Omega)$-measurable non negative
or $\mupalm\times\lebesgueL$- integrable functions $f$ the Campbell
formula
\[
\int_{\Omega}\int_{\Rd}f(x,\tau_{x}\omega)\,\d\muomega(x)\,\d\P(\omega)=\int_{\Rd}\int_{\Omega}f(x,\omega)\,\d\mupalm(\omega)\,\d x
\]
 holds. The measure $\mu_{\omega}$ has finite intensity if $\mupalm(\Omega)<+\infty$.

We denote by \nomenclature[Emup]{$\E_{\mupalm}(f|\sI)$}{Expectation of $f$ wrt. $\mupalm$ and the invariant sets, \eqref{eq:invariant-sets-expectation-mup}}
\begin{equation}
\E_{\mupalm}\of{f|\sI}:=\int_{\Omega}f\text{ the expectation of }f\text{ w.r.t. the }\sigma\text{-algebra }\sI\text{ and }\mupalm\,.\label{eq:invariant-sets-expectation-mup}
\end{equation}
For random measures we find a more general version of Theorem \ref{thm:Ergodic-Theorem}.
\begin{thm}[Ergodic Theorem \cite{Daley1988} 12.2.VIII]
\label{thm:Ergodic-Theorem-ran-meas} Let $\left(\Omega,\sF,\P\right)$
be a probability space, $\left(A_{n}\right)_{n\in\N}\subset\Rd$ be
a convex averaging sequence, let $(\tau_{x})_{x\in\Rd}$ be a dynamical
system on $\Omega$ with invariant $\sigma$-algebra $\sI$ and let
$f:\,\Omega\to\R$ be measurable with $\int_{\Omega}\left|f\right|\d\mupalm<\infty$.
Then for $\P$-almost all $\omega\in\Omega$ 
\begin{equation}
\left|A_{n}\right|^{-1}\int_{A_{n}}f\of{\tau_{x}\omega}\,\d\mu_{\omega}(x)\to\E_{\mupalm}\of{f|\sI}\,.\label{eq:ergodic convergence ran meas}
\end{equation}
\end{thm}

Given a bounded open (and convex) set $\bQ\subset\Omega$, it is not
hard to see that the following generalization holds:
\begin{thm}[General Ergodic Theorem]
\label{thm:Ergodic-Theorem-ran-meas-2} Let $\left(\Omega,\sF,\P\right)$
be a probability space, $\bQ\subset\Rd$ be a convex bounded open
set with $0\in\bQ$, let $(\tau_{x})_{x\in\Rd}$ be a dynamical system
on $\Omega$ with invariant $\sigma$-algebra $\sI$ and let $f:\,\Omega\to\R$
be measurable with $\int_{\Omega}\left|f\right|\d\mupalm<\infty$.
Then for $\P$-almost all $\omega\in\Omega$ it holds
\begin{equation}
\forall\varphi\in C(\overline{\bQ}):\quad n^{-d}\int_{n\bQ}\varphi(\frac{x}{n})f\of{\tau_{x}\omega}\,\d\mu_{\omega}(x)\to\E_{\mupalm}\of{f|\sI}\int_{\bQ}\varphi\,.\label{eq:ergodic convergence ran meas2}
\end{equation}
\end{thm}

\begin{proof}[Sketch of proof]
 Chose a countable family of characteristic functions that spans
$L^{1}(\bQ)$. Use a Cantor argument and Theorem \ref{thm:Ergodic-Theorem-ran-meas}
to prove the statement for a countable dense family of $C(\overline{\bQ})$.
From here, we conclude by density.

The last result can be used to prove the most general ergodic theorem
which we will use in this work:
\end{proof}
\begin{thm}[General Ergodic Theorem for the Lebesgue measure]
\label{thm:general-lebesgue-ergodic-thm}Let $\left(\Omega,\sF,\P\right)$
be a probability space, $\bQ\subset\Rd$ be a convex bounded open
set with $0\in\bQ$, let $(\tau_{x})_{x\in\Rd}$ be a dynamical system
on $\Omega$ with invariant $\sigma$-algebra $\sI$ and let $f\in L^{p}(\Omega;\mupalm)$
and $\varphi\in L^{q}(\bQ)$, where $1<p,q<\infty$, $\frac{1}{p}+\frac{1}{q}=1$.
Then for $\P$-almost all $\omega\in\Omega$ it holds
\[
n^{-d}\int_{n\bQ}\varphi(\frac{x}{n})f\of{\tau_{x}\omega}\,\d x\to\E\of f\int_{\bQ}\varphi\,.
\]
\end{thm}

\begin{proof}
Let $\varphi_{\delta}\in C(\overline{\bQ})$ with $\norm{\varphi-\varphi_{\delta}}_{L^{q}(\bQ)}<\delta$.
Then
\begin{align*}
 & \left|n^{-d}\int_{n\bQ}\varphi(\frac{x}{n})f\of{\tau_{x}\omega}\,\d x-\E\of f\int_{\bQ}\varphi\right|\\
 & \qquad\qquad\leq\norm{\varphi-\varphi_{\delta}}_{L^{q}(\bQ)}\left(n^{-d}\int_{n\bQ}\left|f\of{\tau_{x}\omega}\right|^{p}\,\d x\right)^{\frac{1}{p}}\\
 & \qquad\qquad\quad+\left|n^{-d}\int_{n\bQ}\varphi_{\delta}(x)f\of{\tau_{x}\omega}\,\d x-\E\of f\int_{\bQ}\varphi_{\delta}\right|+\E_{\mupalm}\of{f|\sI}\int_{\bQ}\left|\varphi-\varphi_{\delta}\right|\,,
\end{align*}
which implies the claim.
\end{proof}

\subsection{Random Sets}

The theory of random measures and the theory of random geometry are
closely related. In what follows, we recapitulate those results that
are important in the context of the theory developed below and shed
some light on the correlations between random sets and random measures.

Let $\closedsets(\Rd)$ denote the set of all closed sets in $\Rd$.
We write \nomenclature[F]{$\closedsets_{V}$, $\closedsets^{K}$, $(\closedsets(\Rd),\ttopology F)$}{(Equations \eqref{eq:closedsets-1}, \eqref{eq:closedsets-2}) }
\begin{eqnarray}
\closedsets_{V}:= & \left\{ F\in\closedsets(\Rd)\,:\;F\cap V\not=\emptyset\right\}  & \mbox{if }V\subset\Rd\quad\textnormal{is an open set}\,,\label{eq:closedsets-1}\\
\closedsets^{K}:= & \left\{ F\in\closedsets(\Rd)\,:\;F\cap K=\emptyset\right\}  & \mbox{if }K\subset\Rd\quad\textnormal{is a compact set}\,.\label{eq:closedsets-2}
\end{eqnarray}
The \emph{Fell-topology} $\ttopology F$ is created by all sets $\closedsets_{V}$
and $\closedsets^{K}$ and the topological space $(\closedsets(\Rd),\ttopology F)$
is compact, Hausdorff and separable\cite{Matheron1975}.
\begin{rem}
\label{rem:char-fell}We find for closed sets $F_{n},F$ in $\Rd$
that $F_{n}\to F$ if and only if \cite{Matheron1975}
\begin{enumerate}
\item for every $x\in F$ there exists $x_{n}\in F_{n}$ such that $x=\lim_{n\to\infty}x_{n}$
and
\item if $F_{n_{k}}$ is a subsequence, then every convergent sequence $x_{n_{k}}$
with $x_{n_{k}}\in F_{n_{k}}$ satisfies $\lim_{k\to\infty}x_{n_{k}}\in F$.
\end{enumerate}
\end{rem}

If we restrict the Fell-topology to the compact sets $\mathfrak{K}(\Rd)$
it is equivalent with the Hausdorff topology given by the Hausdorff
distance
\[
\d\of{A,B}=\max\left\{ \sup_{y\in B}\inf_{x\in A}\left|x-y\right|\,,\,\,\sup_{x\in A}\inf_{y\in B}\left|x-y\right|\right\} \,.
\]

\begin{rem}
For $A\subset\Rd$ closed, the set 
\[
\closedsets(A):=\left\{ F\in\closedsets(\Rd)\,:\;F\subset A\right\} 
\]
 is a closed subspace of $\closedsets\of{\Rd}$. This holds since
\[
\closedsets\of{\Rd}\backslash\closedsets\of A=\left\{ B\in\closedsets\of{\Rd}\,:\;B\cap\left(\Rd\backslash A\right)\not=\emptyset\right\} =\closedsets_{\Rd\backslash A}\quad\text{is open.}
\]
.
\end{rem}

\begin{lem}[Continuity of geometric operations]
\label{lem:contin-geom-Ops}The maps $\tau_{x}:\,A\mapsto A+x$ and
$b_{\delta}:\,A\mapsto\overline{\Ball{\delta}A}$ are continuous in
$\closedsets\left(\Rd\right)$.
\end{lem}

\begin{proof}
We show that preimages of open sets are open. For open sets $V$ we
find
\begin{align*}
\tau_{x}^{-1}\left(\closedsets_{V}\right) & =\left\{ F\in\closedsets(\Rd)\,:\;\tau_{x}F\cap V\not=\emptyset\right\} =\left\{ F\in\closedsets(\Rd)\,:\;F\cap\tau_{-x}V\not=\emptyset\right\} =\closedsets_{\tau_{-x}V}\,,\\
b_{\delta}^{-1}\left(\closedsets_{V}\right) & =\left\{ F\in\closedsets(\Rd)\,:\;\overline{\Ball{\delta}F}\cap V\not=\emptyset\right\} =\left\{ F\in\closedsets(\Rd)\,:\;F\cap\Ball{\delta}V\not=\emptyset\right\} =\closedsets_{\left(b_{\delta}V\right)^{\circ}}\,.
\end{align*}
The calculations for $\tau_{x}^{-1}\of{\closedsets^{K}}=\closedsets^{\tau_{-x}K}$
and $b_{\delta}^{-1}\of{\closedsets^{K}}=\closedsets^{b_{\delta}K}$
are analogue.
\end{proof}
\begin{rem}
\label{rem:The-Matheron--field}The \emph{Matheron-$\sigma$-field
$\sigma_{\closedsets}$} is the Borel-$\sigma$-algebra of the Fell-topology
and is fully characterized either by the class $\closedsets_{V}$
of $\closedsets^{K}$. 
\end{rem}

\begin{defn}[Random closed / open set according to Choquet (see \cite{Matheron1975}
for more details)]
\label{def:RACS}\nomenclature[Random closed sets]{Random closed sets}{(Definition \ref{def:RACS})}
\end{defn}

\begin{enumerate}
\item [a)] Let $(\Omega,\sigma,\P)$ be a probability space. Then a \emph{Random
Closed Set (RACS)} is a measurable mapping 
\[
A:(\Omega,\sigma,\P)\longrightarrow(\closedsets,\sigma_{\closedsets})
\]
\item [b)]Let $\tau_{x}$ be a dynamical system on $\Omega$. A random
closed set is called stationary if its characteristic functions $\chi_{A(\omega)}$
are stationary, i.e. they satisfy $\chi_{A(\omega)}(x)=\chi_{A(\tau_{x}\omega)}(0)$
for almost every $\omega\in\Omega$ for almost all $x\in\Rd$. Two
random sets are jointly stationary if they can be parameterized by
the same probability space such that they are both stationary.
\item [c)] A random closed set $\Gamma:(\Omega,\sigma,P)\longrightarrow(\closedsets,\sigma_{\closedsets})\quad\omega\mapsto\Gamma(\omega)$
is called a \emph{Random closed $C^{k}$-Manifold} if $\Gamma(\omega)$
is a piece-wise $C^{k}$-manifold for P almost every $\omega$.
\item [d)] A measurable mapping 
\[
A:(\Omega,\sigma,\P)\longrightarrow(\closedsets,\sigma_{\closedsets})
\]
is called \emph{Random Open Set (RAOS)} if $\omega\mapsto\Rd\backslash A(\omega)$
is a RACS.
\end{enumerate}
The importance of the concept of random geometries for stochastic
homogenization stems from the following Lemma by Z\"ahle. It states
that every random closed set induces a random measure. Thus, every
stationary RACS induces a stationary random measure.
\begin{lem}[\cite{Zaehle1982} Theorem 2.1.3 resp. Corollary 2.1.5]
\label{lemmazaehlerandommeasure}Let $\closedsets_{m}\subset\closedsets$
be the space of closed m-dimensional sub manifolds of $\Rd$ such
that the corresponding Hausdorff measure is locally finite. Then,
the $\sigma$-algebra $\sigma_{\closedsets}\cap\closedsets_{m}$ is
the smallest such that 
\[
M_{B}:\closedsets_{m}\rightarrow\R\quad M\mapsto\mathcal{H}^{m}(M\cap B)
\]
 is measurable for every measurable and bounded $B\subset\Rd$.
\end{lem}

This means that 
\[
M_{\Rd}:\closedsets_{m}\rightarrow\fM(\Rd)\quad M\mapsto\mathcal{H}^{m}(M\cap\cdot)
\]
is measurable with respect to the $\sigma$-algebra created by the
Vague topology on $\fM(\Rd)$. Hence a random closed set always induces
a random measure. Based on Lemma \ref{lemmazaehlerandommeasure} and
on Palm-theory, the following useful result was obtained in \cite{heida2011extension}
(See Lemma 2.14 and Section 3.1 therein).
\begin{thm}
\label{thm:Main-THM-Sto-Geo}Let $(\Omega,\sigma,P)$ be a probability
space with an ergodic dynamical system $\tau$. Let $A:(\Omega,\sigma,P)\longrightarrow(\closedsets,\sigma_{\closedsets})$
be a stationary random closed $m$-dimensional $C^{k}$-Manifold.

a) There exists a separable metric space $\tilde{\Omega}\subset\fM\of{\Rd}$
with an ergodic dynamical system $\tilde{\tau}$ and a mapping $\tilde{A}:\,(\tilde{\Omega},\cB_{\tilde{\Omega}},\P)\to(\closedsets,\sigma_{\closedsets})$
such that $A$ and $\tilde{A}$ have the same law and such that $\tilde{A}$
still is stationary. Furthermore, $(x,\omega)\mapsto\tau_{x}\omega$
is continuous. We identify $\tilde{\Omega}=\Omega$, $\tilde{A}=A$
and $\tilde{\tau}=\tau$.

b) The mapping 
\[
\mu_{\bullet}:\,\,\Omega\to\fM(\Rd)\,,\quad\omega\mapsto\mu_{\omega}(\cdot):=\mathcal{H}^{m}(M\cap\cdot)
\]
is a stationary random measure on $\Rd$ and there exists a corresponding
Palm-measure $\mupalm$ if and only if $\mu_{\bullet}$ has finite
intensity.

c) There exists a measurable set $\hat{A}\subset\Omega$, called the
\emph{prototype }of $A$, such that $\chi_{A(\omega)}(x)=\chi_{\hat{A}}(\tau_{x}\omega)$
for $\lebesgueL+\muomega$-almost every $x$ and $\P$-almost surely.
The Palm-measure $\mupalm$ of $\muomega$ concentrates on $\hat{A}$,
i.e. $\mupalm(\Omega\backslash\hat{A})=0$.

d) If $A$ is a random closed $m$-dimensional $C^{k}$-manifold,
then $\P(\hat{A})=0$.
\end{thm}

Also the following result will be useful below.
\begin{lem}
\label{lem:lower-semi-cont-measures}Let $\mu$ be a Radon measure
on $\Rd$ and let $\bQ\subset\Rd$ be a bounded open set. Let $\closedsets_{0}\subset\closedsets\left(\overline{\bQ}\right)$
be such that $\closedsets_{0}\to\R$, $A\mapsto\mu(A)$ is continuous.
Then 
\[
m:\,\closedsets\times\closedsets_{0}\to\fM\of{\Rd}\,,\quad\left(P,B\right)\mapsto\begin{cases}
A\mapsto\mu\of{A\cap B} & B\subset P\\
0 & \text{else}
\end{cases}
\]
is measurable.
\end{lem}

\begin{proof}
For $f\in C_{c}(\Rd)$ we introduce $m_{f}$ through 
\[
m_{f}:\,\left(P,B\right)\mapsto\begin{cases}
\int_{B}f\,\d\mu & B\subset P\\
0 & \text{else}
\end{cases}
\]
and observe that $m$ is measurable if and only if for every $f\in C_{c}\of{\Rd}$
the map $m_{f}$ is measurable (see Section \ref{subsec:Random-measures-and}).
Hence, if we prove the latter property, the lemma is proved.

We assume $f\geq0$ and we show that the mapping $m_{f}$ is even
upper continuous. In particular, let $(P_{n},B_{n})\to(P,B)$ in $\closedsets\times\closedsets_{0}$
and assume that $B_{n}\subset P_{n}$ for all $n>N_{0}$. Since $\overline{\bQ}$
is compact, Remark \ref{rem:char-fell}.~2. implies that $B\subset P\cap\overline{\bQ}$.
Furthermore, since $f$ has compact support, we find $\left|\int_{B_{n}}f\,\d\mu-\int_{B}f\,\d\mu\right|\leq\norm f_{\infty}\left|\mu\of{B_{n}}-\mu\of B\right|\to0$.
On the other hand, if there exists a subsequence such that $B_{n}\not\subset P_{n}$
for all $n$, then either $B\not\subset P$ and $m_{f}(P_{n},B_{n})=0\to m_{f}(P,B)=0$
or $B\subset P$ and $0=\lim_{n\to\infty}m_{f}(P_{n},B_{n})\leq\int_{B}f\d\mu=m_{f}(P,B)$.
For $f\leq0$ we obtain lower semicontinuity and for general $f$
the map $m_{f}$ is the sum of an upper and a lower semicontinuous
map, hence measurable.
\end{proof}

\subsection{Point Processes}
\begin{defn}[(Simple) point processes]
A $\Z$-valued random measure $\mu_{\omega}$ is called point process.
In what follows, we consider the particular case that for almost every
$\omega$ there exist points $\left(x_{k}(\omega)\right)_{k\in\N}$
and values $\left(a_{k}\left(\omega\right)\right)_{k\in\N}$ in $\Z$
such that
\[
\mu_{\omega}=\sum_{k\in\N}a_{k}\delta_{x_{k}(\omega)}\,.
\]
The point process $\mu_{\omega}$ is called simple if almost surely
for all $k\in\N$ it holds $a_{k}\in\left\{ 0,1\right\} $.
\end{defn}

\begin{example}[Poisson process]
 \label{exa:poisson-point-proc}\nomenclature[Poisson Process]{Poinsson process}{(Example \ref{exa:poisson-point-proc})}A
particular example for a stationary point process is the Poisson point
process $\mu_{\omega}=\X_{\omega}$ with intensity $\lambda$. Here,
the probability $\P\of{\X(A)=n}$ to find $n$ points in a Borel-set
$A$ with finite measure is given by a Poisson distribution
\begin{equation}
\P\of{\X(A)=n}=e^{-\lambda\left|A\right|}\frac{\lambda^{n}\left|A\right|^{n}}{n!}\label{eq:PoisonPointPoc-Prob}
\end{equation}
with expectation $\E\of{\X(A)}=\lambda\left|A\right|$. The last formula
implies that the Poisson point process is stationary.
\end{example}

We can use a given random point process to construct further processes.
\begin{example}[Hard core Matern process]
\label{exa:Matern} \nomenclature[Matern Process]{Matern process}{(Example \ref{exa:Matern}--\ref{exa:Matern-Pois})}The
hard core Matern process is constructed from a given point process
$\X_{\omega}$ by mutually erasing all points with the distance to
the nearest neighbor smaller than a given constant $r$. If the original
process $\X_{\omega}$ is stationary (ergodic), the resulting hard
core process is stationary (ergodic) respectively.
\end{example}

\begin{example}[Hard core Poisson--Matern process]
 \label{exa:Matern-Pois} If a Matern process is constructed from
a Poisson point process, we call it a Poisson--Matern point process.
\end{example}

\begin{lem}
\label{lem:PP-is-RACS}Let $\mu_{\omega}$ be a simple point process
with $a_{k}=1$ almost surely for all $k\in\N$. Then $\X_{\omega}=\left(x_{k}(\omega)\right)_{k\in\N}$
is a random closed set. On the other hand, if $\X_{\omega}=\left(x_{k}(\omega)\right)_{k\in\N}$
is a random closed set that almost surely has no limit points then
$\mu_{\omega}$ is a point process.
\end{lem}

\begin{proof}
Let $\mu_{\omega}$ be a point process. For open $V\subset\Rd$ and
compact $K\subset\Rd$ let 
\begin{gather*}
f_{V,R}(x)=\dist\of{\,x,\,\Rd\backslash\left(V\cap\Ball R0\right)\,}\,,\qquad f_{\delta}^{K}(x)=\max{\left\{ \,1-\frac{1}{\delta}\dist\of{x,K}\,,\,0\,\right\} }\,.
\end{gather*}
Then $f_{V,R}$ is Lipschitz with constant $1$ and $f_{\delta}^{K}$
is Lipschitz with constant $\frac{1}{\delta}$ and support in $\Ball{\delta}K$.
Moreover, since $\mu_{\omega}$ is locally bounded, the number of
points $x_{k}$ that lie within $\Ball 1K$ is bounded. In particular,
we obtain
\begin{align*}
\X^{-1}\of{\closedsets_{V}} & =\bigcup_{R>0}\left\{ \omega\,:\;\int_{\Rd}f_{V,R}\,\d\mu_{\omega}>0\right\} \,,\\
\X^{-1}\of{\closedsets^{K}} & =\bigcap_{\delta>0}\left\{ \omega\,:\;\int_{\Rd}f_{\delta}^{K}\,\d\mu_{\omega}>0\right\} \,,
\end{align*}
are measurable. Since $\closedsets_{V}$ and $\closedsets^{K}$ generate
the $\sigma$-algebra on $\closedsets\left(\Rd\right)$, it follows
that $\omega\to\X_{\omega}$ is measurable.

In order to prove the opposite direction, let $\X_{\omega}=\left(x_{k}(\omega)\right)_{k\in\N}$
be a random closed set of points. Since $\X_{\omega}$ has almost
surely no limit points the measure $\mu_{\omega}$ is locally bounded
almost surely. We prove that $\mu_{\omega}$ is a random measure by
showing that 
\[
\forall f\in C_{c}\of{\Rd}\,:\qquad F:\,\omega\mapsto\int_{\Rd}f\,\d\mu_{\omega}\text{ is measurable.}
\]
For $\delta>0$ let $\mu_{\omega}^{\delta}\of A:=\left(\left|\S^{d-1}\right|\delta^{d}\right)^{-1}\lebesgueL\of{A\cap\Ball{\delta}{\X_{\omega}}}$.
By Lemmas \ref{lem:contin-geom-Ops} and \ref{lem:lower-semi-cont-measures}
we obtain that $F_{\delta}:\,\omega\mapsto\int_{\Rd}f\,\d\mu_{\omega}^{\delta}$
are measurable. Moreover, for almost every $\omega$ we find $F_{\delta}\left(\omega\right)\to F\left(\omega\right)$
uniformly and hence $F$ is measurable.
\end{proof}
\begin{cor}
A random simple point process $\mu_{\omega}$ is stationary iff $\X_{\omega}$
is stationary.
\end{cor}

Hence we can provide the following definition based on Definition
\ref{def:RACS}.
\begin{defn}
\label{def:jontly-stat-point-regular}A point process $\mu_{\omega}$
and a random set $\bP$ are jointly stationary if $\bP$ and $\X$
are jointly stationary.
\end{defn}

\begin{lem}
\label{lem:Matern-is-RACS}Let $\X_{\omega}=\left(x_{i}\right)_{i\in\N}$
be a Matern point process from Example \ref{exa:Matern} with distance
$r$ and let for $\delta<\frac{r}{2}$ be $\bB(\omega):=\bigcup_{i}\overline{\B{\delta}{x_{i}}}$.
Then $\bB(\omega)$ is a random closed set.
\end{lem}

\begin{proof}
This follows from Lemma \ref{lem:contin-geom-Ops}: $\X_{\omega}$
is measurable and $\X\mapsto\overline{\Ball{\delta}\X}$ is continuous.
Hence $\bB\left(\omega\right)$ is measurable.
\end{proof}

\subsection{Unoriented Graphs on Point Processes}
\begin{defn}[(Unoriented) Graph]\label{def:graph-neighbor}
Let $\X=\left(x_{i}\right)_{i\in\N}\subset\Rd$ be a countable set
of points. A graph $\left(\G,\X\right)$ on $\X$ (or simply $\G$
on $\X$) is a subset $\G\subset\X^{2}$. The graph $\G$ is unoriented
if $\left(x,y\right)\in\G$ implies $\left(y,x\right)\in\G$. For
$\left(x,y\right)\in\G$ we write $x\sim y$.
\end{defn}

Elements of $\G$ are usually referred to as \emph{edges}. Classically,
a graph consists of vertices $\X$ and edges $\G$, so the graph is
given through $\left(\G,\X\right)$. However, in this work the set
of points $\X$ will usually be given and we will mostly discuss the
properties of $\G$. This is why we adopt standard notations.
\begin{defn}[Paths and connected graphs]
Let $\X=\left(x_{i}\right)_{i\in\N}\subset\Rd$ be a countable set
of points with a graph $\G\subset\X^{2}$. A \emph{path} in $\X$
is a sorted family of points $\left(y_{1},\dots,y_{N}\right)\in\X^{N}$,
$N\in\N$, such that for every $k\in\left\{ 1,\dots,N-1\right\} $
it holds $y_{k}\sim y_{k+1}$. The family of all paths in $\X$ is
hence a subset of $\bigcup_{N\in\N}\X^{N}$. The graph $\G$ is said
to be \emph{connected} if for every $x,y\in\G$, $x\neq y$, there
exists $N>2$ and a path $\left(y_{1},\dots,y_{N}\right)\in\X^{N}$
such that $y_{1}=x$ and $y_{N}=y$.
\end{defn}

\begin{rem}
Let $\left(y_{1},\dots,y_{k}\right)$ with be a path from $y_{1}$
to $y_{k}$ . A path from $y_{k}$ to $y_{1}$ is given by reversing
the order, i.e. by $\left(y_{k},\dots,y_{1}\right)$.
\end{rem}

\begin{defn}[Local extrema on graphs]
Let $\X\subset\Rd$ be a countable set of points with a graph $\G$.
A function $u:\,A\subset\X\to\R$ has a \emph{local maximum }resp.
\emph{minimum }in $y\in A$ if for all $\tilde{y}\in A$ with $\tilde{y}\sim y$
it holds $u(y)\geq u\of{\tilde{y}}$ resp. $u(y)\leq u\of{\tilde{y}}$
\end{defn}

\subsection{\label{subsec:Dynamical-Systems-on-Zd}Dynamical Systems on $\protect\Zd$}
\begin{defn}
\label{def:A-dynamical-system-Zd}Let $\left(\hat{\Omega},\hat{\sF},\hat{\P}\right)$
be a probability space. A discrete dynamical system on $\hat{\Omega}$
is a family $(\hat{\tau}_{z})_{z\in r\Zd}$ of measurable bijective
mappings $\hat{\tau}_{z}:\hat{\Omega}\mapsto\hat{\Omega}$ satisfying
(i)-(iii) of Definition \ref{Def:Omega-mu-tau}. A set $A\subset\hat{\Omega}$
is almost invariant if for every $z\in r\Zd$ it holds $\P\left(\left(A\cup\hat{\tau}_{z}A\right)\backslash\left(A\cap\hat{\tau}_{z}A\right)\right)=0$
and $\hat{\tau}$ is called ergodic w.r.t. $r\Zd$ if every almost
invariant set has measure $0$ or $1$. 
\end{defn}

Similar to the continuous dynamical systems, also in this discrete
setting an ergodic theorem can be proved.
\begin{thm}[See Krengel and Tempel'man \cite{krengel1985ergodic,tempel1972ergodic}]
\label{thm:Ergodic-Theorem-discrete} Let $\left(A_{n}\right)_{n\in\N}\subset\Rd$
be a convex averaging sequence, let $(\hat{\tau}_{z})_{z\in r\Zd}$
be a dynamical system on $\hat{\Omega}$ with invariant $\sigma$-algebra
$\sI$ and let $f:\,\hat{\Omega}\to\R$ be measurable with $\left|\E(f)\right|<\infty$.
Then for almost all $\hat{\omega}\in\hat{\Omega}$ 
\begin{equation}
\left|A_{n}\right|^{-1}\sum_{z\in A_{n}\cap r\Zd}f\of{\hat{\tau}_{z}\hat{\omega}}\to r^{-d}\E\of{f|\sI}\,.\label{eq:ergodic convergence-discrete}
\end{equation}
\end{thm}

In the following, we restrict to $r=1$ for simplicity of notation.

Let $\Omega_{0}\subset\Rd$. We consider an enumeration $\left(\xi_{i}\right)_{i\in\N}$
of $\Zd$ such that $\hat{\Omega}:=\Omega_{0}^{\Zd}=\Omega_{0}^{\N}$
and write $\hat{\omega}=\left(\hat{\omega}_{\xi_{1}},\hat{\omega}_{\xi_{2}},\dots\right)=\left(\hat{\omega}_{1},\hat{\omega}_{2},\dots\right)$
for all $\hat{\omega}\in\hat{\Omega}$. We define a metric on $\hat{\Omega}$
through 
\[
d(\hat{\omega}_{1},\hat{\omega}_{2})=\sum_{k=1}^{\infty}\frac{1}{2^{k}}\frac{\left|\hat{\omega}_{1,\xi_{k}}-\hat{\omega}_{2,\xi_{k}}\right|}{1+\left|\hat{\omega}_{1,\xi_{k}}-\hat{\omega}_{2,\xi_{k}}\right|}\,.
\]
We write $\Omega_{n}:=\Omega_{0}^{n}$ and $\N_{n}:=\left\{ k\in\N:\,k\geq n+1\right\} $.
The topology of $\hat{\Omega}$ is generated by the open sets $A\times\Omega_{0}^{\N_{n}}$,
where for some $n>0$, $A\subset\Omega_{n}$ is an open set. In case
$\Omega_{0}$ is compact, the space $\hat{\Omega}$ is compact. Further,
$\hat{\Omega}$ is separable in any case since $\Omega_{0}$ is separable
(see \cite{Kelley1955}).

We consider the ring
\[
\cR=\bigcup_{n\in\N}\left\{ A\times\Omega_{0}^{\N_{n}}\,:\;A\subset\Omega_{n}\text{ is measurable}\right\} 
\]
and suppose for every $n\in\N$ that there exists a probability measure
$\P_{n}$ on $\Omega_{n}$ such that for every measurable $A_{n}\subset\Omega_{n}$
it holds $\P_{n+k}\of{A_{n}\times\Omega^{k}}=\P_{n}\of{A_{n}}$. Then
we define 
\[
\P\of{A_{n}\times\Omega_{0}^{\N_{n}}}:=\P_{n}\of{A_{n}}\,.
\]
We make the observation that $\P$ is additive and positive on $\cR$
and $\P(\emptyset)=0$. Next, let $\left(A_{j}\right)_{j\in\N}$ be
an increasing sequence of sets in $\cR$ such that $A:=\bigcup_{j}A_{j}\in\cR$.
Then, there exists $\tilde{A}_{1}\subset\Omega_{0}^{n}$ such that
$A_{1}=\tilde{A}_{1}\times\Omega_{0}^{\N_{n}}$ and since $A_{1}\subset A_{2}\subset\dots\subset A$,
for every $j>1$, we conclude $A_{j}=\tilde{A_{j}}\times\Omega_{0}^{\N_{n}}$
for some $\tilde{A}_{j}\subset\Omega_{n}$. Therefore, $\P(A_{j})=\P_{n}(\tilde{A}_{j})\to\P_{n}(\tilde{A})=\P(A)$
where $A=\tilde{A}\times\Omega_{0}^{\N_{n}}$. We have thus proved
that $\P:\cR\to[0,1]$ can be extended to a measure on the Borel-$\sigma$-Algebra
on $\Omega$ (See \cite[Theorem 6-2]{Berberian1965}).

We define for $z\in\Zd$ the mapping
\[
\hat{\tau}_{z}:\,\hat{\Omega}\to\hat{\Omega}\,,\qquad\hat{\omega}\mapsto\hat{\tau}_{z}\hat{\omega}\,,\quad\mbox{where }\left(\hat{\tau}_{z}\hat{\omega}\right)_{\xi_{i}}=\hat{\omega}_{\xi_{i}+z}\mbox{ component wise}\,.
\]

\begin{rem}
\label{rem:standard-omega-0}In this paper, we consider particularly
$\Omega_{0}=\left\{ 0,1\right\} $. Then $\hat{\Omega}:=\Omega_{0}^{\Zd}$
is equivalent to the power set of $\Zd$ and every $\hat{\omega}\in\hat{\Omega}$
is a sequence of $0$ and $1$ corresponding to a subset of $\Zd$.
Shifting the set $\hat{\omega}\subset\Zd$ by $z\in\Zd$ corresponds
to an application of $\hat{\tau}_{z}$ to $\hat{\omega}\in\hat{\Omega}$.
\end{rem}

Now, let $\bP(\omega)$ be a stationary ergodic random open set and
let $r>0$. Recalling (\ref{eq:Pr}) the map $\omega\mapsto\bP_{-r}(\omega)$
is measurable due to Lemma \ref{lem:contin-geom-Ops} and we can define
$\X_{r}\of{\bP\of{\omega}}:=2r\Zd\cap\bP_{-\frac{r}{2}}\of{\omega}$.
\begin{lem}
\label{lem:X-r-stationary}If $\bP$ is a stationary ergodic random
open set then the set\nomenclature[Xr]{$\X_r(\omega)=\X_r(\bP(\omega))$}{$=2r\Zd\cap \bP_{-r}(\omega)$, \eqref{eq:def-X_r}}
\begin{equation}
\X=\X_{r}(\omega):=\X_{r}\of{\bP\of{\omega}}:=2r\Zd\cap\bP_{-r}\of{\omega}\label{eq:def-X_r}
\end{equation}
 is a stationary random point process w.r.t. $2r\Zd$.
\end{lem}

\begin{proof}
By a simple scaling we can w.l.o.g. assume $2r=1$ and write $\X=\X_{r}$.
Evidently, $\X$ corresponds to a process on $\Zd$ with values in
$\Omega_{0}=\left\{ 0,1\right\} $ writing $\X(z)=1$ if $z\in\X$
and $\X(z)=0$ if $z\not\in\X$. In particular, we write $\left(\omega,z\right)\mapsto\X\of{\omega,z}$.
This process is stationary as  the shift invariance of $\bP$ induces
a shift-invariance of $\hat{\P}$ with respect to $\hat{\tau}_{z}$.
It remains to observe that the probabilities $\P\of{\X(z)=1}$ and
$\P\of{\X(z)=0}$ induce a random measure on $\hat{\Omega}$ in the
way described in Remark \ref{rem:standard-omega-0}. 
\end{proof}
\begin{rem}
If $\bP$ is mixing one can follow the lines of the proof of Lemma
\ref{lem:erg-and-mix-is-erg} to find that $\X_{r}\of{\bP\of{\omega}}$
is ergodic. However, in the general case $\X_{r}\of{\bP\of{\omega}}$
is not ergodic. This is due to the fact that by nature $\left(\tau_{z}\right)_{z\in\Zd}$
on $\Omega$ has more invariant sets than$\left(\tau_{x}\right)_{x\in\Rd}$.
For sufficiently complex geometries the map $\Omega\to\hat{\Omega}$
is onto.
\end{rem}

\begin{defn}[Jointly stationary]
\label{def:jointly-staionary-points}\nomenclature[jointly stationary]{jointly stationary}{Definitions \ref{def:RACS}, \ref{def:jontly-stat-point-regular} and \ref{def:jointly-staionary-points}}We
call a point process $\X$ with values in $2r\Zd$ to be strongly
jointly stationary with a random set $\bP$ if the functions $\chi_{\bP(\omega)}$,
$\chi_{\X(\omega)}$ are strongly jointly stationary w.r.t. the dynamical
system $\left(\tau_{2rx}\right)_{x\in\Zd}$ on $\Omega$.
\end{defn}

\section{\label{sec:Periodic-extension-theorem}Periodic Extension Theorem}

We study extension theorems on periodic geometries. In what follows,
we assume that the torus is split into $\T=\T_{1}\cup\T_{2}$ and
we denote $\bT_{1}$ and $\bT_{2}$ the periodic extensions of $\T_{1}$
and $\T_{2}$ respectively. In order to get familiar with our approach,
we first prove the following standard result, which was already obtained
in \cite{cioranescu1979homogenization} and generalized to $\Rd$
and $W^{1,p}(\bT_{1})$ in \cite{hopker2014diss} (see also \cite{hopker2014note}).
\begin{thm}[Extension Theorem]
\label{thm:Ext-per-isolated}Let $\T=\T_{1}\cup\T_{2}$ with $\T_{2}\subset\subset(0,1)^{d}$
compactly and such that $\partial\T_{2}$ is Lipschitz. Then, for
every $p\in[1,\infty)$ there exists $C$ depending only on $\T_{2}$,
$p$ and $d$ such that for every $u\in W^{1,p}(Y_{1})$: 
\begin{align}
\int_{\Rd}\left|\tilde{\cU}u\right|^{p} & \leq C\int_{Y_{1}}\left|u\right|^{p}\,,\label{eq:thm:Ext-per-isolated-1}\\
\int_{\Rd}\left|\nabla\left(\tilde{\cU}u\right)\right|^{p} & \leq C\int_{Y_{1}}\left|\nabla u\right|^{p}\,.\label{eq:thm:Ext-per-isolated-2}
\end{align}
\end{thm}

\begin{proof}
Since $\T_{2}\subset\subset(0,1)^{d}$ one proves by contradiction
the existence of $C>0$ such that 
\begin{equation}
\forall\varphi\in W^{1,p}((0,1)^{d}\backslash\T_{2})\,:\qquad\int_{\T_{1}}\left|\varphi\right|^{p}\leq C\left(\int_{\T_{1}}\left|\nabla\varphi\right|^{p}+\left|\fint_{\T_{1}}\varphi\right|\right)\,.\label{eq:thm:Ext-per-isolated-poincare}
\end{equation}
In what follows we write $\overline{\varphi}=\fint_{\T_{1}}\varphi$.
Since $\partial\T_{2}$ is Lipschitz, there exists a continuous operator
$\tilde{\cU}:\,W^{1,p}((0,1)^{d}\backslash\T_{2})\to W^{1,p}((0,1)^{d})$.
Due to (\ref{eq:thm:Ext-per-isolated-poincare}) it holds 
\begin{align*}
\int_{\T}\left|\cU\left(u-\overline{u}\right)+\overline{u}\right|^{p} & \leq C\int_{\T_{1}}\left|u\right|^{p}\,,\\
\int_{\T_{2}}\left|\nabla\left(\cU\left(u-\overline{u}\right)+\overline{u}\right)\right|^{p} & =\int_{\T_{2}}\left|\nabla\cU\left(u-\overline{u}\right)\right|^{p}\\
 & \leq C\left(\int_{\T_{1}}\left|u-\overline{u}\right|^{p}+\int_{\T_{1}}\left|\nabla\left(u-\overline{u}\right)\right|^{p}\right)\\
 & \leq C\int_{\T_{1}}\left|\nabla u\right|^{p}\,.
\end{align*}
For $u\in W^{1,p}(\bT_{1})$ and $k\in\Zd$, we define $\cU$ on $\Rd$
by applying it locally on every cell $I_{k}:=k+[0,1)^{d}$. Hence
$\cU$ satisfies (\ref{eq:thm:Ext-per-isolated-1})--(\ref{eq:thm:Ext-per-isolated-2}).
\end{proof}
The last proof heavily relied on the disconnectedness of $\bT_{2}$.
In case $\bT_{2}$ is connected, the ``gluing'' of the local extensions
is more delicate.
\begin{thm}
\label{thm:Ext-per-connected}Let $\T=\T_{1}\cup\T_{2}$ such that
,$\partial\bT_{1}$ is locally Lipschitz. Then there exist an extension
operator 
\[
\cU:\,W^{1,p}(\bT_{1})\to W^{1,p}(\Rd)
\]
such that for some $C>0$ depending only on $\delta$ and $p$ it
holds 
\begin{align}
\int_{\Rd}\left|\cU u\right|^{p} & \leq C\int_{\bT_{1}}\left|u\right|^{p}\,,\label{eq:thm:Ext-per-connected-1}\\
\int_{\Rd}\left|\nabla\left(\cU u\right)\right|^{p} & \leq C\int_{\bT_{1}}\left|\nabla u\right|^{p}\,.\label{eq:thm:Ext-per-connected-2}
\end{align}
\end{thm}

\begin{figure}

\centering{}\includegraphics[width=8cm]{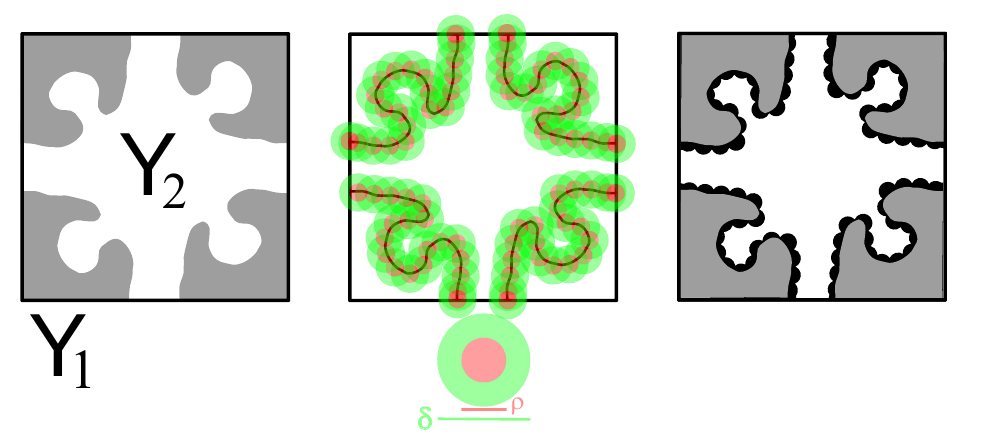}\caption{\label{fig:periodic-ext}Left: The periodic geometry $\protect\bT_{1}$
and $\protect\bT_{2}$. Middle: The boarder $\partial\protect\bT_{1}$
is covered by balls of a uniform size such that on each center $x_{i}$
there exists an extension operator from $\protect\bT_{1}\cap\protect\Ball{\delta}{x_{i}}$
to $\protect\bT_{2}\cap\protect\Ball{\rho}{x_{i}}$. Right: The microscopically
glued extension operator maps functions with support $\protect\bT_{1}$
onto functions with support in the black and gray domain.}
\end{figure}

Idea of Proof: In order to highlight the structure of the following
proof, let us explain how the extension operator is constructed. In
Figure \ref{fig:periodic-ext} we see on the left a Lipschitz surface
$\partial\bT_{1}$ with maximal Lipschitz constant $M$, which can
be locally covered by balls of radius $\rho=\delta\sqrt{4M^{2}+2}^{-1}$
(middle). Using the extension operators given by Lemma \ref{lem:uniform-extension-lemma},
we can extend $u$ to the red balls that intersect $\bT_{2}$. The
extension operators on the various red balls are then glued together
using a suitable partition of unity. However, this leads to steep
gradients in the black region on the right hand side, while $\cU u\equiv0$
in the white region. In particular, if $u(x)\equiv c$ is locally
constant, these gradients are of order $\frac{c}{\rho}$. Hence, proceeding
globally in this way, the gradient $\nabla\left(\tilde{\cU}u\right)$
cannot be bounded by $\nabla u$.

To avoid this problem, in Step 2 we use a mesoscopic correction: Writing
$K_{\alpha}:=(-\alpha,1+\alpha)^{d}$, and $K_{\alpha}(z)=z+K_{\alpha}$
for $z\in\Z^{d}$ with a partition of unity $\tilde{\eta}_{z}$ and
the local extension operator $\cU_{z}$ on $K_{\alpha}(z)$, we define
the global extension operator through:
\begin{equation}
\cU u:=\sum_{z\in\Zd}\tilde{\eta}_{z}\of{\tilde{\cU}_{z}\of{u-\tau_{z}u}+\tau_{z}u}\label{eq:thm:Ext-per-connected-def-global}
\end{equation}
where $\tau_{z}u=\fint_{B(z)}u$ for some suitable ball $B(z)$. By
this, we assign to the void space an averaged value of the surrounding
matrix. In Step 2 we heavily rely on the periodicity, which allows
to apply a $\T$-periodic partitioning to $\Rd$.
\begin{proof}
\textbf{Step 1 (Local extension operator on $(0,1)^{d}$):} W.l.o.g.
we can assume that $\delta\ll1$. Writing $K_{\alpha}:=(-\alpha,1+\alpha)^{d}$
the set $\partial\bT_{1}\cap K_{\delta}$ is precompact and can be
covered by a finite number of balls $B_{\rho/2}(x_{k})$, where $\rho=\delta\sqrt{4M^{2}+2}^{-1}$
and $\left(x_{k}\right)_{k=1,\dots,K}\subset\partial\bT_{1}\cap K_{\delta}$.

In what follows, let $\eta\in C_{0}^{\infty}(-1,1)$ be a positive
symmetric smooth function with $0<\eta(x)\leq1$ on $(-1,1)$, $\eta(0)=1$
and monotone on $(0,1)$. We denote $\eta_{0}:=\eta\circ\mathrm{dist}(\,\cdot\,,\partial\bT_{1}\,)$
and $\eta_{k}(x):=\eta\of{\rho^{-1}\left|x-x_{k}\right|}$ for $k\geq1$.
In what follows we identify $\eta_{k}$ with their periodized versions.
For every $k\geq0$ let $\tilde{\eta}_{k}=\of{\sum_{j=0}^{\infty}\eta_{j}}^{-1}\eta_{k}$
and note that $\tilde{\eta}_{k}$ defines a partition of unity on
$\partial\bT_{1}\cap K_{\delta}$. Writing $\cU_{i}$ for the corresponding
extension operator from Lemma \ref{lem:uniform-extension-lemma} on
$B_{\rho}(x_{i})$, we extend $u$ by $0$ to $\Rd\backslash\bT_{1}$
and consider 
\begin{align}
\tilde{\cU}:\;W^{1,p}(K_{2\delta}\cap\bT_{1}) & \to W^{1,p}(K_{\delta})\nonumber \\
\tilde{\cU}u & :=\sum_{i\in\N}\tilde{\eta}_{i}\cU_{i}u+\eta_{0}u\,.\label{eq:thm:Ext-per-connected-def-local}
\end{align}
For the following calculation, we further note that 
\begin{gather*}
\nabla\tilde{\eta}_{k}=\left(\sum_{j=0}^{\infty}\eta_{j}\right)^{-1}\nabla\eta_{k}-\left(\sum_{j=0}^{\infty}\eta_{j}\right)^{-2}\eta_{k}\sum_{j=0}^{\infty}\nabla\eta_{j}\\
\mbox{and }1\leq\sum_{j=0}^{\infty}\eta_{j}\leq\hat{N}\mbox{ as well as }\left|\sum_{j=0}^{\infty}\nabla\eta_{j}\right|\leq\hat{N}\norm{\nabla\eta}_{\infty}\,,
\end{gather*}
for some $\hat{N}$ depending only on the dimension $d$. Let $\boldsymbol{\tilde{B}}:=\left\{ B_{\rho}(x_{k})\right\} $.
For every $i\in\left\{ 1,\dots,k\right\} $, the number $\#\left\{ \tilde{B}_{j}\in\boldsymbol{\tilde{B}}\,|\;\tilde{B}_{j}\cap\tilde{B}_{i}\not=\emptyset\right\} $
of balls in $\boldsymbol{\tilde{B}}$ intersecting with $\tilde{B}_{i}$
is bounded by $\hat{N}$. On each ball we infer from Lemma \ref{lem:uniform-extension-lemma}
\begin{align*}
\int_{B_{i}}\left|\tilde{\eta}_{i}\cU_{i}u\right|^{p} & \leq7\int_{B_{\delta}(x_{i})\cap\bT_{1}}\left|u\right|^{p}\,,\\
\int_{B_{i}}\left|\nabla\left(\tilde{\eta}_{i}\cU_{i}u\right)\right|^{p} & \leq7\left\Vert \nabla\tilde{\eta}\right\Vert _{\infty}^{p}\int_{B_{\delta}(x_{i})\cap P}\left|u\right|^{p}+14M\int_{B_{\delta}(x_{i})\cap\bT_{1}}\left|\nabla u\right|^{p}\,.
\end{align*}
Similar estimates also hold for $\eta_{0}u$ and summing over $i$,
we obtain
\begin{align}
\int_{K_{\delta}}\left|\tilde{\cU}u\right|^{p} & \leq7\hat{N}\int_{K_{2\delta}\cap\bT_{1}}\left|u\right|^{p}\,,\label{eq:thm:Ext-per-connected-help-1}\\
\int_{K_{\delta}}\left|\nabla\left(\tilde{\cU}u\right)\right|^{p} & \leq7\hat{N}\frac{1}{\rho^{p}}\int_{K_{2\delta}\cap\bT_{1}}\left|u\right|^{p}+14M\hat{N}\int_{K_{2\delta}\cap\bT_{1}}\left|\nabla u\right|^{p}\,.\label{eq:thm:Ext-per-connected-help-2}
\end{align}
Now let $B\subset(2\delta,1-2\delta)^{d}\cap T_{1}$ be a ball with
positive radius. By a contradiction argument, we obtain 
\begin{equation}
\int_{K_{2\delta}\cap\bT_{1}}\left|u\right|^{p}\leq C\left(\int_{K_{2}\cap\bT_{1}}\left|\nabla u\right|^{p}+\left|\fint_{B}u\right|^{p}\right)\label{eq:thm:Ext-per-connected-help-2-a}
\end{equation}
and hence defining $\tau u:=\fint_{B}u$ we find 
\begin{equation}
\int_{K_{\delta}}\left|\nabla\left(\tilde{\cU}\left(u-\tau u\right)\right)\right|^{p}\leq28M\hat{N}\int_{K_{2}\cap\bT_{1}}\left|\nabla u\right|^{p}\,.\label{eq:thm:Ext-per-connected-help-2-b}
\end{equation}

\textbf{Step 2 (gluing together the local extension operators):} In
what follows, for every $z\in\Zd$ let $\left(\tilde{\cU}_{z}u\right)(\cdot):=\tilde{\cU}\of{u(\cdot+z)}(\cdot-z)$
the operator $\tilde{\cU}$ shifted onto the cell $z+K_{2\delta}$.
Given some positive $\overline{\eta}\in C_{c}(K_{\delta})$ with
$\overline{\eta}|_{(0,1)^{d}}\equiv1$ and symmetric w.r.t. the center
of $(0,1)^{d}$ we write $\eta_{z}:=\overline{\eta}(\cdot-z)$ such
that $\eta_{z}|_{z+(0,1)^{d}}\equiv1$ and introduce $\tilde{\eta}_{z}=\eta_{z}/\left(\sum_{x\in\Zd}\eta_{x}\right)$
which provide a $\left(0,1\right)^{d}$-periodic partition of unity.
Note that at each $x\in\Rd$ at most $2^{d}$ functions $\tilde{\eta}_{z}$
are different from $0$. We now define the operator $\cU$ according
to (\ref{eq:thm:Ext-per-connected-def-global}) with $\tau_{z}u:=\fint_{B+z}u$
and $\cU_{z}$ from Step 1 to find
\begin{align}
\int_{\Rd\backslash\bT_{1}}\left|\nabla\cU u\right|^{p} & =\int_{\Rd\backslash\bT_{1}}\left|\nabla\sum_{z\in\Zd}\tilde{\eta}_{z}\left(\tilde{\cU}_{z}\of{u-\tau_{z}u}+\tau_{z}u\right)\right|^{p}\nonumber \\
 & =\int_{\Rd\backslash\bT_{1}}\left|\sum_{z\in\Zd}\left[\nabla\tilde{\eta}_{z}\left(\tilde{\cU}_{z}\of{u-\tau_{z}u}+\tau_{z}u\right)+\tilde{\eta}_{z}\nabla\left(\tilde{\cU}_{z}\of{u-\tau_{z}u}\right)\right]\right|^{p}\nonumber \\
 & \leq C\left\Vert \nabla\tilde{\eta}\right\Vert _{\infty}^{p}\sum_{z\in\Zd}\left\Vert \tilde{\cU}_{z}\of{u-\tau_{z}u}\right\Vert _{L^{p}(z+K_{\delta}\backslash\bT_{1})}^{p}+C\int_{\Rd}\left|\sum_{z\in\Zd}\tau_{z}u\nabla\tilde{\eta}_{z}\right|^{p}\nonumber \\
 & \quad+C\sum_{z\in\Zd}\int_{z+K_{\delta}}\left|\nabla\left(\tilde{\cU}_{z}\of{u-\tau_{z}u}\right)\right|^{p}\,,\label{eq:thm:Ext-per-connected-help-7}
\end{align}
In order to derive an estimate on $\int_{\Rd}\left|\sum_{z\in\Zd}\tau_{z}u\nabla\tilde{\eta}_{z}\right|^{p}$,
note that for $z_{1},z_{2}\in\Zd$ and $x\in\Rd$ for all $i=1,\dots,d$
it holds $\partial_{i}\tilde{\eta}_{z_{1}}=-\partial_{i}\tilde{\eta}_{z_{2}}$
by symmetry and hence (writing $K_{\delta}(z)=z+K_{\delta}$
\[
\int_{\Rd}\left|\sum_{z\in\Zd}\tau_{z}u\nabla\tilde{\eta}_{z}\right|^{p}\leq\sum_{z_{1}\in\Zd}\sum_{z_{2}\in\Zd}\int_{K_{\delta}(z_{1})\cap K_{\delta}(z_{2})}\left|\nabla\tilde{\eta}_{z}\right|^{p}\left|\tau_{z_{1}}u-\tau_{z_{2}}u\right|^{p}\,.
\]
Thus, let $z_{1},z_{2}\in\Zd$ such that $\left(z_{1}+K_{2\delta}\right)\cap\left(z_{2}+K_{2\delta}\right)\not=\emptyset$.
Since $\bT_{1}$ is open and connected, one can prove 
\begin{align}
\left|\tau_{z_{1}}u-\tau_{z_{2}}u\right|^{p} & \leq C\int_{\bT_{1}\cap\left[\left(z_{1}+K_{2}\right)\cup\left(z_{2}+K_{2}\right)\right]}\left|\nabla u\right|^{p}\,,\label{eq:thm:Ext-per-connected-help-6}
\end{align}
where $C$ depends on $d$, $p$ and $\bT_{1}$. Together with (\ref{eq:thm:Ext-per-connected-help-2})--(\ref{eq:thm:Ext-per-connected-help-2-b})
we infer (\ref{eq:thm:Ext-per-connected-2}). Estimate (\ref{eq:thm:Ext-per-connected-1})
can be proved in an analogue way.
\end{proof}

\section{\label{sec:Nonlocal-regularity}Quantifying Nonlocal Regularity Properties
of the Geometry}

We have to account for three types of randomness. One is local, namely
the local Lipschitz regularity. The other is of global nature: We
have to find a partition of $\Rd$ such that on each partition cell
the extension can be explicitly constructed in a well defined way.
In the case of periodicity this is evidently trivial. However, since
we lack periodicity, we have to replace the periodic construction
of the extension operator in Section \ref{sec:Periodic-extension-theorem}
by something similar, but of stochastic nature. The key to this will
be the local $\left(\delta,M\right)$-regularity

The second problem will be overcome using a random distribution of
balls within $\bP(\omega)$ and a Voronoi tessellation which is such
that every Ball is contained in exactly one Voronoi cell. This construction
is based on the following observation.
\begin{lem}
\label{lem:there-are-balls}Let $\bP(\omega)$ be a stationary and
ergodic random open set such that 
\[
\P{\left(\bP\cap\I=\emptyset\right)}<1\,.
\]
Then there exists $\fr>0$ such that with positive probability $p_{\fr}>0$
the set $(0,1)^{d}\cap\bP$ contains a ball with radius $4\sqrt{d}\fr$.
\end{lem}

\begin{proof}
Assume that the lemma was wrong. Then for every $r>0$ the set $(0,1)^{d}\cap\bP$
almost surely does not contain an open ball with radius $r$. In particular
with probability $1$ the set $(0,1)^{d}\cap\bP$ does not contain
any ball. Hence $(0,1)^{d}\cap\bP=\emptyset$ almost surely, contradicting
the assumptions.
\end{proof}
The numbers $\fr$ and $p_{\fr}$ from Lemma \ref{lem:there-are-balls}
will finally lead to the concept of \emph{mesoscopic regularity }of
the geometry $\bP(\omega)$, see Definition \ref{def:meso-regularity}.
Particularly the number $\fr$ is important, as it affects also the
construction of the extension operator on the very microscopic level.

The third problem is the hardest: It is the necessity to quantify
connectedness of a domain geometrically and analytically.

\subsection{\label{subsec:Microscopic-Regularity}Microscopic Regularity}
\begin{defn}[$\left(\delta,M\right)$-Regularity]
\label{def:loc-del-M-reg}\nomenclature[delta]{$(\delta,M)$-regularity}{(Definition \ref{def:loc-del-M-reg})}Let
$\bP\subset\Rd$ be an open set.
\begin{enumerate}
\item $\bP$ is called \emph{$(\delta,M)$-regular }in $p_{0}\in\partial\bP$
if $M(p,\delta)<\infty$ and $M>M(p,\delta)$, i.e. there exists an
open set $U\subset\R^{d-1}$ and a Lipschitz continuous function $\phi:\,U\to\R$
with Lipschitz constant $M$ such that $\partial\bP\cap\Ball{\delta}{p_{0}}$
is graph of the function $\varphi:\,U\to\Rd\,,\;\tilde{x}\mapsto\left(\tilde{x},\phi\of{\tilde{x}}\right)$
in some suitable coordinate system.
\item $\bP$ is called \emph{locally $(\delta,M)$-regular }if for every
$p_{0}\in\partial\bP$ there exists $\delta(p_{0})>0$ and $M(p_{0})>0$
such that $\bP$ is $\left(\delta{\left(p_{0}\right)},M{\left(p_{0}\right)}\right)$-regular
in $p_{0}$.
\item $\bP$ is called \emph{(globally) $(\delta,M)$-regular }or \emph{minimally
smooth} if there exist constants $\delta,M>0$ s.t. $\bP$ is $\left(\delta,M\right)$-regular
in every $p_{0}\in\partial\bP$.
\end{enumerate}
\end{defn}

The concept of (global) $\left(\delta,M\right)$-regularity or minimally
smoothness can be found in the book \cite{stein2016singular}. The
theory of \cite{stein2016singular} was recently used in \cite{guillen2015quasistatic}
to derive extension theorems for minimally smooth stochastic geometries.
A first application of the concept of $\left(\delta,M\right)$-regularity
is the following Lemma, which is important for the application of
the Poincar\'e inequalities proved in Section \ref{sec:Preliminaries}
during the construction of the local extension operators in Section
\ref{sec:Extension-and-Trace-d-M}.

\begin{figure}
\centering{}\includegraphics[width=4cm]{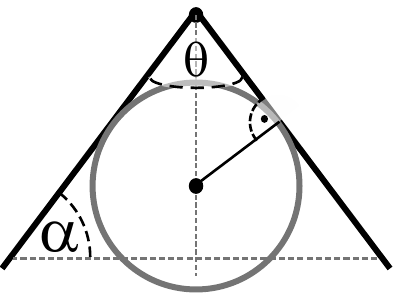}\caption{\label{fig:Ball-and-Cone}How to fit a ball into a cone.}
\end{figure}

\begin{lem}
\label{lem:small-ball-in-cone}Let $\bP$ be locally $\left(\delta,M\right)$-regular.
Then for every $p_{0}\in\partial\bP$ with $\delta(p_{0})>0$ the
following holds: For every $\delta<\delta\left(p_{0}\right)$ let
$M:=M(p_{0},\delta)>0$ such that $\partial\bP\cap\Ball{\delta}{p_{0}}$
is a $M(p_{0},\delta)$ Lipschitz manifold. Then there exists $y\in\bP$
with $\left|p_{0}-y\right|=\frac{\delta}{4}$ such that with $r\left(p_{0}\right):=\frac{\delta}{4\left(1+M\right)}$
it holds $\Ball{r\left(p_{0}\right)}y\subset\Ball{\delta/2}{p_{0}}$.
\end{lem}

\begin{proof}
We can assume that $\partial\bP$ is locally a cone as in Figure \ref{fig:Ball-and-Cone}.
With regard to Figure \ref{fig:Ball-and-Cone}, for $p_{0}\in\partial\bP$
with $\delta$ and $M$ as in the statement we can place a right circular
cone with vertex (apex) $p_{0}$ and axis $\nu$ and an aperture $\theta=\pi-2\arctan M$
inside $\Ball{\delta}{p_{0}}$, where $\alpha=\arctan M\left(p_{0}\right)$.
In other words, it holds $\tan\left(\alpha\right)=\tan\left(\frac{\pi-\theta}{2}\right)=M$.
Along the axis we may select $y$ with $\left|p_{0}-y\right|=\frac{\delta}{4}$.
Then the distance $R$ of $y$ to the cone is given through 
\[
\left|y-p_{0}\right|^{2}=R^{2}+R^{2}\tan^{2}\left(\frac{\pi-\theta}{2}\right)\quad\Rightarrow\quad R=\frac{\left|y-p_{0}\right|}{\sqrt{1+M^{2}}}\,.
\]
In particular $r\left(p_{0}\right)$ as defined above satisfies the
claim.
\end{proof}

\subsubsection*{Continuity properties of $\delta$, $M$ and $\varrho$}

Our main extension and trace theorems will be proved for \emph{locally
}$\left(\delta,M\right)$-regular sets $\bP$ and is based on some
simple properties of such sets which we summarize in this section.
Additionally we introduce the quantity $\rho$.
\begin{lem}
\label{lem:properties-delta-M-regular}Let $\fr>0$, $\bP$ be a locally
$\left(\delta,M\right)$-regular open set and let $M_{0}\in(0,+\infty]$
such that for every $p\in\partial\bP$ there exists $\delta>0$, $M<M_{0}$
such that $\partial\bP$ is $(\delta,M)$-regular in $p$. Define
for every $p\in\partial\bP$ 
\begin{align*}
\Delta{\left(p\right)} & :=\sup_{\delta<\fr}\left\{ \exists M\in(0,M_{0}):\,\bP\text{ is }\left(\delta,M\right)\text{-regular in }p\right\} \,,\quad\delta_{\Delta}{\left(p\right)}:=\frac{\Delta{\left(p\right)}}{2}
\end{align*}
Then $\partial\bP$ is $\delta_{\Delta}$-regular in the sense of
Definition \ref{def:eta-regular} with 
\[
f\of{p,\delta}:=\left(\exists M\in(0,M_{0}):\,\bP\text{ is }\left(\delta,M\right)\text{-regular in }p\right)\,.
\]
In particular, $\delta_{\Delta}:\,\partial\bP\to\R$ is locally Lipschitz
continuous with Lipschitz constant $4$ and for every $\eps\in\left(0,\frac{1}{2}\right)$
and $\tilde{p}\in\Ball{\eps\delta}p\cap\partial\bP$ it holds 
\begin{equation}
\frac{1-\eps}{1-2\eps}\delta_{\Delta}\of p>\delta_{\Delta}\of{\tilde{p}}>\delta_{\Delta}\of p-\left|p-\tilde{p}\right|>\left(1-\eps\right)\delta_{\Delta}\of p\,.\label{eq:lem:properties-delta-M-regular-3}
\end{equation}
\end{lem}

\begin{rem}
\label{rem:no-global-lipschitz}The latter lemma does \textbf{\emph{not
}}imply global Lipschitz regularity of $\delta_{\Delta}$. It could
be that $2\delta_{\Delta}\of p<\left|p-\tilde{p}\right|<3\delta_{\Delta}\of p$
and $p$ and $\tilde{p}$ are connected by a path inside $\partial\bP$
with the shortest path of length $10\delta_{\Delta}\of p$. Then Lemma
\ref{lem:properties-delta-M-regular} would have to be applied successively
along this path yielding an estimate of $\left|\delta_{\Delta}\of p-\delta_{\Delta}\of{\tilde{p}}\right|\leq40\left|p-\tilde{p}\right|$.
\end{rem}

\begin{proof}[Proof of Lemma \ref{lem:properties-delta-M-regular}]
 It is straight forward to verify that $f$ and $\delta_{\Delta}$
satisfy the conditions of Lemma \ref{lem:eta-lipschitz}.
\end{proof}
With regard to Lemma \ref{lem:uniform-extension-lemma}, the relevant
quantity for local extension operators is related to the variable $\delta\of p/\sqrt{4M\of p^{2}+2}$,
where $M(p)$ is the related Lipschitz constant. While we can quantify
$\delta\of p$ in terms of $\delta\of{\tilde{p}}$ and $\left|p-\tilde{p}\right|$,
this does not work for $M\of p$. Hence we cannot quantify $\delta\of p/\sqrt{4M\of p^{2}+2}$
in terms of its neighbors. This drawback is compensated by a variational
trick in the following statement.
\begin{lem}
\label{lem:rho-p-lsc} Let $\bP$ be locally $\left(\delta,M\right)$-regular
and let $\delta\leq\delta_{\Delta}$ satisfy (\ref{eq:lem:properties-delta-M-regular-3})
such that $\partial\bP$ is $\delta$-regular. For $p\in\partial\bP$
and $r<\delta\of p$ let $M_{r}\of p$ be the Lipschitz constant of
$\partial\bP$ in $\Ball rp$ and define \nomenclature[rho]{$\hat\rho(p)$}{$\inf\left\{ \delta\leq\delta(p)\,:\;\sup_{r<\delta}r\sqrt{4M_{r}\of p^{2}+2}^{-1}=\rho\right\} $ \eqref{eq:def-rhohat-of-p}}\nomenclature[rho]{$\rho(p)$}{$\sup_{r<\delta\of p}r\sqrt{4M_{r}\of p^{2}+2}^{-1}$ \eqref{eq:def-rho-of-p}}
\begin{align}
\rho{\left(p\right)} & :=\sup_{r<\delta\of p}r\sqrt{4M_{r}\of p^{2}+2}^{-1}\,,\label{eq:def-rho-of-p}\\
\hat{\rho}(p) & :=\inf\left\{ \delta\leq\delta(p)\,:\;\sup_{r<\delta}r\sqrt{4M_{r}\of p^{2}+2}^{-1}=\rho(p)\right\} \,.\label{eq:def-rhohat-of-p}
\end{align}
Then, $\rho$ and $\hat{\rho}$ are positive and locally Lipschitz
continuous on $\partial\bP$ with Lipschitz constant $4$ and $\partial\bP$
is $\rho$ and $\hat{\rho}$-regular in the sense of Definition \ref{def:eta-regular}.
In particular, for $\left|p-\tilde{p}\right|<\eps\rho{\left(p\right)}$
or $\left|p-\tilde{p}\right|<\eps\hat{\rho}{\left(p\right)}$ it holds
respectively 
\begin{gather*}
\frac{1-\eps}{1-2\eps}\rho\of p>\rho\of{\tilde{p}}>\rho\of p-\left|p-\tilde{p}\right|>\left(1-\eps\right)\rho\of p\,,\\
\frac{1-\eps}{1-2\eps}\hat{\rho}\of p>\hat{\rho}\of{\tilde{p}}>\hat{\rho}\of p-\left|p-\tilde{p}\right|>\left(1-\eps\right)\hat{\rho}\of p\,.
\end{gather*}
\end{lem}

\begin{rem}
For the same reason as in Remark \ref{rem:no-global-lipschitz}. The
latter lemma does \textbf{\emph{not }}imply global Lipschitz regularity
of $\rho$ or $\hat{\rho}$.
\end{rem}

\begin{proof}
Positivity is given by $\rho{\left(p\right)}\geq\delta{\left(p\right)}\sqrt{4M{\left(p\right)}^{2}+2}^{-1}$.
Let $\eps>0$ and $\left|p-\tilde{p}\right|<\eps\hat{\rho}{\left(p\right)}$.
For $r<\hat{\rho}(p)$ sufficiently large it holds $\left|p-\tilde{p}\right|<\eps r$
implying $\tilde{p}$ is $\left((1-\eps)r,M_{r}(p)\right)$-regular.
From here we conclude that $\partial\bP$ is $\hat{\rho}$-regular
and the above chain of inequalities follows from Lemma \ref{lem:eta-lipschitz}.

Now let $\left|p-\tilde{p}\right|<\eps\rho{\left(p\right)}<\eps\delta(p)$
implying $\delta{\left(\tilde{p}\right)}\geq\left(1-\eps\right)\delta{\left(p\right)}$
by Lemma \ref{lem:properties-delta-M-regular}. For every $\eta>0$
let $r_{\eta}<\delta\of p$ such that $\rho{\left(p\right)}\leq\left(1+\eta\right)r_{\eta}\sqrt{4M_{r_{\eta}}{\left(p\right)}^{2}+2}^{-1}$.
Since $r_{\eta}>\rho{\left(p\right)}$ and $\left|p-\tilde{p}\right|<\eps\rho{\left(p\right)}$
we find $\Ball{r_{\eta}}p\supset\Ball{\left(1-\eps\right)r_{\eta}}{\tilde{p}}$
 and hence $M_{\left(1-\eps\right)r_{\eta}}{\left(\tilde{p}\right)}\leq M_{r_{\eta}}{\left(p\right)}$.
This implies at the same time that $\partial\bP$ is $\rho$-regular
and that 
\[
\rho{\left(\tilde{p}\right)}\geq\frac{\left(1-\eps\right)r_{\eta}}{\sqrt{4M_{\left(1-\eps\right)r_{\eta}}{\left(\tilde{p}\right)}^{2}+2}}\geq\frac{\left(1-\eps\right)r_{\eta}}{\sqrt{4M_{r_{\eta}}{\left(p\right)}^{2}+2}}\geq\frac{\left(1-\eps\right)}{\left(1+\eta\right)}\rho{\left(p\right)}\,.
\]
Since $\eta$ was arbitrary, we conclude $\rho{\left(\tilde{p}\right)}\geq\left(1-\eps\right)\rho{\left(p\right)}$.
Moreover, we find $\left|p-\tilde{p}\right|<\frac{\eps}{1-\eps}\rho{\left(\tilde{p}\right)}$.
From here, we conclude with Lemma \ref{lem:eta-lipschitz}.
\end{proof}

\begin{lem}
\label{lem:M-eta}Let $\fr>0$, $\bP\subset\Rd$ be a locally $\left(\delta,M\right)$-regular
open set and let $M_{0}\in(0,+\infty]$ such that for every $p\in\partial\bP$
there exists $\delta>0$, $M<M_{0}$ such that $\partial\bP$ is $(\delta,M)$-regular
in $p$. For $\alpha\in(0,1]$ let $\eta(p)=\alpha\delta(p)$ from
Lemma \ref{lem:properties-delta-M-regular} or $\eta(p)=\alpha\rho(p)$
or $\eta(p)=\alpha\hat{\rho}(p)$ from Lemma \ref{lem:rho-p-lsc}
and define
\begin{align}
\rmM_{[\eta]}{\left(p\right)} & :=\inf_{\delta>\eta{\left(p\right)}}\inf_{M}\left\{ \,\bP\text{ is }\left(\delta,M\right)\text{-regular in }p\right\} \,.\label{eq:lem:M-eta-1}\\
\fm_{[\eta]}{\left(p,\xi\right)} & :=\inf_{\delta>\min\left\{ \delta{\left(p\right)},\xi\right\} }\inf_{M}\left\{ \,\bP\text{ is }\left(\delta,M\right)\text{-regular in }p\right\} \,,\label{eq:lem:M-eta-2}
\end{align}
Then, for fixed $\xi$, $\rmM_{[\eta]}\of{\cdot},\fm_{[\eta]}{\left(p,\xi\right)}:\;\partial\bP\to\R$
are upper semicontinuous and on each bounded measurable set $A\subset\Rd$
the quantity 
\begin{equation}
\rmM_{[\eta],A}:=\sup_{p\in\overline{A}\cap\partial\bP}M_{[\eta]}\of p\label{eq:def-M-set-A}
\end{equation}
 \nomenclature[MA]{$M_{[\eta]},\,M_{[\eta],A}$ }{($A$ a set) Equation \eqref{eq:def-M-set-A}, a quantity on $\partial\bP$}\nomenclature[m]{$\fm_{[\eta]}{\left(p,\xi\right)}$}{Lemma \ref{lem:M-eta}}with
$\rmM_{[\eta],A}=0$ if $\overline{A}\cap\partial\bP=\emptyset$ is
well defined. The functions 
\[
\rmM_{[\eta],A}:\;\Rd\to\R\,,\quad\rmM_{[\eta],A}\of x:=\rmM_{[\eta],A+x}\quad\text{with }\rmM_{[\eta],A}(0)=\rmM_{[\eta],A}
\]
are upper semicontinuous.
\end{lem}

\begin{rem}
\label{rem:difference-M-eta-M-eta}In order to prevent confusion,
let us note at this point that $M_{\eta}$ defined in (\ref{eq:lem:local-delta-M-construction-estimate-1b})
is different from $M_{[\eta]}$. In particular, $M_{\eta}$ is a quantity
on $\Rd$, while $M_{[\eta]}$ is a quantity on $\partial\bP$. Furthermore,
as the last lemma shows, $M_{[\eta]}$ is upper semi continuous, while
$M_{\eta}$ is only measurable.
\end{rem}

\begin{notation}
The infimum in (\ref{eq:lem:M-eta-1}) is a $\liminf$ for $\delta\searrow\eta(p)$.
We sometimes use the special notation 
\[
M_{[\eta],\fr}(x):=M_{[\eta],\Ball{\fr}0}(x)\,.
\]
\end{notation}

\begin{proof}[Proof of Lemma \ref{lem:M-eta}]
Let $p,\tilde{p}\in\partial\bP$ with $\left|p-\tilde{p}\right|<\eps\eta(p)$.
Writing $\tilde{\eps}:=\frac{\eps}{1-\eps}$ and $r\left(p,\eps\right):=\left(\frac{1}{1-2\eps}+\eps\right)\eta\of p$
and 
\[
M\of{p,\eps}:=\inf_{M}\left\{ \Ball{r\left(p,\eps\right)}p\cap\partial\bP\text{ is }M\text{-Lipschitz graph}\right\} 
\]
as well as we observe from $\eta$-regularity that $\Ball{\eta\left(\tilde{p}\right)}{\tilde{p}}\subset\Ball{r\left(p,\eps\right)}p$
and $\Ball{\eta\left(p\right)}p\subset\Ball{r\left(\tilde{p},\tilde{\eps}\right)}{\tilde{p}}$.
Hence we find 
\[
\rmM_{[\eta]}\of{\tilde{p}}\leq M\of{p,\eps}\,.
\]
Observing that $M\of{p,\eps}\searrow\rmM_{[\eta]}\of p$ as $\eps\to0$
we find $\limsup_{\tilde{p}\to p}\rmM_{[\eta]}\of{\tilde{p}}\leq\rmM_{[\eta]}\of p$
and $\rmM$ is u.s.c.

Let $x\to0$. First observe that $\rmM_{[\eta],A}=\max_{y\in\overline{A}}\rmM_{[\eta]}\of y$.
The set $\overline{A}$ is compact and hence $\overline{A}+x\to\overline{A}$
in the Hausdorff metric as $x\to0$. Let $y_{x}\in\overline{A}+x$
such that $\rmM_{[\eta]}\of{y_{x}}=\rmM_{[\eta],A}\left(x\right)$.
Since $\overline{A}+x\to\overline{A}$ w.l.o.g. we find $y_{x}\to y$
converges and $y\in\overline{A}$. Hence 
\[
\rmM_{[\eta]}\of y\geq\limsup_{x\to0}\rmM_{[\eta]}\of{y_{x}}=\limsup_{x\to0}\rmM_{[\eta],A}\of x\,.
\]
In particular, $M_{[\eta],A}\of{\cdot}$ is u.s.c. The u.s.c of $\fm_{[\eta]}{\left(p,\xi\right)}$
can be proved similarly.
\end{proof}

\begin{cor}
\label{cor:cover-boundary}Let $\fr>0$ and let $\bP\subset\Rd$ be
a locally $\left(\delta,M\right)$-regular open set, where we restrict
$\delta$ by $\delta\left(\cdot\right)\leq\frac{\fr}{4}$. Then there
exists a countable number of points $\left(p_{k}\right)_{k\in\N}\subset\partial\bP$
such that $\partial\bP$ is completely covered by balls $\Ball{\tilde{\rho}\left(p_{k}\right)}{p_{k}}$
where $\tilde{\rho}\left(p\right):=2^{-5}\rho\left(p\right)$. Writing
\[
\tilde{\rho}_{k}:=\tilde{\rho}\of{p_{k}}\,,\qquad\delta_{k}:=\delta\of{p_{k}}\,.
\]
For two such balls with $\Ball{\tilde{\rho}_{k}}{p_{k}}\cap\Ball{\tilde{\rho}_{i}}{p_{i}}\neq\emptyset$
it holds 
\begin{equation}
\begin{aligned} & \frac{15}{16}\tilde{\rho}_{i}\leq\tilde{\rho}_{k}\leq\frac{16}{15}\tilde{\rho}_{i}\\
\text{and}\quad & \frac{31}{15}\min\left\{ \tilde{\rho}_{i},\tilde{\rho}_{k}\right\} \geq\left|p_{i}-p_{k}\right|\geq\frac{1}{2}\max\left\{ \tilde{\rho}_{i},\tilde{\rho}_{k}\right\} \,.
\end{aligned}
\label{eq:cor:cover-boundary-h1}
\end{equation}
Furthermore, there exists $\fr_{k}\geq\frac{\tilde{\rho}_{k}}{32\left(1+\fm_{[\tilde{\rho}]}\of{p_{k},\tilde{\rho}_{k}/4}\right)}$
and $y_{k}$ such that $\Ball{\fr_{k}}{y_{k}}\subset\Ball{\tilde{\rho}_{k}/8}{p_{k}}\cap\bP$
and $\Ball{2\fr_{k}}{y_{k}}\cap\Ball{2\fr_{j}}{y_{j}}=\emptyset$
for $k\neq j$.
\end{cor}

\begin{proof}
The existence of the points and Balls satisfying (\ref{eq:cor:cover-boundary-h1})
follows from Theorem \ref{thm:delta-M-rho-covering}, in particular
(\ref{eq:thm:delta-M-rho-covering}). It holds for $\Ball{\tilde{\rho}_{k}}{p_{k}}\cap\Ball{\tilde{\rho}_{i}}{p_{i}}\neq\emptyset$
\[
\left|p_{i}-p_{k}\right|\leq\tilde{\rho_{i}}+\tilde{\rho}_{k}\leq\left(\frac{16}{15}+1\right)\tilde{\rho}_{i}\,.
\]
Lemma \ref{lem:small-ball-in-cone} yields existence of $y_{k}$ such
that $\Ball{\fr_{k}}{y_{k}}\subset\Ball{\tilde{\rho}_{k}/8}{p_{k}}\cap\bP$.
The latter implies $\Ball{\fr_{k}}{y_{k}}\cap\Ball{\fr_{j}}{y_{j}}=\emptyset$
for $k\neq j$.
\end{proof}

\subsubsection*{Measurability and Integrability of Extended Variables}
\begin{lem}
\label{lem:delta-rho-M-measurable}Let $\fr>0$, $\bP\subset\Rd$
be a locally $\left(\delta,M\right)$-regular open set and let $M_{0}\in(0,+\infty]$
such that for every $p\in\partial\bP$ there exists $\delta>0$, $M<M_{0}$
such that $\partial\bP$ is $(\delta,M)$-regular in $p$. For $\alpha\in(0,1]$
let $\eta(p)=\alpha\delta(p)$ from Lemma \ref{lem:properties-delta-M-regular}
or $\eta(p)=\alpha\rho(p)$ or $\eta(p)=\alpha\hat{\rho}(p)$ from
Lemma \ref{lem:rho-p-lsc} and define\nomenclature[M eta]{$\tilde{M}_{\eta}(x)$}{Equation \eqref{eq:lem:local-delta-M-construction-estimate-1b}, a quantity on $\Rd$}
\begin{align}
\tilde{\eta}\of x & :=\inf\left\{ \eta\of{\tilde{x}}\,:\;\tilde{x}\in\partial\bP\,\text{s.t. }x\in\Ball{\frac{1}{8}\eta\of{\tilde{x}}}{\tilde{x}}\right\} \,,\label{eq:lem:local-delta-M-construction-estimate-1}\\
M_{[\eta],\Rd}(x) & :=\sup\left\{ M_{[\eta]}(\tilde{x})\,:\;\tilde{x}\in\partial\bP\,\text{s.t. }x\in\overline{\Ball{\eta\of{\tilde{x}}}{\tilde{x}}}\right\} \,,\label{eq:lem:local-delta-M-construction-estimate-1b}
\end{align}
where $\inf\emptyset=\sup\emptyset:=0$ for notational convenience.
Furthermore, write $A:=F^{-1}\of{(0,\frac{3}{2}\fr)}$ for 
\[
F:=\inf_{p\in\partial\bP}f_{p}\,,\qquad f_{p}\of x:=\begin{cases}
\eta\of p & \text{if }x\in\Ball{\frac{\eta\of p}{4}}p\\
2\fr & \text{else}
\end{cases}\,.
\]
then $\tilde{\eta}$ is measurable and $M_{[\eta]}$ is upper semicontinuous.
\end{lem}

\begin{proof}
Step 1: Let $\left(p_{i}\right)_{i\in\N}\subset\partial\bP$ be a
dense subset. If $x\in\Ball{\frac{1}{8}\eta\of p}p$ for some $p\in\partial\bP$
then also $x\in\Ball{\frac{1}{8}\eta\of{\tilde{p}}}{\tilde{p}}$ for
$\left|p-\tilde{p}\right|$ sufficiently small, by continuity of $\eta$.
For every $p\in\partial\bP$ consider the function $f_{p}\of x$ as
introduced above. Then every $f_{p}$ is upper semicontinuous and
$F:=\inf_{i\in\N}f_{p_{i}}$ is measurable. In particular, the set
$A$ is measurable and thus $\tilde{\eta}=\chi_{A}F$ is measurable.

Step 2: We show that for every $a\in\R$ the preimage $M_{[\eta],\Rd}^{-1}\of{[a,+\infty)}$
is closed. Let $\left(x_{k}\right)_{k\in\N}$ be a sequence with $M_{[\eta],\Rd}\of{x_{k}}\in[a,+\infty)$.
Let $\left(p_{k}\right)\subset\partial\bP$ be a sequence with $\left|x_{k}-p_{k}\right|\leq\eta\of{p_{k}}$.
W.l.o.g. assume $p_{k}\to p\in\partial\bP$ and $x_{k}\to x\in\Rd$.
Since $\eta$ is continuous, it follows $\left|x-p\right|\leq\eta\of p$.
On the other hand $M_{[\eta]}(p)\geq\limsup_{k\to\infty}M_{[\eta]}(p_{k})$
and thus $M_{[\eta],\Rd}\of x\geq M_{[\eta]}(p)\geq a$.
\end{proof}
\begin{lem}
\label{lem:delta-tilde-construction-estimate}Under the assumptions
of Lemma \ref{lem:delta-rho-M-measurable} there exists a constant
$C>0$ only depending on the dimension $d$ such that for every bounded
open domain $\bQ$ it holds
\begin{align}
\int_{A\cap\bQ}\chi_{\tilde{\eta}>0}\tilde{\eta}^{-\alpha} & \leq C\int_{A\cap\Ball{\frac{\fr}{4}}{\bQ}\cap\partial\bP}\eta^{1-\alpha}M_{[\frac{\eta}{4}],\Rd}^{d-2}\,,\label{eq:lem:delta-tilde-construction-estimate-1}\\
\int_{A\cap\bQ}\tilde{\eta}^{-\alpha}M_{[\frac{\eta}{8}]}^{r} & \leq C\int_{A\cap\Ball{\frac{\fr}{4}}{\bQ}\cap\partial\bP}\eta^{1-\alpha}M_{[\frac{\eta}{4}],,\Rd}^{r+d-2}\,.\label{eq:lem:delta-tilde-construction-estimate-2}
\end{align}
Finally, it holds 
\begin{equation}
x\in\Ball{\frac{1}{8}\eta\of p}p\quad\Rightarrow\quad\eta\of p>\tilde{\eta}\of x>\frac{3}{4}\eta\of p\,.\label{eq:lem:local-delta-M-construction-estimate-2}
\end{equation}
\end{lem}

\begin{rem}
Estimates (\ref{eq:lem:delta-tilde-construction-estimate-1})--(\ref{eq:lem:delta-tilde-construction-estimate-2})
are only rough estimates and better results could be obtained via
more sophisticated calculations that make use of particular features
of given geometries.
\end{rem}

\begin{proof}
Step 1: Given $x\in\Rd$ with $\tilde{\eta}(x)>0$ let 
\begin{equation}
p_{x}\in\argmin\left\{ \eta\of{\tilde{x}}\,:\;\tilde{x}\in\partial\bP\,\text{s.t. }x\in\overline{\Ball{\frac{1}{8}\eta\of{\tilde{x}}}{\tilde{x}}}\right\} \,.\label{eq:lem:delta-tilde-construction-estimate-h-1}
\end{equation}
Such $p_{x}$ exists because $\partial\bP$ is locally compact. We
observe with help of the definition of $p_{x}$, the triangle inequality
and (\ref{eq:thm:delta-M-rho-covering-a})
\[
x\in\Ball{\frac{1}{8}\eta\of p}p\quad\Rightarrow\quad\eta\of{p_{x}}\leq\eta\of p\quad\Rightarrow\quad\left|p-p_{x}\right|<\frac{\eta\of p}{4}\quad\Rightarrow\quad\eta\of{p_{x}}>\frac{3}{4}\eta\of p\,.
\]
The last line particularly implies (\ref{eq:lem:local-delta-M-construction-estimate-2})
and 
\[
\forall p\in\partial\bP\;\forall x\in\Ball{\frac{\eta\of p}{8}}p:\quad\tilde{\eta}\of x>\frac{3\eta\of p}{4}\,.
\]
Step 2: By Theorem \ref{thm:delta-M-rho-covering} we can chose a
countable number of points $\left(p_{k}\right)_{k\in\N}\subset\partial\bP$
such that $\Gamma=\partial\bP$ is completely covered by balls $B_{k}:=\Ball{\xi\of{p_{k}}}{p_{k}}$
where $\xi\of p:=2^{-4}\eta\of p$. For simplicity of notation we
write $\eta_{k}:=\eta\of{p_{k}}$ and $\xi_{k}:=\xi\of{p_{k}}$. Assume
$x\in A$ with $p_{x}\in\Gamma$ given by (\ref{eq:lem:delta-tilde-construction-estimate-h-1}).
Since the balls $B_{k}$ cover $\Gamma$, there exists $p_{k}$ with
$\left|p_{x}-p_{k}\right|<\xi_{k}=2^{-4}\eta_{k}$, implying $\eta\of{p_{x}}<\frac{2^{4}}{2^{4}-1}\eta_{k}$
and hence $\left|x-p_{k}\right|\leq\left(2^{-4}+\frac{2^{-3}2^{4}}{2^{4}-1}\right)\eta_{k}<\frac{3}{16}\eta_{k}$.
Hence we find
\[
\forall x\in A\,\;\exists p_{k}\;:\quad x\in\Ball{\frac{3}{16}\eta_{k}}{p_{k}}\,.
\]
Step 3: For $p\in\Gamma$ with $x\in\Ball{\frac{1}{4}\eta\of p}p\cap\Ball{\frac{1}{8}\eta\of{p_{x}}}{p_{x}}$
we can distinguish two cases:
\begin{enumerate}
\item $\eta\of p\geq\eta\of{p_{x}}$: Then $p_{x}\in\Ball{\frac{3}{8}\eta\of p}p$
and hence $\eta\of{p_{x}}\ge\frac{5}{8}\eta\of p$ by (\ref{eq:thm:delta-M-rho-covering-a}).
\item $\eta\of p<\eta\of{p_{x}}$: Then $p\in\Ball{\frac{3}{8}\eta\of{p_{x}}}{p_{x}}$
and hence$\eta\of{p_{x}}>\frac{1-\frac{3}{8}}{1-\frac{6}{8}}\eta\of p=\frac{5}{2}\eta\of p$
by (\ref{eq:thm:delta-M-rho-covering-a}).
\end{enumerate}
and hence 
\[
x\in\Ball{\frac{1}{4}\eta\of p}p\qquad\Rightarrow\qquad\tilde{\eta}\of x=\eta\of{p_{x}}>\frac{5}{8}\eta\of p\,.
\]
Step 4:  Let $k\in\N$ be fixed and define $B_{k}=\Ball{\frac{1}{4}\eta_{k}}{p_{k}}$,
$M_{k}:=M(p_{k},\frac{1}{4}\eta_{k})$. By construction, every $B_{j}$
with $B_{j}\cap B_{k}\neq\emptyset$ satisfies $\eta_{j}\geq\frac{1}{2}\eta_{k}$
and hence if $B_{j}\cap B_{k}\neq\emptyset$ and $B_{i}\cap B_{j}\neq\emptyset$
we find $\left|p_{j}-p_{i}\right|\geq\frac{1}{4}\eta_{k}$ and $\left|p_{j}-p_{k}\right|\leq3\eta_{k}$.
This implies that 
\[
\exists C>0:\;\forall k\quad\#\left\{ j\,:\;B_{j}\cap B_{k}\neq\emptyset\right\} \leq C\,.
\]
We further observe that the minimal surface of $B_{k}\cap\partial\bP$
is given in case when $B_{k}\cap\partial\bP$ is a cone with opening
angle $\frac{\pi}{2}-\arctan M\of{p_{k}}$. The surface area of $B_{k}\cap\partial\bP$
in this case is bounded by $\frac{1}{d-1}\left|\S^{d-2}\right|\eta_{k}^{d-1}\left(M_{k}+1\right)^{2-d}$.
This particularly implies up to a constant independent from $k$:
\begin{align*}
\int_{A\cap\bQ\cap\bP}\tilde{\eta}^{-\alpha} & \lesssim\sum_{k:\,B_{k}\cap\bQ\neq\emptyset}\int_{A\cap B_{k}\cap\bP}\eta_{k}^{-\alpha}\\
 & \lesssim\sum_{k:\,B_{k}\cap\bQ\neq\emptyset}\int_{A\cap B_{k}\cap\partial\bP}\eta^{1-\alpha}M_{[\frac{\eta}{4}]}^{d-2}\\
 & \lesssim\int_{A\cap\Ball{\frac{\fr}{4}}{\bQ}\cap\partial\bP}\eta^{1-\alpha}M_{[\frac{\eta}{4}]}^{d-2}\,.
\end{align*}
The second integral formula follows in a similar way.
\end{proof}

\subsection{\label{subsec:Mesoscopic-Regularity}Mesoscopic Regularity and Isotropic
Cone Mixing}

In what follows, we built upon Lemma \ref{lem:there-are-balls} to
motivate our definition of mesoscopic regularity (Definition \ref{def:iso-cone-mix}
by the following two Lemmas.
\begin{lem}
\label{lem:intensity-mu-omega-r}Recall $\X_{r}\of{\bP\of{\omega}}:=2r\Zd\cap\bP_{-r}\of{\omega}=\left\{ x\in2r\Zd\,:\;\Ball{\frac{r}{2}}x\subset\bP\right\} $
from Lemma \ref{lem:X-r-stationary} and assume $\fr<\frac{1}{8}$.
Let 
\[
\mu_{\omega,\fr}(\,\cdot\,):=\lebesgueL\of{\,\cdot\,\cap\Ball{\frac{\fr}{2}}{\X_{\fr}\of{\bP\of{\omega}}}}\,,
\]
then there exists a constant $\lambda_{0}>0$ such that for almost
every $\omega\in\Omega$ it holds for all regular convex averaging
sequences $A_{n}$
\begin{equation}
\liminf_{n\to\infty}\left|A_{n}\right|^{-1}\mu_{\omega,\fr}(A_{n})\geq\lambda_{0}\,.\label{eq:lem:intensity-mu-omega-r-1}
\end{equation}
\end{lem}

\begin{rem*}
Note that $\mu_{\omega,\fr}$ is stationary with respect to shifts
in $2\fr\Zd$ but not ergodic in general. It corresponds to the function
$\left|\S^{d-1}\right|\left(\frac{\fr}{2}\right)^{d}\X_{\fr}\of{\bP\of{\omega}}$
on $2\fr\Zd$ and by stationarity, Theorem \ref{thm:Ergodic-Theorem-discrete}
yields convergence
\[
\left|A_{n}\right|^{-1}\sum_{z\in A_{n}\cap\X_{r}\left(\bP\left(\omega\right)\right)}\left|\S^{d-1}\right|\left(\frac{\fr}{2}\right)^{d}\X_{\fr}\of{\bP\of{\omega}}\to\E\left(\mu_{\omega,\fr}|\sI\right)\,.
\]
Inequality (\ref{eq:lem:intensity-mu-omega-r-1}) implies $\E\of{\mu_{\omega,\fr}|\sI}\geq\lambda_{0}$
a.s.
\end{rem*}
\begin{proof}[Proof of Lemma \ref{lem:intensity-mu-omega-r}]
Due to Lemma \ref{lem:there-are-balls}, with probability $p_{\fr}>0$
the set $\I\cap\bP$ contains a ball $\Ball{4\sqrt{d}\fr}x$ and thus
the set $\left(\I\cap\bP\right)_{-3\sqrt{d}\fr}$ contains a ball
$\Ball{\sqrt{d}\fr}x$. In particular, the stationary ergodic random
measure $\tilde{\mu}_{\omega}(\,\cdot\,):=\lebesgueL\of{\,\cdot\,\cap\bP_{-3\sqrt{d}\fr}\left(\omega\right)}$
has positive intensity $\tilde{\lambda}_{0}>p_{\fr}\left|\S^{d-1}\left(\sqrt{d}\fr\right)^{d}\right|$.
Let $\tilde{\mu}_{\omega}\of{\I_{-3\sqrt{d}\fr}}>0$. Then there exists
$x\in\left(\I\cap\bP\right)_{-3\sqrt{d}\fr}$ and thus there exists
$x\in\X_{\fr}\of{\bP\of{\omega}}\cap\I$ with $\Ball{\frac{\fr}{2}}x\subset\I$.
It follows 
\[
\tilde{\mu}_{\omega}\of{\I_{-3\sqrt{d}\fr}}\leq\tilde{\mu}_{\omega}\left(\I\right)\leq1=\frac{2^{d}}{\fr^{d}\left|\S^{d-1}\right|}\left|\Ball{\frac{\fr}{2}}x\right|\leq\frac{2^{d}}{\fr^{d}\left|\S^{d-1}\right|}\mu_{\omega,\fr}\of{\I}\,.
\]
Since $\tilde{\mu}_{\omega}$ is stationary ergodic and $A_{n}$ is
regular we find 
\begin{align*}
p_{\fr}\left|\S^{d-1}\left(\sqrt{d}\fr\right)^{d}\right| & <\tilde{\mu}_{\omega}\of{\I_{-3\sqrt{d}\fr}}<\tilde{\lambda}_{0}\leq\E\of{\tilde{\mu}_{\omega}\left(\I\right)}=\liminf_{n\to\infty}\left|A_{n}\right|^{-1}\tilde{\mu}_{\omega}(A_{n})\\
 & \leq\frac{2^{d}}{\fr^{d}\left|\S^{d-1}\right|}\liminf_{n\to\infty}\left|A_{n}\right|^{-1}\mu_{\omega,\fr}(A_{n})\,.
\end{align*}
\end{proof}
Lemma \ref{lem:there-are-balls} suggests that starting at the origin
and walking into an arbitrary direction, it is almost impossible to
not meet a ball of radius $\fr$ that fully lies within $\bP(\omega)$.
However, this is in general wrong, as for a given fixed direction
one may already find periodic counter examples. In what follows, we
will therefore use the weaker concept of isotropic cone mixing (Definition
\ref{def:iso-cone-mix}) which is based on the following observation:
\begin{lem}
\label{lem:Cone-regularity}Let $\left(\left(\nu_{j},\alpha_{j}\right)\right)_{j\in\N}\subset\S^{d-1}\times\left(0,\frac{\pi}{2}\right)$
be countable. Then for every $x\in\Rd$ and each $j\in\N$ there holds
\[
\lim_{R\to\infty}\P\of{\X_{\fr}\left(\bP\right)\cap\cone_{\nu_{j},\alpha_{j},R}(x)\neq\emptyset}=1\,.
\]
\end{lem}

\begin{proof}
By stationarity, we can assume $x=0$ and by Lemma \ref{lem:intensity-mu-omega-r}
the random measure $\mu_{\omega,\fr}$ has strictly positive intensity.

We write $\cone_{R}:=\cone_{\nu_{j},\alpha_{j},e^{R}}(0)$ and denote
by $\tilde{\cone}_{R}$ the cone with the same base as $\cone_{R}$
but with apex $-\nu_{j}R$. Then $\tilde{\cone}_{R}$ is a regular
convex averaging sequence. Furthermore, it holds $\lebesgueL\of{\tilde{\cone}_{R}}/\lebesgueL\of{\left(\cone_{R}\right)}=\frac{e^{R}+R}{e^{R}}\to1$
implying $\lebesgueL\of{\tilde{\cone}_{R}\backslash\cone_{R}}\,\lebesgueL\of{\tilde{\cone}_{R}}^{-1}\to0$
as $R\to\infty$. Thus
\begin{align*}
\mu_{0} & \leq\liminf_{R\to\infty}\lebesgueL\of{\tilde{\cone}_{R}}^{-1}\mu_{\omega,\fr}\of{\tilde{\cone}_{R}}\\
 & =\liminf_{R\to\infty}\lebesgueL\of{\tilde{\cone}_{R}}^{-1}\left(\mu_{\omega,\fr}\of{\cone_{R}}+\mu_{\omega,\fr}\of{\tilde{\cone}_{R}\backslash\cone_{R}}\right)=\liminf_{R\to\infty}\lebesgueL\of{\tilde{\cone}_{R}}^{-1}\mu_{\omega,\fr}\of{\cone_{R}}\,.
\end{align*}
where we use $0\leq\mu_{\omega,\fr}\of{\tilde{\cone}_{R}\backslash\cone_{R}}\leq\lebesgueL\of{\tilde{\cone}_{R}\backslash\cone_{R}}\to0$
as $R\to\infty$. We infer that $\mu_{\omega,\fr}\of{\left(\cone_{R}\right)}=O\of{R^{d}}$
and hence the statement ($\cone_{R}$ has to contain infinitely many
balls $\Ball{\frac{\fr}{2}}{x_{l}}$).
\end{proof}
The following definition is a quantification of Lemma \ref{lem:Cone-regularity}.
\begin{defn}[Isotropic cone mixing]
\label{def:iso-cone-mix}\nomenclature[Isotropic cone mixing]{Isotropic cone mixing}{Definition \ref{def:iso-cone-mix}}A
random set $\bP(\omega)$ is isotropic cone mixing if there exists
a jointly stationary point process $\X$ in $\Rd$ or $2\fr\Zd$,
$\fr>0$, such that almost surely two points $x,y\in\X$ have mutual
minimal distance $2\fr$ and such that $\Ball{\frac{\fr}{2}}{\X(\omega)}\subset\bP(\omega)$.
Further there exists a function $f(R)$ with $f(R)\to0$ as $\R\to\infty$
and $\alpha\in\left(0,\frac{\pi}{2}\right)$ such that with $\bE:=\left\{ e_{1},\dots e_{d}\right\} \cup\left\{ -e_{1},\dots-e_{d}\right\} $
($\left\{ e_{1},\dots e_{d}\right\} $ being the canonical basis of
$\Rd$) 
\begin{equation}
\P\of{\forall e\in\bE:\;\X\cap\cone_{e,\alpha,R}(0)\neq\emptyset}\geq1-f\of R\,.\label{eq:def-iso-cone-mixing}
\end{equation}
\end{defn}

\begin{criterion}[A simple sufficient criterion for (\ref{eq:def-iso-cone-mixing})]
\label{cri:stat-erg-ball}Let $\bP$ be a stationary ergodic random
open set, let $\tilde{f}$ be a positive, monotonically decreasing
function with $\tilde{f}(R)\to0$ as $R\to\infty$ and let $\fr>0$
s.t.
\begin{equation}
\P\of{\exists x\in\Ball R0:\,\Ball{4\sqrt{d}\fr}x\subset\Ball R0\cap\bP}\geq1-\tilde{f}\of R\,.\label{eq:cri:stat-erg-ball}
\end{equation}
Then $\bP$ is isotropic cone mixing with $f(R)=2d\tilde{f}\of{\left(a+1\right)^{-1}R}$
and with $\X=\X_{\fr}(\bP)$. Vice versa, if $\bP$ is isotropic cone
mixing for $f$ then $\bP$ satisfies (\ref{eq:cri:stat-erg-ball})
with $\tilde{f}=f$.
\end{criterion}

\begin{defn}[Mesoscopic regularity]
\label{def:meso-regularity}\nomenclature[Mesoscopic regularity]{Mesoscopic regularity}{Definition \ref{def:meso-regularity}}A
random set $\bP$ satisfying Criterion \ref{cri:stat-erg-ball} is
also called mesoscopically regular and $\tilde{f}$ is the regularity.
$\bP$ is called polynomially (exponentially) regular if $1/\tilde{f}$
grows polynomially (exponentially).
\end{defn}

\begin{proof}[Proof of Criterion \ref{cri:stat-erg-ball}]
Because of $\P\of{A\cup B}\leq\P\of A+\P\of B$ it holds for $a>1$
\[
\P\of{\exists e\in\bE\,:\;\nexists x\in\Ball R{aRe}:\,\Ball{4\sqrt{d}\fr}x\subset\Ball R{aRe}\cap\bP}\leq2d\tilde{f}(R)\,.
\]
The existence of $\Ball{4\sqrt{d}\fr}x\subset\Ball R{aRe}\cap\bP\of{\omega}$
implies that there exists at least one $x\in\X_{\fr}\left(\bP\left(\omega\right)\right)$
such that $\Ball{\frac{\fr}{2}}x\subset\Ball R{aRe}\cap\bP\of{\omega}$
and we find 
\[
\P\of{\exists e\in\bE\,:\;\nexists x\in\X_{\fr}\of{\bP}:\,\Ball{\frac{\fr}{2}}x\subset\Ball R{aRe}\cap\bP}\leq2d\tilde{f}(R)\,.
\]
In particular, for $\alpha=\arccos a$ and $R$ large enough we discover
\[
\P\left(\exists e\in\bE\,:\;\X_{\fr}\of{\bP}\cap\cone_{e,\alpha,\left(a+1\right)R}\left(0\right)=\emptyset\right)\leq2d\tilde{f}(R)\,.
\]
The relation (\ref{eq:def-iso-cone-mixing}) holds with $f(R)=2d\tilde{f}\of{\left(a+1\right)^{-1}R}$.

The other direction is evident.
\end{proof}
Note that Criterion \ref{cri:stat-erg-ball} is much easier to verify
than Definition \ref{def:iso-cone-mix}. However, Definition \ref{def:iso-cone-mix}
is formulated more generally and is easier to handle in the proofs
below, that are all built on properties of Voronoi meshes.

The formulation of Definition \ref{def:iso-cone-mix} is particularly
useful for the following statement.
\begin{lem}[Size distribution of cells]
\label{lem:Iso-cone-geo-estimate}Let $\bP(\omega)$ be a stationary
and ergodic random open set that is isotropic cone mixing for $\X(\omega)$,
$\fr>0$, $f:\,(0,\infty)\to\R$ and $\alpha\in\left(0,\frac{\pi}{2}\right)$.
Then $\X$ and its Voronoi tessellation have the following properties:
\begin{enumerate}
\item If $G(x)$ is the open Voronoi cell of $x\in\X(\omega)$  with diameter
$d\of x$ then $d$ is jointly stationary with $\X$ and for some
constant $C_{\alpha}>0$ depending only on $\alpha$
\begin{equation}
\P(d(x)>D)<f\of{C_{\alpha}^{-1}\frac{D}{2}}\,.\label{eq:estim-diam-Voronoi}
\end{equation}
\item For $x\in\X(\omega)$  let $\cI\of x:=\left\{ y\in\X\,:\;G\of y\cap\Ball{\fr}{G\of x}\not=\emptyset\right\} $.
Then 
\begin{equation}
\#\cI\of x\leq\left(\frac{4d\of x}{\fr}\right)^{d}\,.\label{eq:estim-diam-Voronoi-2}
\end{equation}
\end{enumerate}
\end{lem}

\begin{proof}
1. W.l.o.g. let $x_{k}=0$. The first part follows from the definition
of isotropic cone mixing: We take arbitrary points $x_{\pm j}\in C_{\pm\e_{j},\alpha,R}(0)\cap\X$.
Then the planes given by the respective equations $\left(x-\frac{1}{2}x_{\pm j}\right)\cdot x_{\pm j}=0$
define a bounded cell around $0$, with a maximal diameter $D(\alpha,R)=2C_{\alpha}R$
which is proportional to $R$. The constant $C_{\alpha}$ depends
nonlinearly on $\alpha$ with $C_{\alpha}\to\infty$ as $\alpha\to\frac{\pi}{2}$.
Estimate (\ref{eq:estim-diam-Voronoi}) can now be concluded from
the relation between $R$ and $D(\alpha,R)$ and from (\ref{eq:def-iso-cone-mixing}).

2. This follows from Lemma \ref{eq:lem:estim-diam-Voronoi-cells}.
\end{proof}
\begin{lem}
\label{lem:estim-E-fa-fb}Let $\X_{\fr}$ be a stationary and ergodic
random random points process with minimal mutual distance $2\fr$
for $\fr>0$ and let $f:\,(0,\infty)\to\R$ be such that the Voronoi
tessellation of $\X$ has the property 
\[
\forall x\in\fr\Zd\,:\quad\P(d(x)>D)=f\of D\,.
\]
Furthermore, let $n,s:\,\X_{\fr}\to[1,\infty)$ be measurable and
i.i.d. among $\X_{\fr}$ and let $n,s,d$ be independent from each
other. Let 
\[
G_{n(x)}\of x=n(x)\left(G\of x-x\right)+x
\]
 be the cell $G\of x$ enlarged by the factor $n(x)$, let $d(x)=\diam G(x)$
and let 
\begin{align*}
\fb_{n}\of y & :=\sum_{x\in\X_{\fr}}\chi_{G_{n}\of x}d(x)^{\eta}s(x)^{\xi}n(x)^{\zeta}\,,
\end{align*}
where $\eta,\xi,\zeta>0$ is a constant. Then $\fb_{n}$ is jointly
stationary with $\X_{\fr}$ and for every $r>1$ there exists $C\in(0,+\infty)$
such that 
\begin{multline}
\E\of{\fb_{n}^{p}} \\ \leq C\left(\sum_{k,N=1}^{\infty}\left(k+1\right)^{d\left(p+1\right)+\eta p+r\left(p-1\right)}\left(S+1\right)^{\xi p+r\left(p-1\right)}\left(N+1\right)^{d\left(p+1\right)+\zeta p+r\left(p-1\right)}\P_{d,k}\P_{n,N}\P_{s,S}\right)\,.\label{eq:lem:estim-E-fa-fb-1}
\end{multline}
where 
\begin{align*}
\P_{d,k} & :=\P\of{d(x)\in[k,k+1)}=f\of k-f\of{k+1}\,,\\
\P_{n,N} & :=\P\of{n(x)\in[N,N+1)}\,,\\
\P_{s,S} & :=\P\of{s(x)\in[S,S+1)}\,.
\end{align*}
\end{lem}

\begin{proof}
We write $\X_{\fr}=\left(x_{i}\right)_{i\in\N}$, $d_{i}=d(x_{i})$,
$n_{i}=n(x_{i})$, $s_{i}:=s(x_{i})$. Let 
\[
X_{k,N,S}(\omega):=\left\{ x_{i}\in\X_{\fr}\,:\;d_{i}\in[k,k+1),\,n_{i}\in[N,N+1),\,s_{i}\in[S,S+1)\right\} 
\]
 with $A_{k,N,S}:=\bigcup_{x\in X_{k,N,S}}G_{n(x)}\of x$. We observe
that 
\begin{equation}
\forall x\in\Rd:\quad\#\left\{ x_{i}\in X_{k,N,S}:\,x\in G_{n(x_{i})}\of{x_{i}}\right\} \leq\S^{d-1}\left(N+1\right)^{d}\,\left(k+1\right)^{d}\fr^{-d}\,,\label{eq:lem:estim-E-fa-fb-help-1}
\end{equation}
which follows from the uniform boundedness of cells $G_{n(x)}\of x$,
$x\in X_{k,N}$ and the minimal distance of $\left|x_{i}-x_{j}\right|>2\fr$.
Then, writing $B_{R}:=\Ball R0$ for every $y\in\Rd$ it holds by
stationarity and the ergodic theorem
\begin{align*}
\P\of{y\in G_{n_{i}}\of{x_{i}}\,:\;x_{i}\in X_{k,N}} & =\lim_{R\to\infty}\left|B_{R}\right|^{-1}\left|A_{k,N}\cap B_{R}\right|\\
 & \leq\lim_{R\to\infty}\left|B_{R}\right|^{-1}\left|B_{R}\cap\bigcup_{x_{i}\in X_{k,N}}G_{n_{i}}\of{x_{i}}\right|\\
 & \leq\lim_{R\to\infty}\left|B_{R}\right|^{-1}\sum_{x_{i}\in X_{k,N}\cap B_{R}}\left|\S^{d-1}\right|\left(N+1\right)^{d}\left(k+1\right)^{d}\fr^{-d}\\
 & \to\P_{d,k}\P_{n,N}\left(N+1\right)^{d}\left|\S^{d-1}\right|\left(k+1\right)^{d}\fr^{-d}\,.
\end{align*}
In the last inequality we made use of the fact that every cell $G_{n(x)}(x)$,
$x\in X_{k,N}$, has volume smaller than $\S^{d-1}\left(N+1\right)^{d}\left(k+1\right)^{d}$.
We note that for $\frac{1}{p}+\frac{1}{q}=1$
\begin{align*}
 & \int_{\bQ}\left(\sum_{x\in\X_{\fr}}\chi_{G_{n}\of x}d(x)^{\eta}s(x)^{\xi}n(x)^{\zeta}\right)^{p}\\
 & \qquad\leq\int_{\bQ}\left(\sum_{k=1}^{\infty}\sum_{N=1}^{\infty}\sum_{S=1}^{\infty}\left(\sum_{x\in X_{k,N,S}}\chi_{G_{n(x)}\of x}\left(k+1\right)^{\eta}(N+1)^{\xi}(S+1)^{\zeta}\right)\right)^{p}\\
 & \qquad\leq\int_{\bQ}\left(\sum_{k,N,S=1}^{\infty}\alpha_{k,N,S}^{q}\right)^{\frac{p}{q}}\left(\sum_{k,N,S=1}^{\infty}\alpha_{k,N,S}^{-p}\left(\sum_{x\in X_{k,N,S}}\chi_{G_{n(x)}\of x}\left(k+1\right)^{\eta}(N+1)^{\xi}(S+1)^{\zeta}\right)^{p}\right)\,.
\end{align*}
Due to (\ref{eq:lem:estim-E-fa-fb-help-1}) we find 
\[
\sum_{x\in X_{k,N}}\chi_{G_{n(x)}\of x}\leq\chi_{A_{k,N}}\left(N+1\right)^{d}\left(k+1\right)^{d}\left|\S^{d-1}\right|
\]
and obtain for $q=\frac{p}{p-1}$ and $C_{q}:=\left(\sum_{k,N,S=1}^{\infty}\alpha_{k,N,S}^{q}\right)^{\frac{p}{q}}\left|\S^{d-1}\right|^{p}$:
\begin{align*}
\frac{1}{\left|B_{R}\right|}\int_{B_{R}} & \left(\sum_{x\in\X_{\fr}}\chi_{G_{n}\of x}d(x)^{\eta}s(x)^{\xi}n(x)^{\zeta}\right)^{p}\\
 & \leq C_{q}\frac{1}{\left|B_{R}\right|}\int_{B_{R}}\left(\sum_{k,N,S=1}^{\infty}\alpha_{k,N,S}^{-p}\chi_{A_{k,N,S}}\left(N+1\right)^{dp+\zeta p}\left(k+1\right)^{dp+\eta p}(S+1)^{\xi p}\right)\\
 & \to C_{q}\left(\sum_{k,N,S=1}^{\infty}\alpha_{k,N,S}^{-p}\left(k+1\right)^{d\left(p+1\right)+\eta p}\left(N+1\right)^{d\left(p+1\right)+\zeta p}(S+1)^{\xi p}\P_{s,S}\P_{d,k}\P_{n,N}\right)
\end{align*}
For the sum $\sum_{k,N,S=1}^{\infty}\alpha_{k,N,S}^{q}$ to converge,
it is sufficient that $\alpha_{k,N,S}^{q}=\left(k+1\right)^{-r}\left(N+1\right)^{-r}\left(S+1\right)^{-r}$
for some $r>1$. Hence, for such $r$ it holds $\alpha_{k,N,S}=\left(k+1\right)^{-r/q}\left(N+1\right)^{-r/q}\left(S+1\right)^{-r/q}$
and thus (\ref{eq:lem:estim-E-fa-fb-1}).

\end{proof}

\subsection{\label{subsec:Connectedness-of--Regular-sets}Discretizing the Connectedness
of $\left(\delta,M\right)$-Regular Sets}

Let $\bP\of{\omega}$ be a stationary ergodic random open set which
is isotropic cone mixing for $\fr>0$, $f:\,(0,\infty)\to\R$ and
$\alpha\in\left(0,\frac{\pi}{2}\right)$. Then $\X_{\fr}\of{\bP\of{\omega}}=\left(x_{k}\right)_{k\in\N}$
generates a Voronoi tessellation according to Lemma \ref{lem:Iso-cone-geo-estimate}
with cells $G_{k}$ and balls $B_{k,\fr}=\Ball{\fr/2}{x_{k}}$. While
the $\left(\delta,M\right)$-regularity of $\bP$ is a strictly local
property with a radius of influence of $\delta$, the isotropic cone
mixing is a mesoscopic property, with the influence ranging from $\fr$
to $\infty$.

In this part, we close the gap by introducing graphs on $\bP$ that
connect the small local balls covering $\partial\bP$ with $\X_{\fr}$
in $\bP$. The resulting family of graphs and paths on these graphs
will be essential for the last step in Section \ref{sec:Construction-of-Macroscopic-1}.
\begin{defn}[Admissible and simple graphs]
Let $\partial\X:=\left(p_{k}\right)_{k\in\N}\subset\partial\bP$
with corresponding $\Y_{\partial\X}:=\left(y_{k}\right)_{k\in\N}$
like in Corollary \ref{cor:cover-boundary} and let $\Y\subset\bP$
be a countable set of points with $\partial\X\cup\Y_{\partial\X}\cup\X_{\fr}\subset\Y$
and let $\left(\Y,\G_{\ast}(\bP)\right)$ be a graph. Then the graph
$\G_{\ast}(\bP)$ on $\Y$ is \emph{admissible} if it is connected
and every $p_{k}\in\partial\X$ has exactly one neighbor $y=y_{k}\in\Y_{\partial\X}$.
An admissible graph is called \emph{simple }if every $y_{k}\in\Y_{\partial\X}$
has - besides $p_{k}$ - only neighbors in $\Y\backslash\Y_{\partial\X}$.
\end{defn}

The following concept will become important later in Section . For
reasons of self-containedness, we introduce it already at this point.
\begin{defn}[Locally connected $\bP$ and $\G_{\fl}$]
Assume that $\left(\Y,\G(\bP)\right)$ is an admissible graph on
$\bP$ with the property that for $y_{1},y_{2}\in\Y_{\partial\X}$
with corresponding $p_{1},p_{2}\in\partial\X$ it holds $y_{1}\sim y_{2}$
iff $\Ball{\tilde{\rho}_{1}}{p_{1}}\cap\Ball{\tilde{\rho}_{2}}{p_{2}}\neq\emptyset$.
\nomenclature[locally connected]{locally connected}{$\G_\fl$, Definition \ref{def:simple-border}}The
graph $\G_{\fl}(\bP)$ consists of all elements of $\G(\bP)$, except
those $(y_{1},y_{2})\in\Y_{\partial\X}^{2}$ for which there is no
path in $\Ball{2\tilde{\rho}_{1}}{p_{1}}\cap\bP$ or in $\Ball{2\tilde{\rho}_{2}}{p_{2}}\cap\bP$
connecting $y_{1}$ with $y_{2}$. \\
If $\G_{\fl}(\bP)$ is connected, the set $\bP$ is called \emph{locally
connected}.
\end{defn}

Locally flat geometries will turn out to be particularly useful as
they allow to construct tubes around paths that fully lie within $\bP$
and connect the local with the mesoscopic balls.
\begin{defn}[Admissible paths]
\label{def:admissible-path}Let $(\Y,\G(\bP))$ be an admissible
graph on $\bP$ and let $\Apaths(y,x)$ be a family of paths from
$y\in\Y_{\partial\X}$ to $x\in\X_{\fr}$ which are constructed from
a deterministic algorithm that terminates after finitely many steps.
Assume that for every $Y=(y_{1},\dots,y_{k})\in\Apaths(y,x)$. If
$\fr_{1}=\fr(y)$ is the radius of $y$ from Corollary \ref{cor:cover-boundary}
assume there exists
\[
Y_{0}\in C\of{[0,1]\times\Ball{\fr_{1}}y;\bP}\quad\text{with}\quad Y_{0}\of{t,\Ball{\fr_{1}}y}=\Ball{\frac{\fr}{16}}x\,,
\]
such that $Y_{0}(t,\Ball{\fr_{1}}y)$ is invertible for every $t$
and $Y_{0}(0,x)=x$. Then the family $\Apaths(y,x)$ is called admissible.
\end{defn}

\subsubsection*{A general approach to construct admissible graphs and paths on locally
connected $\protect\bP$}

For a particular family of random geometries, there might be sophisticated
ways to construct $\Y$ and the families $\Apaths(\cdot,\cdot)$.
However, it is interesting to know that such a graph can be constructed
very generally for every locally connected geometry. In this section,
we will thus introduce a concept how to transform the domain $\bP$
into such a graph, thereby bridging the gap between the local regularity
of $\partial\bP$ and the mesoscopic regularity. The basic Idea is
sketched in Figure \ref{fig:graph}.

\begin{figure}
 \begin{minipage}[c]{0.5\textwidth} \includegraphics[width=6cm]{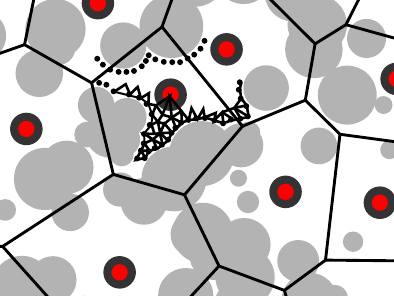}\end{minipage}\hfill   \begin{minipage}[c]{0.45\textwidth}\caption{\label{fig:graph}In order to treat the differences $\left|\tau_{i}u-\protect\cM_{j}u\right|^{s}$
appearing in Theorem \ref{thm:Fine-meso-estimate} below, it is necessary
to construct a graph that connects the boundary with the centers of
the Voronoi tessellation.}
\end{minipage}
\end{figure}

\paragraph*{The grid}

Let $\bP\subset\Rd$ be open and $\fr>0$. For $x\not\in\partial\bP$
let\nomenclature[eta]{$\eta(x)$}{Equation \eqref{eq:eta-distance}}
\begin{equation}
\eta(x):=\min{\left\{ \dist\of{x,\partial\bP}\,,\,2\fr\right\} }\label{eq:eta-distance}
\end{equation}
 and $\tilde{\eta}=\frac{1}{4}\eta$. Then we find the following:
\begin{lem}
\label{lem:cover-iso-mixing-delta-M}Let $\bP$ be a connected open
set which is locally $\left(\delta,M\right)$-regular. For $\fr>0$
let $\X_{\fr}=\left(x_{k}\right)_{k\in\N}$ be a family of points
with a mutual distance of at least $2\fr$ satisfying $\dist\of{x_{k},\partial\bP}>2\fr$
and let $\partial\X:=\left(p_{k}\right)_{k\in\N}\subset\partial\bP$
with corresponding $\left(\tilde{\rho}_{k}\right)_{k\in\N},\left(\fr_{k}\right)_{k\in\N}\subset\R$
and $\Y_{\partial\X}:=\left(y_{k}\right)_{k\in\N}$ like in Corollary
\ref{cor:cover-boundary}. Then there exists a family of points $\mathring{\X}=\left(\hat{p}_{j}\right)_{j\in\N}\subset\bP$
with $\X_{\fr}\subset\mathring{\X}$ such that with $\tilde{\eta}_{k}:=\tilde{\eta}\of{\hat{p}_{k}}$,
$\hat{B}_{k}:=\Ball{\tilde{\eta}_{k}}{\hat{p}_{k}}$ and $B_{k}:=\Ball{\tilde{\rho}_{k}}{p_{k}}$
the family $\left(B_{k}\right)_{k\in\N}\cup\left(\hat{B}_{k}\right)_{k\in\N}$
covers $\bP$ and 
\begin{equation}
\hat{B}_{k}\cap\hat{B}_{i}\neq\emptyset\quad\Rightarrow\quad\left\{ \begin{aligned} & \frac{1}{2}\tilde{\eta}_{i}\leq\tilde{\eta}_{k}\leq2\tilde{\eta}_{i}\\
\text{and}\quad & 3\min\left\{ \tilde{\eta}_{i},\tilde{\eta}_{k}\right\} \geq\left|\hat{p}_{i}-\hat{p}_{k}\right|\geq\frac{1}{2}\max\left\{ \tilde{\eta}_{i},\tilde{\eta}_{k}\right\} \,.
\end{aligned}
\right.\label{eq:min-dist-eta}
\end{equation}
Furthermore, $B_{k}\cap\hat{B}_{j}\neq\emptyset$ implies
\begin{equation}
\frac{3}{14}\tilde{\rho}_{k}\leq\tilde{\eta}_{j}\leq\frac{1}{3}\tilde{\rho}_{k}\,,\quad4\tilde{\eta}_{j}\leq\left|\hat{p}_{j}-p_{k}\right|\leq\frac{4}{3}\tilde{\rho}_{k}\,,\label{eq:min-dist-eta-2}
\end{equation}
i.e. $\Ball{\fr_{k}}{y_{k}}\cap\Ball{\frac{1}{8}\tilde{\eta}_{j}}{\hat{p}_{j}}=\emptyset$.
Finally, there exists $C>0$ such that for every $x\in\bP$
\begin{equation}
\#\left\{ j\in\N:\;x\in\Ball{\frac{1}{8}\tilde{\eta}_{j}}{\hat{p}_{j}}\right\} +\#\left\{ k\in\N:\;x\in\Ball{\fr_{k}}{y_{k}}\right\} \leq C\,.\label{eq:min-dist-eta-3}
\end{equation}
\end{lem}

\begin{notation}
\label{nota:XY}\label{nota:Y}\nomenclature[Y]{$\mathring{\Y},\,\partial\Y,\,\Y$}{Notation \ref{nota:Y}}\nomenclature[X]{$\X_\partial,\,\hat\X$}{Notation \ref{nota:XY}}\nomenclature[X]{$\Y_{\X\partial}$}{Notation \ref{nota:XY}}Summing
up and extending the notation of Lemma \ref{lem:cover-iso-mixing-delta-M}
we write 
\begin{equation}
\begin{aligned}\partial\Y:=\partial\X & :=\left(p_{k}\right)_{k\in\N}\subset\partial\bP\,, & \X_{\fr}\subset\mathring{\X} & :=\left(\hat{p}_{j}\right)_{j\in\N}\subset\bP\,, & \X & :=\partial\X\cup\mathring{\X}\,,\\
\Y_{\partial\X} & :=\left(y_{k}\right)_{k\in\N}\,, & \mathring{\Y} & :=\left(y_{k}\right)_{k\in\N}\cup\mathring{\X} & \Y & :=\mathring{\Y}\cup\partial\Y\,.
\end{aligned}
\label{eq:def:Y-Partial-Y}
\end{equation}
 The meaning of introducing the symbol $\Y$ will be clarified below.

For $p\in\partial\X$ we write $\tilde{\eta}(p):=\tilde{\rho}(p)$
and for $p\in\mathring{\X}$ we use the above notation (\ref{eq:eta-distance})
and further define 
\begin{equation}
\fr(y):=\fr_{j}\text{ for }y=y_{j}\in\Y_{\partial\X}\,,\qquad\fr(y):=\frac{1}{8}\tilde{\eta}(y)\text{ for }y\in\left(\hat{p}_{k}\right)_{k}\,.\label{eq:r-of-y}
\end{equation}
We finally introduce the following bijective mappings
\begin{equation}
x(y)=\begin{cases}
p_{k} & \text{if }y=y_{k}\in\Y_{\partial\X}\\
\hat{p}_{j} & \text{if }y=\hat{p}_{j}\in\mathring{\X}
\end{cases}\,,\qquad y(x)=\begin{cases}
y_{k} & \text{if }x=p_{k}\in\partial\X\\
\hat{p}_{j} & \text{if }x=\hat{p}_{j}\in\mathring{\X}
\end{cases}\,.\label{eq:def-x-of-y-y-of-x}
\end{equation}
\end{notation}

\begin{proof}[Proof of Lemma \ref{lem:cover-iso-mixing-delta-M}]
We recall $\tilde{\rho}_{k}:=\tilde{\rho}\left(p_{k}\right):=2^{-5}\rho\left(p_{k}\right)$
and $\fr_{k}=\frac{\tilde{\rho}_{k}}{32\left(1+M_{k}\right)}$ and
that (\ref{eq:cor:cover-boundary-h1}) holds. Furthermore, $\Ball{\fr_{k}}{y_{k}}\subset\Ball{\tilde{\rho}_{k}/8}{p_{k}}\cap\bP$
and hence $\Ball{\fr_{k}}{y_{k}}\cap\Ball{\fr_{j}}{y_{j}}=\emptyset$
for $k\neq j$.

If we define $\bP_{B}:=\overline{\bP\backslash\bigcup_{k}B_{k}}$
and observe that $\bP_{B}$ is $\eta$-regular (for $\eta$ defined
in (\ref{eq:eta-distance})). Then Lemma \ref{lem:eta-lipschitz}
and Theorem \ref{thm:delta-M-rho-covering} yield a cover of $\bP_{B}$
by a locally finite family of balls $\hat{B}_{k}=\Ball{\tilde{\eta}_{k}}{\hat{p}_{k}}$,
where $\left(\hat{p}_{k}\right)_{k\in\N}\subset\bP_{B}$, and where
(\ref{eq:min-dist-eta}) holds. Looking into the proof of Theorem
\ref{thm:delta-M-rho-covering} we can assume w.l.o.g. that $\left(x_{k}\right)_{k\in\N}\subset\left(\hat{p}_{k}\right)_{k\in\N}$
by suitably bounding $\eta$.

Furthermore, we find for $B_{k}\cap\hat{B}_{j}\neq\emptyset$ that
\[
\tilde{\eta}_{j}+\tilde{\rho}_{k}\geq\left|\hat{p}_{j}-p_{k}\right|>4\tilde{\eta}_{j}\quad\Rightarrow\quad\tilde{\eta}_{j}\leq\frac{1}{3}\tilde{\rho}_{k}\text{ and }\left|\hat{p}_{j}-p_{k}\right|\leq\frac{4}{3}\tilde{\rho}_{k}\,.
\]
Next, for such $p_{k}$ we consider all $B_{i}$ such that $p_{i}\in\Ball{4\tilde{\rho}_{k}}{p_{k}}$
and since $\hat{p}_{j}\not\in B_{i}$ for all such $i$, we infer
$\dist\of{\hat{p}_{j},\partial\bP}\geq\tilde{\rho}(p_{i})$ and hence
by Lemma \ref{lem:rho-p-lsc}
\[
\tilde{\eta}_{j}\geq\frac{1}{4}\tilde{\rho}_{i}=2^{-7}\rho_{i}\geq2^{-7}\frac{1-2\frac{1}{8}}{1-\frac{1}{8}}\rho_{k}>\frac{3}{14}\tilde{\rho}_{k}\,.
\]
Finally, $\Ball{\fr_{k}}{y_{k}}\cap\Ball{\frac{1}{8}\tilde{\eta}_{j}}{\hat{p}_{j}}=\emptyset$
follows from $\frac{12}{14}\tilde{\rho}_{k}\leq4\tilde{\eta}_{j}\leq\left|\hat{p}_{j}-p_{k}\right|$.

To see (\ref{eq:min-dist-eta-3}) let $x\in\bP$ and let $\hat{p_{j}}$
such that $\tilde{\eta}_{j}$ is maximal among all $\hat{B}_{j}$
with $x\in\hat{B}_{j}$. Let $\hat{p}_{i}$ with $x\in\hat{B}_{i}\cap\hat{B}_{j}$
and observe that both $\left|\hat{p}_{i}-\hat{p}_{j}\right|$ and
$\tilde{\eta}_{i}$ are bounded from below and above by a multiple
of $\tilde{\eta}_{j}$. If $x\in\hat{B}_{i}\cap\hat{B}_{k}\cap\hat{B}_{j}$,
$\left|\hat{p}_{i}-\hat{p}_{k}\right|$ is bounded from above and
below by $\tilde{\eta}_{i}$, hence by $\tilde{\eta}_{j}$. This provides
a uniform bound on $\#\left\{ j\in\N:\;x\in\Ball{\tilde{\eta}_{j}}{\hat{p}_{j}}\right\} $.
The second part of (\ref{eq:min-dist-eta-3}) follows in an analogue
way.
\end{proof}
\begin{defn}[Neighbors]
\label{def:Y-Partial-Y}\nomenclature[neighbors]{$x\sim y$}{$x$ and $y$ are neighbors, see Definition \ref{Y-Partial-Y}}Under
the assumptions and notations of Lemma \ref{lem:cover-iso-mixing-delta-M},
for two points $y_{1},y_{2}\in\mathring{\X}\cup\Y_{\partial\X}$ let
$x_{1}=x(y_{1})$, $x_{2}=x(y_{2})$. We say that $y_{1}$ and $y_{2}$
are neighbors, written $y\sim y_{2}$, if $\Ball{\tilde{\eta}(x_{1})}{x_{1}}\cap\Ball{\tilde{\eta}(x_{2})}{x_{2}}\neq\emptyset$.
This implies a definition of ``neighbor'' for $x_{1},x_{2}\in\mathring{\X}$
. For $x\in\partial\X$ and $y\in\Y_{\partial\X}$ we write $x\sim y$
if $x=x(y)$. We denote by $\G_{0}(\bP,\X)$, $\G_{0}(\bP)$ or simply
$\G(\bP)$ the graph on $\mathring{\X}\cup\Y_{\partial X}\cup\partial\X$
generated by $\sim$. \nomenclature[G(P)]{$\G(\bP,\X),\,\G(\bP) $}{Graph constructed from $\bP$, see Definition \ref{def:Y-Partial-Y}.}
\end{defn}

\begin{rem}
\label{rem:def:Y-Partial-Y}a) Every $y\in\Y_{\partial\X}$ has a
neighbor $x\in\partial\X$.

b) Besides $y(x)$, points $x\in\partial\X$ have no other neighbors.
\end{rem}

\paragraph*{The admissible paths}

We will see below that $\G_{0}(\bP)$ is admissible if $\bP$ is connected.
Besides $\G_{0}(\bP)$ we introduce further (reduced) graphs on $\X$,
which are based on continuous paths. For two points $x,y\in\bP$ we
denote 
\[
\bP_{0}(x,y):=\left\{ f\in C([0,1];\bP):\;f(0)=x,\,f(1)=y\right\} \,.
\]

\begin{defn}
\label{def:simple-border}\nomenclature[simple]{simple}{Definition \ref{def:simple-border}}Using
the notation of Lemma \ref{lem:cover-iso-mixing-delta-M}, the graph
\begin{align*}
\G_{\simple}(\bP) & :=\left\{ \left(y_{1},y_{2}\right)\in\G_{\fl}(\bP)\,:\;\left(y_{1},y_{2}\right)\not\in\Y_{\partial\X}^{2}\right\} 
\end{align*}
 is the subset of $\G(\bP)$ where all elements $(y_{1},y_{2})$ and
$(x(y_{1}),x(y_{2}))$ are removed for which $y_{1},y_{2}\in\Y_{\partial\X}$.
Furthermore, if $y_{k}\in\Y_{\partial\X}$ with $p_{k}=x(y_{k})\in\partial\X$
has a neighbor $\hat{p}_{j}\in\mathring{\X}$ such that $y_{k}$ and
$\hat{p}_{j}$ are \emph{not }connected through a path which lies
in $\Ball{3\tilde{\rho}(p_{k})}{p_{k}}\cap\bP$, then $(y(p_{k}),\hat{p}_{j})$,
$(\hat{p}_{j},y(p_{k}))$ are removed.

We write $\G_{\ast}(\bP)$ for either $\G_{0}(\bP)$, $\G_{\fl}(\bP)$,
$\G_{\simple}(\bP)$ or any other subset of $\G_{0}(\bP)$ which is
connected.
\end{defn}

\begin{lem}
\label{lem:gamma-connect}Assume $\left(\X,\G_{\fl}(\bP)\right)$
is connected, assume $y\in\Y_{\partial\X}$ and $y_{1}\sim y$. Then
there exists $\gamma\in C\of{[0,1]\times\Ball{\frac{\fr}{16}}0;\,\bP\cap\Ball{3\tilde{\rho}(x(y))}{x(y)}}$
such that $\gamma(\cdot,x)$ is a path from $y+\frac{16}{\fr}\fr(y)x$
to $y_{1}+\frac{16}{\fr}\fr(y_{1})x$, for two points $x_{1},x_{2}\in\Ball{\frac{\fr}{16}}0$
it holds either $\gamma(\cdot,x_{1})\cap\gamma(\cdot,x_{2})=\emptyset$
or $\gamma(\cdot,x_{1})\subset\gamma(\cdot,x_{2})$ or $\gamma(\cdot,x_{1})\supset\gamma(\cdot,x_{2})$
and there exist constants $c_{1},c_{2},c_{3}$ depending only on the
dimension but not on $y$ or $y_{1}$ such that
\begin{align*}
 & \forall t\in[0,1]\;\Ball{c_{1}\min\left\{ \fr(y),\fr(y_{1})\right\} }{\gamma(t,0)}\subset\gamma\of{[0,1]\times\Ball{\frac{\fr}{16}}0}\\
 & \forall x\in\Ball{\frac{\fr}{16}}0\;\Length\gamma(t,x)\leq c_{2}\left|y-y_{1}\right|\,.
\end{align*}
We denote $\gamma$ as $\gamma[y,y_{1}]$.
\end{lem}

\begin{proof}
Let $\tilde{\gamma}\in\bP_{0}(y,y_{1})$. If $\tilde{y}\in\Y_{\partial\X}$
we infer from Lemma \ref{lem:properties-local-rho-convering} below
that $\Ball{\frac{1}{2}\tilde{\rho}(\tilde{y})}{\tilde{y}}\subset\Ball{3\tilde{\rho}(x(y))}{x(y)}$.
We recall that $\partial\bP\cap\Ball{3\tilde{\rho}(x(y))}{x(y)}$
is a graph $(\cdot,\phi(\cdot))$ of a Lipschitz continuous function
$\phi:\R^{d-1}\to\R$ and that both $\Ball{\fr(y)}y$ and $\Ball{\fr(y_{1})}{y_{1}}$
as well as $\tilde{\gamma}([0,1])$ lie below that graph. We project
$\Ball{\fr(y)}y$ and $\Ball{\fr(y_{1})}{y_{1}}$ as well as $\tilde{\gamma}([0,1])$
onto the sphere $x(y)+2\tilde{\rho}(x(y))\S^{d-1}$, which still do
not intersect with the graph of $\phi$. From here we may construct
$\gamma$ satisfying the claimed estimates. Since $\Ball{\fr(y)}y\subset\Ball{\frac{1}{2}\tilde{\rho}(x(y))}{x(y)}$
and $\Ball{\fr(y_{1})}{y_{1}}\subset\Ball{\frac{1}{2}\tilde{\rho}(x(y_{1}))}{x(y_{1})}$
and $\left|y-y_{1}\right|>\frac{15}{16}\min\left\{ \tilde{\rho}(x(y_{1})),\tilde{\rho}(x(y))\right\} $,
we conclude that the constants can be chosen independently from $y$.

If $y\in\mathring{\Y}$ we can proceed analogously.
\end{proof}
\begin{lem}[$\G_{0}(\bP)$ is admissible]
\label{lem:existence-discrete-path} Under the assumptions and notations
of Lemma \ref{lem:cover-iso-mixing-delta-M} for every $y_{0},y_{1}\in\Y$
there exists a discrete path from $y_{0}$ to $y_{1}$ in $\left(\X,\G_{0}(\bP)\right)$.
\end{lem}

\begin{proof}
Since $\bP$ is connected, there exists a continuous path $\gamma:\,[0,1]\to\bP$
with $\gamma(0)=y_{0}$, $\gamma(1)=y_{1}$. Since $\gamma\of{[0,1]}$
is compact, it is covered by a finite family of balls $\Ball{\tilde{\eta}(y)}y$,
$y\in\Y$. If $\gamma\of{[0,1]}\subset\Ball{\tilde{\eta}(y_{0})}{y_{0}}$
the statement is obvious. Otherwise there exists a maximal interval
$[0,a)$, $a<1$, such that $\gamma\of{[0,a)}\subset\Ball{\tilde{\eta}(y_{0})}{y_{0}}$,
$\gamma(a)\not\in\Ball{\tilde{\eta}(y_{0})}{y_{0}}$and there exists
$y\not=y_{0}$ such that for some $\eps>0$ $\gamma\of{(a-\eps,a+\eps)}\subset\Ball{\tilde{\eta}(y_{0})}{y_{0}}\cap\Ball{\tilde{\eta}(y)}y$.
One may hence iteratively continue with $y_{0}':=y$ on the interval
$[a,1]$.
\end{proof}
Hence, every two points in $\Y$ can be connected by a discrete path.
However, the choice of the path is not unique, there might be even
infinitely many with arbitrary large deviation from the ``shortest''
path. Luckily, it turns out that it suffices to provide a deterministically
constructed finite family of paths.
\begin{defn}[Admissible paths on $\G_{\ast}(\bP)$]
\label{def:admissible-path-gen-sit} Let $\bP\subset\Rd$ be open,
connected and locally connected with $\G_{\ast}(\bP)$ such that the
assumptions of Lemma \ref{lem:cover-iso-mixing-delta-M} are satisfied.
\nomenclature[AX]{$\Apaths(y,x)$}{Admissible paths from $y\in\Y\setminus\{x\}$ to $x\in\X_r$, see Definition, \ref{def:admissible-path}, ref{def:admissible-path-gen-sit}}
Let $x\in\X_{r}$. We call any family of paths which connect $y\in\Y\backslash\{x\}$
to $x$ \emph{admissible}, if it is generated by a deterministic algorithm
that terminates after a finite number of steps. Hence, an \emph{admissible
path} from $y$ to $x$ in $\G_{\ast}(\bP)$ is a path $\left(x_{1},\dots,x_{k}\right)$
with $x_{1}=y$, $x_{k}=x$ generated according to this algorithm.
We denote the set of admissible paths from $y$ to $x$ by $\Apaths_{\ast}(y,x)$.
\end{defn}

\begin{notation}
\label{nota:Apaths}Let $x_{j}\in\X_{\fr}$, $p_{i}\in\Y_{\partial\X}$
and $Y=\left(y_{1},\dots,y_{N}\right)\in\Apaths\of{p_{i},x_{j}}$.
Recalling (\ref{eq:r-of-y}), for $x\in\Ball{\fr_{1}}0$ we define
$Y_{0}(x)$ the set of paths connecting $y_{1}+x$, $y_{2}+\frac{\fr(y_{2})}{\fr_{1}}x$,
... $y_{N}+\frac{\fr(y_{N})}{\fr_{1}}x$ chosen as straight line if
$y_{i},y_{i+1}\in\Y_{0}$ and $\gamma(\cdot,x)$ from Lemma \ref{lem:gamma-connect}
else and 
\[
Y_{0}\of{\Ball{\fr_{1}}0}:=\bigcup_{x\in\Ball{\fr_{1}}0}Y(x)\,.
\]
In what follows, we are usually working with the latter expression
and hence introduce for simplicity of notation the identification
$Y\equiv Y_{0}\of{\Ball{\fr_{i}}0}$. In this way, $Y$ is an open
set and the characteristic function $\chi_{Y}\in L^{1}(\Rd)$ is integrable
as Lemma the next Lemma \ref{lem:radius-admissible-pahts} will reveal.
Finally, by Lemma \ref{lem:cover-iso-mixing-delta-M} there exists
$C>1$ such that independent from $x_{j}$, $p_{i}$ and $x\in\Ball{\fr_{i}}0$
it holds \nomenclature[Length]{$\Length(Y)$}{Length of an admissible path $Y$, see \eqref{eq:def-Length-apath}}
\begin{equation}
\frac{1}{C}\Length(Y(x))\leq\Length(Y):=\Length(Y(0))\leq C\Length(Y(x))\,.\label{eq:def-Length-apath}
\end{equation}
\end{notation}

\begin{rem}
1. Every path admissible in the sense of \ref{def:admissible-path-gen-sit}
is admissible in the sense of \ref{def:admissible-path}. This follows
from Lemma \ref{lem:gamma-connect} and the fact that for $y,\tilde{y}\in\Y_{0}$
with $y\sim\tilde{y}$ it holds $\Ball{\frac{1}{4}\tilde{\eta}(y)}y\subset\Ball{4\tilde{\eta}(\tilde{y})}{\tilde{y}}\subset\bP$.

2. A particular family of admissible paths is given by the shortest
distance. In particular, if $x\in\X_{r}$ and $y\in\Y\backslash\{x\}$
we define the shortest paths as 
\begin{align*}
\Apaths_{\short}(y,x):=\arg\min\left\{ \sum_{i=1}^{k}\left|x_{i+1}-x_{i}\right|:\right.& \,\left(x_{1},\dots,x_{k}\right)\text{path in }\G_{\ast}(\bP)\,, \\ &\,\left.\,k\in\N,\,x_{1}=y,\,x_{k}=x\right\} 
\end{align*}
\end{rem}

\paragraph*{Construct a finite family}

In what follows, we will construct a class of admissible paths on
$\G_{\ast}(\bP)$ which does not rely on the metric graph distance.
We study the discrete Laplacian $\cL_{\ast}:\,L^{2}(\Y)\to L^{2}(\Y)$
on an admissible graph $\G_{\ast}(\bP)$ given by
\[
\left(\cL_{\ast}u\right)(x):=-\sum_{(y,x)\in\G_{\ast}(\bP)}\frac{1}{\left|x-y\right|^{2}}\left(u(y)-u(x)\right)\,.
\]
It is well known that $\cL$ is a discrete version of an elliptic
second order operator, see \cite{biskup2011recent,flegel2019homogenization,heida2020consistency}
and references therein. This may be quickly verified for the ``classical''
choice $\Y=h\Zd$ with $x\sim y$ iff $\left|x-y\right|=h$ (using
Taylor expansion and the limit $h\to0$).

The discrete Laplacian is connected to the following discrete Poincar\'e
inequality.
\begin{lem}
\label{lem:discrete-Poincare}Let $\bP\subset\Rd$ be open, connected
and satisfy the assumptions of Lemma \ref{lem:cover-iso-mixing-delta-M},
let $\left(\X,\G_{\ast}(\bP)\right)$ be admissible and let $0\in\mathring{\Y}$.
Writing 
\begin{align*}
H_{0}\of{\Y} & :=\left\{ u:\Y\to\R\,:\;\forall y\in\partial\Y:\,\,u(y)=0\right\} \,,\\
H_{0}\of{\Y\cap\Ball R0} & :=\left\{ u\in H_{0}\of{\Y}\,:\;\forall y\not\in\Ball R0:\,\,u(y)=0\right\} \,.
\end{align*}
There exists $R_{0}>0$ and $C_{R_{0}}>0$ such that for every $R>R_{0}$
the following discrete Poincaré estimate holds:
\begin{equation}
\forall u\in H_{0}(\Y):\quad u(0)^{2}\leq C_{R_{0}}\sum_{\substack{y_{1},y_{2}\in\Y\cap\Ball R0\\
y_{1}\sim y_{2}
}
}\frac{\left(u(y_{1})-u(y_{2})\right)^{2}}{\left|y_{1}-y_{2}\right|^{2}}\,.\label{eq:lem:existence-cLu-equiv-delta-x-help-1}
\end{equation}
\end{lem}

\begin{proof}
This is straight forward from a contradiction argument (using
connectedness of $\left(\X,\G_{\ast}(\bP)\right)$).
\end{proof}
For the following result we introduce the notation: 
\[
\text{For }x\in\mathring{\Y}\text{ define }\delta_{x}(y):=\begin{cases}
0 & \text{ if }x\neq y\\
1 & \text{ if }x=y
\end{cases}\,.
\]

\begin{lem}[A discrete maximum principle]
\label{lem:existence-cLu-equiv-delta-x}Let $\bP\subset\Rd$ be open,
connected and satisfy the assumptions of Lemma \ref{lem:cover-iso-mixing-delta-M},
let $\left(\X,\G_{\ast}(\bP)\right)$ be admissible and let $x\in\mathring{\Y}$.
Then the equation 
\begin{equation}
\begin{aligned}\of{\cL_{\ast}u}(y)+\left|y-x\right|u(y) & =\delta_{x}(y) &  & \text{ for }y\in\mathring{\Y}\\
u(y) & =0 &  & \text{ for }y\in\partial\Y
\end{aligned}
\label{eq:lem:existence-cLu-equiv-delta-x}
\end{equation}
has a unique solution which satisfies $u(y)>0$ for all $y\in\mathring{\Y}$
and attains its unique local (and thus global) maximum in $x$. Furthermore,
$u(y)\to0$ as $\left|y\right|\to\infty$ and for $C_{R_{0}}>0$ from
Lemma \ref{lem:discrete-Poincare} it holds 
\begin{equation}
u(x)+\sum_{\substack{\left(y_{1},y_{2}\right)\in\G_{\ast}(\bP)}
}\frac{1}{\left|y_{1}-y_{2}\right|^{2}}\left(u(y_{1})-u(y_{2})\right)^{2}+\sum_{y\in\Y}\left|x-y\right|u(y)^{2}\leq5C_{R_{0}}\,.\label{eq:lem:existence-cLu-equiv-delta-x-2}
\end{equation}
\end{lem}

\begin{proof}
W.l.o.g. let $x=0$ and write $y_{1}\sim y_{2}$ iff $\left(y_{1},y_{2}\right)\in\G_{\ast}(\bP)$.
Using the notation of Lemma \ref{lem:discrete-Poincare} and $B_{R}:=\Ball R0$
and $B_{R}^{\complement}:=\Rd\setminus\Ball R0$ we divide the proof
in three parts.

\textbf{Approximation:} We consider the problem
\begin{equation}
\cL_{\ast}u_{R}+\left|\cdot\right|u=\delta_{0}\,,\qquad u_{R}(y)=0\text{ for }y\in\partial\Y,\,\text{and }y\in\Y\cap B_{R}^{\complement}\,.\label{eq:lem:existence-cLu-equiv-delta-x-help-2}
\end{equation}
Putting $v(x)=0$ for $v\in H_{0}\of{\Y\cap B_{R}}$ and all $x\not\in B_{R}$,
we find 
\begin{equation}
\sum_{y\in\mathring{\Y}\cap B_{R}}v(y)\,\cL_{\ast}u(y)=\sum_{z\sim y}\frac{1}{\left|y-z\right|^{2}}\left(u(y)-u(z)\right)\left(v(y)-v(z)\right)\,,\label{eq:discrete-part-int}
\end{equation}
which is a strictly positive definite bilinear symmetric form on $\R^{\mathring{\Y}\cap B_{R}}$.
Hence, there exists a unique solution $u_{R}$ to (\ref{eq:lem:existence-cLu-equiv-delta-x-help-2}).

Since $\Y\cap B_{R}$ is finite, $u_{R}$ attains a maximum and a
minimum. If $u_{R}$ attains a local maximum in $y$, it holds $\cL_{\ast}u_{R}(y)\geq0$
and if $u_{R}$ attains a local miminum in $y$ it holds $\cL_{\ast}u_{R}(y)\leq0$.
If $u_{R}$ attains negative values, it has a negative minimum in
$y_{0}\in\mathring{\Y}$ and hence $\of{\cL_{\ast}u}(y_{0})+\left|y_{0}-x\right|u(y_{0})<0$,
a contradiction. Thus, $u_{R}>0$ in every $y\not\in\partial\Y$.
Furthermore, because of (\ref{eq:lem:existence-cLu-equiv-delta-x-help-2})
$u_{R}$ can attain a local maximum only in $0$.

\textbf{Passage $R\to\infty$:} Using Lemma \ref{lem:discrete-Poincare},
for some large enough $R_{0}\in\R$ we find the following estimate,
which holds for every $R>R_{0}$ due to (\ref{eq:lem:existence-cLu-equiv-delta-x-help-2})
and (\ref{eq:lem:existence-cLu-equiv-delta-x-help-1}) applied to
$R_{0}$
\begin{align}
\sum_{y\in\mathring{\X}\cap B_{R}}\left(u_{R}(y)\,\cL u_{R}(y)+\left|y\right|u(y)^{2}\right) & =\sum_{z\sim y}\frac{\left(u_{R}(y)-u_{R}(z)\right)^{2}}{\left|y-z\right|^{2}}+\sum_{y\in\mathring{\X}\cap B_{R}}\left|y\right|u(y)^{2}\nonumber \\
 & =\sum_{y\in\mathring{\X}\cap B_{R}}u_{R}(y)\,\delta_{0}(y)\leq u(0)\leq2C_{R_{0}}+\frac{1}{2C_{R_{0}}}u_{R}(0)^{2}\nonumber \\
 & \leq2C_{R_{0}}+\frac{1}{2}\sum_{\substack{y_{1},y_{2}\in\Y\cap B_{R_{0}}\\
y_{1}\sim y_{2}
}
}\frac{\left(u_{R}(y_{1})-u_{R}(y_{2})\right)^{2}}{\left|y_{1}-y_{2}\right|^{2}}\label{eq:lem:existence-cLu-equiv-delta-x-help-3}
\end{align}
Together with (\ref{eq:discrete-part-int}), the latter yields a uniform
estimate for all $R>R_{0}$. In particular (due to a Cantor argument),
there exists a subsequence $u_{R'}$ such that $u_{R'}(y)\to u(y)$
converges for every $y\in\Y$ as $R'\to\infty$. Evidently, $u$ solves
(\ref{eq:lem:existence-cLu-equiv-delta-x}), is non-negative, attains
its maximum in $0$ and satisfies the estimate (\ref{eq:lem:existence-cLu-equiv-delta-x-2}).
The limit $u(y)\to0$ as $y\to\infty$ follows from (\ref{eq:lem:existence-cLu-equiv-delta-x-2})
and (\ref{eq:lem:existence-cLu-equiv-delta-x-help-3}). $u$ has a
unique local maximum in $0$ for the same reason as for $u_{R}$.

\textbf{Uniqueness of $u$:} Finally, let $u$ and $\tilde{u}$ be
two solutions such that $v=u-\tilde{u}$ satisfies 
\[
\begin{aligned}\of{\cL_{\ast}v}(y)+\left|y-x\right|v(y) & =0 &  & \text{ for }y\in\Y\setminus\Y_{\partial}\\
v(y) & =0 &  & \text{ for }y\in\Y_{\partial}
\end{aligned}
\,.
\]
Multiplying the above equation with $v$ and summing over all $y$,
we find
\[
\sum_{y\in\mathring{\X}}\left(v(y)\,\cL_{\ast}v(y)+\left|y\right|v(y)^{2}\right)=0\,,
\]
which implies $v=0$.
\end{proof}
\begin{defn}
\label{def:admissible-harmonic-path}Let $x\in\X_{r}$, let $u_{x}$
be the solution of (\ref{eq:lem:existence-cLu-equiv-delta-x}) and
$y\in\Y\backslash\{x\}$. An \emph{admissible harmonic path} from
$y$ to $x$ in $\G_{\ast}(\bP)$ is a path $\left(x_{1},\dots,x_{k}\right)$
with $x_{1}=y$, $x_{k}=x$ such that $u_{x}(x_{i+1})\geq u_{x}(x_{i})$.
We denote the set of admissible harmonic paths from $y$ to $x$ by
$\Apaths_{\ast}(y,x)$. If $\G_{\ast}(\bP)=\G_{0}(\bP)=\G(\bP)$ we
simply write $\Apaths(y,x)$. Note that 
\[
\Apaths(y,x)\supseteq\Apaths_{\ast}(y,x)\,.
\]
\end{defn}

\begin{lem}
\label{lem:radius-admissible-pahts}Let $\bP\subset\Rd$ be open,
connected and satisfy the assumptions of Lemma \ref{lem:cover-iso-mixing-delta-M}.
Let $\left(\Y,\G_{\ast}(\bP)\right)$ be admissible and let $x\in\mathring{\Y}$
and $y\in\Y$. There exists $R>0$, depending on $\bP$, $x$ and
$y$ such that every admissible harmonic path $\left(x_{1},\dots,x_{k}\right)\in\Apaths_{\ast}(y,x)$
from $y$ to $x$ lies in $\Ball Rx$. If $C_{0},C>0$ are the natural
numbers such that for every $y\in\Y$ it holds $C_{0}\leq\#\left\{ z\sim y:\,z\in\Y\right\} \leq C$
(which exist due to Lemma \ref{lem:cover-iso-mixing-delta-M}) then
we can choose \nomenclature[Mc]{$\sfR_{0}(x,y)$}{Equation \eqref{eq:lem:radius-admissible-pahts}}
\begin{equation}
\forall y\in\Y:\quad R\leq\sfR_{0}(x,y):=C\frac{u(x)}{u\of y}\label{eq:lem:radius-admissible-pahts}
\end{equation}
\end{lem}

\begin{proof}
Let us recall that $u(z)>0$ for every $z\not\in\partial\Y$ by Lemma
\ref{lem:existence-cLu-equiv-delta-x}. Again we write $x\sim y$
if $(x,y)\in\G_{\ast}(\bP)$.

For an admissible path $\left(x_{1},\dots,x_{k}\right)$ from $y=x_{1}$
to $x=x_{k}$ it follows $u\of{x_{j}}\geq u\of y>0$ for every $j>1$.
On the other hand 
\[
\of{\cL_{\ast}u}\of{x_{j}}+\left|x_{j}-x\right|u\of{x_{j}}=0
\]
Let us further recall, that with $C_{0}$ and $C$ independent from
$y$. Given $u(y)$ we can therefore conclude the necessary condition
\[
\left(C_{0}+\left|x_{j}-x\right|\right)u\of{x_{j}}-\sum_{z\sim x_{j}}u(z)\leq0\,.
\]
On the other hand, it holds $u(z)\leq u(x)$. This implies that the
left hand side of the last inequality is bounded from below by 
\[
\left(C_{0}+\left|x_{j}-x\right|\right)u\of{x_{j}}-Cu(x).
\]
Hence we conclude (\ref{eq:lem:radius-admissible-pahts}) from 
\[
\left|x_{j}-x\right|\leq C\frac{u(x)}{u\of y}-C_{0}\,.
\]
\end{proof}
The most important and concluding result in this context is the following,
which states that the set of admissible paths is not empty and the
$\G(\bP)$ is connected:
\begin{thm}[Admissible $\G_{\ast}(\bP)$ are connected through admissible harmonic
paths]
\label{thm:Ex-adm-path} Let $\bP\subset\Rd$ be open, connected
and let $\bP$ as well as $(\Y,\G_{\ast}(\bP))$ satisfy the assumptions
of Lemma \ref{lem:existence-cLu-equiv-delta-x}. Then for $x\in\mathring{\Y}$
let $u_{x}$ be the solution of (\ref{eq:lem:existence-cLu-equiv-delta-x})
and for $y\in\Y$ let $x_{1}:=y$. As long as $x_{i}\neq x$ select
iteratively $x_{i+1}\in\left\{ z\in\Y:\,z\sim x_{i},\,u_{x}(z)>u_{x}(x_{i})\right\} $.
Then this algorithm terminates after finite steps, i.e. there exists
$i\in\N$ such that $x_{i}=x$. In particular $\G_{\ast}(\bP)$ is
connected via admissible paths.
\end{thm}

\begin{proof}
According to Lemma \ref{lem:radius-admissible-pahts}, the number
of points that can be reached by the iterative process is finite,
i.e. the algorithm will stop when $x_{i}$ is a local maximum of $u_{x}$.
But this is given by $x_{i}=x$ according to Lemma \ref{lem:existence-cLu-equiv-delta-x}.
\end{proof}

\section{\label{sec:Extension-and-Trace-d-M}Extension and Trace Properties
from $\left(\delta,M\right)$-Regularity}

\subsection{\label{subsec:5-Preliminaries}Preliminaries}

For this whole section, let $\bP$ be a locally $\left(\delta,M\right)$-regular
open set and let $\delta$ be bounded by $\fr>0$ and satisfy (\ref{eq:lem:properties-delta-M-regular-3}).
In view of Corollary \ref{cor:cover-boundary}, there exists a complete
covering of $\partial\bP$ by balls $\Ball{\tilde{\rho}\left(p_{k}\right)}{p_{k}}$,
$\left(p_{k}\right)_{k\in\N}$, where $\tilde{\rho}\of p:=2^{-5}\rho\of p$.
We define with $\tilde{\rho}_{k}:=\tilde{\rho}\of{p_{k}}$, $\hat{\rho}_{k}:=\hat{\rho}\of{p_{k}}$\nomenclature[A1]{$A_{1,k},A_{2,k},A_{3,k}$}{Equation \eqref{eq:A123-k}}
given in Lemma \ref{lem:rho-p-lsc}
\begin{equation}
A_{1,k}:=\Ball{\tilde{\rho}_{k}}{p_{k}}\,,\quad A_{2,k}:=\Ball{3\tilde{\rho}_{k}}{p_{k}}\,,\quad A_{3,k}:=\Ball{\frac{\hat{\rho}_{k}}{8}}{p_{k}}\label{eq:A123-k}
\end{equation}
and recall (\ref{eq:lem:local-delta-M-construction-estimate-1}),
which we apply to $\delta$ in order to obtain the measurable function
\begin{equation}
\tilde{\delta}\of x:=\tilde{\hat{\rho}}(x)=\min\left\{ \hat{\rho}\of{\tilde{x}}\,:\;\tilde{x}\in\partial\bP\,\text{s.t. }x\in\Ball{\frac{1}{8}\hat{\rho}\of{\tilde{x}}}{\tilde{x}}\right\} \,.\label{eq:tilde-delta}
\end{equation}
Similarly, in view of (\ref{eq:lem:local-delta-M-construction-estimate-1b}),
we define the measurable function\nomenclature[delta-tilde]{$\tilde\delta$}{Equation \eqref{eq:tilde-delta}}\nomenclature[M-tilde]{$\tilde M$}{Equation \eqref{eq:tilde-M}}
\begin{align}
\tilde{M}(x) & :=M_{[\frac{1}{8}\hat{\rho}],\Rd}(x)+1=\max\left\{ M_{[\frac{1}{8}\hat{\rho}]}(\tilde{x})+1\,:\;\tilde{x}\in\partial\bP\,\text{s.t. }x\in\overline{\Ball{\frac{1}{8}\hat{\rho}\of{\tilde{x}}}{\tilde{x}}}\right\} \,,\label{eq:tilde-M}
\end{align}
Here we have used the convention $\max\emptyset=\min\emptyset=0$.
\begin{rem}
a) In view of Lemma \ref{lem:M-eta} we recall Remark \ref{rem:difference-M-eta-M-eta}
on the difference between $M_{[\eta]}$ and $M_{\eta}$ and additionally
remark that $\rmM_{[\frac{\hat{\rho}}{8}]}{\left(x\right)}+1\leq\tilde{M}_{\hat{\rho}}(x)$
for every $x\in\partial\bP$.

b) We could equally work with $\delta$ replacing $\hat{\rho}$. However,
Lemma \ref{lem:rho-p-lsc} suggests that the natural choice is $\hat{\rho}$.
\end{rem}

Additionally introduce (recalling (\ref{eq:def-M-set-A}))\nomenclature[Mk]{$M_k,\,M_{\fr,k}$ }{$k\in\N$, $\fr>0$ \eqref{eq:notation-M}}
\begin{equation}
\fm_{k}:=\fm_{[\frac{1}{8}\hat{\rho}]}(p_{k},\frac{1}{4}\tilde{\rho})\,,\qquad\tilde{M}_{k}:=\tilde{M}(p_{k})\,,\qquad M_{k}:=M\of{p_{k},\frac{1}{8}\hat{\rho}(p_{k})}\label{eq:notation-M}
\end{equation}
We further recall that there exists $\fr_{k}=\frac{\tilde{\rho}_{k}}{32\left(1+\fm_{k}\right)}$,
\nomenclature[mk]{$\fm_{k}:=\fm\of{p_{k},\tilde{\rho}_{k}/4}$}{Section \ref{subsec:5-Preliminaries}}
and $y_{k}$ such that 
\[
B_{k}:=\Ball{\fr_{k}}{y_{k}}\subset\bP\cap\Ball{\frac{1}{8}\tilde{\rho}_{k}}{p_{k}}
\]

\begin{lem}
\label{lem:properties-local-rho-convering}For two balls $A_{1,k}\cap A_{1,j}\neq\emptyset$
either $A_{1,k}\subset A_{2,j}$ or $A_{1,j}\subset A_{2,k}$ and
\begin{equation}
A_{1,k}\cap A_{1,j}\neq\emptyset\quad\Rightarrow\quad\Ball{\frac{1}{2}\tilde{\rho}_{k}}{p_{k}}\subset A_{2,j}\text{ and }\Ball{\frac{1}{2}\tilde{\rho}_{j}}{p_{j}}\subset A_{2,k}\,.\label{eq:lem:properties-local-rho-convering-1}
\end{equation}
Furthermore, there exists a constant $C$ depending only on the dimension
$d$ and some $\hat{d}\in[0,d]$ such that 
\begin{align}
\forall k\quad &  & \#\left\{ j\,:\;A_{1,j}\cap A_{1,k}\neq\emptyset\right\} +\#\left\{ j\,:\;A_{2,j}\cap A_{2,k}\neq\emptyset\right\}  & \leq C\,,\label{eq:lem:properties-local-rho-convering-2}\\
\forall x\quad &  & \#\left\{ j\,:\;x\in A_{1,j}\right\} +\#\left\{ j\,:\;x\in A_{2,j}\right\}  & \leq C+1\,,\label{eq:lem:properties-local-rho-convering-3}\\
\forall x\quad &  & \#\left\{ j:\,x\in\overline{\Ball{\frac{1}{8}\hat{\rho}_{j}}{p_{j}}}\right\}  & <C\tilde{M}(x)^{\hat{d}}\,.\label{eq:lem:properties-local-rho-convering-4}
\end{align}
Finally, there exist non-negative functions $\phi_{0}$ and $\left(\phi_{k}\right)_{k\in\N}$
such that for $k\geq1$: $\support\phi_{k}\subset A_{1,k}$, $\phi_{k}|_{B_{j}}\equiv0$
for $k\neq j$. Further, $\phi_{0}\equiv0$ on all $B_{k}$ and on
$\partial\bP$ and $\sum_{k=0}^{\infty}\phi_{k}\equiv1$ and there
exists $C$ depending only on $d$ such that for all $j\in\N\cup\left\{ 0\right\} $,
$k\in\N$ it holds and 
\begin{equation}
x\in A_{1,k}\quad\Rightarrow\quad\left|\nabla\phi_{j}(x)\right|\leq C\tilde{\rho}_{k}^{-1}\,.\label{eq:bound-nabla-phi-j}
\end{equation}
\end{lem}

\begin{rem}
\label{rem:lem:properties-local-rho-convering}We usually can improve
$\hat{d}$ to at least $\hat{d}=d-1$. To see this assume $\partial\bP$
is locally connected. Then all points $p_{i}$ lie on a $d-1$-dimensional
plane and we can thus improve the argument in the following proof
to $\hat{d}=d-1$.
\end{rem}

\begin{proof}
(\ref{eq:lem:properties-local-rho-convering-1}) follows from (\ref{eq:cor:cover-boundary-h1})$_{2}$.

Let $k\in\N$ be fixed. By construction in Corollary \ref{cor:cover-boundary},
every $A_{1,j}$ with $A_{1,j}\cap A_{1,k}\neq\emptyset$ satisfies
$\tilde{\rho}_{j}\geq\frac{1}{2}\tilde{\rho}_{k}$ and hence if $A_{1,j}\cap A_{1,k}\neq\emptyset$
and $A_{1,i}\cap A_{1,k}\neq\emptyset$ we find $\left|p_{j}-p_{i}\right|\geq\frac{1}{4}\tilde{\rho}_{k}$
and $\left|p_{j}-p_{k}\right|\leq3\tilde{\rho}_{k}$. This implies
(\ref{eq:lem:properties-local-rho-convering-2})--(\ref{eq:lem:properties-local-rho-convering-3})
for $A_{1,j}$ and the statement for $A_{2,j}$ follows analogously.

For two points $p_{i}$, $p_{j}$ such that $x\in A_{3,i}\cap A_{3,j}$
it holds due to the triangle inequality $\left|p_{i}-p_{j}\right|\leq\max\left\{ \frac{1}{4}\hat{\rho}_{i},\frac{1}{4}\hat{\rho}_{j}\right\} $.
Let $\X(x):=\left\{ p_{i}\in\X:\,x\in\overline{\Ball{\frac{1}{8}\hat{\rho}_{i}}{p_{i}}}\right\} $
and choose $\tilde{p}(x)=\tilde{p}\in\X(x)$ such that $\hat{\rho}_{\text{m}}:=\hat{\rho}\of{\tilde{p}}$
is maximal. Then $\X(x)\subset\Ball{\frac{1}{4}\hat{\rho}_{\text{m}}}{\tilde{p}}$
and every $p_{i}\in\X(x)$ satisfies $\hat{\rho}_{\text{m}}>\hat{\rho}_{i}>\frac{1}{3}\hat{\rho}_{\text{m}}$.
Correspondingly, $\tilde{\rho}_{i}>\frac{1}{3}\hat{\rho}_{\text{m}}2^{-5}\tilde{M}_{i}^{-1}$
for all such $p_{i}$. In view of (\ref{eq:cor:cover-boundary-h1})
this lower local bound of $\tilde{\rho}_{i}$ implies a lower local
bound on the mutual distance of the $p_{i}$. Since this distance
is proportional to $\hat{\rho}_{\text{m}}\tilde{M}_{i}^{-1}$, and
since $\hat{\rho}_{\text{m}}>\hat{\rho}_{i}>\frac{1}{3}\hat{\rho}_{\text{m}}$,
this implies (\ref{eq:lem:properties-local-rho-convering-4}) with
$\hat{d}=d$. This is by the same time the upper estimate on $\hat{d}$.

Let $\phi:\,\R\to\R$ be symmetric, smooth, monotone on $(0,\infty)$
with $\phi'\leq2$ and $\phi=0$ on $(1,\infty)$. For each $k$ we
consider a radially symmetric smooth function $\tilde{\phi}_{k}\of x:=\phi\of{\frac{\left|x-p_{k}\right|^{2}}{\tilde{\rho}_{k}}}$
and an additional function $\tilde{\phi}_{0}\left(x\right)=\dist\of{\,x,\,\partial\bP\cup\bigcup_{k}\Ball{\fr_{k}}{y_{k}}\,}$.
In a similar way we may modify $\tilde{\phi}_{k}$ such that $\tilde{\phi}_{k}|_{B_{j}}\equiv0$
for $j\neq k$. Then we define $\phi_{k}:=\tilde{\phi}/\left(\tilde{\phi}_{0}+\sum_{j}\tilde{\phi}_{j}\right)$.
Note that by construction of $\fr_{k}$ and $y_{k}$ we find $\phi_{k}|_{B_{k}}\equiv1$
and $\sum_{k\geq1}\phi_{k}\equiv1$ on $\partial\bP$.

Estimate (\ref{eq:bound-nabla-phi-j}) follows from (\ref{eq:lem:properties-local-rho-convering-2}).
\end{proof}

\subsection{\label{subsec:Microscopic-Extension-dm}Extension Estimate Through
$\left(\delta,M\right)$-Regularity of $\partial\protect\bP$}

By Lemmas \ref{lem:rho-p-lsc} and \ref{lem:uniform-extension-lemma}
the local extension operator 
\begin{equation}
\cU_{k}:\,\,W^{1,p}\of{\bP\cap A_{3,k}}\,\to\,W^{1,p}\of{\Ball{\frac{1}{8}\rho_{k}}{p_{k}}\backslash\bP}\,\hookrightarrow\,W^{1,p}\of{A_{2,k}\backslash\bP}\label{eq:very-local-extension}
\end{equation}
is linear continuous with bounds 
\begin{align}
\left\Vert \nabla\cU_{k}u\right\Vert _{L^{p}(A_{2,k}\backslash\bP)} & \leq14M_{k}\left\Vert \nabla u\right\Vert _{L^{p}\left(A_{3,k}\cap\bP\right)}\,,\label{eq:lem:local-delta-M-extension-estimate-h2}\\
\left\Vert \cU_{k}u\right\Vert _{L^{p}(A_{2,k}\backslash\bP)} & \leq7\left\Vert u\right\Vert _{L^{p}\left(A_{3,k}\cap\bP\right)}\,,\label{eq:lem:local-delta-M-extension-estimate-h2-b}
\end{align}
and for constants $c$ we find 
\begin{equation}
\left\Vert c-\cU_{k}c\right\Vert _{L^{p}(A_{2,k}\backslash\bP)}=0\,.\label{eq:lem:local-delta-M-extension-estimate-h2-c}
\end{equation}

\begin{defn}
\label{def:cU-Q}For every $\bQ\subset\Rd$ let $\tau_{i}u:=\frac{1}{\left|\Ball{\fr_{i}}{y_{i}}\right|}\int_{\Ball{\fr_{i}}{y_{i}}}u$
and 
\begin{align*}
\cU_{\bQ}:\,C^{1}\of{\overline{\bP\cap\Ball{\frac{\fr}{2}}{\bQ}}} & \to C^{1}\of{\overline{\bQ\backslash\bP}}\,,\\
u & \mapsto\chi_{\bQ\backslash\bP}\sum_{k}\phi_{k}\left(\cU_{k}\of{u-\tau_{k}u}+\tau_{k}u\right)
\end{align*}
where $\cU_{k}$ are the extension operators on $A_{3,k}$ given by
Lemma \ref{lem:uniform-extension-lemma}, respectively (\ref{eq:very-local-extension})--(\ref{eq:lem:local-delta-M-extension-estimate-h2-c}).
Furthermore, we observe 
\begin{equation}
\cU_{\bQ}=\tilde{\cU}_{\bQ}+\hat{\cU}_{\bQ}\,,\qquad\text{with}\qquad\tilde{\cU}_{\bQ}u:=\chi_{\bQ\backslash\bP}\sum_{k}\phi_{k}\,\cU_{k}\of{u-\tau_{k}u}\,,\quad\hat{\cU}_{\bQ}u:=\chi_{\bQ\backslash\bP}\sum_{k}\phi_{k}\,\tau_{k}u\label{eq:def:cU-Q-2}
\end{equation}
\end{defn}

For two points $p_{i}$ and $p_{j}$ such that $A_{1,i}\cap A_{1,j}\neq\emptyset$
we find 
\begin{align}
\left|\tau_{i}u-\tau_{j}u\right|^{r} & =\left(\left|\Ball{\frac{\fr_{j}}{2}}0\right|^{-1}\left|\int_{\Ball{\frac{\fr_{j}}{2}}{x_{j}}}\left(u(\cdot)-\tau_{i}u\right)\right|\right)^{r}\nonumber \\
 & \leq\left|\Ball{\frac{\fr_{j}}{2}}0\right|^{-1}\int_{\Ball{\frac{\fr_{j}}{2}}{x_{j}}}\left|u(\cdot)-\tau_{i}u\right|^{r}\leq\left|\Ball{\frac{\fr_{j}}{2}}0\right|^{-1}\int_{A_{1,i}}\left|u(\cdot)-\tau_{i}u\right|^{r}\nonumber \\
 & \leq\left|\Ball{\frac{\fr_{j}}{2}}0\right|^{-1}\rho_{i}^{r}\int_{\conv\left(A_{1,i}\cup A_{1,j}\right)}\left|\nabla\cU_{i}u\right|^{r}\leq\left|\Ball{\frac{\fr_{j}}{2}}0\right|^{-1}\rho_{i}^{r}\int_{A_{2,i}}\left|\nabla\cU_{i}u\right|^{r}\,.\label{eq:tau-i-tau-j-difference-estimate}
\end{align}
The latter expression is not symmetric in $i,j$. Hence we can play
a bit with the indices in order to optimize our estimates below. We
have seen that $\fr_{j}\simeq\rho_{j}M_{j}^{-1}$, and hence we expect
in view of (\ref{eq:lem:local-delta-M-extension-estimate-h2})
\begin{equation}
\left|\tau_{i}u-\tau_{j}u\right|^{r}\leq CM_{j}^{d}\tilde{\rho}_{j}^{-d}\tilde{\rho}_{i}^{r}\int_{A_{3,i}\cap\bP}\left|\nabla u\right|^{r}\,.\label{eq:assu:M-alpha-bound-rough}
\end{equation}
However, this needs not to be the optimal estimate. Instead of the
general and restrictive estimate (\ref{eq:assu:M-alpha-bound-rough}),
we make the following Assumption:
\begin{assumption}
\label{assu:M-alpha-bound}There exists $\alpha\in[0,d]$ and $C>0$
such that for every $k$ it holds $\fr_{k}\geq C\hat{\rho}_{k}M_{k}^{-\frac{\alpha}{d}}$.
In particular, for two points $p_{i},p_{j}\in\partial\Y$ with $p_{i}\sim p_{j}$
it holds
\begin{equation}
\left|\tau_{i}u-\tau_{j}u\right|^{r}\leq C\tilde{\rho_{j}}^{-d}M_{j}^{\alpha}\tilde{\rho}_{i}^{r}\int_{A_{3,i}\cap\bP}\left|\nabla u\right|^{r}\,.\label{eq:assu:M-alpha-bound-improved}
\end{equation}
\end{assumption}

In order to formulate our main results we define the general sets\nomenclature[Rd]{$\Rd_1,\,\Rd_3$}{\eqref{eq:def-Rd1-Rd3}}
\begin{equation}
\Rd_{1}:=\bigcup_{k}A_{1,k}\,,\qquad\Rd_{3}:=\bigcup_{k}A_{3,k}\label{eq:def-Rd1-Rd3}
\end{equation}
 and for every bounded set $\bQ\subset\Rd$ we define\nomenclature[Q1]{$\bQ_1,\,\bQ_3$}{\eqref{eq:def-Q1-Q3}}
\begin{equation}
\bQ_{1}:=\bQ\cap\Rd_{1}\,,\qquad\bQ_{3}:=\bQ\cap\Rd_{3}\,.\label{eq:def-Q1-Q3}
\end{equation}

\begin{lem}
\label{lem:local-delta-M-extension-estimate} Let $\bP\subset\Rd$
be a locally $\left(\delta,M\right)$-regular open set with delta
bounded by $\fr>0$ and let Assumption \ref{assu:M-alpha-bound} hold
and let $\hat{d}$ be the constant from (\ref{eq:lem:properties-local-rho-convering-4}).
Then for every bounded open $\bQ\subset\Rd$, $1\leq r<p$ the operators
\begin{align*}
\tilde{\cU}_{\bQ},\hat{\cU}_{\bQ}:\,W^{1,p}\left(\bP\cap\Ball{\frac{\fr}{2}}{\bQ}\right) & \to W^{1,r}\left(\bQ\backslash\bP\right)
\end{align*}
are linear, well defined and satisfy 
\begin{align}
\norm{\nabla\tilde{\cU}_{\bQ}u}_{L^{r}\left(\bQ\backslash\bP\right)}^{r} & \leq C_{0}\left(\frac{1}{\left|\bQ\right|}\int_{\Ball{\frac{\fr}{2}}{\bQ}\cap\bP}\tilde{M}^{\frac{p\left(\hat{d}+1\right)}{p-r}}\right)^{\frac{p-r}{p}}\norm{\nabla u}_{L^{p}\left(\bP\cap\Ball{\frac{\fr}{2}}{\bQ}\right)}^{\frac{r}{p}}\label{lem:local-delta-M-extension-estimate-estim-1}\\
\norm{\nabla\hat{\cU}_{\bQ}u}_{L^{r}\left(\bQ\backslash\bP\right)}^{r} & \leq C_{0}\left(\frac{1}{\left|\bQ\right|}\int_{\Ball{\frac{\fr}{2}}{\bQ}\cap\bP}\tilde{M}^{\frac{p\left(\hat{d}+\alpha\right)}{p-r}}\right)^{\frac{p-r}{p}}\norm{\nabla u}_{L^{p}\left(\bP\cap\Ball{\frac{\fr}{2}}{\bQ}\right)}^{\frac{r}{p}}\label{lem:local-delta-M-extension-estimate-estim-1-a}\\
 & \quad+C_{0}\frac{1}{\left|\bQ\right|}\int_{\Ball{\frac{\fr}{2}}{\bQ}\backslash\bP}\left|\nabla\phi_{0}\right|^{r}\sum_{\substack{j\neq0:\,\partial_{l}\phi_{j}\partial_{l}\phi_{0}<0}
}\frac{\left|\partial_{l}\phi_{j}\right|}{D_{l+}}\left|\tau_{j}u\right|^{r}\,,\\
\norm{\cU_{\bQ}u}_{L^{r}\left(\bQ\backslash\bP\right)}^{r} & \leq C_{0}\left(\frac{1}{\left|\bQ\right|}\int_{\Ball{\frac{\fr}{2}}{\bQ}\cap\bP}\tilde{M}^{\frac{p\hat{d}}{p-r}}\right)^{\frac{p-r}{p}}\norm u_{L^{p}\left(\Ball{\frac{\fr}{2}}{\bQ}\right)}^{\frac{r}{p}}\label{lem:local-delta-M-extension-estimate-estim-1-b}
\end{align}
where $D_{l+}:=\sum_{\substack{j\neq0:\,\partial_{l}\phi_{j}\partial_{l}\phi_{0}<0}
}\left|\partial_{l}\phi_{j}\right|$. Furthermore, for constant functions $x\mapsto c\in\R$ it holds
\begin{equation}
\norm{c-\cU_{\bQ}c}_{L^{r}\left(\bQ\backslash\bP\right)}\leq\left|c\right|\left|\bQ\backslash\bP\right|^{\frac{1}{r}}\,.\label{eq:lem:local-delta-M-extension-estimate-8}
\end{equation}
\end{lem}

The second term in (\ref{lem:local-delta-M-extension-estimate-estim-1-a})
imposes severe problems, as we will see in Sections \ref{sec:Construction-of-Macroscopic-1},
\ref{subsec:The-Issue-of-Connectedness} or even in Lemma \ref{lem:first-estim-nabla-phi-0}
below.
\begin{lem}
\label{lem:conv-sum-0}Let $\alpha_{i}$, $u_{i}$, $i=1\dots n$,
be a family of real numbers such that $\sum_{i}\alpha_{i}=0$ and
let $\alpha_{+}:=\sum_{i:\,\alpha_{i}>0}\alpha_{i}$. Then 
\[
\sum_{i}\alpha_{i}u_{i}=\sum_{i:\,\alpha_{i}>0}\sum_{j:\,\alpha_{j}<0}\frac{\alpha_{i}\left|\alpha_{j}\right|}{\alpha_{+}}\left(u_{i}-u_{j}\right)\,.
\]
\end{lem}

\begin{proof}
\begin{align*}
\sum_{i}\alpha_{i}u_{i} & =\sum_{i:\,\alpha_{i}>0}\alpha_{i}u_{i}+\sum_{j:\,\alpha_{j}<0}\alpha_{j}u_{j}\\
 & =\sum_{i:\,\alpha_{i}>0}\alpha_{i}\sum_{j:\,\alpha_{j}<0}\frac{-\alpha_{j}}{\alpha_{+}}u_{i}+\sum_{j:\,\alpha_{j}<0}\alpha_{j}\sum_{i:\,\alpha_{i}>0}\frac{\alpha_{i}}{\alpha_{+}}u_{j}\\
 & =\sum_{i:\,\alpha_{i}>0}\sum_{j:\,\alpha_{j}<0}\frac{\alpha_{i}\left|\alpha_{j}\right|}{\alpha_{+}}\left(u_{i}-u_{j}\right)\,.
\end{align*}
\end{proof}

\begin{proof}[Proof of Lemma \ref{lem:local-delta-M-extension-estimate}]
 For shortness of notation (and by abuse of notation) we write
\[
\fint_{\bP\cap\bQ}g:=\frac{1}{\left|\bQ\right|}\int_{\bP\cap\bQ}g\,,\qquad\fint_{\bQ\backslash\bP}g:=\frac{1}{\left|\bQ\right|}\int_{\bQ\backslash\bP}g
\]
and similar for integrals over $\Ball{\frac{\fr}{2}}{\bQ}\cap\bP$
and $\Ball{\frac{\fr}{2}}{\bQ}\backslash\bP$.

\textbf{Step 1: }We note that $\tilde{\rho}_{k}\leq\frac{1}{8}\delta_{k}$
as well as $\sqrt{4M_{k}^{2}+2}\leq2\tilde{M}_{k}$. The integral
over $\nabla\left(\tilde{\cU}_{\bQ}u\right)$ can be estimated via
\begin{gather}
\fint_{\bQ\backslash\bP}\left|\nabla\sum_{i\neq0}\phi_{i}\cU_{i}\left(u-\tau_{i}u\right)\right|^{r}\leq C_{r}\left(I_{1}+I_{2}\right)\label{eq:lem:dm-general-estimate-1}\\
I_{1}=\fint_{\bQ\backslash\bP}\left|\sum_{i\neq0}\cU_{i}\left(u-\tau_{i}u\right)\nabla\phi_{i}\right|^{r}\,,\qquad I_{2}:=\fint_{\bQ\backslash\bP}\left|\sum_{i\neq0}\phi_{i}\nabla\cU_{i}\left(u-\tau_{i}u\right)\right|^{r}\,.\nonumber 
\end{gather}
(\ref{eq:bound-nabla-phi-j}) together with Jensen's yields 
\begin{align*}
\fint_{\bQ\backslash\bP}\left|\sum_{i\neq0}\cU_{i}\left(u-\tau_{i}u\right)\nabla\phi_{i}\right|^{r} & \leq C\,\sum_{i\neq0}\fint_{\bQ\backslash\bP}\left|\cU_{i}\left(u-\tau_{i}u\right)\right|^{r}\delta_{i}^{-r}\chi_{A_{1,i}}\\
 & \leq C\,\sum_{i\neq0}\fint_{\bQ}\chi_{A_{2,i}}\left|\nabla\cU_{i}\left(u-\tau_{i}u\right)\right|^{r}\\
 & \leq C\,\sum_{i\neq0}\tilde{M}_{i}\fint_{\bQ\cap\bP}\chi_{A_{3,i}}\left|\nabla u\right|^{r}
\end{align*}
where we used Lemma \ref{lem:scaled-poincare} with $\frac{R}{r}=3$
and inequality (\ref{eq:lem:local-delta-M-extension-estimate-h2}).
In a similar way, we conclude 
\begin{align*}
&\fint_{\bQ\backslash\bP}\left|\sum_{i\neq0}\phi_{i}\nabla\cU_{i}\left(u-\tau_{i}u\right)\right|^{r} \\
\qquad\qquad & \leq\fint_{\bQ\backslash\bP}\sum_{i\neq0}\phi_{i}\left|\nabla\cU_{i}\left(u-\tau_{i}u\right)\right|^{r}\leq\fint_{\bQ\backslash\bP}\sum_{i\neq0}\chi_{A_{1,i}}\left|\nabla\cU_{i}\left(u-\tau_{i}u\right)\right|^{r}\\
 \qquad\qquad& \leq C\,\sum_{i\neq0}\fint_{\bQ}\chi_{A_{2,i}}\left|\nabla\cU_{i}\left(u-\tau_{i}u\right)\right|^{r}\leq C\,\sum_{i\neq0}\tilde{M}_{i}\fint_{\Ball{\frac{\fr}{2}}{\bQ}\cap\bP}\chi_{A_{3,i}}\left|\nabla u\right|^{r}\,.
\end{align*}
It only remains to estimate $\sum_{i}\chi_{A_{3,i}}(x)$. Inequality
(\ref{eq:lem:properties-local-rho-convering-4}) yields 
\begin{align}
\sum_{i\neq0}\tilde{M}_{i}\fint_{\bQ\cap\bP}\chi_{A_{3,i}}\left|\nabla u\right|^{r} & \leq\fint_{\Ball{\frac{\fr}{2}}{\bQ}\cap\bP}\sum_{i\neq0}\chi_{A_{3,i}}\tilde{M}\left|\nabla u\right|^{r}\nonumber \\
 & \leq\left(\fint_{\Ball{\frac{\fr}{2}}{\bQ}\cap\bP}\left(\sum_{i\neq0}\chi_{A_{3,i}}\right)^{\frac{p}{p-r}}\tilde{M}^{\frac{p}{p-r}}\right)^{\frac{p-r}{p}}\left(\fint_{\Ball{\frac{\fr}{2}}{\bQ}\cap\bP}\left|\nabla u\right|^{p}\right)^{\frac{r}{p}}\nonumber \\
 & \leq\left(\fint_{\Ball{\frac{\fr}{2}}{\bQ}\cap\bP}\tilde{M}^{\frac{p\left(\hat{d}+1\right)}{p-r}}\right)^{\frac{p-r}{p}}\left(\fint_{\Ball{\frac{\fr}{2}}{\bQ}\cap\bP}\left|\nabla u\right|^{p}\right)^{\frac{r}{p}}\,.\label{eq:thm:dm-Ext-sto-connected-help-5-0}
\end{align}

\textbf{Step2:} We now study $\hat{\cU}_{\bQ}$ and use Lemma \ref{lem:conv-sum-0}
which yields
\begin{equation}
\sum_{j}\partial_{l}\phi_{j}=0\quad\Rightarrow\quad D_{l+}:=\sum_{j:\,\partial_{l}\phi_{j}>0}\partial_{l}\phi_{j}=-\sum_{j:\,\partial_{l}\phi_{j}<0}\partial_{l}\phi_{j}\label{eq:dm-sum-partial-Phi-1}
\end{equation}
that
\begin{align}
\fint_{\bQ\backslash\bP}\left|\nabla\sum_{j}\phi_{j}\tau_{j}u\right|^{r} & \leq C\sum_{l=1}^{d}\fint_{\bQ\backslash\bP}\left|\sum_{j}\partial_{l}\phi_{j}\tau_{j}u\right|^{r}\nonumber \\
 & \leq C\sum_{l=1}^{d}\fint_{\bQ\backslash\bP}\left|\sum_{\substack{i\neq0:\,\partial_{l}\phi_{i}>0}
}\sum_{\substack{j\neq0:\,\partial_{l}\phi_{j}<0}
}\frac{\partial_{l}\phi_{i}\left|\partial_{l}\phi_{j}\right|}{D_{l+}}\left|\tau_{i}u-\tau_{j}u\right|\right|^{r}+I_{3}\,,\label{eq:thm:dm-Ext-sto-connected-help-6-0}
\end{align}
where 
\begin{equation}
I_{3}=C\sum_{l=1}^{d}\fint_{\bQ\backslash\bP}\left|\sum_{\substack{j\neq0:\,\partial_{l}\phi_{j}\partial_{l}\phi_{0}<0}
}\frac{\left|\partial_{l}\phi_{0}\right|\left|\partial_{l}\phi_{j}\right|}{D_{l+}}\left|\tau_{j}u\right|\right|^{r}\,.\label{eq:thm:dm-Ext-sto-connected-help-def-I-3}
\end{equation}
Since in (\ref{eq:thm:dm-Ext-sto-connected-help-def-I-3}) $\sum_{\substack{j\neq0:\,\partial_{l}\phi_{j}\partial_{l}\phi_{0}<0}
}\left|\partial_{l}\phi_{j}\right|=D_{l+}$ we obtain 
\begin{align*}
I_{3} & =C\sum_{l=1}^{d}\fint_{\bQ\backslash\bP}\left|\partial_{l}\phi_{0}\right|^{r}\sum_{\substack{j\neq0:\,\partial_{l}\phi_{j}\partial_{l}\phi_{0}<0}
}\frac{\left|\partial_{l}\phi_{j}\right|}{D_{l+}}\left|\tau_{j}u\right|^{r}\,.
\end{align*}
We will now derive an estimate on $\left|\tau_{i}u-\tau_{j}u\right|$.
For this reason, denote $l_{ij}$ the line from $x_{i}$ to $x_{j}$
and by $\Ball{\frac{\fr}{2}}{l_{ij}}$ the set of all points with
distance less than $\frac{\fr}{2}$ to $l_{ij}$. We exploit the fact
that every term in the sum on the right hand side of (\ref{eq:thm:dm-Ext-sto-connected-help-6-0})
appears only once and introduce 
\[
E_{l}(x)=\left\{ (i,j)\,:\;\partial_{l}\phi_{i}\partial_{l}\phi_{j}<0\,\text{and }\fr_{i}<\fr_{j}\text{ or }(\fr_{i}=\fr_{j}\text{ and }i<j)\right\} \,.
\]
We make use of (\ref{eq:assu:M-alpha-bound-improved}) and successively
apply Jensen's inequality, $\left|\nabla\phi_{i}\right|\leq C\tilde{\rho}_{i}^{-1}$,
$\frac{1}{C}\rho_{i}\leq\rho_{j}\leq C\rho_{i}$ and $\left|\Ball{\frac{\fr_{j}}{2}}0\right|^{-1}\leq\tilde{M}_{i}^{d}\left|A_{1,i}\right|^{-1}$
to obtain 
\begin{align*}
S:= & \left|\sum_{i:\,\partial_{l}\phi_{i}>0}\sum_{j:\,\partial_{l}\phi_{j}<0}\frac{\partial_{l}\phi_{i}\left|\partial_{l}\phi_{j}\right|}{D_{l+}}\left|\tau_{i}u-\tau_{j}u\right|\right|^{r}=\left|\sum_{\left(i,j\right)\in E_{l}}\frac{\left|\partial_{l}\phi_{i}\right|\left|\partial_{l}\phi_{j}\right|}{D_{l+}}\left|\tau_{i}u-\tau_{j}u\right|\right|^{r}\\
 & \leq\sum_{\left(i,j\right)\in E_{l}}\frac{\tilde{\rho}_{i}^{-r}\left|\partial_{l}\phi_{j}\right|}{D_{l+}}C\tilde{\rho_{j}}^{-d}M_{j}^{\alpha}\tilde{\rho}_{i}^{r}\int_{A_{3,i}\cap\bP}\left|\nabla u\right|^{r}\,.
\end{align*}
Hence we find 
\begin{align}
\fint_{\bQ\backslash\bP}\left|\nabla\sum_{j}\phi_{j}\tau_{j}u\right|^{r} & \apprle C\sum_{l=1}^{d}\sum_{\left(i,j\right)\in E_{l}}\frac{\tilde{\rho}_{i}^{-r}\left|\partial_{l}\phi_{j}\right|}{D_{l+}}C\tilde{\rho_{j}}^{-d}M_{j}^{\alpha}\tilde{\rho}_{i}^{r}\int_{A_{3,i}\cap\bP}\left|\nabla u\right|^{r}\,.\nonumber \\
 & \apprle C\frac{1}{\left|\bQ\right|}\sum_{i}\tilde{M}_{i}^{\alpha}\int_{A_{3,i}\cap\bP}\left|\nabla u\right|^{r}\label{eq:thm:dm-Ext-sto-connected-help-6-2}
\end{align}
Similar to (\ref{eq:thm:dm-Ext-sto-connected-help-5-0}) we may conclude
(\ref{lem:local-delta-M-extension-estimate-estim-1-a}).

\textbf{Step 3:} We observe with Jensen's inequality and the fact
that $\cU_{i}$ are linear with $\cU_{i}c=c$ for constants $c$ that
\begin{align*}
\fint_{\bQ\backslash\bP}\left|\sum_{i}\phi_{i}\,\cU_{i}\left(u-\tau_{i}u\right)+\phi_{i}\tau_{i}u\right|^{r} & \leq\fint_{\bQ\backslash\bP}\sum_{i}\phi_{i}\,\left(\cU_{i}u\right)^{r}\leq\fint_{\bQ\backslash\bP}\sum_{i}\phi_{i}\,\left(\cU_{i}u\right)^{r}\\
 & \leq\fint_{\bQ\backslash\bP}\sum_{i}\chi_{A_{1,i}}\,\left(\cU_{i}u\right)^{r}\\
 & \leq7\fint_{\Ball{\frac{\fr}{2}}{\bQ}\cap\bP}\sum_{i}\chi_{A_{3,i}}\,u^{r}
\end{align*}
From here we may proceed as in (\ref{eq:thm:dm-Ext-sto-connected-help-5-0})
to conclude .
\end{proof}
\begin{lem}
\label{lem:first-estim-nabla-phi-0} Let $\bP\subset\Rd$ be a locally
$\left(\delta,M\right)$-regular open set with delta bounded by $\fr>0$
and let Assumption \ref{assu:M-alpha-bound} hold and let $\hat{d}$
be the constant from (\ref{eq:lem:properties-local-rho-convering-4}).
Then for every bounded open $\bQ\subset\Rd$, $1\leq r<p_{0}<p_{1}<p$
\begin{multline*}
\frac{1}{\left|\bQ\right|}\int_{\Ball{\frac{\fr}{2}}{\bQ}\backslash\bP}\left|\nabla\phi_{0}\right|^{r}\sum_{\substack{j\neq0:\,\partial_{l}\phi_{j}\partial_{l}\phi_{0}<0}
}\frac{\left|\partial_{l}\phi_{j}\right|}{D_{l+}}\left|\tau_{j}u\right|^{r}\\
\leq C\left(\frac{1}{\left|\bQ\right|}\int_{\Ball{\frac{\fr}{2}}{\bQ}\backslash\bP}\left|\nabla\phi_{0}\right|^{\frac{rp_{0}}{p_{0}-r}}\tilde{M}^{2-d}\right)^{\frac{p_{0}-r}{p_{0}}}\left(\frac{1}{\left|\bQ\right|}\int_{\Ball{\frac{\fr}{2}}{\bQ}\backslash\bP}\tilde{M}^{\frac{p_{1}\left(d-2\right)(p_{0}-r)}{r(p_{1}-p_{0})}}\right)^{r\frac{p_{1}-p_{0}}{p_{1}p_{0}}}\\
\left(\frac{1}{\left|\bQ\right|}\int_{\Ball{\frac{\fr}{2}}{\bQ}}\tilde{M}^{\frac{\alpha p_{1}p}{p-p_{1}}}\right)^{r\frac{p-p_{1}}{pp_{1}}}\left(\frac{1}{\left|\bQ\right|}\int_{\Ball{\frac{\fr}{2}}{\bQ}}\left|u\right|^{p}\right)^{\frac{r}{p}}
\end{multline*}
and
\begin{multline*}
\frac{1}{\left|\bQ\right|}\int_{\Ball{\frac{\fr}{2}}{\bQ}\backslash\bP}\left|\nabla\phi_{0}\right|^{r}\sum_{\substack{j\neq0:\,\partial_{l}\phi_{j}\partial_{l}\phi_{0}<0}
}\frac{\left|\partial_{l}\phi_{j}\right|}{D_{l+}}\left|\tau_{j}u\right|^{r}\\
\leq C\left(\frac{1}{\left|\bQ\right|}\int_{\Ball{\frac{\fr}{2}}{\bQ}\backslash\bP}\left|\nabla\phi_{0}\right|^{\frac{rp_{0}}{p_{0}-r}}\right)^{\frac{p_{0}-r}{p_{0}}}\left(\frac{1}{\left|\bQ\right|}\int_{\Ball{\frac{\fr}{2}}{\bQ}}\tilde{M}^{\frac{\alpha p_{0}p}{p-p_{0}}}\right)^{r\frac{p-p_{0}}{pp_{0}}}\left(\frac{1}{\left|\bQ\right|}\int_{\Ball{\frac{\fr}{2}}{\bQ}}\left|u\right|^{p}\right)^{\frac{r}{p}}
\end{multline*}
\end{lem}

\begin{proof}
We observe with H\"older and Jensens inequality on $\Rd$ 
and in the sum $\sum_{\substack{j\neq0:\,\partial_{l}\phi_{j}\partial_{l}\phi_{0}<0}
}\frac{\left|\partial_{l}\phi_{j}\right|}{D_{l+}}=1$ respectively that 
\begin{multline*}
\frac{1}{\left|\bQ\right|}\int_{\Ball{\frac{\fr}{2}}{\bQ}\backslash\bP}\left|\nabla\phi_{0}\right|^{r}\sum_{\substack{j\neq0:\,\partial_{l}\phi_{j}\partial_{l}\phi_{0}<0}
}\frac{\left|\partial_{l}\phi_{j}\right|}{D_{l+}}\left|\tau_{j}u\right|^{r}\leq\left(\frac{1}{\left|\bQ\right|}\int_{\Ball{\frac{\fr}{2}}{\bQ}\backslash\bP}\left|\nabla\phi_{0}\right|^{\frac{rp_{0}}{p_{0}-r}}\tilde{M}^{2-d}\right)^{\frac{p_{0}-r}{p_{0}}}\\
\left(\frac{1}{\left|\bQ\right|}\int_{\Ball{\frac{\fr}{2}}{\bQ}\backslash\bP}\tilde{M}^{\frac{1}{r}\left(d-2\right)(s-r)}\sum_{\substack{j\neq0:\,\partial_{l}\phi_{j}\partial_{l}\phi_{0}<0}
}\frac{\left|\partial_{l}\phi_{j}\right|}{D_{l+}}\left|\tau_{j}u\right|^{p_{0}}\right)^{\frac{r}{p_{0}}}\,.
\end{multline*}
Applying the same trick again we find 
\begin{align*}
 & \frac{1}{\left|\bQ\right|}\int_{\Ball{\frac{\fr}{2}}{\bQ}\backslash\bP}\tilde{M}^{\frac{1}{r}\left(d-2\right)(p_{0}-r)}\sum_{\substack{j\neq0:\,\partial_{l}\phi_{j}\partial_{l}\phi_{0}<0}
}\frac{\left|\partial_{l}\phi_{j}\right|}{D_{l+}}\left|\tau_{j}u\right|^{p_{0}}\\
 & \qquad\qquad\leq\left(\frac{1}{\left|\bQ\right|}\int_{\Ball{\frac{\fr}{2}}{\bQ}\backslash\bP}\tilde{M}^{\frac{p_{1}\left(d-2\right)(p_{0}-r)}{r(p_{1}-p_{0})}}\right)^{\frac{p_{1}-p_{0}}{p_{1}}}\left(\frac{1}{\left|\bQ\right|}\int_{\Ball{\frac{\fr}{2}}{\bQ}\backslash\bP}\sum_{\substack{j\neq0:\,\partial_{l}\phi_{j}\partial_{l}\phi_{0}<0}
}\frac{\left|\partial_{l}\phi_{j}\right|}{D_{l+}}\left|\tau_{j}u\right|^{p_{1}}\right)^{\frac{p_{0}}{p_{1}}}
\end{align*}
From the definition of $\tau_{j}$ and (\ref{eq:lem:properties-local-rho-convering-3})
we find 
\begin{align*}
 & \frac{1}{\left|\bQ\right|}\int_{\Ball{\frac{\fr}{2}}{\bQ}\backslash\bP}\sum_{\substack{j\neq0:\,\partial_{l}\phi_{j}\partial_{l}\phi_{0}<0}
}\frac{\left|\partial_{l}\phi_{j}\right|}{D_{l+}}\left|\tau_{j}u\right|^{p_{1}}\\
 & \qquad\qquad\leq\frac{1}{\left|\bQ\right|}\sum_{p_{j}\in\Ball{\frac{\fr}{4}}{\bQ}}\tilde{\rho}_{j}^{d}\frac{M_{j}^{\alpha p_{1}}}{\tilde{\rho}_{j}^{d}}\int_{\Ball{\tilde{\rho}_{j}}{p_{j}}}\left|u\right|^{p_{1}}\\
 & \qquad\qquad\leq\frac{1}{\left|\bQ\right|}\int_{\Ball{\frac{\fr}{2}}{\bQ}}\tilde{M}^{\alpha p_{1}}\left|u\right|^{p_{1}}\\
 & \qquad\qquad\leq\left(\frac{1}{\left|\bQ\right|}\int_{\Ball{\frac{\fr}{2}}{\bQ}}\tilde{M}^{\frac{\alpha p_{1}p}{p-p_{1}}}\right)^{\frac{p-p_{1}}{p}}\left(\frac{1}{\left|\bQ\right|}\int_{\Ball{\frac{\fr}{2}}{\bQ}}\left|u\right|^{p}\right)^{\frac{p_{1}}{p}}\,.
\end{align*}
\end{proof}

\subsection{Traces on $\left(\delta,M\right)$-Regular Sets}
\begin{thm}
\label{thm:uniform-trace-estimate-1}Let $\bP\subset\Rd$ be a locally
$\left(\delta,M\right)$-regular open set, $\frac{1}{8}>\fr>0$ and
let $\bQ\subset\Rd$ be a bounded open set and let $1\leq r<p_{0}<p$.
Then the trace operator $\cT$ satisfies for every $u\in W_{\loc}^{1,p}\of{\bP}$
\[
\frac{1}{\left|\bQ\right|}\int_{\bQ\cap\partial\bP}\left|\cT u\right|^{r}\leq C\left(\frac{1}{\left|\bQ\right|}\int_{\Ball{\frac{1}{4}}{\bQ}\cap\bP}\left|u\right|^{p}+\left|\nabla u\right|^{p}\right)^{\frac{r}{p}}
\]
where for some constant $C_{0}$ depending only on $p_{0}$, $p$
and $r$ and $d$ and for $\eta\in\left\{ \rho,\hat{\rho},\delta\right\} $
one may chose between 
\begin{align}
C & =C_{0}\left(\frac{1}{\left|\bQ\right|}\int_{\Ball{\frac{1}{4}\fr}{\bQ}\cap\partial\bP}\eta^{-\frac{1}{p_{0}-r}}\right)^{\frac{p_{0}-r}{p_{0}}}\left(\frac{1}{\left|\bQ\right|}\int_{\Ball{\frac{1}{4}\fr}{\bQ}\cap\bP}\tilde{M}_{[\frac{1}{8}\eta],\Rd}^{\left(\frac{1}{p_{0}}+1\right)\frac{p}{p-p_{0}}}\right)^{\frac{p-p_{0}}{p_{0}p}}\,,\label{eq:lem:uniform-trace-estimate-1-1}\\
C & =C_{0}\left(\frac{1}{\left|\bQ\right|}\int_{\Ball{\frac{1}{4}\fr}{\bQ}\cap\partial\bP}\left(\eta M_{[\frac{1}{16}\eta],\Rd}\right)^{-\frac{1}{p-r}}\right)^{\frac{p-r}{p}}\,.
\end{align}
\end{thm}

\begin{proof}
Using Theorem \ref{thm:delta-M-rho-covering}, we cover $\partial\bP$
by balls $B_{k}=\Ball{\frac{1}{16}\eta\of{p_{k}}}{p_{k}}$ with $\left(p_{k}\right)_{k\in\N}\subset\partial\bP$
and define $\hat{B}_{k}=\Ball{\frac{1}{8}\eta\of{p_{k}}}{p_{k}}$
and $M_{k}=M_{[\frac{1}{16}\eta]}(p_{k})$. Like for (\ref{eq:lem:properties-local-rho-convering-3})
we can show that the covering with both $B_{k}$ and $\hat{B}_{k}$
is locally uniformly bounded by a constant $C$. Due to Lemma \ref{lem:basic-trace}
we find locally
\begin{equation}
\norm{\cT u}_{L^{p_{0}}(\partial\bP\cap B_{k})}\leq C_{p_{0},p_{0}}\eta^{-\frac{1}{p_{0}}}\sqrt{4M_{k}^{2}+2}^{\frac{1}{p_{0}}+1}\norm u_{W^{1,p_{0}}\left(\hat{B}_{k}\right)}\,.\label{eq:lem:uniform-trace-estimate-2-2}
\end{equation}
If $\phi_{k}$ is a partition of $1$ on $\partial\bP$ with respective
support $B_{k}$ we obtain 
\begin{multline*}
\frac{1}{\left|\bQ\right|}\int_{\bQ\cap\partial\bP}\left|\sum_{k}\phi_{k}\cT_{k}u\right|^{r}\\
\leq\left(\frac{1}{\left|\bQ\right|}\int_{\Ball{\frac{1}{4}}{\bQ}\cap\partial\bP}\sum_{k}\chi_{B_{k}}\eta_{k}^{-\frac{1}{p_{0}-r}}\right)^{\frac{p_{0}-r}{p_{0}}}\left(\frac{1}{\left|\bQ\right|}\sum_{k}\int_{\Ball{\frac{1}{4}}{\bQ}\cap\partial\bP}\chi_{B_{k}}\eta_{k}\left|\cT_{k}u\right|^{p_{0}}\right)^{\frac{r}{p_{0}}}
\end{multline*}
which yields by the uniform local bound of the covering, $\tilde{\eta}$
defined in Lemma \ref{lem:delta-tilde-construction-estimate}, twice
the application of (\ref{eq:lem:local-delta-M-construction-estimate-2})
and (\ref{eq:lem:uniform-trace-estimate-2-2})
\begin{align*}
\frac{1}{\left|\bQ\right|}\int_{\bQ\cap\partial\bP}\left|\sum_{k}\phi_{k}\cT_{k}u\right|^{r} & \leq\left(\frac{1}{\left|\bQ\right|}\int_{\bQ\cap\partial\bP}\eta^{-\frac{1}{p_{0}-r}}\right)^{\frac{p_{0}-r}{p_{0}}}\cdot\\
 & \qquad\cdot\left(\frac{1}{\left|\bQ\right|}\int_{\bQ\cap\bP}\sum_{k}\chi_{\hat{B}_{k}}\sqrt{4M_{k}^{2}+2}^{\frac{1}{p_{0}}+1}\left(\left|\nabla u\right|^{p_{0}}+\left|u\right|^{p_{0}}\right)\right)^{\frac{r}{p_{0}}}\,.
\end{align*}
With H\"olders inequality, the last estimate leads to (\ref{eq:lem:uniform-trace-estimate-1-1}).
The second estimate goes analogue since the local covering by $A_{2,k}$
is finite.
\end{proof}

\section{\label{sec:Construction-of-Macroscopic}Construction of Macroscopic
Extension Operators I: General Considerations}

\begin{figure}
 \begin{minipage}[c]{0.5\textwidth} \includegraphics[width=6cm]{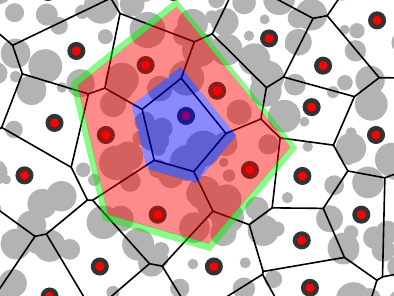}\end{minipage}\hfill   \begin{minipage}[c]{0.45\textwidth}\caption{\label{fig:sketch-extension}Gray: a Poisson ball process. Black balls:
balls of radius $\protect\fr>0$. Red Balls: radius $\frac{\protect\fr}{2}$.
The Voronoi tessellation is generated from the centers of the red
balls. The existence of such tessellations is discussed in Section
\ref{subsec:Mesoscopic-Regularity}. Blue region: $\protect\fA_{1,k}$
according to Assumption \ref{assu:mesoscopic-regular}. Red region:
$\protect\fA_{2,k}$. Green region: an alternative choice of $\protect\fA_{3,k}$.}
\end{minipage}
\end{figure}

In this section, we provide the extension results which answer the
question of the existence of such uniformly bounded families of operators
up to the issue of quantifying \emph{connectedness}. We will discuss
what we mean by that in Section \ref{subsec:The-Issue-of-Connectedness}.
In Section \ref{subsec:Extension-for-Statistically} we provide a
first attempt to from the point of view of continuous PDE, which is
- in some sense - a tautology. However, verifying the conditions of
Theorem \ref{thm:final-estimate} in a computer based approach (for
real life geometries) leads to a discretization of an elliptic second
order operator. Therefore, in Section \ref{sec:Construction-of-Macroscopic-1}
we use the construction of Section \ref{subsec:Connectedness-of--Regular-sets}
to introduce a quantity which can be directly calculated from a numerical
algorithm.

\subsection{Extension for Voronoi Tessellations}
\begin{assumption}
\label{assu:mesoscopic-voronoi}Let $\bP$ be an open set and let
$\X_{\fr}=\left(x_{i}\right)_{i=\in\N}$ have mutual distance $\left|x_{i}-x_{k}\right|>2\fr$
if $i\neq k$ and with $\Ball{\frac{\fr}{2}}{x_{i}}\subset\bP$ for
every $i\in\N$ (e.g. $\X_{\fr}(\bP)$, see (\ref{eq:def-X_r})).
We construct from $\X_{\fr}$ a Voronoi tessellation and denote by
$G_{i}:=G(x_{i})$ the Voronoi cell corresponding to $x_{i}$ with
diameter $d_{i}$. We denote $\fA_{1,i}:=\Ball{\frac{\fr}{2}}{G_{i}}$
and \nomenclature[Af1]{$\fA_{1,i}:=\Ball{\frac{\fr}{2}}{G_{i}}$}{Assumption \ref{assu:mesoscopic-voronoi}}
\nomenclature[Mc]{$\cM_i$}{Equation \eqref{eq:def-mean-cM}}
\begin{equation}
\cM_{i}u:=\left|\Ball{\frac{\fr}{16}}0\right|^{-1}\int_{\Ball{\frac{\fr}{16}}{x_{i}}}u\,.\label{eq:def-mean-cM}
\end{equation}
Let $\tilde{\Phi}_{0}\in C^{\infty}(\R;[0,1])$ be monotone decreasing
with $\tilde{\Phi}_{0}'>-\frac{4}{\fr}$, $\tilde{\Phi}_{0}(x)=1$
if $x\leq0$ and $\tilde{\Phi}_{0}(x)=0$ for $x\geq\frac{\fr}{2}$.
We define on $\Rd$ the functions 
\begin{equation}
\tilde{\Phi}_{i}(x):=\tilde{\Phi}_{0}\left(\dist\left(x,G_{i}\right)\right)\quad\text{and}\quad\Phi_{i}(x):=\tilde{\Phi}_{i}(x)\left(\sum_{j}\tilde{\Phi}_{j}(x)\right)^{-1}\,.\label{eq:def-Phi-i}
\end{equation}
\end{assumption}

Lemma \ref{lem:Iso-cone-geo-estimate}.2) implies 
\begin{equation}
\forall x\in\Ball{\frac{\fr}{2}}{G_{i}}\,:\quad\#\left\{ k\,:\;x\in\fA_{1,k}\right\} \leq\left(\frac{4d_{i}}{\fr}\right)^{d}\label{eq:assu:mesoscopic-regular-1}
\end{equation}
and thus (\ref{eq:def-Phi-i}) yields for some $C$ depending only
on $\tilde{\Phi}_{0}$ that
\begin{equation}
\left|\nabla\Phi_{i}\right|\leq Cd_{i}^{d}\quad\text{and}\quad\forall k:\,\left|\nabla\Phi_{k}\right|\chi_{\fA_{1,i}}\leq Cd_{i}^{d}\,.\label{eq:estiamte-nabla-Phi-i}
\end{equation}

\begin{defn}[Weak Neighbors]
\label{def:neighbors}\nomenclature[neighbors weak]{$x_i\sim\sim x_j$}{$x_i$ and $x_j$ are weak neighbors or weakly connected, see Definition \ref{Y-Partial-Y}}Under
the Assumption \ref{assu:mesoscopic-voronoi}, two points $x_{i}$
and $x_{j}$ are called to be  weakly connected (or weak neighbors),
written $i\sim\sim j$ or $x_{i}\sim\sim x_{j}$ if $\Ball{\frac{\fr}{2}}{G_{i}}\cap\Ball{\frac{\fr}{2}}{G_{j}}\neq\emptyset$.
For $\bQ\subset\Rd$ open we say $\fA_{1,j}\sim\sim\bQ$ if $\Ball{\frac{\fr}{2}}{\fA_{1,j}}\cap\bQ\neq\emptyset$.
We then define \nomenclature[XrQ]{$\X_\fr(\bQ)$}{\eqref{eq:Xr-Q-Q-simsim}}\nomenclature[Qsim]{$\bQ^{\sim\sim}$}{\eqref{eq:Xr-Q-Q-simsim}}
\begin{equation}
\X_{\fr}(\bQ):=\left\{ x_{j}\in\X_{\fr}:\;\fA_{1,j}\sim\sim\bQ\neq\emptyset\right\} \,,\quad\bQ^{\sim\sim}:=\bigcup_{\fA_{1,j}\sim\sim\bQ}\fA_{1,j}\,.\label{eq:Xr-Q-Q-simsim}
\end{equation}
\end{defn}

Let $\bP$ be locally $\left(\delta,M\right)$-regular and satisfy
Assumption \ref{assu:mesoscopic-voronoi}. Then we can construct continuous
local extension operators $\cU_{G_{j}}:\,W^{1,p}\of{\Ball{\fr}{G_{i}}}\to W^{1,r}\of{\Ball{\frac{\fr}{2}}{G_{i}}}$
from Lemma \ref{lem:local-delta-M-extension-estimate}. These can
be glued together via 
\[
\cU_{\bQ}u:=\sum_{j}\Phi_{j}\left(\cU_{G_{j}}\left(u-\cM_{j}u\right)+\cM_{j}u\right)\,.
\]
However, using the partition of unity from Lemma \ref{lem:properties-local-rho-convering}
and the definition of $U_{G_{i}}$ from (\ref{eq:def:cU-Q-2}) we
obtain 
\[
\cU_{\bQ}u=\sum_{j}\Phi_{j}\left(\sum_{i}\phi_{i}\left[\cU_{i}\left(u-\cM_{j}u-\tau_{i}\left(u-\cM_{j}u\right)\right)+\tau_{i}\left(u-\cM_{j}u\right)\right]+\cM_{j}u\right)\,.
\]
Using $\tau_{i}\cM_{j}u=\cM_{j}u$ the latter yields
\begin{align*}
\cU_{\bQ}u & =\sum_{j}\Phi_{j}\left(\sum_{i}\phi_{i}\left[\cU_{i}\left(u-\tau_{i}u\right)+\tau_{i}u-\cM_{j}u\right]+\cM_{j}u\right)\\
 & =\sum_{i}\sum_{j}\Phi_{j}\left(\phi_{i}\,\left(\cU_{i}\of{u-\tau_{i}u}+\tau_{i}u-\cM_{j}u\right)+\cM_{j}u\right)\,,
\end{align*}
where we used that $\cU_{i}$ maps constants onto constants via the
identity. Note that $$\cU_{\bQ}u\neq\sum_{i}\phi_{i}\left[\cU_{i}\left(u-\tau_{i}u\right)+\tau_{i}u\right]\,,$$
as $\sum_{i\neq0}\phi_{i}\neq1$ in most points.
\begin{thm}[Extensions for locally regular, isotropic cone mixing geometries]
\label{thm:Fine-meso-estimate}Let the open set $\bP$ be locally
$(\delta,M)$-regular, $\delta$ bounded by $\frac{\fr}{2}>0$, and
satisfy Assumptions \ref{assu:M-alpha-bound}, \ref{assu:mesoscopic-voronoi}
and $\hat{d}$ be the constant from (\ref{eq:lem:properties-local-rho-convering-4}),.
Let $1<r<s<t<p<+\infty$ and $s<p_{0}\leq p$ with $1-\frac{\hat{d}}{r}\geq\frac{\hat{d}}{s}$.

Recalling (\ref{eq:notation-M}) and defining $\bQ_{\fr}:=\Ball{\fr}{\bQ}$
as well as 
\begin{align}
\cU u & :=\sum_{i}\sum_{j}\Phi_{j}\left(\phi_{i}\,\left(\cU_{i}\of{u-\tau_{i}u}+\tau_{i}u-\cM_{j}u\right)+\cM_{j}u\right)\label{eq:def-global-U}
\end{align}
the following estimates hold:
\begin{align}
\frac{1}{\left|\bQ\right|}\int_{\bQ\backslash\bP}\left|\nabla\cU u\right|^{r} & \leq C_{0}\left(\frac{1}{\left|\bQ\right|}\int_{\bQ_{\fr}\cap\bP}\tilde{M}^{\frac{p\left(\hat{d}+\alpha\right)}{p-r}}\right)^{\frac{p-r}{p}}\norm{\nabla u}_{L^{p}\left(\bP\cap\Ball{\frac{\fr}{2}}{\bQ_{\fr}}\right)}^{\frac{r}{p}}\nonumber \\
 & \quad+\frac{1}{\left|\bQ\right|}\int_{\bQ^{\sim\sim}}\left|f\of u\right|^{r}\label{eq:thm:Fine-meso-estimate-1}\\
\frac{1}{\left|\bQ\right|}\int_{\bQ\backslash\bP}\left|\cU u\right|^{r} & \leq C_{0}\left(\frac{1}{\left|\bQ\right|}\int_{\bQ_{\fr}\cap\bP}\tilde{M}^{\frac{p\hat{d}}{p-r}}\right)^{\frac{p-r}{p}}\cdot\left(\frac{1}{\left|\bQ\right|}\int_{\bP\cap\bQ_{\fr}}\left|u\right|^{p}\right)^{\frac{r}{p}}\nonumber \\
 & \qquad+C_0 \left(\frac{1}{\left|\bQ\right|}\int_{\bQ\cap\bP}\left(\sum_{x_{j}\in\X_{\fr}(\bQ)}\chi_{G_{j}}\left|\fA_{1,j}\right|\right)^{\frac{p}{p-r}}\right)^{\frac{p-r}{p}}\cdot\left(\frac{1}{\left|\bQ\right|}\int_{\bP\cap\bQ_{\fr}}\left|u\right|^{p}\right)^{\frac{r}{p}}\,,\label{eq:thm:Fine-meso-estimate-2}
\end{align}
where 
\begin{align*}
f\of u & =\sum_{l=1}^{d}\sum_{k:\,\partial_{l}\Phi_{k}>0}\sum_{j:\,\partial_{l}\Phi_{j}<0}\frac{\partial_{l}\Phi_{k}\left|\partial_{l}\Phi_{j}\right|}{D_{l+}^{\Phi}}\left(2-\phi_{0}\right)\left(\cM_{k}u-\cM_{j}u\right)\\
 & \quad-\sum_{l=1}^{d}\sum_{\substack{i\neq0:\,\partial_{l}\phi_{i}\partial_{l}\phi_{0}<0}
}\sum_{j}\frac{\partial_{l}\phi_{0}\left|\partial_{l}\phi_{i}\right|}{D_{l+}}\Phi_{j}\left(\tau_{i}u-\cM_{j}u\right)\,.
\end{align*}
with functions
\[
D_{l+}:=\sum_{\substack{j\neq0:\,\partial_{l}\phi_{j}\partial_{l}\phi_{0}<0}
}\left|\partial_{l}\phi_{j}\right|\,,\qquad D_{l+}^{\Phi}:=\sum_{\substack{j\neq0:\,\partial_{l}\Phi_{j}<0}
}\left|\partial_{l}\Phi_{j}\right|\,.
\]
\end{thm}

\begin{proof}
 Let us note that on $\bQ$ it holds 
\begin{align}
\cU u: & =\sum_{i}\phi_{i}\,\cU_{i}\of{u-\tau_{i}u}+\sum_{x_{j}\in\X_{\fr}(\bQ)}\sum_{i}\Phi_{j}\phi_{i}\left(\tau_{i}u-\cM_{j}u\right)+\sum_{x_{j}\in\X_{\fr}(\bQ)}\Phi_{j}\cM_{j}u\label{eq:thm:Fine-meso-estimate-help-1}\\
 & =\sum_{i}\phi_{i}\,\left(\cU_{i}\of{u-\tau_{i}u}+\tau_{i}u\right)+\sum_{x_{j}\in\X_{\fr}(\bQ)}\Phi_{j}\phi_{0}\cM_{j}u\,.\label{eq:thm:Fine-meso-estimate-help-2}
\end{align}
We first observe in (\ref{eq:thm:Fine-meso-estimate-help-2}) that
\begin{align*}
\frac{1}{\left|\bQ\right|}\int_{\bQ}\left|\sum_{x_{j}\in\X_{\fr}(\bQ)}\Phi_{j}\phi_{0}\cM_{j}u\right|^{r} & \leq\frac{1}{\left|\bQ\right|}\int_{\bQ}\sum_{x_{j}\in\X_{\fr}(\bQ)}\Phi_{j}\phi_{0}\left|\cM_{j}u\right|^{r}\leq\frac{1}{\left|\bQ\right|}\int_{\bQ}\sum_{x_{j}\in\X_{\fr}(\bQ)}\chi_{\fA_{1,j}}\left|\cM_{j}u\right|^{r}\\
 & \leq\frac{1}{\left|\bQ\right|}\sum_{x_{j}\in\X_{\fr}(\bQ)}\left|\fA_{1,j}\right|\left|\S^{d-1}\right|\left(\frac{\fr}{2}\right)^{d}\int_{\Ball{\frac{\fr}{2}}{x_{i}}}\left|u\right|^{r}\\
 & \leq\frac{1}{\left|\bQ\right|}\int_{\bQ}\left|u\right|^{r}\sum_{x_{j}\in\X_{\fr}(\bQ)}\chi_{G_{j}\cap\bP}\left|\fA_{1,j}\right|\left|\S^{d-1}\right|\left(\frac{\fr}{2}\right)^{d}
\end{align*}

From the last inequality and Lemma \ref{lem:local-delta-M-extension-estimate}
we obtain (\ref{eq:thm:Fine-meso-estimate-2}). Furthermore, the first
term on the right hand side of (\ref{eq:thm:Fine-meso-estimate-help-1})
with Lemma \ref{lem:local-delta-M-extension-estimate} provides the
first line of (\ref{eq:thm:Fine-meso-estimate-1}).

In what follows, we write for simplicity $\sum_{x_{j}\in\X_{\fr}(\bQ)}=\sum_{j}$
but have in mind the respective meaning. The same holds for $\sum_{k:\,\partial_{l}\Phi_{k}>0}$.

Concerning the second term in (\ref{eq:thm:Fine-meso-estimate-help-1}),
we observe 
\begin{multline*}
\nabla\sum_{x_{j}\in\X_{\fr}(\bQ)}\sum_{i\in\N}\Phi_{j}\phi_{i}\left(\tau_{i}u-\cM_{j}u\right)\\
	=\sum_{j}\sum_{i\in\N}\phi_{i}\left(\tau_{i}u-\cM_{j}u\right)\nabla\Phi_{j}+\sum_{j}\sum_{i\in\N}\Phi_{j}\left(\tau_{i}u-\cM_{j}u\right)\nabla\phi_{i}\,,
\end{multline*}
and obtain with help of Lemma \ref{lem:conv-sum-0} and $\sum_{j}\nabla\Phi_{j}(x)=0$
for $x\in\bQ$
\begin{align*}
 & \sum_{j}\sum_{i\in\N}\phi_{i}\left(\tau_{i}u-\cM_{j}u\right)\nabla\Phi_{j}\\
 & \qquad=\sum_{l=1}^{d}\sum_{k:\,\partial_{l}\Phi_{k}>0}\sum_{j:\,\partial_{l}\Phi_{j}<0}\frac{\partial_{l}\Phi_{k}\left|\partial_{l}\Phi_{j}\right|}{D_{l+}^{\Phi}}\left(\sum_{i\in\N}\phi_{i}\left(\tau_{i}u-\cM_{k}u\right)-\sum_{i\in\N}\phi_{i}\left(\tau_{i}u-\cM_{j}u\right)\right)\\
 & \qquad=\sum_{l=1}^{d}\sum_{k:\,\partial_{l}\Phi_{k}>0}\sum_{j:\,\partial_{l}\Phi_{j}<0}\frac{\partial_{l}\Phi_{k}\left|\partial_{l}\Phi_{j}\right|}{D_{l+}^{\Phi}}\left(1-\phi_{0}\right)\left(\cM_{k}u-\cM_{j}u\right)\,.
\end{align*}
Similarly, we use Lemma \ref{lem:conv-sum-0} together with $\sum_{i}\nabla\phi_{i}=-\nabla\phi_{0}$
and find
\begin{align*}
 & \sum_{j}\sum_{i\in\N}\Phi_{j}\left(\tau_{i}u-\cM_{j}u\right)\nabla\phi_{i}\,,\\
 & \qquad=\sum_{l=1}^{d}\sum_{k:\,\partial_{l}\phi_{k}>0}\sum_{i:\,\partial_{l}\phi_{i}<0}\frac{\partial_{l}\phi_{k}\left|\partial_{l}\phi_{i}\right|}{D_{l+}}\left(\sum_{j}\Phi_{j}\left(\tau_{k}u-\cM_{j}u\right)-\sum_{j}\Phi_{j}\left(\tau_{i}u-\cM_{j}u\right)\right)\\
 & \qquad\quad-\sum_{l=1}^{d}\sum_{\substack{i\neq0:\,\partial_{l}\phi_{j}\partial_{l}\phi_{0}<0}
}\frac{\partial_{l}\phi_{0}\left|\partial_{l}\phi_{i}\right|}{D_{l+}}\Phi_{j}\left(\tau_{i}u-\cM_{j}u\right)\\
 & \qquad=\sum_{l=1}^{d}\sum_{k:\,\partial_{l}\phi_{k}>0}\sum_{j:\,\partial_{l}\phi_{i}<0}\frac{\partial_{l}\phi_{k}\left|\partial_{l}\phi_{i}\right|}{D_{l+}}\left(\tau_{k}u-\tau_{i}u\right)\\
 & \qquad\quad-\sum_{l=1}^{d}\sum_{\substack{i\neq0:\,\partial_{l}\phi_{j}\partial_{l}\phi_{0}<0}
}\frac{\partial_{l}\phi_{0}\left|\partial_{l}\phi_{i}\right|}{D_{l+}}\Phi_{j}\left(\tau_{i}u-\cM_{j}u\right)\,,
\end{align*}
where the first term on the right hand side can be estimated like
in Lemma \ref{lem:local-delta-M-extension-estimate}. Finally, from
a similar calculation using Lemma \ref{lem:conv-sum-0} it is now
obvious for the third term in (\ref{eq:thm:Fine-meso-estimate-help-1})
that
\begin{align*}
\sum_{j}\cM_{j}u\nabla\Phi_{j} & \leq\sum_{l=1}^{d}\sum_{k:\,\partial_{l}\Phi_{k}>0}\sum_{j:\,\partial_{l}\Phi_{j}<0}\frac{\partial_{l}\Phi_{k}\left|\partial_{l}\Phi_{j}\right|}{D_{l+}^{\Phi}}\left(\cM_{k}u-\cM_{j}u\right)\,.
\end{align*}
\end{proof}

\subsection{\label{subsec:The-Issue-of-Connectedness}The Issue of Connectedness}

In Theorem \ref{thm:Fine-meso-estimate} we discovered the integral
$\int_{\bQ}\left|f\of u\right|^{r}$ as part of the estimate for $\frac{1}{\left|\bQ\right|}\int_{\bQ\backslash\bP}\left|\nabla\cU u\right|^{r}$,
where we recall that $f$ was given through
\begin{align*}
f\of u & =\sum_{l=1}^{d}\sum_{k:\,\partial_{l}\Phi_{k}>0}\sum_{j:\,\partial_{l}\Phi_{j}<0}\frac{\partial_{l}\Phi_{k}\left|\partial_{l}\Phi_{j}\right|}{D_{l+}^{\Phi}}\left(2-\phi_{0}\right)\left(\cM_{k}u-\cM_{j}u\right)\\
 & \quad-\sum_{l=1}^{d}\sum_{\substack{i\neq0:\,\partial_{l}\phi_{i}\partial_{l}\phi_{0}<0}
}\sum_{j}\frac{\partial_{l}\phi_{0}\left|\partial_{l}\phi_{i}\right|}{D_{l+}}\Phi_{j}\left(\tau_{i}u-\cM_{j}u\right)\,.
\end{align*}
We seek for an interpretation of the two sums appearing on the right
hand side. The first one is related to the difference of mean values
around $x_{k}$ and $x_{j}$ in case they are weak neighbors, i.e.
$x_{k}\sim\sim x_{j}$. In Theorem \ref{thm:Rough-meso-estimate}
below we provide a rough estimate on this part in terms of $\tau_{i}u-\cM_{j}u$
but on a larger area. In the present section, we first want to ``isolate''
$\left|\tau_{i}u-\cM_{j}u\right|$ and $\left|\cM_{k}u-\cM_{j}u\right|$
from the other geometric properties of $\bP$. In Section \ref{sec:Construction-of-Macroscopic-1}
we will see how these quantities are related to the connectivity of
$\bP$.
\begin{lem}
\label{lem:6-4}Under Assumptions \ref{assu:M-alpha-bound}, \ref{assu:mesoscopic-voronoi}
and using the notation of Theorem \ref{thm:Fine-meso-estimate} let
$\left(f_{j}\right)_{j\in\N}$ be non-negative and have support $\support f_{j}\supset\Ball{\frac{\fr}{2}}{x_{j}}$
and let $\sum_{j\in\N}f_{j}\equiv1$. \\
Writing $\X(\bQ):=\left\{ x_{j}:\;\support f_{j}\cap\bQ\neq\emptyset\right\} $,
for every $l=1,\dots d$ and $r<\tilde{s}<s<p$ it holds 
\begin{multline*}
\left|\frac{1}{\left|\bQ\right|}\int_{\bQ_{\fr}}\sum_{\substack{i\neq0:\,\partial_{l}\phi_{i}\partial_{l}\phi_{0}<0}
}\sum_{j}\frac{\partial_{l}\phi_{0}\left|\partial_{l}\phi_{i}\right|}{D_{l+}}f_{j}\left(\tau_{i}u-\cM_{j}u\right)\right|\\
\leq\left(\frac{1}{\left|\bQ\right|}\int_{\bP\cap\bQ_{\fr}\cap\Rd_{3}}\left|\partial_{l}\phi_{0}\right|^{\frac{sr}{s-r}}\right)^{\frac{s-r}{s}}\left(\frac{1}{\left|\bQ\right|}\int_{\bP\cap\bQ_{\fr}}\sum_{\substack{i\neq0:\,\partial_{l}\phi_{i}\partial_{l}\phi_{0}<0}
}\sum_{x_{j}\in\X(\bQ)}f_{j}\frac{\left|\partial_{l}\phi_{i}\right|}{D_{l+}}\left|\tau_{i}u-\cM_{j}u\right|^{s}\right)^{\frac{r}{s}}
\end{multline*}
and
\begin{multline*}
\left|\frac{1}{\left|\bQ\right|}\int_{\bQ_{\fr}}\sum_{\substack{i\neq0:\,\partial_{l}\phi_{i}\partial_{l}\phi_{0}<0}
}\sum_{j}\frac{\partial_{l}\phi_{0}\left|\partial_{l}\phi_{i}\right|}{D_{l+}}f_{j}\left(\tau_{i}u-\cM_{j}u\right)\right|\\
\leq\left(\frac{1}{\left|\bQ\right|}\int_{\bP\cap\bQ_{\fr}\cap\Rd_{3}}\left|\partial_{l}\phi_{0}\right|^{\frac{\tilde{s}r}{\tilde{s}-r}}\tilde{M}^{2-d}\right)^{\frac{\tilde{s}-r}{\tilde{s}}}\left(\frac{1}{\left|\bQ\right|}\int_{\Ball{\frac{\fr}{2}}{\bQ}\backslash\bP}\tilde{M}^{\frac{p_{1}\left(d-2\right)(\tilde{s}-r)}{r(s-\tilde{s})}}\right)^{r\frac{s-\tilde{s}}{\tilde{s}s}}\\
\left(\frac{1}{\left|\bQ\right|}\int_{\bP\cap\bQ_{\fr}}\sum_{\substack{i\neq0:\,\partial_{l}\phi_{i}\partial_{l}\phi_{0}<0}
}\sum_{x_{j}\in\X(\bQ)}f_{j}\frac{\left|\partial_{l}\phi_{i}\right|}{D_{l+}}\left|\tau_{i}u-\cM_{j}u\right|^{s}\right)^{\frac{r}{s}}\,.
\end{multline*}
\end{lem}

\begin{proof}
We find from H\"older's and Jensen's inequality
\begin{align*}
 & \frac{1}{\left|\bQ\right|}\int_{\bP\cap\bQ}\left|\sum_{\substack{i\neq0:\,\partial_{l}\phi_{i}\partial_{l}\phi_{0}<0}
}\sum_{j}\frac{\partial_{l}\phi_{0}\left|\partial_{l}\phi_{i}\right|}{D_{l+}}f_{j}\left(\tau_{i}u-\cM_{j}u\right)\right|^{r}\\
 & \qquad\leq\frac{1}{\left|\bQ\right|}\int_{\bP\cap\bQ}\left|\partial_{l}\phi_{0}\right|^{r}\sum_{\substack{i\neq0:\,\partial_{l}\phi_{i}\partial_{l}\phi_{0}<0}
}\sum_{j}\frac{\left|\partial_{l}\phi_{i}\right|}{D_{l+}}f_{j}\left|\tau_{i}u-\cM_{j}u\right|^{r}\\
 & \qquad\leq\left(\frac{1}{\left|\bQ\right|}\int_{\bP\cap\bQ\cap\Rd_{3}}\left|\partial_{l}\phi_{0}\right|^{\frac{sr}{s-r}}\right)^{\frac{s-r}{s}}\left(\frac{1}{\left|\bQ\right|}\int_{\bP\cap\bQ}\sum_{\substack{i\neq0:\,\partial_{l}\phi_{i}\partial_{l}\phi_{0}<0}
}\sum_{j}f_{j}\frac{\left|\partial_{l}\phi_{i}\right|}{D_{l+}}\left|\tau_{i}u-\cM_{j}u\right|^{s}\right)^{\frac{r}{s}}\,.
\end{align*}
The other inequality can be derived similarly, see also the proof
of Lemma \ref{lem:first-estim-nabla-phi-0}.
\end{proof}
\begin{lem}
\label{lem:6-5}Under Assumptions \ref{assu:M-alpha-bound}, \ref{assu:mesoscopic-voronoi}
for every $l=1,\dots d$ it holds
\begin{multline*}
\frac{1}{\left|\bQ\right|}\int_{\bP\cap\bQ}\left|\sum_{k:\,\partial_{l}\Phi_{k}>0}\sum_{j:\,\partial_{l}\Phi_{j}<0}\frac{\partial_{l}\Phi_{k}\left|\partial_{l}\Phi_{j}\right|}{D_{l+}^{\Phi}}\left(2-\phi_{0}\right)\left(\cM_{k}u-\cM_{j}u\right)\right|^{r}\\
\leq\left(\frac{1}{\left|\bQ\right|}\int_{\bP\cap\bQ}\left(\sum_{j:\,\partial_{l}\Phi_{j}<0}d_{j}^{\frac{r(d-1)+drs}{s-r}}\chi_{\nabla\Phi_{j}\neq0}\right)^{\frac{s}{s-r}}\right)^{\frac{s-r}{s}}\frac{1}{\left|\bQ\right|}\sum_{\substack{x_{k}\sim\sim x_{j}\\
x_{k},x_{j}\in\X_{\fr}(\bQ)
}
}\left|\cM_{k}u-\cM_{j}u\right|^{s}
\end{multline*}
\end{lem}

\begin{proof}
For this we observe with help of (\ref{eq:estiamte-nabla-Phi-i})
and with Lemma \ref{lem:Iso-cone-geo-estimate}.2) 
\begin{align}
\forall x:\quad\sup_{k}\left|\partial_{l}\Phi_{k}\right|(x) & \leq\sup\left\{ \left|\nabla\Phi_{k}(x)\right|\,:\;x\in\Ball{\frac{\fr}{2}}{G_{k}}\right\} \nonumber \\
 & \leq C\sup\left\{ d_{k}^{d}\,:\;x\in G_{k}\right\} \,,\label{eq:thm:all-in-help-1}\\
\sup_{x\in\Ball{\frac{\fr}{2}}{G_{j}}}\left|\partial_{l}\Phi_{j}\right|(x) & \leq Cd_{j}^{d}\,.\label{eq:thm:all-in-help-2}
\end{align}
We write 
\[
I:=\frac{1}{\left|\bQ\right|}\int_{\bP\cap\bQ}\left|\sum_{k:\,\partial_{l}\Phi_{k}>0}\sum_{j:\,\partial_{l}\Phi_{j}<0}\frac{\partial_{l}\Phi_{k}\left|\partial_{l}\Phi_{j}\right|}{D_{l+}^{\Phi}}\left(2-\phi_{0}\right)\left(\cM_{k}u-\cM_{j}u\right)\right|^{r}
\]
and find 
\begin{align*}
I & \leq C\frac{1}{\left|\bQ\right|}\int_{\bP\cap\bQ}\sum_{k:\,\partial_{l}\Phi_{k}>0}\sum_{j:\,\partial_{l}\Phi_{j}<0}\frac{\left|\partial_{l}\Phi_{k}\right|^{r}\left|\partial_{l}\Phi_{j}\right|}{D_{l+}^{\Phi}}\left|\cM_{k}u-\cM_{j}u\right|^{r}\\
 & \leq CC\frac{1}{\left|\bQ\right|}\int_{\bP\cap\bQ}\left(\sum_{k:\,\partial_{l}\Phi_{k}>0}\sum_{j:\,\partial_{l}\Phi_{j}<0}\frac{d_{j}^{\alpha\frac{s}{s-r}}\left|\partial_{l}\Phi_{k}\right|^{\frac{sr}{s-r}}\left|\partial_{l}\Phi_{j}\right|}{D_{l+}^{\Phi}}\right)^{\frac{s-r}{s}}\cdot\dots\\
 & \qquad\dots\cdot\left(\sum_{k:\,\partial_{l}\Phi_{k}>0}\sum_{j:\,\partial_{l}\Phi_{j}<0}\frac{d_{j}^{-\alpha\frac{s}{r}}\left|\partial_{l}\Phi_{j}\right|}{D_{l+}^{\Phi}}\left|\cM_{k}u-\cM_{j}u\right|^{s}\right)^{\frac{r}{s}}\,.
\end{align*}
Now we make use of (\ref{eq:thm:all-in-help-1}) and once more of
Lemma \ref{lem:Iso-cone-geo-estimate}.2) to obtain for the first
bracket on the right hand side an estimate of the form 
\[
\left|\partial_{l}\Phi_{k}\right|^{\frac{sr}{s-r}}\left|\partial_{l}\Phi_{j}\right|\leq\left|\partial_{l}\Phi_{k}\right|\left|\partial_{l}\Phi_{k}\right|^{\frac{sr}{s-r}-1}\left|\partial_{l}\Phi_{j}\right|\leq C\left|\partial_{l}\Phi_{k}\right|d_{j}^{d\frac{sr-s+r}{s-r}}d_{j}^{d}\leq C\left|\partial_{l}\Phi_{k}\right|d_{j}^{d\frac{sr}{s-r}}\,,
\]
which implies 
\begin{align*}
\sum_{k:\,\partial_{l}\Phi_{k}>0}\sum_{j:\,\partial_{l}\Phi_{j}<0}\frac{d_{j}^{\alpha\frac{s}{s-r}}\left|\partial_{l}\Phi_{k}\right|^{\frac{sr}{s-r}}\left|\partial_{l}\Phi_{j}\right|}{D_{l+}^{\Phi}} & \leq C\sum_{k:\,\partial_{l}\Phi_{k}>0}\sum_{j:\,\partial_{l}\Phi_{j}<0}\frac{d_{j}^{\alpha\frac{s}{s-r}}d_{j}^{\frac{dsr}{s-r}}\left|\partial_{l}\Phi_{k}\right|}{D_{l+}^{\Phi}}\\
 & \leq C\sum_{j:\,\partial_{l}\Phi_{j}<0}d_{j}^{\alpha\frac{s}{s-r}}d_{j}^{\frac{dsr}{s-r}}\chi_{\nabla\Phi_{j}\neq0}\,,
\end{align*}
where we used $\sum\left|\partial_{l}\Phi_{k}\right|=D_{l+}^{\Phi}$.
We make use of $\alpha=rs^{-1}(d-1)$ in the above estimates and H\"older's
inequality to find
\begin{align*}
 & \frac{1}{\left|\bQ\right|}\int_{\bP\cap\bQ}\left|\sum_{k:\,\partial_{l}\Phi_{k}>0}\sum_{j:\,\partial_{l}\Phi_{j}<0}\frac{\partial_{l}\Phi_{k}\left|\partial_{l}\Phi_{j}\right|}{D_{l+}^{\Phi}}\left(2-\phi_{0}\right)\left(\cM_{k}u-\cM_{j}u\right)\right|^{r}\\
 & \qquad\leq C\left(\frac{1}{\left|\bQ\right|}\int_{\bP\cap\bQ}\left(\sum_{j:\,\partial_{l}\Phi_{j}<0}d_{j}^{\frac{r(d-1)+drs}{s-r}}\chi_{\nabla\Phi_{j}\neq0}\right)^{\frac{s}{s-r}}\right)^{\frac{s-r}{s}}\cdot\dots\\
 & \qquad\qquad\dots\cdot\left(\frac{1}{\left|\bQ\right|}\int_{\bP\cap\bQ}\sum_{k:\,\partial_{l}\Phi_{k}>0}\sum_{j:\,\partial_{l}\Phi_{j}<0}\frac{d_{j}^{1-d}\left|\partial_{l}\Phi_{j}\right|}{D_{l+}^{\Phi}}\left|\cM_{k}u-\cM_{j}u\right|^{s}\right)^{\frac{r}{s}}\,.
\end{align*}
Since $\int_{\bP\cap\bQ}\frac{d_{j}^{1-d}\left|\partial_{l}\Phi_{j}\right|}{D_{l+}^{\Phi}}\leq C$
for some $C>0$ independent from $j$, we obtain 
\begin{align*}
 & \frac{1}{\left|\bQ\right|}\int_{\bP\cap\bQ}\sum_{k:\,\partial_{l}\Phi_{k}>0}\sum_{j:\,\partial_{l}\Phi_{j}<0}\frac{d_{j}^{1-d}\left|\partial_{l}\Phi_{j}\right|}{D_{l+}^{\Phi}}\left|\cM_{k}u-\cM_{j}u\right|^{s}\\
 & \qquad\leq\frac{1}{\left|\bQ\right|}\sum_{x_{k}\sim\sim x_{j}}\left|\cM_{k}u-\cM_{j}u\right|^{s}\,.
\end{align*}
\end{proof}

\subsection{Estimates Related to Mesoscopic Regularity of the Geometry}
\begin{assumption}[Mesoscopic Regularity]
\label{assu:mesoscopic-regular} Under the Assumption \ref{assu:mesoscopic-voronoi}
and introducing the notation $\cI_{i}:=\left\{ x_{j}\in\X_{\fr}\,:\;\cH^{d-1}\left(\partial G_{i}\cap\partial G_{j}\right)\geq0\right\} $
we construct $\fA_{2,i}$ and $\fA_{3,i}$ from $\fA_{1,i}$ by \nomenclature[Af23]{$\fA_{2,i},\,\fA_{3,i}$}{\eqref{eq:def-fA-2-3}}
\begin{equation}
\fA_{2,i}:=\Ball{2d_{i}}{\fA_{1,i}}\,,\qquad\fA_{3,i}:=\Ball{2d_{i}+\fr}{\fA_{2,i}}\,.\label{eq:def-fA-2-3}
\end{equation}
We infer from Lemma \ref{lem:local-delta-M-extension-estimate}
that $\cU:\,W^{1,p}\of{\fA_{3,i}}\to W^{1,r}\of{\fA_{2,i}}$ is continuous
with the estimate and constants given by Lemma \ref{lem:local-delta-M-extension-estimate}.
\end{assumption}

\begin{thm}[Extensions for mesoscopic regular, isotropic cone mixing geometries]
\label{thm:Rough-meso-estimate} Let $\bP(\omega)$ be an open connected
set and let Assumption \ref{assu:mesoscopic-regular} hold. Let $\bP$
be locally $(\delta,M)$-regular and satisfy Assumptions \ref{assu:M-alpha-bound},
\ref{assu:mesoscopic-voronoi} and $\hat{d}$ be the constant from
(\ref{eq:lem:properties-local-rho-convering-4}). Then for almost
every $\omega$ it holds: for every $l=1,\dots,d$ and $1\leq r<s,\tilde{s}<p$:
\begin{align}
 & \frac{1}{\left|\bQ\right|}\int_{\bQ}\left|\sum_{k:\,\partial_{l}\Phi_{k}>0}\sum_{j:\,\partial_{l}\Phi_{j}<0}\frac{\partial_{l}\Phi_{k}\left|\partial_{l}\Phi_{j}\right|}{D_{l+}^{\Phi}}\left(2-\phi_{0}\right)\left|\cM_{j}u-\cM_{k}u\right|\right|^{r}\nonumber \\
 & \qquad\leq C\of{\bP(\omega)}\left(\frac{1}{\left|\bQ\right|}\int_{\bQ\cap\bP}\left|\nabla u\right|^{p}\right)^{\frac{r}{p}}\label{eq:thm:Rough-meso-estimate-1}\\
 & \qquad\quad+C_{P}\of{\bP(\omega)}\sum_{l}\left(\frac{1}{\left|\bQ\right|}\int_{\bQ\backslash\bP}\sum_{k}\frac{\chi_{\fA_{3,k}}}{\fa}\left|\nabla\phi_{0}\right|^{\tilde{s}}\sum_{\substack{j\neq0:\,\partial_{l}\phi_{j}\partial_{l}\phi_{0}<0}
}\frac{\left|\partial_{l}\phi_{j}\right|}{D_{l+}}\left|\tau_{j}u-\cM_{k}u\right|^{\tilde{s}}\right)^{\frac{r}{\tilde{s}}}\nonumber 
\end{align}
where with $P(x)=x^{d(2r-1)+r}(x^{r+1}+x^{d+1})$, $\fa:=\sum_{k}\chi_{\fA_{3,k}}$
and it holds 
\begin{align*}
C\of{\bP(\omega)} & =\left(\frac{C}{\left|\bQ\right|}\int_{\bP\cap\bQ}\left(\sum_{k}P\of{d_{k}}\chi_{\fA_{3,k}}\right)^{\frac{p}{p-s}}\right)^{\frac{p-s}{p}}\left(\frac{C}{\left|\bQ\right|}\int_{\bP\cap\bQ}\tilde{M}^{\frac{2p\hat{d}}{s-r}}\right)^{\frac{s-r}{p}}\\
C_{P}\of{\bP(\omega)} & =\left(\frac{1}{\left|\bQ\right|}\int_{\bQ\backslash\bP}\sum_{k}P\of{d_{k}}^{\frac{\tilde{s}}{\tilde{s}-r}}\fa^{\frac{\tilde{s}}{\tilde{s}-r}}\frac{\chi_{\fA_{3,k}}}{\fa}\right)^{\frac{\tilde{s}-r}{\tilde{s}}}
\end{align*}
\end{thm}

\begin{rem}
A combination with Lemma \ref{lem:6-4} is possible.
\end{rem}

\begin{proof}
We make use of (\ref{eq:assu:mesoscopic-regular-1}) as well as the
following observation: for each $k=1,\dots K$ let $\alpha_{k}\geq K$.
Then 
\[
\left(\sum_{k=1}^{K}f_{k}\right)^{r}\leq K^{r-1}\sum_{k=1}^{K}f_{k}^{r}\leq\sum_{k=1}^{K}\alpha_{k}^{r-1}f_{k}^{r}\,.
\]
Hence 
\begin{align*}
 & \left|\sum_{k:\,\partial_{l}\Phi_{k}>0}\sum_{j:\,\partial_{l}\Phi_{j}<0}\frac{\partial_{l}\Phi_{k}\left|\partial_{l}\Phi_{j}\right|}{D_{l+}^{\Phi}}\left(2-\phi_{0}\right)\left|\cM_{j}u-\cM_{k}u\right|\right|^{r}\\
 & \qquad\qquad\leq\sum_{k:\,\partial_{l}\Phi_{k}>0}\sum_{j:\,\partial_{l}\Phi_{j}<0}\left(\frac{4d_{k}}{\fr}\right)^{d(r-1)}\frac{\left|\partial_{l}\Phi_{k}\right|^{r}\left|\partial_{l}\Phi_{j}\right|}{D_{l+}^{\Phi}}2\left|\cM_{j}u-\cM_{k}u\right|^{r}
\end{align*}
Given $\left|\nabla\Phi_{k}\right|\leq\chi_{\fA_{1,k}}\left(\frac{4d_{k}}{\fr}\right)^{d}$
we hence find an estimate by 
\[
\sum_{k:\,\partial_{l}\Phi_{k}>0}\sum_{j:\,\partial_{l}\Phi_{j}<0}\left(\frac{4d_{k}}{\fr}\right)^{d(2r-1)}\chi_{\fA_{1,k}}\frac{\left|\partial_{l}\Phi_{j}\right|}{D_{l+}^{\Phi}}2\left|\cM_{j}u-\cM_{k}u\right|^{r}\,.
\]
Next, we obtain 
\[
\left|\cM_{j}u-\cM_{k}u\right|^{r}\leq\left|\fr^{d}\S^{d-1}\right|^{-1}\int_{\Ball{\frac{\fr}{2}}{x_{j}}}\left|u-\cM_{k}u\right|^{r}
\]
and thus 
\begin{align*}
 & \left|\sum_{k:\,\partial_{l}\Phi_{k}>0}\sum_{j:\,\partial_{l}\Phi_{j}<0}\frac{\partial_{l}\Phi_{k}\left|\partial_{l}\Phi_{j}\right|}{D_{l+}^{\Phi}}\left(2-\phi_{0}\right)\left|\cM_{j}u-\cM_{k}u\right|\right|^{r}\\
 & \qquad\leq2\left|\fr^{d}\S^{d-1}\right|^{-1}\sum_{k}\sum_{j:\,\partial_{l}\Phi_{j}<0}\left(\frac{4d_{k}}{\fr}\right)^{d(2r-1)}\chi_{\fA_{1,k}}\frac{\left|\partial_{l}\Phi_{j}\right|}{D_{l+}^{\Phi}}\int_{\Ball{\frac{\fr}{2}}{x_{j}}}\left|u-\cM_{k}u\right|^{r}\\
 & \qquad\leq2\left|\fr^{d}\S^{d-1}\right|^{-1}\sum_{k}\left(\frac{4d_{k}}{\fr}\right)^{d(2r-1)}\chi_{\fA_{1,k}}\int_{\fA_{2,k}}\left|u-\cM_{k}u\right|^{r}\\
 & \qquad\leq2\left|\fr^{d}\S^{d-1}\right|^{-1}\sum_{k}\left(\frac{4d_{k}}{\fr}\right)^{d(2r-1)}\chi_{\fA_{1,k}}C_{k}\int_{\fA_{2,k}}\left|\nabla\cU u\right|^{r}\\
 & \qquad\leq C\sum_{k}\left(\frac{4d_{k}}{\fr}\right)^{d(2r-1)}\chi_{\fA_{1,k}}C_{k}\hat{C}_{k,r,s}\left(\frac{1}{\left|\fA_{3,k}\right|}\int_{\fA_{3,k}\cap\bP}\left|\nabla u\right|^{s}\right)^{\frac{r}{s}}\\
 & \qquad+C\sum_{k}\left(\frac{4d_{k}}{\fr}\right)^{d(2r-1)}\chi_{\fA_{1,k}}C_{k}\frac{C_{0}}{\left|\fA_{3,k}\right|}\int_{\fA_{3,k}\backslash\bP}\left|\nabla\phi_{0}\right|^{r}\sum_{\substack{j\neq0:\,\partial_{l}\phi_{j}\partial_{l}\phi_{0}<0}
}\frac{\left|\partial_{l}\phi_{j}\right|}{D_{l+}}\left|\tau_{j}u-\cM_{k}u\right|^{r}
\end{align*}
where according to Lemmas \ref{lem:scaled-poincare} and \ref{lem:local-delta-M-extension-estimate}
for some $C$ depending only on $r$ and $\fr$:
\begin{align*}
C_{k} & =Cd_{k}^{r}\left(d_{k}^{r+1}+d_{k}^{d+1}\right)\,,\\
\hat{C}_{k,r,s} & =\left(\fint_{\fA_{3,k}\cap\bP}\tilde{M}^{\frac{2s\hat{d}}{s-r}}\right)^{\frac{s-r}{s}}\,.
\end{align*}
We integrate with respect to $\bQ$ and obtain with $P(x)=x^{d(2r-1)+r}(x^{r+1}+x^{d+1})$
\begin{align}
 & \frac{1}{\left|\bQ\right|}\int_{\bQ}\sum_{k}\left(\frac{4d_{k}}{\fr}\right)^{d(2r-1)}\chi_{\fA_{1,k}}C_{k}\hat{C}_{k,r,s}\left(\frac{1}{\left|\fA_{3,k}\right|}\int_{\fA_{3,k}\cap\bP}\left|\nabla u\right|^{s}\right)^{\frac{r}{s}}\nonumber \\
 & \qquad\leq\sum_{k}\left(\frac{4d_{k}}{\fr}\right)^{d(2r-1)}\left|\fA_{1,k}\right|C_{k}\hat{C}_{k,r,s}\left(\frac{1}{\left|\fA_{3,k}\right|}\int_{\fA_{3,k}\cap\bP}\left|\nabla u\right|^{s}\right)^{\frac{r}{s}}\label{eq:thm:Rough-meso-estimate-help-1}\\
 & \qquad\leq C\left(\frac{1}{\left|\bQ\right|}\sum_{k}P\of{d_{k}}\frac{\left|\fA_{1,k}\right|}{\left|\fA_{3,k}\right|}\int_{\fA_{3,k}\cap\bP}\left|\nabla u\right|^{s}\right)^{\frac{r}{s}}\left(\frac{1}{\left|\bQ\right|}\sum_{k}P\of{d_{k}}\frac{\left|\fA_{1,k}\right|}{\left|\fA_{3,k}\right|}\int_{\fA_{3,k}\cap\bP}\tilde{M}^{\frac{2s\hat{d}}{s-r}}\right)^{\frac{s-r}{s}}\,.\nonumber 
\end{align}
For the measurable function $g=\tilde{M}^{\frac{2s\hat{d}}{s-r}}$
on $\Rd$ we find for every $\frac{1}{\tilde{p}}+\frac{1}{\tilde{q}}=1$
\begin{align}
 & \frac{C}{\left|\bQ\right|}\sum_{k}P\of{d_{k}}\frac{\left|\fA_{1,k}\right|}{\left|\fA_{3,k}\right|}\int_{\fA_{3,k}\cap\bP}g(x)\,\d x\nonumber \\
 & \qquad\leq\frac{C}{\left|\bQ\right|}\int_{\bP\cap\bQ}g(x)\sum_{k}P\of{d_{k}}\chi_{\fA_{3,k}}(x)\,\d x\nonumber \\
 & \qquad\leq\left(\frac{C}{\left|\bQ\right|}\int_{\bP\cap\bQ}g^{\tilde{p}}\right)^{\frac{1}{\tilde{p}}}\left(\frac{C}{\left|\bQ\right|}\int_{\bP\cap\bQ}\left(\sum_{k}P\of{d_{k}}\chi_{\fA_{3,k}}\right)^{\tilde{q}}\right)^{\frac{1}{\tilde{q}}}\,.\label{eq:thm:Rough-meso-estimate-help-2}
\end{align}
For the remaining expression note that
\begin{align}
 & \frac{1}{\left|\bQ\right|}\int_{\bQ}\sum_{k}\left(\frac{4d_{k}}{\fr}\right)^{d(2r-1)}\chi_{\fA_{1,k}}C_{k}\frac{C_{0}}{\left|\fA_{3,k}\right|}\int_{\fA_{3,k}\backslash\bP}\left|\nabla\phi_{0}\right|^{r}\sum_{\substack{j\neq0:\,\partial_{l}\phi_{j}\partial_{l}\phi_{0}<0}
}\frac{\left|\partial_{l}\phi_{j}\right|}{D_{l+}}\left|\tau_{j}u-\cM_{k}u\right|^{r}\nonumber \\
 & \qquad\leq\frac{1}{\left|\bQ\right|}\int_{\bQ\backslash\bP}\sum_{k}P\of{d_{k}}\chi_{\fA_{3,k}}\left|\nabla\phi_{0}\right|^{r}\sum_{\substack{j\neq0:\,\partial_{l}\phi_{j}\partial_{l}\phi_{0}<0}
}\frac{\left|\partial_{l}\phi_{j}\right|}{D_{l+}}\left|\tau_{j}u-\cM_{k}u\right|^{r}\label{eq:thm:Rough-meso-estimate-help-3}
\end{align}
We denote the right hand side of (\ref{eq:thm:Rough-meso-estimate-help-3})
by $I_{1}$. Using $\fa$ we obtain from H\"older's inequality together
with Jensen's inequality 
\begin{equation}
I_{l}\leq C_{P}\left(\frac{1}{\left|\bQ\right|}\int_{\bQ\backslash\bP}\sum_{k}\frac{\chi_{\fA_{3,k}}}{\fa}\left|\nabla\phi_{0}\right|^{\tilde{s}}\sum_{\substack{j\neq0:\,\partial_{l}\phi_{j}\partial_{l}\phi_{0}<0}
}\frac{\left|\partial_{l}\phi_{j}\right|}{D_{l+}}\left|\tau_{j}u-\cM_{k}u\right|^{\tilde{s}}\right)^{\frac{r}{\tilde{s}}}\,,\label{eq:thm:Rough-meso-estimate-help-4}
\end{equation}
where $C_{P}$ and $\fa$ are defined in the statement and where we
used $\sum_{k}\frac{\chi_{\fA_{3,k}}}{\fa}\equiv1$ and $\sum\frac{\left|\partial_{l}\phi_{j}\right|}{D_{l+}}\equiv1$.

Taking together (\ref{eq:thm:Rough-meso-estimate-help-1})--(\ref{eq:thm:Rough-meso-estimate-help-4})
we conclude for $\tilde{p}=\frac{p}{s}$ and with boundedness $0<c<\frac{\left|\fA_{1,k}\right|}{\left|\fA_{3,k}\right|}<C<\infty$.
\end{proof}

\subsection{\label{subsec:Extension-for-Statistically}Extension for Statistically
Harmonic Domains}
\begin{defn}
\label{def:statis-mixing}A random geometry $\bP(\omega)$ is statistically
$s$-harmonic if there exist constants $C_{k}>0$, $k\in\N$ and sets
$\fA_{4,k}\supset\fA_{3,k}$ such that for every $x_{k}\in\X_{\omega}$
\[
\int_{\fA_{3,k}\cap\Rd_{3}\cap\bP}\left|u-\cM_{k}u\right|^{s}\leq\int_{\fA_{4,k}\cap\bP}C_{k}\left|\nabla u\right|^{s}\,.
\]
\end{defn}

\begin{thm}
\label{thm:final-estimate}Let $\bP(\omega)$ be a stationary ergodic
random open set which is $\left(\delta,M\right)$-regular, isotropic
cone mixing for $\fr>0$ and $f(R)$, statistically $s$-harmonic
and let Assumption \ref{assu:mesoscopic-regular} hold. Then for every
$l=1,\dots,d$ and $1\leq r<s<p$ and every $1<\alpha,\tilde{p}<\infty$
it holds 
\begin{align*}
 & \frac{1}{\left|\bQ\right|}\int_{\bQ\backslash\bP}\sum_{k}P\of{d_{k}}\chi_{\fA_{3,k}}\left|\nabla\phi_{0}\right|^{r}\sum_{\substack{j\neq0:\,\partial_{l}\phi_{j}\partial_{l}\phi_{0}<0}
}\frac{\left|\partial_{l}\phi_{j}\right|}{D_{l+}}\left|\tau_{j}u-\cM_{k}u\right|^{r}\\
 & \qquad\leq\left(\frac{C}{\left|\bQ\right|}\int_{\bP\cap\bQ\cap\Rd_{3}}\left(\sum_{k}P\of{d_{k}}\chi_{\fA_{3,k}}\right)^{\frac{\tilde{p}}{\tilde{p}-1}}\right)^{\frac{\left(s-1\right)\left(\tilde{p}-1\right)}{\tilde{p}s}}\left(\frac{C}{\left|\bQ\right|}\int_{\partial\bP\cap\bQ}\tilde{\delta}^{1-\alpha r\tilde{p}\frac{s}{s-r}}\right)^{\frac{1}{\alpha\tilde{p}}}\cdot\\
 & \qquad\quad\left(\frac{C}{\left|\bQ\right|}\int_{\bP\cap\bQ\cap\Rd_{3}}\left(\left(\tilde{M}^{d+r}\right)^{\tilde{p}\frac{s}{s-r}}\tilde{M}^{\frac{\left(d-2\right)}{\alpha}}\right)^{\frac{\alpha}{\alpha-1}}\right)^{\frac{\alpha-1}{\alpha\tilde{p}}}\cdot\left(\frac{1}{\left|\bQ\right|}\int_{\bQ}\left|\nabla u\right|^{p}\right)^{\frac{s}{p}}\cdot\\
 & \qquad\quad\left(\frac{1}{\left|\bQ\right|}\int_{\bQ}\left(\sum_{k}P\of{d_{k}}\chi_{\fA_{4,k}}C_{k}\right)^{\frac{p}{p-s}}\right)^{\frac{p-s}{ps}r}
\end{align*}
\end{thm}

\begin{proof}
We make use of $\left|\nabla\phi_{0}\right|\leq C\rho_{j}$ on $A_{1,j}$
as well as the definition of $\tau_{j}u$ to obtain that the latter
expression is bounded by (compare also with the calculation leading
to (\ref{eq:assu:M-alpha-bound-rough}))
\begin{align*}
 & \frac{1}{\left|\bQ\right|}\int_{\bQ\backslash\bP}\sum_{k}P\of{d_{k}}\chi_{\fA_{3,k}}\left|\nabla\phi_{0}\right|^{r}\sum_{\substack{j\neq0:\,\partial_{l}\phi_{j}\partial_{l}\phi_{0}<0}
}\frac{\left|\partial_{l}\phi_{j}\right|}{D_{l+}}\left|\tau_{j}u-\cM_{k}u\right|^{r}\\
 & \qquad\leq\frac{1}{\left|\bQ\right|}\int_{\bQ\backslash\bP}\sum_{k}P\of{d_{k}}\chi_{\fA_{3,k}}\sum_{\substack{j\neq0}
}\rho_{j}^{-r}\rho_{j}^{-d}M_{j}^{d}\chi_{A_{1,j}}\int_{\Ball{\fr_{j}}{y_{j}}\cap\fA_{3,k}}\left|u-\cM_{k}u\right|^{r}\\
 & \qquad\leq\frac{1}{\left|\bQ\right|}\sum_{k}P\of{d_{k}}\sum_{\substack{j\neq0}
}\rho_{j}^{-r}M_{j}^{d}\int_{\Ball{\fr_{j}}{y_{j}}\cap\fA_{3,k}}\left|u-\cM_{k}u\right|^{r}\\
 & \qquad\leq\frac{1}{\left|\bQ\right|}\int_{\bQ}\sum_{k}P\of{d_{k}}\sum_{\substack{\substack{j\neq0}
}
}\rho_{j}^{-r}M_{j}^{d}\chi_{\Ball{\fr_{j}}{y_{j}}\cap\fA_{3,k}}\left|u-\cM_{k}u\right|^{r}\\
 & \qquad\leq\left(\frac{1}{\left|\bQ\right|}\int_{\bQ}\sum_{k}P\of{d_{k}}\sum_{\substack{j\neq0}
}\left(\rho_{j}^{-r}M_{j}^{d}\right)^{\frac{s}{s-r}}\chi_{\Ball{\fr_{j}}{y_{j}}\cap\fA_{3,k}}\right)^{\frac{s-r}{s}}\\
 & \qquad\qquad\left(\frac{1}{\left|\bQ\right|}\int_{\bQ}\sum_{k}P\of{d_{k}}\sum_{\substack{j\neq0}
}\chi_{\Ball{\fr_{j}}{y_{j}}\cap\fA_{3,k}}\left|u-\cM_{k}u\right|^{s}\right)^{\frac{r}{s}}\,.
\end{align*}
We use that $\Ball{\fr_{j}}{y_{j}}$ are mutually disjoint and $\Ball{\fr_{j}}{y_{j}}\subset\Rd_{3}$
to find

\begin{align*}
 & \frac{1}{\left|\bQ\right|}\int_{\bQ}\sum_{k}P\of{d_{k}}\sum_{\substack{\substack{j\neq0}
}
}\chi_{\Ball{\fr_{j}}{y_{j}}\cap\fA_{3,k}}\left|u-\cM_{k}u\right|^{s}\\
 & \qquad\leq\frac{1}{\left|\bQ\right|}\int_{\bQ}\sum_{k}P\of{d_{k}}\chi_{\fA_{3,k}\cap\Rd_{3}}\left|u-\cM_{k}u\right|^{s}\\
 & \qquad\leq\frac{1}{\left|\bQ\right|}\int_{\bQ}\sum_{k}P\of{d_{k}}\chi_{\fA_{4,k}}C_{k}\left|\nabla u\right|^{s}
\end{align*}
where we have used the statistical $s$-connectedness. Similar to
the proof of Theorem \ref{thm:Rough-meso-estimate} we observe that
\begin{multline*}
\left(\frac{1}{\left|\bQ\right|}\int_{\bQ}\sum_{k}P\of{d_{k}}\sum_{\substack{\substack{j\neq0}
}
}\left(\rho_{j}^{-r}M_{j}^{d}\right)^{\frac{s}{s-r}}\,\chi_{\Ball{\fr_{j}}{y_{j}}\cap\fA_{3,k}}\d x\right)\\
\leq\left(\frac{C}{\left|\bQ\right|}\int_{\bP\cap\bQ\cap\Rd_{3}}\left(\tilde{\rho}^{-r}M_{j}^{d}\right)^{\tilde{p}\frac{s}{s-r}}\right)^{\frac{1}{\tilde{p}}}\left(\frac{C}{\left|\bQ\right|}\int_{\bP\cap\bQ\cap\Rd_{3}}\left(\sum_{k}P\of{d_{k}}\chi_{\fA_{3,k}}\right)^{\tilde{q}}\right)^{\frac{1}{\tilde{q}}}
\end{multline*}
and 
\begin{align*}
 & \frac{1}{\left|\bQ\right|}\int_{\bQ}\sum_{k}P\of{d_{k}}\chi_{\fA_{4,k}}C_{k}\left|\nabla u\right|^{s}\\
 & \qquad\leq\left(\frac{1}{\left|\bQ\right|}\int_{\bQ}\left(\sum_{k}P\of{d_{k}}\chi_{\fA_{4,k}}C_{k}\right)^{\frac{p}{p-s}}\right)^{\frac{p-s}{p}}\left(\frac{1}{\left|\bQ\right|}\int_{\bQ}\left|\nabla u\right|^{p}\right)^{\frac{s}{p}}\,.
\end{align*}
Finally, Lemma \ref{lem:delta-tilde-construction-estimate} yields
with $\tilde{\rho}\geq C\tilde{\delta}/\tilde{M}$
\begin{multline*}
\frac{C}{\left|\bQ\right|}\int_{\bP\cap\bQ\cap\Rd_{3}}\left(\tilde{\rho}^{-r}\tilde{M}^{d}\right)^{\tilde{p}\frac{s}{s-r}}\\
\leq\left(\frac{C}{\left|\bQ\right|}\int_{\partial\bP\cap\bQ}\tilde{\delta}^{1-\alpha r\tilde{p}\frac{s}{s-r}}\right)^{\frac{1}{\alpha}}\left(\frac{C}{\left|\bQ\right|}\int_{\bP\cap\bQ\cap\Rd_{3}}\left(\left(\tilde{M}^{d+r}\right)^{\tilde{p}\frac{s}{s-r}}\tilde{M}^{\frac{\left(d-2\right)}{\alpha}}\right)^{\frac{\alpha}{\alpha-1}}\right)^{\frac{\alpha-1}{\alpha}}
\end{multline*}
\end{proof}

\section{\label{sec:Construction-of-Macroscopic-1}Construction of Macroscopic
Extension Operators II: Admissible Paths}

In this section, we will use admissible paths on connected sets in
order to estimate the (so far uncontrolled) terms $\left|\tau_{i}u-\cM_{j}u\right|$ in
Theorems \ref{thm:Fine-meso-estimate} and \ref{thm:Rough-meso-estimate}
in terms of $\nabla u$.

Knowing there exists an admissible path (by Theorem \ref{thm:Ex-adm-path}),
it remains to deal with the non-uniqueness of the path. Note there
is no clear distinction which puts one path in favor of others. While
this could be seen as a drawback, it can also be considered as an
opportunity, since it allows to distribute the ``weight'' of integration
along the paths more uniformly among the total volume. This is the
basic idea of this section.

\subsection{\label{subsec:Connected-Preliminaries}Preliminaries}

Given an open connected set $\bP$ and a countable family of points
$\X_{\fr}$ satisfying Assumption \ref{assu:mesoscopic-voronoi} we
extend the covering $A_{1,j}$ resp. $A_{2,j}$ of $\partial\bP$
from Section \ref{subsec:5-Preliminaries} (e.g. (\ref{eq:A123-k}))
to the inner of $\bP$ using Lemma \ref{lem:cover-iso-mixing-delta-M}.
In this context, we remind the reader of (\ref{eq:def:Y-Partial-Y})
and Definition \ref{def:Y-Partial-Y} and introduce the notation 
\[
A_{1}(y)=\begin{cases}
A_{1,k} & \text{if }y=p_{k}\in\X_{\partial}\\
\Ball{\tilde{\eta}\of y}y & \text{if }y\in\mathring{\Y}
\end{cases}\,,\quad A_{2}(y)=\begin{cases}
A_{2,k} & \text{if }y=p_{k}\in\X_{\partial}\\
\Ball{3\frac{1}{2}\tilde{\eta}\of y}y & \text{if }y\in\mathring{\Y}
\end{cases}\,.
\]
We find the following
\begin{lem}
\label{lem:bounded-num-a2}There exists $C>0$ independent from $\bP$
such that for every $x\in\bP$ 
\[
\#\left\{ y\in\mathring{\Y}:\;x\in A_{2}(y)\right\} \leq C\,.
\]
\end{lem}

\begin{proof}
For two points $p_{i},p_{j}\in\partial\X$ such that $x\in A_{2,i}\cap A_{2,j}$
it holds due to the triangle inequality 
\begin{equation}
\left|p_{i}-p_{j}\right|\leq\left|x-p_{j}\right|+\left|p_{i}-x\right|\leq3\left(\tilde{\rho_{i}}+\tilde{\rho}_{j}\right)\leq\max\left\{ 6\tilde{\rho}_{i},6\tilde{\rho}_{j}\right\} \,.\label{eq:lem:bounded-num-a2-1}
\end{equation}
Let $\X_{\partial}(x):=\left\{ p_{i}\in\partial\X:\,x\in\overline{\Ball{3\tilde{\rho}_{i}}{p_{i}}}\right\} $
and choose $\tilde{p}\in\X_{\partial}(x)$ such that $\tilde{\rho}_{\text{m}}:=\tilde{\rho}\of{\tilde{p}}$
is maximal. Then $\X_{\partial}(x)\subset\Ball{6\tilde{\rho}_{\text{m}}}{\tilde{p}}$
by (\ref{eq:lem:bounded-num-a2-1}) and every $p_{i}\in\X_{\partial}(x)$
satisfies $\tilde{\rho}_{\text{m}}>\tilde{\rho}_{i}>\frac{1}{3}\tilde{\rho}_{\text{m}}$
(Lemma \ref{lem:eta-lipschitz}). In view of (\ref{eq:cor:cover-boundary-h1})
this lower local bound of $\tilde{\rho}_{i}$ implies a lower local
bound on the mutual distance of the $p_{i}$. Since this distance
is proportional to $\tilde{\rho}_{\text{m}}$, and since $\tilde{\rho}_{\text{m}}>\tilde{\rho}_{i}>\frac{1}{3}\tilde{\rho}_{\text{m}}$,
this implies for some constant $C>0$ independent of $x$ or $\bP$
that
\[
\#\left\{ y\in\partial\X:\;x\in A_{2}(y)\right\} \leq C\,.
\]
Now let $y\in\mathring{\Y}\backslash\partial\X$ and $x\in A_{2}(y)=\Ball{\frac{7}{8}\eta}y$.
We show 
\[
\eta(y)<8\eta(x)<16\eta(y)\,.
\]
For the first inequality, observe that $\eta(x)\leq\frac{1}{8}\eta(y)$
is equivalent with $\dist\of{x,\partial\bP}\leq\frac{1}{8}\dist\of{y,\partial\bP}$
and hence
\begin{align*}
 &  & \dist\of{y,\partial\bP} & \leq\dist\of{x,\partial\bP}+\left|x-y\right|\\
 &  &  & \leq\frac{1}{8}\dist\of{y,\partial\bP}+\left|x-y\right|\\
 & \Rightarrow & \left|x-y\right| & \geq\frac{7}{8}\dist\of{y,\partial\bP}\,.
\end{align*}
For the second inequality, assume $\tilde{\eta}(y)<\tilde{\eta}(x)$.
Then $y$ lies closer to the boundary than $x$ and $x\in A_{2}(y)$
implies 
\[
\eta(x)=\dist\left(x,\partial\bP\right)\leq\dist(y,\partial\bP)+\left|x-y\right|\leq\eta(y)+\frac{7}{2}\tilde{\eta}(y)\leq2\eta(y)\,.
\]
The mutual minimal distance of neighboring points in terms of $\tilde{\eta}$
now implies for some $C$ independent from $x$ and $\bP$ 
\[
\#\left\{ y\in\mathring{\Y}\backslash\partial\X:\;x\in A_{2}(y)\right\} \leq C\,.
\]
\end{proof}

\begin{defn}
Let $\G_{\ast}(\bP)$ be a connected sub-graph of $\G_{0}(\bP)$.
Let $x_{i}\in\X_{\fr}$ and $u_{i}:=u_{x_{i}}$ be the solution of
the discrete Laplace equation (\ref{eq:lem:existence-cLu-equiv-delta-x})
for $x=x_{i}$ on the graph $\G_{\ast}(\bP)$. For every $z\in\mathring{\Y}\backslash\{x_{i}\}$
let 
\[
\O_{\ast,i}(z):=\left\{ \tilde{y}\in\mathring{\Y}:\,u_{i}\of{\tilde{y}}>u_{i}(z)\right\} 
\]
 the neighbors corresponding the \emph{outgoing branches of admissible
paths through $y$}, and we assign to each $\tilde{y}\in\O_{\ast,i}(z)$
the weight $w_{\ast}\of{z,\tilde{y}}=w_{\ast,1,2}\of{z,\tilde{y}}$
of the branch $\of{z,\tilde{y}}$ where either 
\begin{align*}
w_{\ast,1}\of{z,\tilde{y}} & =\left(u_{i}\of{\tilde{y}}-u_{i}(z)\right)/\left(\sum_{y\in\O_{\ast,i}(z)}\left(u_{i}\of y-u_{i}(z)\right)\right)\,,\\
w_{\ast,2}\of{z,\tilde{y}} & =\#\O_{\ast,i}(z)^{-1}\,.
\end{align*}
For $Y=\left(y_{1},\dots y_{N}\right)\in\Apaths_{\ast}\of{p_{j},x_{i}}$
we define the weight of the path $Y$ by 
\[
W_{\ast}(Y):=W_{\ast}\of{y_{1},\dots y_{N}}:=\prod_{i=1}^{N-1}w_{\ast}\of{y_{i},y_{i+1}}\,.
\]
\end{defn}

\begin{rem}
We observe 
\[
\sum_{Y\in\Apaths_{\ast}\of{p_{j},x_{i}}}W_{\ast}(Y)=1\,.
\]
This holds by induction along the path and different branches since
in every $z\in\mathring{\Y}\backslash\{x_{i}\}$ it holds $\sum_{y\in\O_{\ast,i}(z)}w_{\ast}(z,y)=1$.
\end{rem}

\subsection{Extension for Connected Domains}

In this section, we discuss how the graphs built in Section \ref{subsec:Connectedness-of--Regular-sets}
can be used to derive estimates on $f(u)$ given in Theorem \ref{thm:Fine-meso-estimate}.
The remaining constant on the right hand side is given in terms of
the balls $\Ball{\fr_{i}}{p_{i}}$ and length of the paths between
$p_{i}$ and $x_{j}$ or $x_{j}$ and $x_{k}$ respectively. Although
one could go even more into details and try to generally decouple
these effects, this is not helpful for our examples in Section \ref{sec:Sample-Geometries}
below. Hence we leave the results of this section as they are but
encourage further investigation in the future.

\subsubsection*{The idea}

We first consider the case of a general graph $\left(\Y,\G(\bP)\right)$
on $\bP$ and do not claim that paths in the classes $\Apaths$ are
fully embedded into $\bP$. In particular, we drop for a moment the
concept of local connectivity and we allow paths to intersect with
$\Rd\backslash\bP$. Let $x_{j}\in\X_{\fr}$, $p_{i}\in\Y_{\X\partial}$
and $Y=\left(y_{1},\dots,y_{N}\right)\in\Apaths\of{p_{i},x_{j}}$.
In the following short calculation, one may think of $\widetilde{\nabla u}$
as a function related to $\nabla\left(\cU u\right)$, though the following
calculations will reveal that it is not exactly what we mean. Nevertheless,
recalling Notation \ref{nota:Apaths} for $Y(x)$ and $Y=\bigcup_{x}Y(x)$
it holds 
\begin{align}
\left|\tau_{i}u-\cM_{j}u\right|^{s} & =\left|\frac{1}{\left|\Ball{\fr_{i}}0\right|}\int_{\Ball{\fr_{i}}0}u(x+p_{i})-\frac{1}{\left|\Ball{\frac{\fr}{16}}{x_{j}}\right|}\int_{\Ball{\frac{\fr}{16}}0}u(x+x_{j})\right|^{s}\nonumber \\
 & =\left|\frac{1}{\left|\Ball{\frac{\fr}{16}}0\right|}\int_{\Ball{\frac{\fr}{16}}0}\left(u\of{\frac{16}{\fr}\fr_{i}(x+p_{i})}-u(x+x_{j})\right)\right|^{s}\nonumber \\
 & \leq\sum_{Y\in\Apaths\of{p_{i},x_{j}}}W(Y)\left|\Ball{\frac{\fr}{16}}0\right|^{-1}\int_{\Ball{\fr_{i}}0}\left|\int_{Y(x)}\left|\widetilde{\nabla u}\right|\right|^{s}\d x\nonumber \\
 & \leq C\sum_{Y\in\Apaths\of{p_{i},x_{j}}}W(Y)\left|\Ball{\frac{\fr}{16}}0\right|^{-1}\int_{Y}\left|\widetilde{\nabla u}\right|^{s}\Length\of Y^{\frac{s-1}{s}}\,.\label{eq:estimate-rough-general}
\end{align}
Since $\widetilde{\nabla u}$ is related to $\nabla\cU u$, the latter
formula reveals that the terms $\left|\tau_{i}u-\cM_{j}u\right|^{s}$
may lead to an ``entanglement'' of $\tilde{M}_{\hat{\rho}}$ and
the properties of the paths $\Apaths$. In what follows, we will resolve
the latter calculation in more details to prepare this discussion.

In what follows, we will make use of $Y=\left(y_{1}=p_{i},\dots y_{N}=x_{j}\right)$
and 
\[
u\left(\frac{16}{\fr}\fr_{i}(x+p_{i})\right)-u(x+x_{j})=\sum_{k=1}^{N-1}u\of{\frac{16}{\fr}\fr(y_{k})x+y_{k}}-u\of{\frac{16}{\fr}\fr(y_{k+1})x+y_{k+1}}\,,
\]
and we write $Y(y_{k},y_{k+1},x)$ for the straight line segment connecting
$\frac{16}{\fr}\fr(y_{k})x+y_{k}$ with $\frac{16}{\fr}\fr(y_{k+1})x+y_{k+1}$.
We distinguish 4 cases:

\emph{Case $y_{k},y_{k+1}\in\Y_{\partial\X}$}: According to Lemma
\ref{lem:properties-local-rho-convering} it holds $\Ball{\fr(y_{k+1})}{y_{k+1}}\subset A_{2}(y_{k})$
and if $\cU_{k}:\,W^{1,p}(A_{3,k})\to W^{1,r}(A_{2,k})$ is the corresponding
local extension operator it holds 
\[
u\of{\frac{16}{\fr}\fr(y_{k})x+y_{k}}-u\of{\frac{16}{\fr}\fr(y_{k+1})x+y_{k+1}}\leq\int_{Y(y_{k},y_{k+1},x)}\nabla\cU_{k}u\,.
\]
\emph{Case $y_{k}\in\Y_{\partial\X}$, $y_{k+1}\in\mathring{\Y}$}:
According to Lemma \ref{lem:cover-iso-mixing-delta-M} it holds $\Ball{\fr(y_{k+1})}{y_{k+1}}\subset A_{2}(y_{k})$
and if $\cU_{k}:\,W^{1,p}(A_{3,k})\to W^{1,r}(A_{2,k})$ is the corresponding
local extension operator it holds 
\[
u\of{\frac{16}{\fr}\fr(y_{k})x+y_{k}}-u\of{\frac{16}{\fr}\fr(y_{k+1})x+y_{k+1}}\leq\int_{Y(y_{k},y_{k+1},x)}\nabla\cU_{k}u\,.
\]
\emph{Case $y_{k+1}\in\Y_{\partial\X}$, $y_{k}\in\mathring{\Y}$}:
According to Lemma \ref{lem:cover-iso-mixing-delta-M} it holds $\Ball{\fr(y_{k})}{y_{k}}\subset A_{2}(y_{k+1})$
and if $\cU_{k+1}:\,W^{1,p}(A_{3,k+1})\to W^{1,r}(A_{2,k+1})$ is
the corresponding local extension operator it holds 
\[
u\of{\frac{16}{\fr}\fr(y_{k})x+y_{k}}-u\of{\frac{16}{\fr}\fr(y_{k+1})x+y_{k+1}}\leq\int_{Y(y_{k},y_{k+1},x)}\nabla\cU_{k+1}u\,.
\]
\emph{Case $y_{k},y_{k+1}\in\mathring{\Y}$}: According to Lemma \ref{lem:cover-iso-mixing-delta-M}
it holds $\Ball{\fr(y_{k})}{y_{k}}\subset A_{2}(y_{k+1})\subset\bP$
and 
\[
u\of{\frac{16}{\fr}\fr(y_{k})x+y_{k}}-u\of{\frac{16}{\fr}\fr(y_{k+1})x+y_{k+1}}\leq\int_{Y(y_{k},y_{k+1},x)}\nabla u\,.
\]

However, in case of local connectivity, we face a simpler situation.
In case $y_{k},y_{k+1}\in\mathring{\Y}$ we can use the above estimates
while in the other cases, we can use the Lemma \ref{lem:gamma-connect}.

\subsubsection*{Locally connected $\protect\bP$}

In what follows, we consider $\G_{\ast}(\bP)=\G_{\fl}(\bP)$ (see
Definition \ref{def:simple-border}) with a suitable family of admissible
paths $\Apaths_{\fl}$, and we also recall $Y(x)$ from Notation \ref{nota:Apaths}.
We repeat the calculations of (\ref{eq:estimate-rough-general}) in
view of Lemma \ref{lem:gamma-connect}. In particular, if $\tilde{y}\sim y$
are connected via a path $\gamma$ in $\Ball{3\tilde{\rho}(x(y))}{x(y)}$,
which additionally has the property that the corresponding tube exists,
then the length of $\gamma$ is bounded by $C\left|y-\tilde{y}\right|$,
where $C$ is determined by the dimension. Hence we have
\begin{align}
\left|\tau_{i}u-\cM_{j}u\right|^{s} & =\left|\frac{1}{\left|\Ball{\fr_{i}}0\right|}\int_{\Ball{\fr_{i}}0}u(x+p_{i})-\frac{1}{\left|\Ball{\frac{\fr}{16}}{x_{j}}\right|}\int_{\Ball{\frac{\fr}{16}}0}u(x+x_{j})\right|^{s}\nonumber \\
 & =\left|\frac{1}{\left|\Ball{\frac{\fr}{16}}0\right|}\int_{\Ball{\frac{\fr}{16}}0}\left(u\of{\frac{16}{\fr}\fr_{i}(x+p_{i})}-u(x+x_{j})\right)\right|^{s}\nonumber \\
 & \leq\sum_{Y\in\Apaths_{\fl}\of{p_{i},x_{j}}}W(Y)\frac{1}{\left|\Ball{\frac{\fr}{16}}0\right|}\int_{\Ball{\fr_{i}}0}\left|\int_{Y(x)}\left|\nabla u\right|\right|^{s}\d x\nonumber \\
 & \leq C\sum_{Y\in\Apaths_{\fl}\of{p_{i},x_{j}}}W(Y)\frac{1}{\left|\Ball{\frac{\fr}{16}}0\right|}\int_{Y}\left|\nabla u\right|^{s}\Length\of Y^{\frac{s-1}{s}}\,.\label{eq:estimate-flate-geo-path}
\end{align}
The last calculation is at the heart of the results in this section.
In what follows, we adopt the situation of Lemma \ref{lem:6-4}:
\begin{lem}
\label{lem:locally-flat-1}Let $\bP$ be locally connected. Under
Assumptions \ref{assu:M-alpha-bound}, \ref{assu:mesoscopic-voronoi}
and using the notation of Theorem \ref{thm:Fine-meso-estimate} let
$\left(f_{j}\right)_{j\in\N}$ be non-negative and have support $\support f_{j}\supset\Ball{\frac{\fr}{2}}{x_{j}}$
and let $\sum_{j\in\N}f_{j}\equiv1$. Let $\G_{\ast}(\bP)=\G_{\fl}(\bP)$
(see Definition \ref{def:simple-border}) with a suitable family of
admissible paths $\Apaths_{\fl}$. Writing $\X(\bQ):=\left\{ x_{j}:\;\support f_{j}\cap\bQ\neq\emptyset\right\} $
\[
Y_{\mathrm{all\,paths}}^{\mathrm{local}}(\bQ):=\bigcup_{x_{j}\in\X(\bQ)}\bigcup_{p_{i}\in\support f_{j}\cap\Y_{\partial\X}}\bigcup_{Y\in\Apaths_{\fl}\of{p_{i},x_{j}}}\chi_{Y}
\]
$\chi_{f_{j}}(x):=\left(x\in\support f_{j}\right)$ and for every
$l=1,\dots d$ it holds 
\begin{align*}
 & \frac{1}{\left|\bQ\right|}\int_{\bP}\sum_{\substack{i\neq0:\,\partial_{l}\phi_{i}\partial_{l}\phi_{0}<0}
}\sum_{x_{j}\in\X(\bQ)}f_{j}\frac{\left|\partial_{l}\phi_{i}\right|}{D_{l+}}\left|\tau_{i}u-\cM_{j}u\right|^{s}\\
 & \qquad\leq C\left(\frac{1}{\left|\bQ\right|}\int_{Y_{\mathrm{all\,paths}}^{\mathrm{local}}(\bQ)}\left|\nabla u\right|^{p}\right)^{\frac{s}{p}}\\
 & \qquad\qquad\left(\frac{1}{\left|\bQ\right|}\int_{\Rd}\left(\sum_{x_{j}\in\X(\bQ)}\sum_{i}\chi_{f_{j}}(p_{i})\tilde{\rho}_{i}^{d}\sum_{Y\in\Apaths_{\fl}\of{p_{i},x_{j}}}\chi_{Y}W(Y)\,\Length\of Y^{\frac{s-1}{s}}\right)^{\frac{p}{p-s}}\right)^{\frac{p-s}{p}}
\end{align*}
\end{lem}

\begin{proof}
We find 
\begin{align*}
 & \frac{1}{\left|\bQ\right|}\int_{\bP\cap\bQ_{\fr}}\sum_{\substack{i\neq0:\,\partial_{l}\phi_{i}\partial_{l}\phi_{0}<0}
}\sum_{x_{j}\in\X(\bQ)}f_{j}\frac{\left|\partial_{l}\phi_{i}\right|}{D_{l+}}\left|\tau_{i}u-\cM_{j}u\right|^{s}\\
 & \qquad\leq\frac{C}{\left|\bQ\right|}\sum_{\substack{i\neq0:\,\partial_{l}\phi_{i}\partial_{l}\phi_{0}<0}
}\sum_{x_{j}\in\X(\bQ)}\chi_{f_{j}}(p_{i})\tilde{\rho}_{i}^{d}\sum_{Y\in\Apaths_{\fl}\of{p_{i},x_{j}}}W(Y)\int_{Y}\left|\nabla u\right|^{s}\Length\of Y^{\frac{s-1}{s}}
\end{align*}
which leads to the result.
\end{proof}
And finally, we provide an estimate for the remaining term in Lemma
\ref{lem:6-5}. The proof is similar to the last Lemma.
\begin{lem}
\label{lem:locally-flat-2}Let $\bP$ be locally connected. Under
Assumptions \ref{assu:M-alpha-bound}, \ref{assu:mesoscopic-voronoi}
it holds for 
\[
Y_{\mathrm{all\,paths}}^{\mathrm{global}}(\bQ):=\bigcup_{\substack{x_{k}\sim\sim x_{j}\\
x_{k},x_{j}\in\X_{\fr}(\bQ)
}
}\bigcup_{Y\in\Apaths_{\fl}\of{x_{k},x_{j}}}\chi_{Y}
\]
that
\begin{align*}
 & \frac{1}{\left|\bQ\right|}\sum_{\substack{x_{k}\sim\sim x_{j}\\
x_{k},x_{j}\in\X_{\fr}(\bQ)
}
}\left|\cM_{k}u-\cM_{j}u\right|^{s}\leq C\left(\frac{1}{\left|\bQ\right|}\int_{Y_{\mathrm{all\,paths}}^{\mathrm{global}}(\bQ)}\left|\nabla u\right|^{p}\right)^{\frac{s}{p}}\\
 & \qquad\qquad\left(\frac{1}{\left|\bQ\right|}\int_{\Rd}\left(\sum_{\substack{x_{k}\sim\sim x_{j}\\
x_{k},x_{j}\in\X_{\fr}(\bQ)
}
}\sum_{Y\in\Apaths_{\fl}\of{x_{k},x_{j}}}W(Y)\Length\of Y^{\frac{s-1}{s}}\right)^{\frac{p}{p-s}}\right)^{\frac{p-s}{p}}\,.
\end{align*}
\end{lem}

\subsection{\label{subsec:Statistical-Strech-Factor}Statistical Stretch Factor
for Locally Connected Geometries}
\begin{defn}
Let $\bP\subset\Rd$ be an open set with $\X_{\fr}$ satisfying Assumption
\ref{assu:mesoscopic-voronoi}. Generalizing the notation of Lemma
\ref{lem:radius-admissible-pahts} and recalling the Notation \ref{nota:Apaths}
let for $x\in\X_{\fr}$ and $y\in\Y$ and a family of admissible paths
$\Apaths(y,x)$
\[
\sfR_{0}(x,y):=\inf\left\{ R>0:\;\bigcup_{Y\in\Apaths(y,x)}Y\subset\Ball Rx\right\} \,.
\]
For an open set $\fA$ with $x\in\fA$ we denote 
\[
\sfR_{0}(x,\fA):=\sup_{y\in\Y\cap\fA}\sfR_{0}(x,y)\,.
\]
\end{defn}

\begin{thm}
\label{thm:strech-main-thm}Let the Assumptions of Theorem \ref{thm:Fine-meso-estimate}
hold and let $\bP$ be locally connected. . For every $x_{j}\in\X_{\fr}$
let 
\[
\sfS_{j}:=\sfS(x_{j}):=d_{j}^{-1}\sup_{p_{i}\in\Y\cap\fA_{2,j}}\sup_{Y\in\Apaths_{\fl}\of{p_{i},x_{j}}}\Length\of Y\,.
\]
Defining $\sfR_{0}(x_{j}):=\sfR_{0}(x_{j},\fA_{2,j})$ and 
\begin{equation}
\Apaths(\bQ):=\bigcup_{x_{j}\in\bQ^{\sim\sim}}\Ball{\sfR_{0}(x_{j})}{x_{j}}\label{eq:thm:strech-main-thm-1}
\end{equation}
it holds 
\[
\frac{1}{\left|\bQ\right|}\int_{\bQ^{\sim\sim}}\left|f\of u\right|^{r}\leq C_{1}(C_{2}+C_{3})\left(\int_{\Apaths(\bQ)}\left|\nabla u\right|^{p}\right)^{\frac{r}{p}}\,,
\]
where for some $s\in(r,p)$
\begin{align*}
C_{1} & =\left(\frac{1}{\left|\bQ\right|}\int_{\Apaths(\bQ)}\left(\sum_{x_{j}\in\X_{\fr}(\bQ)}\chi_{\Ball{\sfR_{0}(x_{j},\fA_{2,j})}{x_{j}}}d_{j}^{d+\frac{s-1}{s}}\sfS_{j}^{\frac{s-1}{s}}\right)^{\frac{p}{p-s}}\right)^{\frac{p-s}{p}}\\
 & \leq\left(\frac{1}{\left|\bQ\right|}\int_{\Apaths(\bQ)}\left(\sum_{x_{j}\in\X_{\fr}(\bQ)}\chi_{\Ball{\sfS_{j}d_{j}}{x_{j}}}d_{j}^{d+\frac{s-1}{s}}\sfS_{j}^{\frac{s-1}{s}}\right)^{\frac{p}{p-s}}\right)^{\frac{p-s}{p}}\\
C_{2} & =\left(\frac{1}{\left|\bQ\right|}\int_{\bP\cap\bQ_{\fr}\cap\Rd_{3}}\left|\partial_{l}\phi_{0}\right|^{\frac{sr}{s-r}}\right)^{\frac{s-r}{s}}\,,\\
C_{3} & =\left(\frac{1}{\left|\bQ\right|}\int_{\bP\cap\bQ}\left(\sum_{j:\,\partial_{l}\Phi_{j}<0}d_{j}^{\frac{r(d-1)+drs}{s-r}}\chi_{\nabla\Phi_{j}\neq0}\right)^{\frac{s}{s-r}}\right)^{\frac{s-r}{s}}\,.
\end{align*}
\end{thm}

\begin{defn}
\label{def:thm:strech-main-thm}We call $\sfS_{j}$ the statistical
stretch factor.
\end{defn}

\begin{cor}
\label{cor:R0-Sj}It holds $\sfR_{0}(x_{j})\leq d_{j}\sfS_{j}$.
\end{cor}

\begin{cor}
\label{cor:support-extension}If $u\in W^{1,p}(\bP)$ satisfies $u\equiv0$
on $\Rd\backslash\bQ$ then $\cU$ has support on $\Apaths(\bQ)$.
\end{cor}

\begin{proof}
This follows since 
\[
\Apaths(\bQ)\supset\bigcup_{j\sim\sim\bQ}\fA_{1,j}\,.
\]
\end{proof}

\begin{proof}[Proof of Theorem \ref{thm:strech-main-thm}]
 With regard to Lemma \ref{lem:locally-flat-1}, we observe that
$f_{j}=\Phi_{j}$ with $\X(\bQ)=\X_{\fr}(\bQ)$ and $\chi_{f_{j}}(p_{i})=1$
only if $p_{i}\in\fA_{1,j}$. Furthermore, $W(Y)\leq1$ and we define
\[
L_{j}:=\sup_{p_{i}\in\Y\cap\fA_{2,j}}\sup_{Y\in\Apaths_{\fl}\of{p_{i},x_{j}}}\Length\of{Y_{\fl}}\,.
\]
Hence we find for given $x_{j}$ using Corollary \ref{cor:R0-Sj}:
\begin{align*}
\sum_{i}\chi_{f_{j}}(p_{i})\tilde{\rho}_{i}^{d}\sum_{Y\in\Apaths_{\fl}\of{p_{i},x_{j}}}\chi_{Y_{\fl}}W(Y)\,\Length\of{Y_{\fl}}^{\frac{s-1}{s}} & \leq\chi_{\Ball{\sfR_{0}(x_{j},\fA_{1,j})}{x_{j}}}\left|\fA_{1,j}\right|L_{j}^{\frac{s-1}{s}}\\
 & \leq\chi_{\Ball{\sfS_{j}d_{j}}{x_{j}}}\left|\fA_{1,j}\right|L_{j}^{\frac{s-1}{s}}\,.
\end{align*}
Also with regard to Lemma \ref{lem:locally-flat-1} we find for given
$x_{j}$
\begin{align*}
\sum_{\substack{x_{k}\sim\sim x_{j}\\
x_{k}\in\X_{\fr}(\bQ)
}
}\sum_{Y\in\Apaths_{\fl}\of{x_{k},x_{j}}}W(Y)\Length\of{Y_{\fl}}^{\frac{s-1}{s}} & \leq\chi_{\Ball{\sfR_{0}(x_{j},\fA_{2,j})}{x_{j}}}\left|\fA_{2,j}\right|L_{j}^{\frac{s-1}{s}}\\
 & \leq\chi_{\Ball{\sfS_{j}d_{j}}{x_{j}}}\left|\fA_{2,j}\right|L_{j}^{\frac{s-1}{s}}\,.
\end{align*}
The statement now follows from the definition of $\sfS_{j}$, Lemmas
\ref{lem:6-4} and \ref{lem:6-5}.
\end{proof}
Finally, the following result allows us to estimate the difference
of $\bQ$ and $\Apaths(\bQ)$.
\begin{thm}
\label{thm:ergodic-Apaths}Let the Assumptions of Theorem \ref{thm:Fine-meso-estimate}
hold, let $\bQ$ have a $C^{1}$-boundary and let $\Apaths(\bQ)$
be given by (\ref{eq:thm:strech-main-thm-1}). Furthermore, let $\sfR_{0}$
be ergodic such that for every $\eps>0$ 
\begin{equation}
\lim_{n\to\infty}\sum_{k=1}^{\infty}\left(1+\eps\right)^{k}\E\of{\sfR_{0}(x_{j})\geq(1+\eps)^{k}n}=0\,.\label{eq:thm:ergodic-Apaths}
\end{equation}
Then
\[
\lim_{n\to\infty}\frac{\left|n\bQ\right|}{\left|\Apaths\left(n\bQ\right)\right|}\to1\,.
\]
\end{thm}

\begin{rem}
Condition (\ref{eq:thm:ergodic-Apaths}) is satisfied if e.g. $\E\of{\sfR_{0}(x_{j})\geq r_{0}}\leq r_{0}^{-a}$
for some $a>1$ as then 
\[
\sum_{k=1}^{\infty}\left(1+\eps\right)^{k}\E\of{\sfR_{0}(x_{j})\geq(1+\eps)^{k}n}\leq\frac{1}{n^{\alpha}}\sum_{k=1}^{\infty}\left(\frac{1}{\left(1+\eps\right)^{a-1}}\right)^{k}\,.
\]
\end{rem}

\begin{proof}
Since $n\bQ\subset\Apaths(n\bQ)$ we have to estimate the excess mass
of $\Apaths(n\bQ)$ over $\left|n\bQ\right|$. If we define 
\begin{align*}
\X_{n\bQ} & :=\left\{ x_{j}\in\X_{\fr}\cap\bQ:\,\Ball{\sfR_{0}(x_{j})}{x_{j}}\backslash\left(n\bQ\right)\neq\emptyset\right\} \,,\\
\X_{n\bQ^{\complement}} & :=\left\{ x_{j}\in\X_{\fr}\backslash\bQ:\,\Ball{\fr}{\fA_{1,j}}\cap\left(n\bQ\right)\neq\emptyset\right\} \,,
\end{align*}
we find 
\[
\left|\Apaths(n\bQ)\backslash\left(n\bQ\right)\right|\leq\sum_{x_{j}\in\X_{n\bQ}\cup\X_{n\bQ^{\complement}}}\left|\Ball{\sfR_{0}(x_{j})}{x_{j}}\right|\,,
\]
and we thus derive an estimate on the contribution from $\X_{n\bQ}$
and $\X_{n\bQ^{\complement}}$ respectively.

Let $\eps>0$. Then for $\bQ_{n,k}^{\eps}:=\left((1+\eps)^{k}n\bQ\right)\backslash\left((1+\eps)^{k-1}n\bQ\right)$
\begin{align*}
\sum_{x_{j}\in\X_{n\bQ^{\complement}}}\left|\Ball{\sfR_{0}(x_{j},\fA_{2,j})}{x_{j}}\right| & \leq\sum_{x_{j}\in\X_{\fr}\cap\bQ_{n,1}^{\eps}}\left|\Ball{\sfR_{0}(x_{j})}{x_{j}}\right|+\sum_{k=2}^{\infty}\sum_{\substack{x_{j}\in\X_{\fr}\cap\bQ_{n,k}^{\eps}\\
d(x_{j})\geq(1+\eps)^{k-1}n
}
}\left|\Ball{\sfR_{0}(x_{j})}{x_{j}}\right|\\
 & \leq\sum_{x_{j}\in\X_{\fr}\cap\bQ_{n,1}^{\eps}}\left|\Ball{\sfR_{0}(x_{j})}{x_{j}}\right|+\sum_{k=2}^{\infty}\sum_{\substack{x_{j}\in\X_{\fr}\cap\bQ_{n,k}^{\eps}\\
\sfR_{0}(x_{j})\geq(1+\eps)^{k-1}n
}
}\left|\Ball{\sfR_{0}(x_{j})}{x_{j}}\right|\\
 & \leq\sum_{x_{j}\in\X_{\fr}\cap\bQ_{n,1}^{\eps}}\left|\Ball{\sfR_{0}(x_{j})}{x_{j}}\right|+\sum_{k=2}^{\infty}\sum_{\substack{x_{j}\in\X_{\fr}\cap\bQ_{n,k}^{\eps}\\
\sfR_{0}(x_{j})\geq(1+\eps)n
}
}\left|\Ball{\sfR_{0}(x_{j})}{x_{j}}\right|
\end{align*}
Due to the ergodic theorem, we obtain for every $n_{0}\in\N$
\begin{align*}
\frac{1}{\left|n\bQ\right|}\sum_{\substack{x_{j}\in\X_{\fr}\cap\bQ_{n,k}^{\eps}\\
\sfR_{0}(x_{j})\geq(1+\eps)^{k-1}n
}
}\left|\Ball{\sfR_{0}(x_{j})}{x_{j}}\right| & \leq\frac{1}{\left|n\bQ\right|}\sum_{\substack{x_{j}\in\X_{\fr}\cap\bQ_{n,k}^{\eps}\\
\sfR_{0}(x_{j})\geq(1+\eps)^{k-1}n_{0}
}
}\left|\Ball{\sfR_{0}(x_{j})}{x_{j}}\right|\\
 & \to\left(\left(1+\eps\right)^{k}-\left(1+\eps\right)^{k-1}\right)\E\of{\sfR_{0}(x_{j})\geq(1+\eps)^{k-1}n_{0}}\\
 & \leq\eps\left(1+\eps\right)^{k-1}\E\of{\sfR_{0}(x_{j})\geq(1+\eps)^{k-1}n_{0}}
\end{align*}
and similarly 
\[
\lim_{n\to\infty}\frac{1}{\left|n\bQ\right|}\sum_{x_{j}\in\X_{\fr}\cap\bQ_{n,1}^{\eps}}\left|\Ball{\sfR_{0}(x_{j})}{x_{j}}\right|=\eps\E(\sfR_{0})\,.
\]
Since the above estimates hold for every $\eps$ and every $n_{0}$,
we find
\[
\frac{1}{\left|n\bQ\right|}\sum_{x_{j}\in\X_{n\bQ^{\complement}}}\left|\Ball{\sfR_{0}(x_{j},\fA_{2,j})}{x_{j}}\right|\to0\,.
\]
In a similar way, we prove
\[
\frac{1}{\left|n\bQ\right|}\sum_{x_{j}\in\X_{n\bQ}}\left|\Ball{\sfR_{0}(x_{j},\fA_{2,j})}{x_{j}}\right|\to0\,.
\]
\end{proof}

\section{\label{sec:Sample-Geometries}Sample Geometries}

\subsection{Boolean Model for the Poisson Ball Process}

Recalling Example \ref{exa:poisson-point-proc} we consider a Poisson
point process $\X_{\pois}(\omega)=\left(x_{i}(\omega)\right)_{i\in\N}$
with intensity $\lambda$ (recall Example \ref{exa:poisson-point-proc}).
To each point $x_{i}$ a random ball $B_{i}=\Ball 1{x_{i}}$ is assigned
and the family $\B:=\left(B_{i}\right)_{i\in\N}$ is called the Poisson
ball process. We then denote $\bP\left(\omega\right):=\Rd\backslash\overline{\bigcup_{i}B_{i}}$
and seek for a corresponding uniform extension operator. The following
argumentation will be strongly based on the so called void probability.
This is the probability $\P_{0}(A)$ to not find any point of the
point process in a given open set $A$ and is given by (\ref{eq:PoisonPointPoc-Prob})
i.e. $\P_{0}(A):=e^{-\lambda\left|A\right|}$.

The void probability for the ball process is given accordingly by
\[
\P_{0}(A):=e^{-\lambda\left|\overline{\Ball 1A}\right|}\,,\qquad\overline{\Ball 1A}:=\left\{ x\in\Rd\,:\;\dist\of{x,A}\leq1\right\} \,,
\]
which is the probability that no ball intersects with $A\subset\Rd$.
\begin{thm}
Let $\bP\left(\omega\right):=\bigcup_{i}B_{i}(\omega)$ and define
\begin{align*}
\tilde{\delta}\of x & :=\min\left\{ \delta\of{\tilde{x}}\,:\;\tilde{x}\in\partial\bP\,\text{s.t. }x\in\Ball{\frac{1}{8}\delta\of{\tilde{x}}}{\tilde{x}}\right\} \,,\\
\tilde{\hat{\rho}}\of x & :=\min\left\{ \hat{\rho}\of{\tilde{x}}\,:\;\tilde{x}\in\partial\bP\,\text{s.t. }x\in\Ball{\frac{1}{8}\hat{\rho}\of{\tilde{x}}}{\tilde{x}}\right\} \,,
\end{align*}
where $\min\emptyset:=0$ for convenience. Then $\partial\bP$ is
almost surely locally $\left(\delta,M\right)$ regular and for every
$\gamma<1$, $\beta<d+2$ and $1\leq r<2$ and $2\frac{sr}{2(s-1)-sr}\leq d+2$
it holds
\[
\E\of{\delta^{-\gamma}}+\E\of{\tilde{\delta}^{-\gamma-1}}+\E\of{\tilde{M}^{\beta}}+\E\of{\tilde{\hat{\rho}}^{-\frac{rs}{s-1}}}<\infty\,.
\]
Furthermore, it holds $\hat{d}\le d-1$ and $\alpha=0$ in inequalities
(\ref{eq:lem:properties-local-rho-convering-4}) and (\ref{eq:assu:M-alpha-bound-improved}).
The same holds if $\bP\left(\omega\right):=\Rd\backslash\overline{\bigcup_{i}B_{i}}(\omega)$
with $\alpha$ replaced by $d$.
\end{thm}

\begin{rem}
We observe that the union of balls has better properties than the
complement.
\end{rem}

\begin{proof}
We study only $\bP\left(\omega\right):=\bigcup_{i}B_{i}(\omega)$
since $\Rd\backslash\overline{\bigcup_{i}B_{i}}(\omega)$ is the complement
sharing the same boundary. Hence, in case $\bP(\omega)=\Rd\backslash\overline{\bigcup_{i}B_{i}}(\omega)$,
all calculations remain basically the same. However, in the first
case, we assume that $\fr(y_{k})=\frac{1}{4}\tilde{\rho}(y_{k})$
, which we cannot assume in the other case, where $\fr(y_{k})$ is
proportional to $\tilde{\rho}_{k}\tilde{M}_{k}^{-1}$. This is the
reason for the different $\alpha$ in the two cases.

In what follows, we use that the distribution of balls is mutually
independent. That means, given a ball around $x_{i}\in\X_{\pois}$,
the set $\X_{\pois}\backslash\left\{ x_{i}\right\} $ is also a Poisson
process. W.l.o.g. , we assume $x_{i}=x_{0}=0$ with $B_{0}:=\Ball 10$.
First we note that $p\in\partial B_{0}\cap\partial\bP$ if and only
if $p\in\partial B_{0}\backslash\bP$, which holds with probability
$\P_{0}\of{\Ball 1p}=\P_{0}\of{B_{0}}$. This is a fixed quantity,
independent from $p$.

Now assuming $p\in\partial B_{0}\backslash\bP$, the distance to the
closest ball besides $B_{0}$ is denoted 
\[
r(p)=\dist(p,\partial\bP\backslash\partial B_{0})
\]
with a probability distribution 
\[
\P_{\dist}(r):=\P_{0}\of{\Ball{1+r}p}/\P_{0}\of{\Ball 1p}\,.
\]
It is important to observe that $\partial B_{0}$ is $r$-regular
in the sense of Lemma \ref{lem:eta-lipschitz}. Another important
feature in view of Lemma \ref{lem:properties-delta-M-regular} is
$r(p)<\Delta(p)$. In particular, $\delta(p)>\frac{1}{2}r(p)$ and
$\partial B_{0}$ is $(\delta,1)$-regular in case $\delta<\sqrt{\frac{1}{2}}$.
Hence, in what follows, we will derive estimates on $r^{-\gamma}$,
which immediately imply estimates on $\delta^{-\gamma}$.

\textbf{Estimate on $\gamma$:} A lower estimate for the distribution
of $r(p)$ is given by 
\begin{equation}
\P_{\dist}(r):=\P_{0}\of{\Ball{1+r}p}/\P_{0}\of{\Ball 1p}\approx1-\lambda\left|\S^{d-1}\right|r\,.\label{eq:PP-lower-estim-delta}
\end{equation}
This implies that almost surely for $\gamma<1$
\[
\limsup_{n\to\infty}\frac{1}{(2n)^{d}}\int_{(-n,n)^{d}\cap\partial\bP}r(p)^{-\gamma}\,\d\cH^{d-1}(p)<\infty\,,
\]
i.e. $\E\of{\delta^{-\gamma}}<\infty$.

\textbf{Intersecting balls:} Now assume there exists $x_{i}$, $i\neq0$
such that $p\in\partial B_{i}\cap\partial B_{0}$. W.l.o.g. assume
$x_{i}=x_{1}:=(2x,0,\dots,0)$ and $p=\left(\sqrt{1-x^{2}},0,\dots,0\right)$.
Then 
\[
\delta(p)\leq\delta_{0}(p):=2\sqrt{1-x^{2}}
\]
 and $p$ is at least $M(p)=\frac{x}{\sqrt{1-x^{2}}}$-regular. Again,
a lower estimate for the probability of $r$ is given by (\ref{eq:PP-lower-estim-delta})
on the interval $(0,\delta_{0})$. Above this value, the probability
is approximately given by $\lambda\left|\S^{d-1}\right|\delta_{0}$
(for small $\delta_{0}$i.e.~$x\approx1$). We introduce as a new
variable $\xi=1-x$ and obtain from $1-x^{2}=\xi(1+x)$ that 
\begin{equation}
\delta_{0}\leq C\xi^{\frac{1}{2}}\quad\text{and}\quad M(p)\leq C\xi^{-\frac{1}{2}}\,.\label{eq:delta-0-xi}
\end{equation}

\textbf{No touching:} At this point, we observe that $M$ is almost
surely locally finite. Otherwise, we would have $x=1$ and for every
$\eps>0$ we had $x_{1}\in\Ball{2+\eps}{x_{0}}\backslash\Ball{2-\eps}{x_{0}}$.
But 
\[
\P_{0}\of{\Ball{2+\eps}{x_{0}}\backslash\Ball{2-\eps}{x_{0}}}\approx1-\lambda2\left|\S^{d-1}\right|\eps\;\to\;1\qquad\text{as }\eps\to0\,.
\]
Therefore, the probability that two balls ``touch'' (i.e. that $x=1$)
is zero. The almost sure local boundedness of $M$ now follows from
the countable number of balls.

\textbf{Extension to $\tilde{\delta}$:} We again study each ball
separately. Let $p\in\partial B_{0}\backslash\overline{\bP}$ with
tangent space $T_{p}$ and normal space $N_{p}$. Let $x\in N_{p}$
and $\tilde{p}\in\partial B_{0}$ such that $x\in\Ball{\frac{1}{8}\delta(\tilde{p})}{\tilde{p}}$,
then also $p\in\Ball{\frac{1}{8}\delta(\tilde{p})}{\tilde{p}}$ and
$\delta(p)\in(\frac{7}{8},\frac{7}{6})\delta(\tilde{p})$ and $\delta(\tilde{p})\in(\frac{7}{8},\frac{7}{6})\delta(p)$
by Lemma \ref{lem:eta-lipschitz}. Defining 
\[
\tilde{\delta}_{i}\of x:=\min\left\{ \delta\of{\tilde{x}}\,:\;\tilde{x}\in\partial B_{i}\backslash\bP\,\text{s.t. }x\in\Ball{\frac{1}{8}\delta\of{\tilde{x}}}{\tilde{x}}\right\} \,,
\]
we find 
\[
\tilde{\delta}^{-\gamma}\leq\sum_{i}\chi_{\tilde{\delta}_{i}>0}\tilde{\delta}_{i}^{-\gamma}\,.
\]
Studying $\delta_{0}$ on $\partial B_{0}$ we can assume $M\leq M_{0}$
in (\ref{eq:lem:delta-tilde-construction-estimate-1}) and we find
\[
\int_{\bP}\chi_{\tilde{\delta}_{0}>0}\tilde{\delta}_{0}^{-\gamma-1}\leq C\int_{\partial B_{0}\backslash\bP}\delta^{-\gamma}\,.
\]
Hence we find
\[
\int_{\bP}\tilde{\delta}^{-\gamma-1}\leq\sum_{i}\int_{\bP}\chi_{\tilde{\delta}_{i}>0}\tilde{\delta}_{i}^{-\gamma-1}\leq\sum_{i}C\int_{\partial B_{i}\backslash\bP}\delta^{-\gamma}\,.
\]

\textbf{Estimate on $\beta$:} For two points $x_{i},x_{j}\in\X_{\pois}$
let $\mathrm{Circ}_{ij}:=\partial B_{i}\cap\partial B_{j}$ and $\Ball{\frac{1}{8}\tilde{\delta}}{\mathrm{Circ}_{ij}}:=\bigcup_{p\in\mathrm{Circ}_{ij}}\Ball{\frac{1}{8}\tilde{\delta}(p)}p$.
For the fixed ball $B_{i}=B_{0}$ we write $\mathrm{Circ}_{0j}$ and
obtain $\left|\mathrm{Circ}_{0j}\right|\leq C\delta_{0}^{d}$ with
$\delta_{0}$ from (\ref{eq:delta-0-xi}). Therefore, we find 
\[
\int_{\mathrm{Circ}_{0j}}(1+M(p))^{\beta}\leq\delta_{0}^{d}(1+M(p))^{\beta}\leq C\xi^{-\frac{1}{2}(\beta-d)}\,.
\]

We now derive an estimate for $\E\of{\int_{\Ball{1+\fr}0}\tilde{M}^{\beta}}.$

To this aim, let $q\in(0,1)$. Then $x\in\Ball{2-q^{k+1}}0\backslash\Ball{2-q^{k}}0$
implies $\xi\geq q^{k+1}$ and 
\begin{align*}
\int_{\Ball{1+\fr}0}\tilde{M}^{\beta} & \leq C+\sum_{k=1}^{\infty}\sum_{x_{j}\in\Ball{2-q^{k+1}}0\backslash\Ball{2-q^{k}}0}\int_{\mathrm{Circ}_{0j}}(1+M(p))^{\beta}\\
 & \leq C+\sum_{k=1}^{\infty}\sum_{x_{j}\in\Ball{2-q^{k+1}}0\backslash\Ball{2-q^{k}}0}C\left(q^{k+1}\right)^{-\frac{1}{2}(\beta-d)}
\end{align*}
The only random quantity in the latter expression is $\#\left\{ x_{j}\in\Ball{2-q^{k+1}}0\backslash\Ball{2-q^{k}}0\right\} $.
Therefore, we obtain with $\E\of{\X(A)}=\lambda\left|A\right|$ that
\begin{align*}
\E\of{\int_{\Ball{1+\fr}0}\tilde{M}^{\beta}} & \leq C\left(1+\sum_{k=1}^{\infty}\left(q^{k}-q^{k+1}\right)\left(q^{k+1}\right)^{-\frac{1}{2}(\beta-d)}\right)\\
 & \leq C\left(1+\sum_{k=1}^{\infty}\left(q^{k}\right)^{-\frac{1}{2}(\beta-d-2)}\right)\,.
\end{align*}
Since the point process has finite intensity, this property carries
over to the whole ball process and we obtain the condition $\beta<d+2$
in order for the right hand side to remain bounded.

\textbf{Estimate on $\tilde{\gamma}$:} We realize that $\tilde{\hat{\rho}}\geq\frac{\tilde{\delta}}{\tilde{M}}\geq\frac{\tilde{r}}{\tilde{M}}$.
Hence we obtain from H\"older's inequality

\[
\E\of{\tilde{\hat{\rho}}^{-\frac{rs}{s-1}}}\leq\E\of{\tilde{\delta}^{-\tilde{s}}}^{\frac{1}{q}}\E\of{\tilde{M}^{\frac{sr}{(s-1)}p}}^{\frac{1}{p}}\,,
\]
where $\tilde{s}=\frac{rs}{s-1}q$ and $\frac{1}{p}+\frac{1}{q}=1$.
From the right hand side of the last inequality, we infer boundedness
of the first expectation value for $\tilde{s}<2$ implying $q<\frac{2(s-1)}{sr}$.
Since we have to require $q>1$, this implies $r<2$ and $s>\frac{2}{2-r}$.
On the other hand, we know that the second expectation is finite if
$\frac{sr}{(s-1)}p<d+2$. For $q=\frac{2(s-1)}{sr}$, we obtain the
lower bound for $p=\frac{q}{q-1}$ and hence we conclude the sufficient
condition 
\[
2\frac{1}{\frac{2(s-1)}{sr}-1}\leq d+2\,,
\]
which implies our claim.

\textbf{Estimate on $\hat{d}$:} We have to estimate the local maximum
number of $A_{3,k}$ overlapping in a single point in terms of $\tilde{M}$.
We first recall that $\hat{\rho}(p)\approx8\tilde{M}(p)\tilde{\rho}(p)$.
Thus large discrepancy between $\hat{\rho}$ and $\tilde{\rho}$ occurs
in points where $\tilde{M}$ is large. This is at the intersection
of at least two balls. Despite these ``cusps'', the set $\partial\bP$
consists locally on the order of $\hat{\rho}$ of almost flat parts.
Arguing like in Lemma \ref{lem:properties-local-rho-convering} resp.
Remark \ref{rem:lem:properties-local-rho-convering} this yields $\hat{d}\leq d-1$.

\textbf{Estimate on $\alpha$:} Given two points $y_{1},y_{2}$ with
radii $\fr(y_{1})$, $\fr(y_{2})$, $\B_{y_{i}}:=\Ball{\fr(y_{i})}{y_{i}}$
and $\cM_{y_{i}}u:=\left|\B_{y_{1}}\right|^{-1}\int_{\B_{y_{1}}}u$
we find 
\[
\left|\cM_{y_{1}}u-\cM_{y_{2}}u\right|\leq\frac{\left|y_{1}-y_{2}\right|+\left|\left(\frac{\fr\of{y_{2}}}{\fr\of{y_{1}}}-1\right)\fr\of{y_{1}}\right|}{\left|\B_{y_{1}}\right|}\int_{\conv\of{\B_{y_{1}}\cup\B_{y_{2}}}}\left|\nabla u\right|\,.
\]
By our initial assumptions on $\fr(y_{i})$ we prove our claim on
$\alpha$.
\end{proof}
It remains to verify bounded average connectivity of the Boolean set
$\bP\left(\omega\right):=\bigcup_{i}B_{i}(\omega)$ or its complement.
In what follows we restrict to the Boolean set and use the following
result.
\begin{thm}
\label{thm:large-dev-pois}\cite{yao2011large}Let $\bP$ have a connected
component and let $\G(\X_{\pois})$ be the graph on $\X_{\pois}$
constructed from $x\sim y$ iff $\Ball 1x\cap\Ball 1y\neq\emptyset$.
Let $\tilde{\bP}$ be the connected component of $\bP$ and $\tilde{\X}_{\pois}:=\X_{\pois}\cap\tilde{\bP}$.
For $x,y\in\tilde{\X}_{\pois}$ let $d(x,y)$ be the graph distance.
Then for every $\eps>0$ there exists $\mu>1$, $\nu>0$ such that
\[
\P\of{\frac{d(x,y)}{\mu\left|x-y\right|}\not\in(1-\eps,1+\eps)}\leq e^{-\nu\left|x-y\right|}\,.
\]
\end{thm}

The latter result enables us to prove the following.
\begin{lem}
\label{lem:factor-poisson}Using the notation of Theorem \ref{thm:large-dev-pois},
let $x,y\in\tilde{\X}_{\pois}$ and $a>2$. Then
\[
\P\of{d(x,y)\geq4\mu a\left|x-y\right|\left(1+\eps\right)}\leq2e^{-\frac{\nu}{2}a\left|x-y\right|}\,.
\]
\end{lem}

In other words, the probability that the distance between $x$ and
$y$ on the grid is stretched by more than $5\mu a$ is decreasing
exponentially in $a$.
\begin{proof}
Let $x,y\in\tilde{\X}_{\pois}$. Let $a>2$ and let $n\in\N$ such
that $a\in[2^{n},2^{n+1})$. With probability $1-\exp\of{-\lambda\left|\S^{d-1}\right|\of{2^{dn+d}-2^{d}}\left|x-y\right|^{d}}>\frac{1}{2}$
there exists $z\in\Ball{2^{n+1}\left|x-y\right|}x\backslash\Ball{2^{n}\left|x-y\right|}x$.
For such $z$ it holds 
\begin{gather*}
2^{n}\left|x-y\right|\leq\left|z-x\right|<2^{n+1}\left|x-y\right|\\
2^{n}\left|x-y\right|\leq\left|z-y\right|<\left(2^{n+1}+1\right)\left|x-y\right|
\end{gather*}
 In particular, we obtain for $a_{n+1}:=2^{n+1}+1$
\begin{align*}
\frac{d(x,y)}{4\mu a\left|x-y\right|}\leq\frac{d(x,y)}{2\mu a_{n+1}\left|x-y\right|} & \leq\frac{d(x,z)}{2\mu a_{n+1}\left|x-y\right|}+\frac{d(z,y)}{2\mu a_{n+1}\left|x-y\right|}\\
 & \leq\frac{d(x,z)}{2\mu\left|x-z\right|}+\frac{d(z,y)}{2\mu\left|z-y\right|}
\end{align*}
Hence, assuming $1+\eps\leq\frac{d(x,y)}{4\mu a\left|x-y\right|}$
we find that at least one of the conditions $\frac{d(x,z)}{\mu\left|x-z\right|}\geq1+\eps$
or $\frac{d(z,y)}{\mu\left|z-y\right|}\geq1+\eps$ has to hold, which
implies 
\[
\P\of{1+\eps\leq\frac{d(x,y)}{4\mu a\left|x-y\right|}}\leq\P\of{\frac{d(x,z)}{\mu\left|x-z\right|}\geq1+\eps\text{ or }\frac{d(z,y)}{\mu\left|z-y\right|}\geq1+\eps}\,.
\]
Now it holds under the condition that $z$ exists 
\[
\P\of{\frac{d(x,z)}{\mu\left|x-z\right|}\geq1+\eps\text{ or }\frac{d(z,y)}{\mu\left|z-y\right|}\geq1+\eps}<e^{-\nu\left|x-z\right|}+e^{-\nu\left|y-z\right|}<2e^{-\nu2^{n}\left|x-y\right|}<2e^{-\frac{\nu}{2}a\left|x-y\right|}\,,
\]
which implies the statement.
\end{proof}
We construct a suitable graph $\left(\Y,\G(\bP)\right)$. For this
we choose $\X_{\fr}:=\X_{\fr}(\bP)$ according to Lemma \ref{lem:X-r-stationary}
and define 
\[
\Y_{\pois}=\Y_{\partial\X}\cup\partial\X\cup\X_{\fr}\cup\X_{\pois}\,.
\]
For $\Y_{\partial\X}$ and $\partial\X$ we choose the standard neighborhood
relation. Furthermore, we say for $y\in\Y_{\partial\X}$ and $x\in\X_{\fr}$
that $y\sim x$ iff there exists $\tilde{x}\in\X_{\pois}$ with $x,y\in\Ball 1{\tilde{x}}$
and for $x\in\X_{\fr}$, $\tilde{x}\in\X_{\pois}$ we say $x\sim\tilde{x}$
iff $x\in\Ball 1{\tilde{x}}$. This graph is called $\G_{\pois}$.
\begin{thm}
Let $\bP$ be the connected component of $\bigcup_{i}B_{i}(\omega)$.
Then $\bP$ is locally connected and for $\left(\Y_{\pois},\G_{\pois}\right)$
we find for every $\gamma>0$ that $\E\of{\sfS_{j}^{\gamma}}\leq\infty$.
\end{thm}

\begin{proof}
We write $a=\fr^{-1}$. Let $x_{1}\in\X_{\fr}$ with diameter $d_{1}$
of the Voronoi cell and let $\X_{\fr,1}:=\X_{\fr}\cap\Ball{3d_{1}}x$.
We can chose $\X_{\pois,1}\subset\X_{\pois}\cap\Ball{3d_{1}+a\fr}x$
with $\#\X_{\pois,1}=\#\X_{\fr,1}$ such that $\X_{\fr,1}\subset\Ball 1{\X_{\pois,1}}$.
Note in particular, that $\#\X_{\pois,1}\leq Cd_{1}^{d}$. Now let
$y\in\Y_{\partial\X}\cap\Ball{3d}x$ and let $Y=(y_{1},\dots,y_{k})\in\Apaths(x,y)$.
If $y_{1}=y$, then $y_{2}\in\X_{\fr,1}$ and, w.l.o.g., $y_{3},y_{k-1}\in\X_{\pois,1}$.
For the graph distances it holds 
\begin{align*}
d(x,y) & \leq d\of{x,y_{k-1}}+d\of{y_{k-1},y_{3}}+d\of{y_{2},y_{3}}+d\of{y_{1},y_{2}}\\
 & \leq a\fr+d\of{y_{k-1},y_{3}}+a\fr+4\sqrt{d}\,\fr\,.
\end{align*}
In case $d\of{y_{k-1},y_{3}}\leq1$ we conclude with $d(x,y)\leq\left(3a+4\sqrt{d}\right)\fr\leq\left(3a+4\sqrt{d}\right)d_{1}$.
If $d\of{y_{k-1},y_{3}}\geq4\sqrt{d}$ we obtain $d(x,y)\leq4d\of{y_{k-1},y_{3}}$.

Hence, because $\#\X_{\pois,1}\leq Cd_{1}^{d}$, it only remains to
observe that Lemma \ref{lem:factor-poisson} yields an exponential
decrease for the probability of large stretch factors for $d\of{y_{k-1},y_{3}}$.
\end{proof}

\subsection{\label{subsec:Matern-Process}Delaunay Pipes for a Matern Process}

For two points $x,y\in\Rd$, we denote 
\[
P_{r}(x,y):=\left\{ y+z\in\Rd:\,0\leq z\cdot(x-y)\leq\left|x-y\right|^{2},\,\left|z-z\cdot(x-y)\frac{x-y}{\left|x-y\right|}\right|<r\right\} \,,
\]
the cylinder (or pipe) around the straight line segment connecting
$x$ and $y$ with radius $r>0$.

Recalling Example \ref{exa:poisson-point-proc} we consider a Poisson
point process $\X_{\pois}(\omega)=\left(x_{i}(\omega)\right)_{i\in\N}$
with intensity $\lambda$ (recall Example \ref{exa:poisson-point-proc})
and construct a hard core Matern process $\X_{\mat}$ by deleting
all points with a mutual distance smaller than $d\fr$ for some $\fr>0$
(refer to Example \ref{exa:Matern}). From the remaining point process
$\X_{\mat}$ we construct the Delaunay triangulation $\D(\omega):=\D(X_{\mat}(\omega))$
and assign to each $(x,y)\in\D$ a random number $\delta(x,y)$ in
$(0,\fr)$ in an i.i.d. manner from some probability distribution
$\delta(\omega)$. We finally define 
\[
\bP(\omega):=\bigcup_{(x,y)\in\D(\omega)}P_{\delta(x,y)}(x,y)\bigcup_{x\in\X_{\mat}}\Ball{\fr}x
\]
 the family of all pipes generated by the Delaunay grid ``smoothed''
by balls with the fix radius $\fr$ around each point of the generating
Matern process.

Since the Matern process is mixing and $\delta$ is mixing, Lemma
\ref{lem:erg-and-mix-is-erg} yields that the whole process is still
ergodic.
\begin{rem}
The family of balls $\Ball{\fr}x$ can also be dropped from the model.
However, this would imply we had to remove some of the points from
$\X_{\mat}$ for the generation of the Voronoi cells. This would cause
technical difficulties which would not change much in the result,
as the probability for the size of Voronoi cells would still decrease
subexponentially.
\end{rem}

\begin{lem}
\label{lem:matern-delaunay-distribution}$\X_{\mat}$ is a point process
for $\bP(\omega)$ that satisfies Assumption \ref{assu:mesoscopic-voronoi}
and $\bP$ is isotropic cone mixing for $\X_{\mat}$ with exponentially
decreasing $f(R)\leq Ce^{-R^{d}}$. Furthermore, assume there exists
$C_{\delta},a_{\delta}>0$ such that $\P(\delta(x,y)<\delta_{0})\leq C_{\delta}e^{-a_{\delta}\frac{1}{\delta_{0}}}$,
then $\P(\tilde{M}>M_{0})\leq Ce^{-aM_{0}}$ for some $C,a>0$.
\end{lem}

\begin{proof}
\emph{Isotropic cone mixing:} For $x,y\in2d\fr\Zd$ the events $\left(x+[0,1]^{d}\right)\cap\X_{\mat}$
and $\left(y+[0,1]^{d}\right)\cap\X_{\mat}$ are mutually independent.
Hence
\[
\P\of{\left(k2dr\,[-1,1]^{d}\right)\cap\X_{\mat}=\emptyset}\leq\P\of{[-1,1]^{d}\cap\X_{\mat}=\emptyset}^{k^{d}}\,.
\]
Hence the open set $\bP$ is isotropic cone mixing for $\X=\X_{\mat}$
with exponentially decaying $f(R)\leq Ce^{-R^{d}}$.

\emph{Estimate on $\delta$:} There exists $C>0$ such that $\bP$
is $\left(\delta(x,y),C\delta(x,y)^{-1}\right)$-regular in every
$x\in\partial P_{\delta(x,y)}(x,y)$. Since the distribution of $\delta(x,y)$
is independent from $x$ and $y$, this implies that $\P(\delta<\delta_{0})\leq C_{\delta}e^{-a_{\delta}\frac{1}{\delta_{0}}}$.

\emph{Estimate on the distribution of $M$:} By definition of the
Delaunay triangulation, two pipes intersect only if they share one
common point $x\in\X_{\mat}$.

Given three points $x,y,z\in\X_{\mat}$ with $x\sim y$ and $x\sim z$,
the highest local Lipschitz constant on $\partial\left(P_{\delta(x,y)}(x,y)\cup P_{\delta(x,z)}(x,z)\right)$
is attained in 
\[
\tilde{x}=\arg\max\left\{ \left|x-\tilde{x}\right|:\,\tilde{x}\in\partial P_{\delta(x,y)}(x,y)\cap\partial P_{\delta(x,z)}(x,z)\right\} \,.
\]
It is bounded by 
\[
\max\left\{ \arctan\left(\frac{1}{2}\sphericalangle\left((x,y),(x,z)\right)\right),\,\frac{1}{\delta(x,y)},\,\frac{1}{\delta(x,z)}\right\} \,,
\]
where $\alpha:=\sphericalangle\left((x,y),(x,z)\right)$ in the following
denotes the angle between $(x,y)$ and $(x,z)$, see Figure \ref{fig:Voronoi-delaunay-proof}.
If $d_{x}$ is the diameter of the Voronoi cell of $x$, we show that
a necessary (but not sufficient) condition that the angle $\alpha$
can be smaller than some $\alpha_{0}$ is given by 
\begin{equation}
d_{x}\geq C\frac{1}{\sin\alpha_{0}}\,,\label{eq:lem:matern-delaunay-distribution-1}
\end{equation}
where $C>0$ is a constant depending only on the dimension $d$. Since
for small $\alpha$ we find $M\approx\frac{1}{\sin\alpha}$, and since
the distribution for $d_{x}$ decays subexponentially, also the distribution
for $M$ decays subexponentially.
\begin{figure}
 \begin{minipage}[c]{0.6\textwidth} \includegraphics[width=6.5cm]{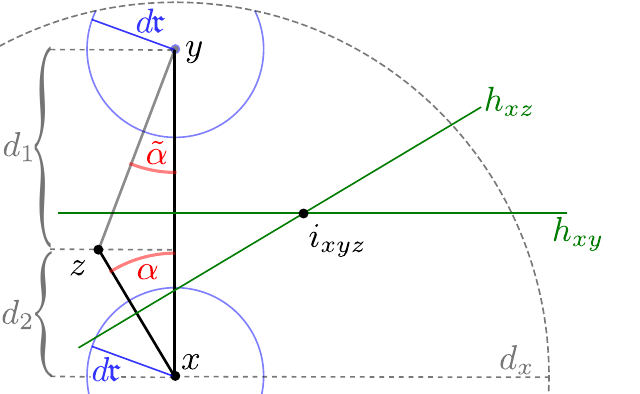}\end{minipage}\hfill   \begin{minipage}[c]{0.35\textwidth}\caption{\label{fig:Voronoi-delaunay-proof}Sketch of the proof of Lemma \ref{lem:matern-delaunay-distribution}
and estimate (\ref{eq:lem:matern-delaunay-distribution-1})\emph{.}}
\end{minipage}
\end{figure}

\emph{Proof of (\ref{eq:lem:matern-delaunay-distribution-1}):} Given
an angle $\alpha>0$ and $x\in\X_{\mat}$ we derive a lower bound
for the diameter of $G(x)$ such that for two neighbors $y,z$ of
$x$ it can hold $\sphericalangle\left((x,y),(x,z)\right)\leq\alpha$.
With regard to Figure \ref{fig:Voronoi-delaunay-proof}, we assume
$\left|x-y\right|\geq\left|x-z\right|$.

Writing $d_{x}:=d(x)$ the diameter of $G(x)$ and $\tilde{\alpha}=\sphericalangle\left((x,z),(z,y)\right)$,
w.l.o.g let $y=(d_{1}+d_{2},0,\dots,0)$, where $d_{1}+d_{2}<d_{x}$
and $d_{1}=\left|y-z\right|\cos\tilde{\alpha}$. Hence we can assume that $z$ takes the form 
$z=(d_{2},-\left|y-z\right|\sin\tilde{\alpha},0\dots0)$ and in what
follows, we focus on the first two coordinates only. The boundaries
between the cells $x$ and $z$ and $x$ and $y$ lie on the planes
\[
h_{xz}(t)=\frac{1}{2}z+t\left(\begin{array}{c}
\left|y-z\right|\sin\tilde{\alpha}\\
d_{2}
\end{array}\right)\,,\quad h_{xy}(s)=\frac{1}{2}y+s\left(\begin{array}{c}
0\\
1
\end{array}\right)
\]
respectively. The intersection of these planes has the first two coordinates
\[
i_{xyz}:=\left(\frac{d_{1}+d_{2}}{2},-\frac{1}{2}\left|y-z\right|\sin\tilde{\alpha}+\frac{1}{2}\frac{d_{1}d_{2}}{\left|y-z\right|\sin\tilde{\alpha}}\right)
\]
. Using the explicit form of $d_{2}$, the latter point has the distance
\[
\left|i_{xyz}\right|^{2}=\frac{1}{4}\left|y-z\right|^{2}+\frac{1}{4}d_{2}^{2}+\frac{1}{4}\frac{d_{2}^{2}\cos^{2}\tilde{\alpha}}{\sin^{2}\tilde{\alpha}}
\]
to the origin $x=0$. Using $\left|y-z\right|\sin\tilde{\alpha}=\left|z\right|\sin\alpha$
and $d_{2}=\left|y\right|-\left|z\right|\cos\alpha$ we obtain 
\[
\left|i_{xyz}\right|^{2}=\frac{1}{4}\left(\left|y-z\right|^{2}\left(1+\frac{\left(\left|y\right|-\left|z\right|\cos\alpha\right)^{2}\cos^{2}\tilde{\alpha}}{\left|z\right|^{2}\sin^{2}\alpha}\right)+\left(\left|y\right|-\left|z\right|\cos\alpha\right)^{2}\right)\,.
\]
Given $y$, the latter expression becomes small for $\left|y-z\right|$
small, with the smallest value being $\left|y-z\right|=d\fr$. But
then 
\[
\cos^{2}\tilde{\alpha}=1-\sin^{2}\tilde{\alpha}=1-\frac{\left(\left|z\right|\sin\alpha\right)^{2}}{\left|y-z\right|^{2}}
\]
and hence the distance becomes 
\[
\left|i_{xyz}\right|^{2}=\frac{1}{4}\left(\left(d\fr\right)^{2}\left(1+\frac{\left(\left|y\right|-\left|z\right|\cos\alpha\right)^{2}\left(\left(d\fr\right)^{2}+\left|z\right|^{2}\sin^{2}\alpha\right)}{\left(d\fr\right)^{2}\left|z\right|^{2}\sin^{2}\alpha}\right)+\left(\left|y\right|-\left|z\right|\cos\alpha\right)^{2}\right)\,.
\]
We finally use $\left|y\right|=\left|z\right|\cos\alpha-\sqrt{\left(d\fr\right)^{2}-\left|z\right|^{2}\sin^{2}\alpha}$
and obtain 
\[
\left|i_{xyz}\right|^{2}=\frac{1}{4}\left(\left(d\fr\right)^{2}\left(1+\frac{\left(\left(d\fr\right)^{4}-\left|z\right|^{4}\sin^{4}\alpha\right)}{\left(d\fr\right)^{2}\left|z\right|^{2}\sin^{2}\alpha}\right)+\left(\left(d\fr\right)^{2}-\left|z\right|^{2}\sin^{2}\alpha\right)\right)\,.
\]
The latter expression now needs to be smaller than $d_{x}$. We observe
that the expression on the right hand side decreases for fixed $\alpha$
if $\left|z\right|$ increases.

On the other hand, we can resolve $\left|z\right|(y)=\left|y\right|\cos\alpha-\sqrt{\left|y\right|^{2}\sin^{2}\alpha+\left(d\fr\right)^{2}}$.
From the conditions $\left|y\right|\leq d_{x}$ and $\left|i_{xyz}\right|\leq d_{x}$,
we then infer (\ref{eq:lem:matern-delaunay-distribution-1}).
\end{proof}
\begin{lem}
Let $\Y$ be constructed from Lemma \ref{lem:cover-iso-mixing-delta-M}
for $\X_{\fr}=\X_{\mat}$ with the corresponding standard graph $\G_{\simple}(\bP)$
(see Definition \ref{def:simple-border}). Let the admissible paths
$\Apaths(y,x)$, $x\in\X_{\fr}$, $y\in\Y$, be the set of shortest
paths on the graph between $x$ and $y$. Then there exists $C>0$
such that for every $x_{j}\in\X_{\fr}$ it holds $\sfR_{0}(x_{j},\fA_{2,j})/d_{j}+\sfS_{j}\leq C$.
In particular, for every $1<s<p$ it holds 
\[
\lim_{n\to\infty}\frac{1}{\left|n\bQ\right|}\int_{\Apaths(n\bQ)}\left(\sum_{x_{j}\in\X_{\fr}(\bQ)}\chi_{\Ball{\sfR_{0}(x_{j},\fA_{2,j})}{x_{j}}}d_{j}^{d+\frac{s-1}{s}}\sfS_{j}^{\frac{s-1}{s}}\right)^{\frac{p}{p-s}}<\infty\,.
\]
\end{lem}

\begin{proof}
Since the admissible paths are the shortest paths, there exists $C>0$
such that for every $Y\in\Apaths(y,x_{j})$, $x_{j}\in\X_{\fr}$,
$y\in\Y\cap\Ball{\frac{\fr2}{}}{\fA_{1,j}}$ it holds $\Length Y\leq C\left|x_{j}-y\right|$.
Furthermore, for $x_{k}\in\X_{\fr}$ with $x_{k}\sim\sim x_{j}$ we
find $\left|x_{k}-x_{j}\right|\leq2d_{j}$ and since $x_{k}$ and
$x_{j}$ are connected through a path lying inside $\Ball{2d_{j}}{x_{j}}$
possibly crossing other $x_{i}\in\fA_{2,j}\cap\Ball{2d_{j}}{x_{j}}$
we can assume for the same $C$ that for every $Y\in\Apaths(x_{k},x_{j})$,
$x_{k}\sim\sim x_{j}$ it holds $\Length Y\leq C\left|x_{k}-x_{j}\right|$.
This provides a uniform bound on $\sfR_{0}(x_{j},\fA_{2,j})+\sfS_{j}d_{j}\leq Cd_{j}$.
The lemma now follows from Lemma \ref{lem:estim-E-fa-fb} and Theorem
\ref{thm:ergodic-Apaths}.
\end{proof}

\section{\label{sec:Sobolev-spaces-on}Sobolev Spaces on the Probability Space
$(\Omega,\protect\P)$}

Based on Assumption \ref{assu:separable}, we want to achieve a better
understanding of the mapping $f\mapsto f_{\omega}$. For this we make
the following basic assumption throughout this section.
\begin{assumption}
\label{assu:Omega-mu-tau-fundmental}Let $\left(\Omega,\sigma,\P\right)$
be a probability space satisfying Assumption \ref{assu:separable}
and let $\tau$ be a dynamical system on $\Omega$ in the sense of
Definition \ref{Def:Omega-mu-tau}.
\end{assumption}

For the introduction of traces of $W^{1,p}(\Omega)$-functions below
we will need the following (uncommon) stronger assumption. It is motivated
by Theorem \ref{thm:Main-THM-Sto-Geo}, which states that we can assume
$\Omega$ to be a separable metric space.
\begin{assumption}
\label{assu:Omega-mu-tau-fundmental-strong}Let $\left(\Omega,\sigma,\P\right)$
be a probability space satisfying Assumption \ref{assu:separable}
and let $\tau$ be a dynamical system on $\Omega$ in the sense of
Definition \ref{Def:Omega-mu-tau}. Furthermore, let $\Omega$ be
a separable metric space such that $\sigma$ is the completion of
the Borel algebra $\borelB(\Omega)$ under the construction of the
Lebesgue space $L^{1}(\Omega;\P)$.
\end{assumption}

Assumption \ref{assu:Omega-mu-tau-fundmental-strong} will pay of
due to the second part of the following lemma, which is a fundamental
property of separable $\sigma$-algebras.
\begin{lem}
\label{lem:Lp-dense}Let $(A,\Sigma,\mu)$ be a measure space with
a countably generated $\sigma$-algebra $\Sigma$. Then for every
$1\leq p<\infty$ the space $L^{p}(A;\mu)$ is separable. If $A$
is a separable metric space and $\Sigma$ the completion of the Borel
algebra with respect to $\mu$ then $C_{b}(A)\hookrightarrow L^{p}(\Omega;\mu)$
densely and continuously, where $C_{b}(\Omega)$ are the bounded continuous
functions on $\Omega$.
\end{lem}

The following lemma is a fundamental observation which will be frequently
used throughout the rest of this work. It relies on the following
notation. For $f:\Omega\to X$, $X$ a metric space, and $\omega\in\Omega$
we define the \emph{realization} $f_{\omega}$ of $f$ as
\[
f_{\omega}:\Rd\to X\,,\qquad x\mapsto f\of{\tau_{x}\omega}\,.
\]
Then we find the following behavior.
\begin{lem}
\label{lem:f-omega-in-Lp}Let Assumption \ref{assu:Omega-mu-tau-fundmental}
hold and let $f\in L^{p}(\Omega)$ for $1\leq p\leq\infty$. Then
for almost every $\omega\in\Omega$ and for every bounded domain $\bQ$
it holds $f_{\omega}\in L^{p}(\bQ)$.
\end{lem}

\begin{proof}
For $1\leq p<\infty$ observe that 
\begin{align*}
\lebesgueL(\bQ)\int_{\Omega}\left|f(\omega)\right|^{p}\d\P(\omega) & =\int_{\bQ}\int_{\Omega}\left|f(\omega)\right|^{p}\d\P(\omega)\d x=\int_{\bQ}\int_{\Omega}\left|f(\tau_{x}\omega)\right|^{p}\d\P(\omega)\d x\\
 & =\int_{\Omega}\int_{\bQ}\left|f(\tau_{x}\omega)\right|^{p}\d x\,\d\P(\omega)\,.
\end{align*}
From Fubini's theorem it follows that $\int_{\bQ}\left|f(\tau_{x}\omega)\right|^{p}\d x$
exists for a.e. $\omega\in\Omega$. For $p=\infty$ the statement
follows since $\int_{\bQ}\left|f(\tau_{x}\omega)\right|^{p}\d x$
exists for every $p<\infty$.
\end{proof}

\subsection{The Semigroup $\protect\rmT$ on $L^{p}(\Omega)$ and its Generators}

For every $x\in\Rd$ we define the mapping 
\[
\rmT(x):\,f\mapsto\rmT(x)f\,,
\]
through $\rmT(x)f(\omega):=f(\tau_{x}\omega)$. This mapping is well
defined for every measurable function $f:\,\Omega\to\R$. Moreover,
we have the following properties.
\begin{lem}
\label{lem:T-strongly-continuous}Let Assumption \ref{assu:Omega-mu-tau-fundmental}
hold. For every $1\leq p<\infty$, the family $\left(\rmT(x)\right)_{x\in\Rd}$
is a strongly continuous unitary group on $L^{p}(\Omega)$.
\end{lem}

\begin{proof}
Every $\rmT(x)$ is linear on $L^{p}(\Omega)$ and the group property
follows from $\left(\rmT(x)\rmT(y)f\right)(\omega)=f(\tau_{x}\tau_{y}\omega)=\rmT(x+y)f(\omega)$.
Since $\tau_{x}$ is measure preserving, we find $\left\Vert f\right\Vert _{L^{p}(\Omega)}=\left\Vert \rmT(x)f\right\Vert _{L^{p}(\Omega)}$
and hence $\rmT(x)$ is unitary.

In order to prove the strong continuity, observe 
\begin{align*}
\left\Vert \rmT(x)f-f\right\Vert _{L^{p}(\Omega)}^{p} & =\int_{\Omega}\left|f(\tau_{x}\omega)-f(\omega)\right|^{p}\d\P(\omega)\\
 & =\int_{\Omega}\int_{\Y}\left|f(\tau_{x+y}\omega)-f(\tau_{y}\omega)\right|^{p}\d y\,\d\P(\omega)\\
 & =\int_{\Omega}\int_{\Y}\left|f_{\omega}(x+y)-f_{\omega}(y)\right|^{p}\d y\,\d\P(\omega)\,,
\end{align*}
where we used that $\tau_{y}$ preserves measure and Fubini's theorem.
By Lemma \ref{lem:f-omega-in-Lp} $f_{\omega}\in L_{\loc}^{p}(\Rd)$
for almost every $\omega\in\Omega$ and for such $\omega$ it holds
$\lim_{|h|\to0}\norm{f_{\omega}-f_{\omega}(\cdot+h)}_{L^{p}(4\Y)}=0$.
Furthermore, for $|x|<\frac{1}{2}$ we have 
\[
\int_{\Y}\left|f_{\omega}(x+y)-f_{\omega}(y)\right|^{p}\d y<2p\int_{2\Y}\left|f_{\omega}(y)\right|^{p}\d y\,.
\]
Thus, the Lebesgue dominated convergence theorem yields 
\[
\sup_{|x|<t}\left\Vert \rmT(x)f-f\right\Vert _{L^{p}(\Omega)}^{p}\to0\qquad\mbox{as }t\to0\,.
\]
\end{proof}
For $i=1,\dots,d$, let $\be_{i}$ be the $i$-th canonical basis
vector in $\Rd$. Since $\rmT(x)$ define a strongly continuous group
we can draw the conclusion that the operators $\rmT_{i}(t)f:=\rmT(t\be_{i})f$,
define $d$ independent one-parameter strongly continuous semigroups
on $L^{p}(\Omega)$ that commute with each other and jointly generate
$\left(\rmT(x)\right)_{x\in\Rd}$ on $L^{p}(\Omega)$. Each of these
one-parameter groups has a generator $\rmD_{i}$ defined by 
\[
\rmD_{i}f(\omega)=\lim_{t\to0}\frac{\rmT_{i}(t)f(\omega)-f(\omega)}{t}=\lim_{t\to0}\frac{f(\tau_{t\be_{i}}\omega)-f(\omega)}{t}\,.
\]
The expression $\rmD_{i}f$ is called $i$-th derivative of $f$ and
is skew adjoint:
\[
\int_{\Omega}g\rmD_{i}f\d\P=-\int_{\Omega}f\rmD_{i}g\d\P\,.
\]
The joint domain of all $\rmD_{i}$ in $L^{p}(\Omega)$ is denote
\[
W^{1,p}(\Omega):=\left\{ f\in L^{p}(\Omega)\,|\;\forall i=1,\dots,d\,:\,\,\rmD_{i}f\in L^{p}(\Omega)\right\} \,,
\]
with the natural norm
\[
\left\Vert f\right\Vert _{W^{1,p}(\Omega)}:=\left\Vert f\right\Vert _{L^{p}(\Omega)}+\sum_{i=1}^{d}\left\Vert \rmD_{i}f\right\Vert _{L^{p}(\Omega)}\,.
\]
In case $p=2$, this is a Hilbert space with scalar product
\[
\left\langle f,g\right\rangle _{W^{1,2}(\Omega)}^{2}:=\int_{\Omega}fg\d\P+\sum_{i=1}^{d}\int_{\Omega}\rmD_{i}f\rmD_{i}g\,\d\P\,.
\]
We finally denote $\rmD_{\omega}f:=\left(\rmD_{1}f,\dots,\rmD_{d}f\right)^{T}$
the gradient with respect to $\omega$ and by $-\diver_{\omega}$
the adjoint of $\rmD_{\omega}$. Sometimes we write $\nabla_{\omega}f:=\rmD_{\omega}f$
to underline the gradient aspect. Similar to distributional derivatives
in $\Rd$, we may define $\rmD_{\omega}^{k}f$ through iterated application
of $\rmD_{\omega}$ and 
\[
W^{k,p}(\Omega):=\left\{ f\in L^{p}(\Omega)\,|\;\forall j=1,\dots,k\,:\,\,\rmD_{\omega}^{j}f\in L^{p}(\Omega)^{d^{j}}\right\} \,.
\]
In case Assumption \ref{assu:Omega-mu-tau-fundmental-strong} holds,
we denote 
\[
C_{b}^{1}(\Omega):=\left\{ f\in C_{b}(\Omega):\,\nabla f\in C_{b}(\Omega;\Rd)\right\} \,.
\]

\begin{lem}
For every $f\in W^{1,p}(\Omega)$ for almost every $\omega\in\Omega$
it holds $f_{\omega}\in W_{\loc}^{1,p}(\Rd)$. In particular, for
every bounded domain $\bQ\subset\Rd$ it holds
\begin{equation}
\forall\psi\in C_{c}^{1}(\bQ)\,:\qquad\int_{\bQ}f_{\omega}\partial_{i}\psi=-\int_{\bQ}\psi\left(\rmD_{i}f\right)_{\omega}\,.\label{eq:partial-integr-realization}
\end{equation}
\end{lem}

\begin{proof}
Let $\psi\in C_{c}^{\infty}(\Rd)$ and let $g\in L^{q}(\Omega)$,
$\frac{1}{p}+\frac{1}{q}=1$. Using Lebesgue's dominated convergence
theorem it follows:
\begin{align*}
\int_{\Omega}g(\omega)\int_{\Rd}f_{\omega}\partial_{i}\psi\d\lebesgueL\d\P(\omega) & =\int_{\Omega}g(\omega)\int_{\Rd}f_{\omega}(x)\lim_{t\to0}\frac{\psi(x+t\be_{i})-\psi(x)}{t}\d x\d\P(\omega)\\
 & =\lim_{t\to0}\int_{\Omega}g(\omega)\int_{\Rd}f_{\omega}(x)\frac{\psi(x+t\be_{i})-\psi(x)}{t}\d x\d\P(\omega)\\
 & =\lim_{t\to0}\int_{\Omega}g(\omega)\int_{\Rd}\psi(x+t\be_{i})\frac{f_{\omega}(x)-f_{\omega}(x+t\be_{i})}{t}\d x\d\P(\omega)\,.
\end{align*}
Since $\tau_{x}$ preserves measure, we obtain 
\begin{align*}
\int_{\Omega}\int_{\Rd}f_{\omega}\partial_{i}\psi\d\lebesgueL\d\P(\omega) & =\lim_{t\to0}\int_{\Rd}\int_{\Omega}g(\tau_{-x}\omega)\psi(x+t\be_{i})\frac{f(\omega)-\rmT_{i}f(\omega)}{t}\d x\d\P(\omega)\\
 & =\lim_{t\to0}\int_{\Omega}\frac{f(\omega)-\rmT_{i}f(\omega)}{t}\int_{\Rd}g(\tau_{-x}\omega)\psi(x+t\be_{i})\d x\d\P(\omega)\\
 & =-\int_{\Rd}\int_{\Omega}g(\tau_{-x}\omega)\psi(x)\rmD_{i}f(\omega)\d x\d\P(\omega)\\
 & =-\int_{\Omega}g(\omega)\int_{\Rd}\left(\rmD_{i}f\right)_{\omega}\psi\d\lebesgueL\d\P(\omega)\,.
\end{align*}
Using a countable dense subset $\left(\psi_{i}\right)_{i\in\N}\subset L^{q}(\Rd)$,
$\psi_{i}\in C_{c}^{\infty}(\Rd)$ and a suitable family of testfunctions
$\left(g_{j}\right)_{j\in\N}\subset L^{q}(\Omega)$, we obtain that
for almost every $\omega\in\Omega$ equation (\ref{eq:partial-integr-realization})
holds for every $\psi_{i}$. Hence, by density, it holds for all $\psi\in C_{c}^{1}(\bQ)$.
\end{proof}
\begin{lem}
\label{lem:Omega-convolution}Let $1\leq p<\infty$ and let $\eta\in C_{c}^{\infty}(\Rd)$.
For every $f\in L^{p}(\Omega)$ let 
\[
\left(\eta\ast f\right)(\omega):=\int_{\Rd}\eta(x)f\of{\tau_{x}\omega}\d x\,.
\]
Then for every $k\in\N$ it holds $\eta\ast f\in W^{k,p}(\Omega)$
with $\rmD_{i}\of{\eta\ast f}=(\partial_{i}\eta)\ast f$ and almost
every realization of $\sI_{\delta}f$ is an element of $C^{\infty}(\Rd)$.
Furthermore, the estimates
\begin{equation}
\left\Vert \eta\ast f\right\Vert _{L^{p}(\Omega)}^{p}\leq\left\Vert \eta\right\Vert _{L^{1}(\Rd)}\left\Vert f\right\Vert _{L^{p}(\Omega)}\,,\qquad\left\Vert \rmD_{i}\of{\eta\ast f}\right\Vert _{L^{p}(\Omega)}^{p}\leq\left\Vert \partial_{i}\eta\right\Vert _{L^{1}(\Rd)}\left\Vert f\right\Vert _{L^{p}(\Omega)}\label{eq:lem:Omega-convolution-1}
\end{equation}
hold and we have $\rmD_{i}\left(\eta\ast f\right)=\eta\ast\rmD_{i}f$.
\end{lem}

\begin{proof}
Let $k\in\N$ and observe 
\begin{align*}
\left\Vert \eta\ast f\right\Vert _{L^{p}(\Omega)}^{p} & =\int_{\Omega}(2k)^{-d}\int_{(-k,k)^{d}}\left|\left(\eta\ast f\right)(\tau_{y}\omega)\right|^{p}\d y\,\d\P(\omega)\\
 & \leq\int_{\Omega}(2k)^{-d}\int_{(-k,k)^{d}}\int_{\Rd}\left|\eta(x)f(\tau_{y+x}\omega)\right|^{p}\d x\,\d y\,\d\P(\omega)\,.
\end{align*}
Due to the convolution inequality we have 
\begin{align*}
\left\Vert \eta\ast f\right\Vert _{L^{p}(\Omega)}^{p} & \leq\left\Vert \eta\right\Vert _{L^{1}(\Rd)}(2k)^{-d}\int_{\Omega}\left\Vert f_{\omega}\right\Vert _{L^{p}((-k-1,k+1)^{d})}^{p}\d\P(\omega)\\
 & \leq\left\Vert \eta\right\Vert _{L^{1}(\Rd)}\left(\frac{k+1}{k}\right)^{-d}\int_{\Omega}\left|f(\omega)\right|^{p}\d\P(\omega)
\end{align*}
and since $k$ is arbitrary, the we obtain $\left\Vert \eta\ast f\right\Vert _{L^{p}(\Omega)}^{p}\leq\left\Vert \eta\right\Vert _{L^{1}(\Rd)}\left\Vert f\right\Vert _{L^{p}(\Omega)}$,
the first part of (\ref{eq:lem:Omega-convolution-1}).

In order to show $\sI_{\delta}f\in W^{k,p}(\Omega)$ observe 
\[
\frac{1}{t}\left(\eta\ast f(\tau_{t\be_{i}}\omega)-\eta\ast f(\omega)\right)=\int_{\Rd}\frac{1}{t}\left(\eta(x+t\be_{i})-\eta(x)\right)f(\tau_{x}\omega)\,.
\]
Taking the limit $t\to0$ in $L^{p}(\Omega)$ on both sides using
Lebesgue's dominated convergence theorem implies 
\begin{equation}
\rmD_{i}\left(\eta\ast f\right)=\int_{\Rd}\partial_{i}\eta(x)f(\tau_{x}\omega)\,,\label{eq:prop-si-delta-help-b}
\end{equation}
and hence $\rmD_{i}\left(\sI_{\delta}f\right)\in L^{p}(\Omega)$ with
$\rmD_{i}\of{\eta\ast f}=(\partial_{i}\eta)\ast f$ and the second
part of (\ref{eq:lem:Omega-convolution-1}) follows. Equation (\ref{eq:prop-si-delta-help-b})
also shows that 
\[
\left(\eta\ast f\right)(\tau_{y}\omega)=\int_{\Rd}\eta(x)f\left(\tau_{x+y}\omega\right)\d x=\int_{\Rd}\eta(x-y)f\left(\tau_{x}\omega\right)\d x
\]
and hence almost every realization of $\eta\ast f$ has $C^{\infty}$-regularity.
Furthermore, (\ref{eq:prop-si-delta-help-b}) implies 
\begin{align*}
\rmD_{i}\left(\sI_{\delta}f\right) & =\lim_{t\to0}\frac{1}{t}\left(\left(\eta\ast f\right)(\tau_{t\be_{i}}\omega)-\left(\eta\ast f\right)(\omega)\right)\\
 & =\eta\ast\lim_{t\to0}\frac{1}{t}\left(f(\tau_{t\be_{i}}\omega)-f(\omega)\right)\\
 & =\eta\ast\rmD_{i}f\,,
\end{align*}
where we used continuity of $f\mapsto\eta\ast f$ and strong convergence
of $\frac{1}{t}\left(f(\tau_{t\be_{i}}\omega)-f(\omega)\right)\to\rmD_{i}f$.
\end{proof}
Similar to $L^{p}(\Rd)$- and Sobolev spaces on $\Rd$, we can introduce
a family of smoothing operators. Let $\left(\eta_{\delta}\right)_{\delta>0}$
be a standard sequence of mollifiers which are symmetric w.r.t. $0$
and define 
\begin{equation}
\sI_{\delta}:\,L^{p}(\Omega)\to L^{p}(\Omega)\,,\qquad\sI_{\delta}f(\omega):=\int_{\Rd}\eta_{\delta}(x)f\left(\tau_{x}\omega\right)\d x\,.\label{eq:eta-delta}
\end{equation}

\begin{lem}
\label{lem:Properties-sI-delta}For every $\delta>0$, $1\leq p<\infty$,
the operator $\sI_{\delta}$ is unitary and selfadjoint. For every
$f\in L^{p}(\Omega)$, $k\in\N$ it holds $\sI_{\delta}f\in W^{k,p}(\Omega)$,
$\sI_{\delta}f\to f$ strongly in $L^{p}(\Omega)$, and almost every
realization of $\sI_{\delta}f$ is an element of $C^{\infty}(\Rd)$.
Finally, for $f\in W^{1,p}(\Omega)$ it holds 
\begin{equation}
\lim_{\delta\to0}\left\Vert \sI_{\delta}f-f\right\Vert _{W^{1,p}(\Omega)}=0\label{eq:lem:Properties-sI-delta}
\end{equation}
and $\rmD_{i}\sI_{\delta}f=\sI_{\delta}\rmD_{i}f$.
\end{lem}

\begin{proof}
The selfadjointness follows from the definition of $\sI_{\delta}$,
symmetry of $\eta_{\delta}$ and invariance of $\P$ w.r.t. $\tau_{x}$.
All other parts except for (\ref{eq:lem:Properties-sI-delta}) follow
from Lemma \ref{lem:Omega-convolution}.

Finally, observe that the the convolution inequality and the strong
continuity of $\rmT(x)$ yield 
\begin{align*}
\int_{\Omega}\left|\sI_{\delta}f-f\right|^{p} & =\int_{\Omega}\left|\int_{\Rd}\eta_{\delta}(x)\left(f(\tau_{x}\omega)-f(\omega)\right)\d x\right|^{p}\d\P(\omega)\\
 & \leq\int_{\Omega}\left\Vert \eta_{\delta}\right\Vert _{L^{1}(\Rd)}^{p}\int_{[-\delta,\delta]^{d}}\left|\left(f(\tau_{x}\omega)-f(\omega)\right)\right|^{p}\d x\,\d\P(\omega)\\
 & \leq\int_{[-\delta,\delta]^{d}}\int_{\Omega}\left|\left(f(\tau_{x}\omega)-f(\omega)\right)\right|^{p}\d\P(\omega)\,\d x\\
 & \to0\,.
\end{align*}
Since $\rmD_{i}\sI_{\delta}f=\sI_{\delta}\rmD_{i}f$, it also holds
$\rmD_{i}\sI_{\delta}f\to\rmD_{i}f$ strongly in $L^{p}(\Omega)$.
\end{proof}

\subsection{Gradients and Solenoidals}

We denote by $L_{\loc}^{p}(\Rd;\Rd)$ the set of measurable functions
$f:\,\Rd\to\Rd$ such that $f|_{\bU}\in L^{p}(\bU;\Rd)$ for every
bounded domain $\bU$ and we define 
\begin{align*}
L_{\pot,\loc}^{p}(\Rd) & :=\left\{ u\in L_{\loc}^{2}(\Rd;\Rd)\,\,|\,\,\forall\bU\,\,\mbox{bounded domain, }\exists\varphi\in H^{1}(\bU)\,:\,u=\nabla\varphi\right\} \,,\\
L_{\sol,\loc}^{p}(\Rd) & :=\left\{ u\in L_{\loc}^{2}(\Rd;\Rd)\,\,|\,\,\int_{\Rd}u\cdot\nabla\varphi=0\,\,\forall\varphi\in C_{c}^{1}(\Rd)\right\} \,.
\end{align*}

\begin{rem}
\label{rem:Lppot-invariant}The space $L_{\pot,\loc}^{p}(\Rd)$ is
invariant under convolution. This follows immediately from the fact
that if $u=\nabla\varphi$ locally, then $\eta_{\delta}\ast u=\nabla\left(\eta_{\delta}\ast\varphi\right)$.
\end{rem}

Recalling the notation for a realization $u_{\omega}(x):=u(\tau_{x}\omega)$
for $u\in L^{p}(\Omega)$, we can then define corresponding spaces
on $\Omega$ through
\begin{align}
L_{\pot}^{p}(\Omega) & :=\left\{ u\in L^{p}(\Omega;\Rd)\,:\,u_{\omega}\in L_{\pot,\loc}^{p}(\Rd)\,\,\mbox{for }\P-\mbox{a.e. }\omega\in\Omega\right\} \,,\nonumber \\
L_{\sol}^{p}(\Omega) & :=\left\{ u\in L^{p}(\Omega;\Rd)\,:\,u_{\omega}\in L_{\sol,\loc}^{p}(\Rd)\,\,\mbox{for }\P-\mbox{a.e. }\omega\in\Omega\right\} \,,\label{eq:sto-Lp-pot-sol-omega-general-1}\\
\cV_{\pot}^{p}(\Omega) & :=\left\{ u\in L_{\pot}^{p}(\Omega)\,:\,\int_{\Omega}u\,\d\P=0\right\} \,.\nonumber 
\end{align}
The spaces $L_{\pot}^{p}(\Omega)$ and $W^{1,p}(\Omega)$ are connected
as the following theorem shows.
\begin{thm}
\label{thm:Lp-W1p-structure}For $1<p,q<\infty$ with $\frac{1}{p}+\frac{1}{q}=1$
the spaces $\cV_{\pot}^{p}(\Omega)$ and $L_{\sol}^{p}(\Omega)$ are
closed and it holds
\begin{equation}
\left(\cV_{\pot}^{p}(\Omega)\right)^{\bot}=L_{\sol}^{q}(\Omega)\,,\qquad\left(L_{\sol}^{p}(\Omega)\right)^{\bot}=\cV_{\pot}^{q}(\Omega)\label{eq:ortho-relation-pot-sol-omega}
\end{equation}
in the sense of duality. Furthermore, $W^{1,p}(\Omega)$ lies densely
in $L^{p}(\Omega)$ and 
\begin{equation}
\cV_{\pot}^{p}(\Omega)=\mathrm{closure}_{L^{p}}\left\{ \rmD u\,|\;u\in W^{1,p}(\Omega)\right\} \,.\label{eq:sto-Lp-pot-alternative-def}
\end{equation}
\end{thm}

\begin{proof}
The density of $W^{1,p}(\Omega)$ in $L^{p}(\Omega)$ follows from
Lemma \ref{lem:Properties-sI-delta}. We furthermore observe that
$\cV_{\pot}^{p}(\Omega)$ is invariant with respect to $\sI_{\delta}$.
In fact, let $u\in\cV_{\pot}^{p}(\Omega)$ and consider $\omega\in\Omega$
such that $u_{\omega}\in L_{\pot,\loc}^{p}(\Rd)$. Then 
\[
\left(\sI_{\delta}u\right)_{\omega}(x)=\int_{\Rd}\eta_{\delta}(y)u\left(\tau_{x+y}\omega\right)\d y
\]
and hence $\left(\sI_{\delta}u\right)_{\omega}\in L_{\pot,\loc}^{p}(\Rd)$
due to Remark \ref{rem:Lppot-invariant}. Furthermore, the space $L_{\sol}^{p}(\Omega)$
is closed as can be seen from the continuity of the expression
\[
L^{p}(\Omega;\Rd)\to\R\,,\qquad u\mapsto\int_{\Omega}g(\omega)\int_{(-1,1)^{d}}u(\tau_{x}\omega)\cdot\nabla\varphi(x)\,\d x\,\d\P(\omega)\,,
\]
where $\varphi\in W^{1,q}(\Rd)$ and $g\in L^{q}(\Omega)$ are arbitrary.

It remains to show (\ref{eq:ortho-relation-pot-sol-omega}), (\ref{eq:sto-Lp-pot-alternative-def})
and closedness of $\cV_{\pot}^{p}(\Omega)$.

Step 1: We first show that $\cV_{\pot}^{p}(\Omega)$ and $L_{\sol}^{q}(\Omega)$
are mutually orthogonal in the sense of duality. Let $\v\in\cV_{\pot}^{p}(\Omega)$
and $p\in L_{{\rm sol}}^{q}(\Omega)$ and chose $\omega\in\Omega$
such that for $\v^{\eps}(x)=\v(\tau_{\frac{x}{\eps}}\omega)$, $p^{\eps}(x)=p(\tau_{\frac{x}{\eps}}\omega)$
and $\v^{\eps}\cdot p^{\eps}$ the ergodic theorem \ref{thm:Main-ergodic Theorem - extended}
holds. Thus, we get $\v^{\eps}\cdot p^{\eps}\weakto\E(v\cdot p|\sI)$
weakly in $L_{\loc}^{1}(\Rd)$. It remains to show that $\v^{\eps}\cdot p^{\eps}\weakto^{\ast}0$.
Since $\v\in L_{\pot}^{p}(\Omega)$, we find for every $\eps>0$ some
$\ue\in W^{1,p}(\bQ)$ such that $\nabla\ue=\v^{\eps}$ and $\int_{\bQ}\ue=0$.
By the ergodic theorem $\nabla\ue=\v^{\eps}\weakto^{\ast}\E(\v|\sI)=0$
and $\ue\weakto u$ has average $0$. Due to the Poincaré inequality
and the compact embedding $W^{1,p}(\bQ)\hookrightarrow L^{p}(\bQ)$,
we find $\ue\to0$ strongly in $L^{p}(\bQ)$. Therefore, for all $\psi\in C_{c}^{\infty}(\bQ)$,
we find 
\[
\int_{\bQ}\psi\ve\cdot p^{\eps}\,dx=\int_{\bQ}\psi p^{\eps}\cdot\nabla\ue\,dx=-\int_{\bQ}\ue p^{\eps}\cdot\nabla\psi\,dx\to0\quad\mbox{for }\eps\to0\,.
\]
Therefore, we obtain 
\begin{equation}
L_{{\rm sol}}^{q}(\Omega)\subset\left(\cV_{\pot}^{p}(\Omega)\right)^{\bot}\quad\mbox{and}\quad\cV_{\pot}^{p}(\Omega)\subset\left(L_{{\rm sol}}^{q}(\Omega)\right)^{\bot}\,.\label{eq:Lpot-sol-orth-proof-1}
\end{equation}

Step 2: We prove (\ref{eq:ortho-relation-pot-sol-omega}) and closedness
of $\cV_{\pot}^{p}(\Omega)$ in case $p=2$. From Step 1 we know that
$L_{{\rm sol}}^{2}(\Omega)\subset\left(\cV_{\pot}^{2}(\Omega)\right)^{\bot}$
and it remains to show that $\left(\cV_{\pot}^{2}(\Omega)\right)^{\bot}\subseteq L_{\sol}^{2}(\Omega)$.
Let $u\in L^{2}(\Omega;\Rd)$ and use the decomposition $u=u_{\pot}+\tilde{u}$
where $u_{\pot}\in\cV_{\pot}^{2}(\Omega)$ and $\tilde{u}\in\left(\cV_{\pot}^{2}(\Omega)\right)^{\bot}$.
Since $\sI_{\delta}$ is symmetric and $\cV_{\pot}^{2}(\Omega)$ is
invariant with respect to $\sI_{\delta}$, we observe that 
\[
\forall v\in\cV_{\pot}^{2}(\Omega)\,:\quad\left\langle \sI_{\delta}\tilde{u},\,v\right\rangle =\left\langle \tilde{u},\,\sI_{\delta}v\right\rangle =0
\]
and hence $\sI_{\delta}\tilde{u}\in\left(\cV_{\pot}^{2}(\Omega)\right)^{\bot}$.
In particular, for every $\eps>0$ and every $\phi\in L^{2}(\Omega)$
it holds
\[
0=\left\langle \sI_{\delta}\tilde{u},\,\rmD_{\omega}\sI_{\eps}\phi\right\rangle =-\left\langle \diver_{\omega}\sI_{\delta}\tilde{u},\,\sI_{\eps}\phi\right\rangle 
\]
and as $\eps\to0$ it holds
\[
0=-\left\langle \diver_{\omega}\sI_{\delta}\tilde{u},\,\phi\right\rangle \,.
\]
Since $\phi\in L^{2}(\Omega)$ was arbitrary, this implies $\sum_{i}\rmD_{i}\sI_{\delta}\tilde{u}=0$
almost everywhere, i.e. $\sI_{\delta}\tilde{u}\in L_{\sol}^{2}(\Omega)$.
Since $\sI_{\delta}\tilde{u}\to\tilde{u}$ as $\delta\to0$, the closedness
of $L_{\sol}^{2}(\Omega)$ implies $\tilde{u}\in L_{\sol}^{2}(\Omega)$.
Hence $L_{\sol}^{2}(\Omega)\supset\cV_{\pot}^{2}(\Omega)^{\bot}$
and Step 1 implies $L_{\sol}^{2}(\Omega)=\cV_{\pot}^{2}(\Omega)^{\bot}$
and closedness of $\cV_{\pot}^{2}(\Omega)$. 

Step 3: For $p\in[1,2]$ we deduce from Step 2
\begin{equation}
\left(\cV_{\pot}^{p}(\Omega)\right)^{\bot}\subseteq L^{q}(\Omega;\Rd)\cap\left(\cV_{\pot}^{2}(\Omega)\right)^{\bot}=L^{q}(\Omega;\Rd)\cap L_{\sol}^{2}(\Omega)\subseteq L_{\sol}^{q}(\Omega)\,.\label{eq:Lpot-sol-orth-proof-2}
\end{equation}
Interchanging the role of $\cV_{\pot}$ and $L_{\sol}$ yields 
\begin{equation}
\left(L_{\sol}^{p}(\Omega)\right)^{\bot}\subseteq\cV_{\pot}^{q}(\Omega)\,.\label{eq:Lpot-sol-orth-proof-3}
\end{equation}
Inclusions (\ref{eq:Lpot-sol-orth-proof-1}), (\ref{eq:Lpot-sol-orth-proof-2})
and (\ref{eq:Lpot-sol-orth-proof-3}) imply (\ref{eq:ortho-relation-pot-sol-omega}).

Step 4: For $1<p<\infty$ we denote 
\[
V:=\left\{ \rmD\phi\,|\;\phi\in W^{1,p}(\Omega)\right\} \subset L_{\pot}^{p}(\Omega)\,.
\]
Let $u\in L^{q}(\Omega;\Rd)$ satisfy 
\[
\forall\phi\in W^{1,p}(\Omega)\,:\qquad\left\langle u,\,\rmD_{\omega}\phi\right\rangle =0\,.
\]
According to Lemma \ref{lem:Properties-sI-delta}, $\rmD_{i}$ and
$\sI_{\delta}$ commute for $\phi\in W^{1,p}(\Omega)$. Furthermore,
$\sI_{\delta}\phi\in W^{1,p}(\Omega)$ and hence 
\[
0=\left\langle u,\,\rmD_{\omega}\sI_{\delta}\phi\right\rangle =\left\langle u,\,\sI_{\delta}\rmD_{\omega}\phi\right\rangle =-\left\langle \diver_{\omega}\sI_{\delta}u,\,\phi\right\rangle \,.
\]
Since $\phi\in W^{1,p}(\Omega)$ was arbitrary and $W^{1,p}(\Omega)$
is dense in $L^{p}(\Omega$), it follows $\diver_{\omega}\sI_{\delta}u=0$,
which implies $u\in L_{\sol}^{q}(\Omega)$ by closedness of $L_{\sol}^{q}(\Omega)$.
To conclude, we have shown $L_{\sol}^{q}(\Omega)=\left(\cV_{\pot}^{p}(\Omega)\right)^{\bot}\subseteq V^{\bot}\subseteq L_{\sol}^{q}(\Omega)$,
and hence (\ref{eq:sto-Lp-pot-alternative-def}).
\end{proof}

\subsection{Stampaccia's Lemma}
\begin{lem}[Stampaccia]
\label{lem:Stampaccia}Let $G:\,\R\to\R$ be Lipschitz continuous
and let $u\in W^{1,p}(\Omega)$. Then $G\circ u\in W^{1,p}(\Omega)$.
\end{lem}

\begin{proof}
Let $u\in W^{1,p}(\Omega)$. It holds 
\begin{align*}
\limsup_{h\to0}\norm{\frac{T_{h\be_{i}}G(u)-G(u)}{h}}_{L^{p}(\Omega)} & =\limsup_{h\to0}\norm{\frac{G\left(T_{h\be_{i}}u\right)-G(u)}{h}}_{L^{p}(\Omega)}\\
 & \leq\limsup_{h\to0}\norm{\frac{G\left(T_{h\be_{i}}u\right)-G\left(u\right)}{T_{h\be_{i}}u-u}}_{\infty}\norm{\frac{T_{h\be_{i}}u-u}{h}}_{L^{p}(\Omega)}\\
 & \leq\norm{G'}_{\infty}\norm{\rmD_{i}u}_{L^{p}(\Omega)}\,.
\end{align*}
Hence, we find that there exists $w_{i}\in L^{p}(\Omega)$ such that
$\frac{1}{h}\left(T_{h\be_{i}}G(u)-G(u)\right)\weakto w_{i}$ weakly
along a further subsequence. Testing this limit with a function $\varphi\in W^{1,q}(\Omega)$,
we obtain that $w=(w_{i})_{i=1\dots d}$ is the weak derivative of
$G(u)$ as
\begin{align*}
\int_{\Omega}w_{i}\varphi\d\P & =\lim_{h\to0}\int_{\Omega}\frac{1}{h}\left(T_{h\be_{i}}G(u)-G(u)\right)\varphi\d\P\\
 & =-\lim_{h\to0}\int_{\Omega}\frac{1}{h}\left(T_{h\be_{i}}\varphi-\varphi\right)G(u)\d\P=-\int_{\Omega}G(u)\rmD_{i}\varphi\d\P\,.
\end{align*}
 
\end{proof}
\begin{rem}
Lemma \ref{lem:Stampaccia} is well known in Sobolev theory in $\Rd$
and is due to Stampaccia. It can be found for example in the book
by Evans \cite{Evans2010}. Stampaccia \cite{stampaccia1964equations}
also showed for functions $u\in W^{1,p}(\Rd)$ that $\nabla\left(G\circ u\right)=G'(u)\nabla u$.
However, to proof such a result in the case of general $\Omega$ goes
beyond the scope of this chapter.
\end{rem}

\begin{thm}
For every $1\leq p<\infty$ the embedding $W^{1,\infty}(\Omega)\hookrightarrow W^{1,p}(\Omega)$
is dense. In particular, 
\[
\cV_{\pot}^{p}(\Omega)=\mathrm{closure}_{L^{p}}\left\{ \rmD u\,|\;u\in W^{1,\infty}(\Omega)\right\} \,.
\]
\end{thm}

\begin{proof}
Let $u\in W^{1,p}(\Omega)$ and let $k\in\N$. By Lemma \ref{lem:Stampaccia}
we obtain that the function $u_{k}:=\max\left\{ -k,\min\left\{ k,u\right\} \right\} $
satisfies $u_{k}\in W^{1,p}(\Omega)$ and $\norm{u_{k}}_{\infty}\leq k$.
Since $u_{k}\to u$ as $k\to\infty$, it remains to show that $u_{k}$
can be approximated by functions in $W^{1,\infty}(\Omega)$. To see
this, note that for $u_{k}^{\delta}:=\sI_{\delta}u_{k}$ it holds
\begin{align*}
\lim_{t\to0}\frac{1}{t}\left(\sI_{\delta}u_{k}(\tau_{t\be_{i}}\omega)-\sI_{\delta}u_{k}(\omega)\right) & =\lim_{t\to0}\int_{\Rd}\frac{1}{t}\left(\eta_{\delta}(x+t\be_{i})-\eta_{\delta}(x)\right)u_{k}(\tau_{x}\omega)\\
 & =\int_{\Rd}u_{k}(\tau_{x}\omega)\partial_{i}\eta_{\delta}(x)\,.
\end{align*}
and since $\eta_{\delta}\in C_{c}^{\infty}(\Rd)$ we find $u_{k}^{\delta}\in W^{1,\infty}(\Omega)$.
Since $u_{k}^{\delta}\to u_{k}$ in $W^{1,p}(\Omega)$ as $\delta\to0$
by Lemma \ref{lem:Properties-sI-delta}, the theorem is proved.
\end{proof}
The last Theorem has an important implication for the existence of
suitable countable and dense family of functions.
\begin{thm}
\label{thm:dense-functions-Omega} Let Assumption \ref{assu:Omega-mu-tau-fundmental}
hold. For every $1\leq p<\infty$ there exists a countable dense family
of functions $\left(u_{k}\right)_{k\in\N}\subset W^{1,p}(\Omega)$
such that $\left(u_{k}\right)_{k\in\N}\subset W^{1,\infty}(\Omega)$
and $\left(u_{k}\right)_{k\in\N}$ is stable under addition and scalar
multiplication with $q\in\Q$. Furthermore, every $u_{k}$ has almost
surely bounded and continuously differentiable realizations with $\norm{u_{k,\omega}}_{W^{1,\infty}(\Rd)}\leq\norm{u_{k}}_{W^{1,\infty}(\Omega)}$.
If additionally Assumption \ref{assu:Omega-mu-tau-fundmental-strong}
holds, then $\left(u_{k}\right)_{k\in\N}$ can be chosen such that
for every $k$ it holds $u_{k}\in C_{b}^{1}(\Omega)$, $\nabla_{\omega}u_{k}\in C_{b}(\Omega)$.
\end{thm}

\begin{proof}
Let $\left(v_{k}\right)_{k\in\N}\subset W^{1,p}(\Omega)$ be dense.
Then for every $k$ consider $v_{k,n}:=\max\left\{ -n,\min\left\{ n,v_{k}\right\} \right\} $
and for $m\in\N$ define $v_{k,n,m}:=\sI_{\frac{1}{m}}v_{k,n}=\eta_{\frac{1}{m}}\ast v_{k,n}$.
Then $\norm{\rmD_{i}v_{k,n,m}}_{\infty}\leq\norm{\partial_{i}\eta_{\frac{1}{m}}}_{\infty}\norm{v_{k,n}}_{\infty}$.
Moreover, for every $\eps>0$ and every $\phi\in W^{1,p}(\Omega)$
there exists $k$ with $\norm{v_{k}-\phi}_{W^{1,p}(\Omega)}\leq\frac{\eps}{3}$,
$n$ with $\norm{v_{k}-v_{k,n}}_{W^{1,p}(\Omega)}\leq\frac{\eps}{3}$
and $m$ with $\norm{v_{k,n}-v_{k,n,m}}_{W^{1,p}(\Omega)}\leq\frac{\eps}{3}$.
Based on the countable family $\left(v_{k,n,m}\right)_{k,n,m\in\N}$,
we find that 
\[
\left(u_{k}\right)_{k\in\N}:=\left\{ \sum_{k,n,m=1}^{N}\lambda_{k,n,m}v_{k,n,m}:\;\lambda_{k,n,m}\in\Q,\,N\in\N\right\} 
\]
satisfies all demanded properties.

If Assumption \ref{assu:Omega-mu-tau-fundmental-strong} holds we
find $\left(c_{l}\right)_{l\in\N}\subset C_{b}(\Omega)\cap L^{p}(\Omega)$
dense in $L^{p}(\Omega)$. For every $v_{k}$ like above and every
$\delta>0$ we observe by Lemma \ref{lem:Omega-convolution} that
\begin{align*}
\norm{\eta_{\delta}\ast(v_{k}-c_{l})}_{L^{p}(\Omega)}&\leq\norm{v_{k}-c_{l}}_{L^{p}(\Omega)}\,,\\
\norm{\rmD_{i}\of{\eta_{\delta}\ast(v_{k}-c_{l})}}_{L^{p}(\Omega)}&\leq\norm{\partial_{i}\eta_{\delta}}_{L^{1}(\Rd)}\norm{v_{k}-c_{l}}_{L^{p}(\Omega)}\,.
\end{align*}
Hence the family $\left(c_{l,j}\right)_{l,j\in\N}:=\left(\eta_{\frac{1}{j}}\ast c_{l}\right)_{l,j\in\N}$
is countable and dense $W^{1,p}(\Omega)$. From here we can proceed
similarly with the modification that $c_{l,j}$ are already in $W^{1,\infty}(\Omega)$.
Based on the countable family $\left(c_{l,j}\right)_{l,j\in\N}$,
we find that 
\[
\left(u_{k}\right)_{k\in\N}:=\left\{ \sum_{l,j=1}^{N}\lambda_{l,j}c_{l,j}:\;\lambda_{l,j}\in\Q,\,N\in\N\right\} 
\]
satisfies all demanded properties.

The bound $\norm{u_{k,\omega}}_{W^{1,\infty}(\Rd)}\leq\norm{u_{k}}_{W^{1,\infty}(\Omega)}$
and continuous differentiability of realizations are a direct consequence
of the construction of $u_{k}$.
\end{proof}

\subsection{Traces and Extensions}

For the remainder of this section, we make the following assumption.
\begin{assumption}
\label{assu:General-Assump-Tr-Ext}Under the Assumption \ref{assu:Omega-mu-tau-fundmental}
let $\bP(\omega)$ be a random open set with boundary $\Gamma(\omega):=\partial\bP(\omega)$
such that $\Gamma(\omega)$ is a random closed set. The corresponding
prototypes $\bP,\Gamma\subset\Omega$ in the sense of Theorem \ref{thm:Main-THM-Sto-Geo}
have Palm measures $\chi_{\bP}\P$ and $\mugammapalm$ respectively.
\end{assumption}

We then introduce the following function spaces.
\begin{defn}
\label{def:W1p-bP}Under the Assumption \ref{assu:General-Assump-Tr-Ext}
we introduce for $1\leq p\leq\infty$ the space
\begin{align*}
W^{1,p}(\bP) & :=\left\{ u\in L^{p}(\bP;\P):\;\text{for a.e. }\omega\text{ holds }u_{\omega}\in W_{\loc}^{1,p}(\bP(\omega))\text{ and }\right.\\
 & \phantom{:)=}\left.\qquad\text{there exists }\rmD u\in L^{p}(\bP)^{d}\text{ s.t. for a.e. }\omega:\;\nabla u_{\omega}=\left(\rmD u\right)_{\omega}\right\} \,,\\
\norm u_{W^{1,p}(\bP)} & :=\norm u_{L^{p}(\bP)}+\norm{\rmD u}_{L^{p}(\bP)}\,.
\end{align*}
\end{defn}

Based on Definition \ref{def:W1p-bP}, we also introduce the following
properties of $\bP$ and $\Gamma$.
\begin{defn}
We say for the corresponding prototypes $\bP,\Gamma\subset\Omega$
in the sense of Theorem \ref{thm:Main-THM-Sto-Geo} that
\begin{enumerate}
\item $\bP$ has the weak $(r,p)$-extension property for $1\leq r\leq p$
if  Assumption \ref{assu:Omega-mu-tau-fundmental} holds and there
exists a continuous linear operator $\cU_{\Omega}:W^{1,p}(\bP)\to W^{1,r}(\Omega)$
such that $\of{\cU_{\Omega}u}|_{\bP}=u$. In this context, we define
\begin{align*}
W^{1,r,p}(\Omega,\bP) & :=\left\{ u\in W^{1,r}(\Omega):\;u|_{\bP}\in L^{p}(\bP),\,\rmD_{\omega}u\in L^{p}(\bP;\Rd)\right\} \,,\\
\cV_{\pot}^{p}(\bP) & :=\mathrm{closure}_{L^{p}}\left\{ \rmD u\,|\;u\in W^{1,p}(\bP)\right\} \,,\\
\cV_{\pot}^{r,p}(\bP) & :=\left\{ \rmD u\in\cV_{\pot}^{r}(\Omega)\,|\;\rmD u\in\cV_{\pot}^{p}(\bP)\right\} \,.
\end{align*}
\item $\bP$ has the strong $(r,p)$-extension property for $1\leq r\leq p$
if  Assumption \ref{assu:Omega-mu-tau-fundmental} holds and there
exists a continuous linear operator $\cU_{\Omega}:W^{1,p}(\bP)\to W^{1,r}(\Omega)$
such that $\of{\cU_{\Omega}u}|_{\bP}=u$ and such that 
\[
\norm{\rmD_{\omega}\cU_{\Omega}u}_{L^{r}(\Omega)}\leq C\norm{\rmD_{\omega}u}_{L^{p}(\Omega)}\,.
\]
\item $\Gamma$ has the strong $(r,p)$-trace property for $1\leq r\leq p$
if  Assumption \ref{assu:Omega-mu-tau-fundmental-strong} holds and
there exists a continuous linear operator $\cT_{\Omega}:W^{1,1,p}(\Omega)\to L^{r}(\Gamma;\mugammapalm)$
such that for every $u\in C_{b}(\Omega)$ it holds $\cT_{\Omega}u=u|_{\Gamma}$
in the sense of $\mugammapalm$.
\end{enumerate}
\end{defn}

We already mention at this point a very important property which holds
for $\bP=\Omega$, but which we are not able to reproduce for general
$\bP$ in this work. Hence we formulate it as a conjecture, and will
avoid to use it in the remainder of this work. Fortunately, it turns
out to be non-essential up to uniqueness properties of the homogenized
problem in Section \ref{sec:Homogenization-of-p-Laplace}.
\begin{conjecture}
\label{conj:decomposition}If $\bP$ has the strong extension property
it holds 
\[
\Rd\cap\cV_{\pot}^{p}(\bP)=\emptyset\,.
\]
\end{conjecture}

\begin{thm}
\label{thm:weak-trace}Let Assumptions \ref{assu:Omega-mu-tau-fundmental-strong}
and \ref{assu:Trace} hold for the random open set $\bP(\omega)$
with $1\leq r<p$ and let $\tau$ be ergodic. Then $\Gamma$ has the
strong $(r,p)$-trace property.
\end{thm}

In order to prove Theorem \ref{thm:weak-trace} we first need the
following lemma.
\begin{lem}
\label{lem:dense-family-w1rp}Let Assumption \ref{assu:Omega-mu-tau-fundmental}
hold and let $1\leq r<p$, then there exists a family $\left(u_{k}\right)_{k\in\N}\subset W^{1,\infty}(\Omega)$
which is dense in $W^{1,r,p}(\Omega,\bP)$. If Assumption \ref{assu:Omega-mu-tau-fundmental-strong}
holds then we can additionally assume $\left(u_{k}\right)_{k\in\N}\subset W^{1,\infty}(\Omega)\cap C_{b}^{1}(\Omega)$.
In both cases $\left(u_{k}\right)_{k\in\N}$ is stable under addition
and scalar multiplication with $q\in\Q$. Furthermore, every $u_{k}$
has almost surely bounded and continuously differentiable realizations
with $\norm{u_{k,\omega}}_{W^{1,\infty}(\Rd)}\leq\norm{u_{k}}_{W^{1,\infty}(\Omega)}$
\end{lem}

\begin{proof}
By Theorem \ref{thm:dense-functions-Omega} there exists $\left(u_{k}\right)_{k\in\N}\subset W^{1,\infty}(\Omega)$
which is at the same time dense in $W^{1,r}(\Omega)$ and $W^{1,p}(\Omega)$.
The statement now follows from $W^{1,r}(\Omega)\supset W^{1,r,p}(\Omega,\bP)\supset W^{1,p}(\Omega)$.
If Assumption \ref{assu:Omega-mu-tau-fundmental-strong} holds Theorem
\ref{thm:dense-functions-Omega} yields $\left(u_{k}\right)_{k\in\N}\subset W^{1,\infty}(\Omega)\cap C_{b}^{1}(\Omega)$.
\end{proof}

\begin{proof}[Proof of Theorem \ref{thm:weak-trace}]
 Let $\left(u_{k}\right)_{k\in\N}\subset W^{1,\infty}(\Omega)\cap C_{b}^{1}(\Omega)$
be dense in $W^{1,1,p}(\Omega,\bP)$ according to Theorem \ref{thm:dense-functions-Omega}.
For each $u\in\left(u_{k}\right)_{k\in\N}$ the function $u|_{\Gamma}$
is well defined. Writing $\bQ_{n}:=\left[-n,n\right]^{d}$ and using
Theorem \ref{thm:uniform-trace-estimate-1} as well as the Ergodic
Theorems we find
\begin{align*}
\int_{\Gamma}\left|u\right|^{r}\d\mugammapalm & =\frac{1}{(2n)^{d}}\int_{\bQ_{n}}\int_{\Gamma}\left|u\right|^{r}\d\mugammapalm=\frac{1}{\left|\bQ_{n}\right|}\E\int_{\bQ_{n}\cap\partial\bP(\omega)}\left|\cT u_{\omega}\right|^{r}\\
 & \leq\E\of{C_{\omega}\left(\frac{1}{\left|\bQ_{n}\right|}\int_{\bQ_{n+1}\cap\bP(\omega)}\left|u_{\omega}\right|^{p}+\left|\nabla u_{\omega}\right|^{p}\right)^{\frac{r}{p}}}\\
 & \leq\E\of{C_{\omega}^{\frac{p}{p-r}}}^{\frac{p-r}{p}}\,\E\of{\frac{1}{\left|\bQ_{n}\right|}\int_{\bQ_{n+1}\cap\bP(\omega)}\left|u_{\omega}\right|^{p}+\left|\nabla u_{\omega}\right|^{p}}^{\frac{r}{p}}\\
 & \to\E\of{C_{\omega}^{\frac{p}{p-r}}}^{\frac{p-r}{p}}\,\left(\int_{\Omega}\left|u\right|^{p}+\left|\nabla_{\omega}u\right|^{p}\d\P\right)^{\frac{r}{p}}
\end{align*}
as $n\to\infty$. Using the definition of $C_{\omega}$ in Theorem
\ref{thm:uniform-trace-estimate-1} we conclude.
\end{proof}
A generalization of Theorem \ref{thm:weak-trace} to the general case
of Assumption \ref{assu:Omega-mu-tau-fundmental} is difficult, since
the trace property does not apply for general $L^{\infty}$-functions,
even in $\Rd$. However, for the sake of homogenization, there exists
a workaround.
\begin{defn}
\label{def:weak-trace-property}We say for the corresponding prototypes
$\bP,\Gamma\subset\Omega$ in the sense of Theorem \ref{thm:Main-THM-Sto-Geo}
that $\Gamma$ has the weak $(r,p)$-trace property for $1\leq r\leq p$
if  Assumption \ref{assu:Omega-mu-tau-fundmental} holds and for
every family of functions $\left(u_{k}\right)_{k\in\N}\subset W^{1,\infty}(\Omega)$
according to Lemma \ref{lem:dense-family-w1rp} which is dense in
$W^{1,1,p}(\Omega,\bP)$ there exists a continuous linear operator
$\cT_{\Omega}:\,W^{1,1,p}(\Omega,\bP)\to L^{r}(\Gamma;\mugammapalm)$
such that for almost every $\omega\in\Omega$ and every $u_{k}$ it
holds $\left(\cT_{\Omega}u_{k}\right)_{\omega}=\cT u_{k,\omega}$
on $\Gamma(\omega)$.
\end{defn}

\begin{thm}
\label{thm:weak-trace-general}Let Assumption \ref{assu:Omega-mu-tau-fundmental}
hold, let $\tau$ be ergodic and let $\Gamma(\omega)$ be almost surely
locally $\left(\delta,M\right)$-regular satisfying Assumption \ref{assu:Trace}.
Then $\Gamma$ has the weak $(r,p)$-trace property.
\end{thm}

\begin{proof}
We define $\cT_{\Omega}u_{k}$ pointwise in $\omega$ through $\left(\cT_{\Omega}u_{k}\right)_{\omega}=\cT u_{k,\omega}$
and observe that $\cT_{\Omega}$ is bounded by the argument in the
proof of Theorem \ref{thm:weak-trace}. It thus remains to show that
$\cT_{\Omega}u_{k}$ is measurable, because then, we can simply extend
$\cT_{\Omega}$ to $W^{1,1,p}(\Omega,\bP)$.

We use Lemma \ref{lem:contin-geom-Ops} and obtain that $\Gamma_{\delta}(\omega):=\Ball{\delta}{\Gamma(\omega)}$
is a RACS with prototype $\Gamma_{\delta}$ due to Theorem \ref{thm:Main-THM-Sto-Geo}.
We observe that $\Gamma=\bigcap_{\delta}\Gamma_{\delta}$ as well
as (by definition) $\cT_{\Omega}u_{k}=\inf_{\delta}\chi_{\Gamma_{\delta}}u_{k}$,
hence $\cT_{\Omega}u_{k}$ is measurable.
\end{proof}
We will now turn our focus to the extension properties. We start with
an important implication by the strong extension property.
\begin{thm}
Let Assumption \ref{assu:Omega-mu-tau-fundmental} hold, let $\tau$
be ergodic and let $\bP$ have the strong $(r,p)$-extension property.
Then the operator $\cU_{\Omega}$ can be extended to a continuous
operator $\cU_{\Omega}:\,\cV_{\pot}^{p}(\bP)\to\cV_{\pot}^{r,p}(\Omega,\bP)$.
More precisely we can identify $\cV_{\pot}^{p}(\bP)$ with 
\begin{align*}
\tilde{\cV}_{\pot}^{p}(\bP) & =\mathrm{closure}_{L^{r,p}(\Omega,\P)}\left\{ \cU_{\Omega}\rmD_{\omega}u:\,u\in W^{1,p}(\Omega)\right\} \,,\\
 & =\mathrm{closure}_{L^{r,p}(\Omega,\P)}\left\{ \cU_{\Omega}\rmD_{\omega}u:\,u\in W^{1,r,p}(\Omega;\bP)\right\} \,,\\
\norm{\xi}_{L^{r,p}(\Omega,\P)} & =\norm{\xi}_{L^{r}(\Omega)}+\norm{\xi}_{L^{p}(\bP)}\,.
\end{align*}
This means that for $\phi\in\cV_{\pot}^{p}(\bP)$ and $\tilde{\phi}\in\tilde{\cV}_{\pot}^{p}(\bP)$
it holds $\tilde{\phi}|_{\bP}=\phi$ iff $\tilde{\phi}=\cU_{\Omega}\phi$.
\end{thm}

\begin{proof}
The first part follows immediately from the definition of the spaces
and of the strong extension property.

For the second part, remark that $\cU_{\Omega}W^{1,p}(\bP)\subset W^{1,r,p}(\Omega,\bP)$
and $(\cU_{\Omega}\phi)|_{\bP}=\phi$. Furthermore, $W^{1,p}(\Omega)$
is dense in $W^{1,r,p}(\Omega,\bP)$ by Lemma \ref{lem:dense-family-w1rp}.
Finally, $\cU_{\Omega}\cU_{\Omega}\phi=\cU_{\Omega}\phi$ and for
$\phi\in\cV_{\pot}^{p}(\bP)$ and $\tilde{\phi}\in\tilde{\cV}_{\pot}^{p}(\bP)$
it holds $\tilde{\phi}|_{\bP}=\phi$ iff $\tilde{\phi}=\cU_{\Omega}\phi$.
\end{proof}
\begin{thm}
\label{thm:weak-extension-property}Let Assumption \ref{assu:Omega-mu-tau-fundmental}
hold, let $\tau$ be ergodic and let $\Gamma(\omega)$ be almost surely
locally $\left(\delta,M\right)$-regular satisfying Assumption \ref{assu:local-extension}
for $1<r<p_{0}<p_{1}<p$. Then $\Gamma$ has the weak $(r,p)$-extension
property.
\end{thm}

\begin{thm}
\label{thm:strong-extension-property}Let Assumption \ref{assu:Omega-mu-tau-fundmental}
hold, let $\tau$ be ergodic and let $\Gamma(\omega)$ be almost surely
locally $\left(\delta,M\right)$-regular satisfying Assumption \ref{assu:dirichlet-extension}
for $1<r<p_{0}<p_{1}<p$. Then $\Gamma$ has the strong $(r,p)$-extension
property.
\end{thm}

We will prove Theorems \ref{thm:weak-extension-property} and \ref{thm:strong-extension-property}
in Section \ref{subsec:Domains-with-holes} using homogenization theory.

\subsection{The Outer Normal Field of $\protect\bP$}
\begin{thm}
\label{thm:normal-field}Let Assumptions \ref{assu:Omega-mu-tau-fundmental-strong}
and \ref{assu:General-Assump-Tr-Ext} hold and let $\Gamma$ have
the strong $(r,p)$-trace property for $1<r<p$. Let $\tau$ be ergodic,
let $\Gamma(\omega)$ be almost surely locally $\left(\delta,M\right)$-regular
and let $\nu_{\Gamma(\omega)}$ be the outer normal of $\bP(\omega)$
on $\Gamma(\omega)$. Then there exists a measurable function $\nu_{\Gamma}:\,\Gamma\to\S^{d-1}$
such that almost surely $\nu_{\Gamma(\omega)}(x)=\nu_{\Gamma}(\tau_{x}\omega)$.
Furthermore, for $f\in C_{b}^{1}(\Omega;\Rd)$ and $\phi\in C_{b}^{1}(\Omega)$
it holds 
\begin{equation}
\int_{\bP}\diver_{\omega}(f\phi)\,\d\P=\int_{\Gamma}\phi f\cdot\nu_{\Gamma}\,\d\mugammapalm\,.\label{eq:thm:normal-field}
\end{equation}
If $\Gamma$ satisfies the weak $(1,p)$-extension property, the equation
(\ref{eq:thm:normal-field}) extends to $\phi\in W^{1,1,p}(\Omega,\bP)$
and $f\in C_{b}^{1}(\Omega;\Rd)$ or to $f\in W^{1,1,p}(\Omega,\bP)^{d}$
and $\phi\in C_{b}^{1}(\Omega)$.
\end{thm}

\begin{proof}
For $\delta>0$ define $\chi_{\delta}(\omega):=\left(\eta_{\delta}\ast\chi_{\bP}\right)(\omega)$.
We observe that 
\begin{equation}
\left|\rmD_{\omega}\chi_{\delta}\right|(\tau_{x}\omega)=\left|\rmD_{\omega}(\eta_{\delta}\ast\chi_{\bP})\right|(\tau_{x}\omega)=\left|\eta_{\delta}\ast(\rmD_{\omega}\chi_{\bP})(\tau_{\cdot}\omega)\right|(x)=\left|\eta_{\delta}\ast\nabla\chi_{\bP(\omega)}\right|(x)\,,\label{eq:thm:normal-field-help-1}
\end{equation}
and hence for almost every $\omega$ we have $\left|\rmD_{\omega}\chi_{\delta}\right|\to\left|\nabla\chi_{\bP(\omega)}\right|=\hausdorffH^{d-1}(\Gamma(\omega)\cap\,\cdot\,)$
weakly. Then for $\varphi\in C_{c}^{\infty}(\Rd)$ and $f\in C_{b}(\Omega)$
it holds by the Palm formula and (\ref{eq:thm:normal-field-help-1})
\begin{align*}
\int_{\Rd}\varphi\int_{\Omega}f\left|\rmD_{\omega}\chi_{\delta}\right| & =\int_{\Omega}\int_{\Rd}f(\tau_{x}\omega)\varphi(x)\left|\rmD_{\omega}\chi_{\delta}\right|(\tau_{x}\omega)\,\d x\,\d\P(\omega)\\
 & =\int_{\Omega}\int_{\Rd}f(\tau_{x}\omega)\varphi(x)\left|\eta_{\delta}\ast\nabla\chi_{\bP(\omega)}\right|(x)\,\d x\,\d\P(\omega)\\
 & \leq\int_{\Omega}\int_{\Rd}f(\tau_{x}\omega)\varphi(x)\left(\left|\nabla\chi_{\bP(\omega)}\right|\of{\Ball{\delta}{\support\varphi}}\right)\,\d x\,\d\P(\omega)\,,
\end{align*}
where $\left|\nabla\chi_{\bP(\omega)}\right|=\hausdorffH^{d-1}(\,\cdot\,\cap\Gamma(\omega))=\mugammaomega$.
From the ergodic theorem, the $\P$-almost sure pointwise weak convergence
and the Lebesgue dominated convergence theorem, we conclude
\begin{align*}
\int_{\Rd}\varphi\int_{\Omega}f\left|\rmD_{\omega}\chi_{\delta}\right| & \to\int_{\Omega}\int_{\Rd}f(\tau_{x}\omega)\varphi(x)\,\d\mugammaomega(x)\,\d\P(\omega)\\
 & =\int_{\Rd}\varphi\int_{\Omega}f\d\mugammapalm\,,
\end{align*}
which implies $\int_{\Omega}f\left|\rmD_{\omega}\chi_{\delta}\right|\to\int_{\Omega}f\d\mugammapalm$.
In a similar way, we show $\int_{\Omega}f\rmD_{\omega}\chi_{\delta}\to\int_{\Omega}f\d\widetilde{\mugammapalm}$,
where $\widetilde{\mugammapalm}$ is a $\Rd$-valued measure on $\Gamma$.
Furthermore, for every $\e_{i}$ in the canonical basis of $\Rd$,
$\e_{i}\cdot\widetilde{\mugammapalm}\ll\mugammapalm$, which implies
by the Radon-Nikodym theorem the existence of a measurable $\nu_{\Gamma}$
with values in $\S^{d-1}$ such that $\widetilde{\mugammapalm}=\nu_{\Gamma}\mugammapalm$.
The property $\nu_{\Gamma(\omega)}(x)=\nu_{\Gamma}(\tau_{x}\omega)$
follows from the fact that $\widetilde{\mugammapalm}$ is the Palm
measure of $\nabla\chi_{\bP(\omega)}$.

For $f\in C_{b}^{1}(\Omega;\Rd)$ and $\phi\in C_{b}^{1}(\Omega,\bP)$
and $\varphi\in C_{c}^{\infty}(\Rd)$ it holds 
\begin{align*}
\int_{\Rd}\varphi\int_{\bP}\diver_{\omega}(f\phi) & =\int_{\Omega}\int_{\Rd}\varphi(x)\,\diver(f\phi)_{\omega}\\
 & =\int_{\Omega}\int_{\Gamma(\omega)}\varphi(x)\,\phi_{\omega}f_{\omega}\cdot\nu_{\Gamma(\omega)}\\
 & =\int_{\Rd}\int_{\Gamma}\varphi(x)\,\phi f\cdot\nu_{\Gamma}\d\mugammapalm\,,
\end{align*}
which implies (\ref{eq:thm:normal-field}) by a density argument.
\end{proof}
\begin{defn}
Let $\Gamma$ have the strong $(r,p)$-Trace property for $1<r<p$
and the weak $(1,p)$-extension property. We say that $f\in L^{p}(\bP;\Rd)$
has the weak normal trace $f_{\nu}\in L^{r}(\Gamma)$ and weak divergence
$\diver_{\omega}f\in L^{1}(\bP)$ if for all $\phi\in C_{b}^{1}(\Omega)$
\[
\int_{\bP}\left(\phi\diver_{\omega}f+f\cdot\nabla_{\omega}\phi\right)\,\d\P=\int_{\Gamma}\phi f_{\nu}\,\d\mugammapalm\,.
\]
\end{defn}

\begin{thm}
\label{thm:Harmonic-P}Let Assumptions \ref{assu:Omega-mu-tau-fundmental-strong}
and \ref{assu:General-Assump-Tr-Ext} hold and for some $r\in(1,2)$
let $\Gamma$ have the strong $(r,2)$-Trace property and the weak
$(r,2)$-extension property and let $\Gamma^{\eps}(\omega)$ have
the strong uniform trace property (see Definition \ref{def:r-p-Ex-Tr-prop-weak}
below). Let $\tau$ be ergodic, let $\Gamma(\omega)$ be almost surely
locally $\left(\delta,M\right)$-regular and let $\nu_{\Gamma(\omega)}$
be the outer normal of $\bP(\omega)$ on $\Gamma(\omega)$. Then there
exists $u_{\Omega}\in W^{1,r}(\Omega)\cap W^{1,2}(\bP;\Rd)$, such
that $\nabla_{\omega}u_{\Omega}$ has a weak normal trace $f_{\nu}\in L^{1}(\Gamma)$
and weak divergence $u_{\Omega}$, i.e.
\[
\forall\phi\in C_{b}^{1}(\omega):\quad\int_{\bP}\left(\phi u_{\Omega}+\nabla u_{\Omega}\cdot\nabla_{\omega}\phi\right)\,\d\P=\int_{\Gamma}\phi f_{\nu}\,\d\mugammapalm\,.
\]
\end{thm}

The last theorem is less trivial than one might think. In particular,
we lack a Poincaré-type inequality on $\Omega$, which is typically
used to prove corresponding results in $\Rd$. We shift the proof
to Section \ref{subsec:Domains-with-holes}.

\section{\label{sec:Two-scale-convergence}Two-Scale Convergence and Application}

As we have already explained in the introduction, there have been
several approaches to the introduction of two-scale convergence in
stochastic homogenization. In this work, we chose a modification of
\cite{heida2017stochastic} because it does not rely on compactness
of the underlying probability space.

\subsection{\label{subsec:General-Setting}General Setting}

For the rest of this work, we consider a stationary random measure
$\omega\to\mu_{\omega}$ with Palm measure $\mupalm$ and we define
\begin{equation}
\muomegaeps(A):=\eps^{d}\muomega\of{\eps^{-1}A}\,.\label{eq:ts-measure}
\end{equation}
For the corresponding Lebesgue spaces we write $L^{p}(\Omega;\mupalm)$
or $L^{p}(\bQ;\muomegaeps)$, where $\bQ\subset\Rd$ is a convex domain
with $C^{1}$-boundary. If $\muomega=\lebesgueL$, i.e. $\mupalm=\P$,
or $\mu_{\omega}=\chi_{\bP(\omega)}\lebesgueL$ we omit the notion
of $\muomegaeps$ and $\mupalm$.

In our applications, $\d\muomega=\chi_{\bP(\omega)}\d\lebesgueL$
for the characteristic function of the prototype $\bP\subset\Omega$
of the random set $\bP(\omega)$ with Palm measure $\chi_{\bP}\P$
or $\d\muomega=\d\mugammaomega:=\chi_{\Gamma(\omega)}\d\hausdorffH^{d-1}$,
with Palm measure located on $\Gamma\subset\Omega$, the prototype
of $\Gamma(\omega):=\partial\bP(\omega)$. If we explicitly study
the latter case, we write $\mugammapalm$ for the Palm measure.

Moreover, in view of (\ref{eq:ts-measure}), we write $\mugammaomegaeps(A):=\eps^{d}\mugammaomega\of{\eps^{-1}A}=\eps\hausdorffH^{d-1}\of{A\cap\eps\Gamma(\omega)}$.
In case of $\muomega=\chi_{\bP(\omega)}\lebesgueL$, we drop the notation
$\muomegaeps$.
\begin{assumption}
\label{assu:TS-general}Let $\left(\Omega,\sigma,\P\right)$ be a
probability space with ergodic dynamical system $\left(\tau_{x}\right)_{\in\Rd}$
in the sense of Definition \ref{Def:Omega-mu-tau}. Let $1<q,p<\infty$
with $\frac{1}{p}+\frac{1}{q}=1$ and 
\[
\Phi_{\cP,q}\subset L^{q}(\Omega;\mupalm)
\]
be a countable dense subset, which is stable under scalar multiplication
and linear combination. Finally, let $\Omega_{\Phi}$ be such that
(\ref{eq:ergodic convergence ran meas2}) holds for all $\varphi\in C(\overline{\bQ})$,
$\omega\in\Omega_{\Phi}$, $f\in\Phi_{\cP,q}$.
\end{assumption}

\begin{rem*}
In some proofs below we will assume w.l.o.g. that some particular,
essentially bounded functions are elements of $\Phi_{\cP,q}$. These
will always be countably many and hence $\Omega_{\Phi}$ has to be
changed only by a set of measure $0$.
\end{rem*}
\begin{defn}
\label{def:two-scale-conv} Let Assumption \ref{assu:TS-general}
hold. Let $\omega\in\Omega_{\Phi}$ and let $\ue\in L^{p}(\bQ;\muomegaeps)$
for all $\eps>0$. We say that $(\ue)$ converges (weakly) in two
scales to $u\in L^{p}(\bQ;L^{p}(\Omega;\mupalm))$ and write $\ue\stackrel{2s}{\weakto}u$
if $\sup_{\eps>0}\norm{\ue}_{L^{p}(\bQ;\muomegaeps)}<\infty$ and
if for every $\psi\in\Phi_{\cP,q}$, $\varphi\in C(\overline{\bQ})$
there holds with $\phi_{\omega,\eps}(x):=\varphi(x)\psi(\tau_{\frac{x}{\eps}}\omega)$
\[
\lim_{\eps\to0}\int_{\bQ}\ue(x)\phi_{\omega,\eps}(x)\d\muomegaeps(x)=\int_{\bQ}\int_{\Omega}u(x,\tilde{\omega})\varphi(x)\psi(\tilde{\omega})\,\d\mupalm(\tilde{\omega})\,\d x\,.
\]
\end{defn}

We note that the definition of two-scale convergence in \cite{heida2017stochastic}
is formulated more generally, in particular for a more general class
of test-functions.
\begin{lem}[\cite{heida2017stochastic} Lemma 4.4-1.]
\label{lem:Existence-ts-lim}Let Assumption \ref{assu:TS-general}
hold. Let $\omega\in\Omega$ and $\ue\in L^{p}(\bQ;\muomegaeps)$
be a sequence of functions such that $\norm{\ue}_{L^{p}(\bQ)}\leq C$
for some $C>0$ independent of $\eps$. Then there exists a subsequence
of $(u^{\eps'})_{\eps'\to0}$ and $u\in L^{p}(\bQ;L^{p}(\Omega;\mupalm))$
such that $u^{\eps'}\stackrel{2s}{\weakto}u$ and 
\begin{equation}
\norm u_{L^{p}(\bQ;L^{p}(\Omega;\mupalm))}\leq\liminf_{\eps'\to0}\norm{u^{\eps'}}_{L^{p}(\bQ;\muomegaeps)}\,.\label{eq:two-scale-limit-estimate}
\end{equation}
\end{lem}

\begin{proof}[Sketch of proof]
 The proof is standard and has been carried out in various publications
under various assumptions \cite{allaire1992homogenization,heida2011extension,heida2017stochastic,heida2019lambda,Zhikov2006}.
The important point is the separability of $C(\overline{\bQ})$, which
allows to pass to the limit for a countable number of test functions
$\left(\varphi_{k}\right)_{k\in\N}\in C(\overline{\bQ})$ first, and
then apply a density argument.
\end{proof}
Furthermore, we will need the following result on the lower estimate
in homogenization of convex functionals using two-scale convergence,
which was obtained in \cite{heidaNesen2017stochastic}.
\begin{lem}
\label{lem:General-Hom-Convex}Let Assumption \ref{assu:TS-general}
hold and let $\muomega$ be a random measure. Let $f:\,\bQ\times\Omega\times\R^{N}\to\R$
be a convex functional in $\Rd$. For almost all $\omega\in\Omega_{\Phi_{p}}$
the following holds: Let $\ue\in L^{q}(\bQ;\muomegaeps)$ be a sequence
such that $\norm{\ue}_{L^{q}(\bQ;\muomegaeps)}\leq C$ for some $0<C<\infty$
and such that $\ue\stackrel{{\scriptstyle 2s}}{\weakto}u\in L^{q}(\bQ\times\Omega;\lebesgueL\otimes\mupalm)$.
Then, it holds 
\[
\int_{\bQ}\int_{\Omega}f(x,\tilde{\omega},u(x,\tilde{\omega}))\,d\mupalm(\tilde{\omega})\,dx\leq\liminf_{\eps\to0}\int_{\bQ}f(x,\tau_{\frac{x}{\eps}}\omega,\ue(x))\,d\muomegaeps(x)\,.
\]
\end{lem}

\subsection{The ``Right'' Choice of Oscillating Test Functions}

In what follows, we will have to deal with two-scale limits of functions
on $\Rd$, but also on $\bP(\omega)$ or $\Gamma(\omega)$. Hence
we deal with two-scale convergence w.r.t. to $\P$, $\chi_{\bP}\P$
and $\mugammapalm$. In order to keep notation of the set(s) of testfunctions
short and concise, we make the following choice: 
\[
\Phi_{q}=\Phi_{\P,q}=\left(u_{k}\right)_{k\in\N}
\]
is the set countable set of functions $\left(u_{k}\right)_{k\in\N}\subset W^{1,p}(\Omega)\cap W^{1,\infty}(\Omega)$
from Theorem \ref{thm:dense-functions-Omega}. Hence $\left(u_{k}\right)_{k\in\N}$
is dense in $L^{p}(\Omega)$ and $\left(\nabla u_{k}\right)_{k\in\N}$
is dense in $\cV_{\pot}^{p}(\Omega)$ (see Theorem \ref{thm:Lp-W1p-structure}).

If $\Gamma$ has the strong or weak $(r,p)$-trace property, using
Theorem \ref{thm:weak-trace} and \ref{thm:weak-trace-general} we
define 
\[
\Phi_{r,\Gamma}=\cT_{\Omega}\Phi_{p}\cup\tilde{\Phi}_{r,\Gamma}\,,
\]
where $\tilde{\Phi}_{r,\Gamma}$ is dense in $L^{r}(\Gamma,\mugammapalm)$.
In case of Assumption \ref{assu:Omega-mu-tau-fundmental-strong},
we note that $\cT_{\Omega}\Phi_{p}$ is dense in $L^{r}(\Gamma,\mugammapalm)$
because $C_{b}(\omega)$ is dense in $L^{r}(\Gamma,\mugammapalm)$.
However, in case of Assumption \ref{assu:Omega-mu-tau-fundmental}
it is not clear that $\cT_{\Omega}\Phi_{p}$ is dense in $L^{r}(\Gamma,\mugammapalm)$,
which is why $\tilde{\Phi}_{r,\Gamma}$ is needed.

\subsection{Homogenization of Gradients}

In what follows, we introduce two-scale convergence of gradients.
This result has been proven in various work under various assumptions,
see e.g. \cite{allaire1992homogenization} for the periodic case and
\cite{Zhikov2006,neukammvarga2018stochastic,heida2017stochastic}
in the stochastic case. We provide the proof here for self-containedness
of this outline.
\begin{thm}
\label{thm:sto-conver-grad}Under Assumption \ref{assu:TS-general}
for almost every $\omega\in\Omega$ the following holds: 

If $u^{\eps}\in W^{1,p}(\bQ;\Rd)$
for all $\eps$ and if there exists $0<C_{u}<\infty$ with 
\[
\sup_{\eps>0}\norm{u^{\eps}}_{L^{p}(\bQ)}+\eps^{\gamma}\norm{\nabla u^{\eps}}_{L^{p}(\bQ)}<C_{u}
\]
Then there exists $u\in L^{p}(\bQ L^{p}(\Omega;\P))$ such that $\ue\tsweakto u$.
Depending on the choice of $\gamma$, the following holds: 

\begin{enumerate}
\item If $\gamma=0$, then $u\in W^{1,p}(\bQ)$ with $u^{\eps}\weakto u$
weakly in $W^{1,p}(\bQ)$ and there exists $\vel_{1}\in L^{p}(\bQ;\cV_{\pot}^{p}(\Omega))$
such that $\nabla\ue\tsweakto\nablax u+\vel_{1}$weakly in two scales.
\item If $\gamma\in(0,1)$ then $\eps^{\gamma}\nabla\ue\tsweakto\vel_{1}$
for some $\vel_{1}\in L^{p}(\bQ;\cV_{\pot}^{p}(\Omega))$.
\item If $\gamma=1$ then $u\in L^{p}(\bQ;W^{1,p}(\Omega))$ and $\eps\nabla\ue\tsweakto\rmD_{\omega}u$.
\item If $\gamma>1$ then $\eps^{\gamma}\nabla\ue\tsweakto0$.
\end{enumerate}
\end{thm}

\begin{lem}
\label{lem:per-ts-no-grad-omega}Under Assumption \ref{assu:TS-general}
for almost all $\omega\in\Omega$ the following holds: Let $p>1$
and $\left(\ue\right)_{\eps>0}$ be a sequence of functions satisfying
\begin{equation}
\sup_{\eps>0}\left\Vert \ue\right\Vert _{L^{p}(\bQ)}<+\infty\,,\qquad\lim_{\eps\to0}\eps\left\Vert \nabla\ue\right\Vert _{L^{p}(\bQ)}=0\,.\label{eq:lem:per-ts-no-grad-omega}
\end{equation}
If $\ue\tsweakto u$ along a subsequence, then $u\in L^{p}(\bQ)$
is independent of $\Omega$.
\end{lem}

\begin{proof}
We obtain that $\ue\stackrel{2s}{\weakto}u\in L^{p}(\bQ;L^{p}(\Omega))$
along a subsequence. We show that $u$ does not depend on the $\Omega$-coordinate
using ergodicity. We recall that $\tau_{\bullet}$ are all measure
preserving for $\P$. Hence, for any $\varphi\in C_{c}^{\infty}(\bQ)$
and $\psi\in\Phi_{q}$, we find for any $a\in\Q^{d}$ it holds
\begin{align*}
\int_{\bQ}\int_{\Omega}\left(u(x,\tau_{a}\omega)-u(x,\omega)\right) & \varphi(x)\psi(\omega)\,d\P(\omega)\,dx\\
 & =\int_{\bQ}\int_{\Omega}u(x,\omega)\varphi(x)\left(\psi(\tau_{-a}\omega)-\psi(\omega)\right)\,d\P(\omega)\,dx\\
 & =\lim_{\eps\to0}\int_{\bQ}u^{\eps}(x)\varphi(x)\left(\psi(\tau_{\frac{-\eps a+x}{\eps}}\omega)-\psi(\tau_{\frac{x}{\eps}}\omega)\right)\,dx\\
 & =\lim_{\eps\to0}\int_{\bQ}\left(u^{\eps}(x+\eps a)\varphi(x+\eps a)-u^{\eps}(x)\varphi(x)\right)\psi(\tau_{\frac{x}{\eps}}\omega)\,dx\\
 & =\lim_{\eps\to0}\int_{\bQ}u^{\eps}(x+\eps a)\left(\varphi(x+\eps a)-\varphi(x)\right)\psi(\tau_{\frac{x}{\eps}}\omega)\,dx\\
 & \qquad+\lim_{\eps\to0}\int_{\bQ}\left(u^{\eps}(x+\eps a)-u^{\eps}(x)\right)\varphi(x)\psi(\tau_{\frac{x}{\eps}}\omega)\,dx\,.
\end{align*}
The first integral on the right hand side can be easily estimated
through
\[
\norm{\ue}_{L^{p}(\bQ)}\norm{\psi(\tau_{\frac{\cdot}{\eps}}\omega)}_{L^{q}(\bQ)}\eps|a|\norm{\nabla\varphi}_{\infty}\to0\quad\mbox{as }\eps\to0\,.
\]
The second integral can be estimated through 
\[
\norm{\varphi}_{\infty}\int_{\bQ}\left|\int_{0}^{\eps}\nabla u^{\eps}(x+ta)\cdot a\,\d t\right|\,\left|\psi(\tau_{\frac{x}{\eps}}\omega)\right|\,dx\leq\norm{\varphi}_{\infty}\eps\norm{\nabla\ue}_{L^{p}(\bQ)}|a|\norm{\psi(\tau_{\frac{\cdot}{\eps}}\omega)}_{L^{q}(\bQ)}\,.
\]
Due to (\ref{eq:lem:per-ts-no-grad-omega}) the right hand side of
the above inequality converges to $0$. Since $\varphi$ and $\psi$
were arbitrary, we obtain $u(x,\tau_{a}\omega)=u(x,\omega)$ for every
$a\in\Rd$. Hence $u$ is invariant under all translations $\tau_{a}$,
which implies for almost every $x\in\bQ$ that $u(x,\cdot)=const$
by ergodicity of $\tau_{\bullet}$.
\end{proof}
Based on Lemma \ref{lem:per-ts-no-grad-omega} we can now prove Theorem
\ref{thm:sto-conver-grad}.
\begin{proof}[Proof of Theorem \ref{thm:sto-conver-grad}]
\textbf{ }We note that $\ue\tsweakto u\in L^{p}(\bQ;L^{p}(\Omega))$
and $\nabla\ue\tsweakto\vel\in L^{p}(\bQ;L^{p}(\Omega;\Rd))$ along
a subsequence.

\emph{Proof of 1:} We consider a countable set $\Phi_{\sol}\subset L_{\sol}^{q}(\Omega)$
which is dense in $L_{\sol}^{q}(\Omega)$. Then, by definition of
$L_{\sol}^{p}(\Omega)$ we find for all $b\in\Phi_{\sol}$ and all
$\varphi\in C_{c}^{\infty}(\bQ)$
\[
\int_{\bQ}\left(\varphi\nabla\ue+\ue\nabla\varphi\right)\cdot b(\tau_{\frac{\bullet}{\eps}}\omega)\,d\lebesgueL=\int_{\bQ}\nabla\left(\ue\varphi\right)\cdot b(\tau_{\frac{\bullet}{\eps}}\omega)\,d\lebesgueL=0\,.
\]
We take the limit $\eps\to0$ on the left hand side and obtain 
\[
\int_{\bQ}\left(\varphi(x)v(x,\tilde{\omega})+u\nabla\varphi(x)\right)\cdot b(\tilde{\omega})\,d\P(\tilde{\omega})\,dx=0\,.
\]
After integration by parts, this implies 
\[
\int_{\bQ}\varphi(x)\left(\nabla u(x)-v(x,\tilde{\omega})\right)\cdot b(\tilde{\omega})\,d\P(\tilde{\omega})\,dx=0\,.
\]
As $\varphi\in C_{c}^{\infty}(\bQ)$ and $b\in\Phi_{\sol}$ were arbitrary
and since $\Phi_{\sol}\subset L_{\sol}^{q}(\Omega)$ is dense, the
last equation and Lemma \ref{thm:Lp-W1p-structure} imply that $\nabla u(x)-\vel(x,\cdot)\in\cV_{\pot}^{p}(\Omega)$
for almost every $x\in\bQ$.

\emph{Proof of 2:} We apply Part 1 to $\tilde{u}^{\eps}:=\eps^{\gamma}\ue$.
Evidently, $\tilde{u}^{\eps}\tsweakto0$ and hence there exists $\vel_{1}\in L^{p}(\bQ;\cV_{\pot}^{p}(\Omega))$
such that $\eps^{\gamma}\nabla\ue=\nabla\tilde{u}^{\eps}\tsweakto\vel_{1}$.

\emph{Proof of 3:} Let $\psi\in L_{\sol}^{q}(\Omega)$ and $\varphi\in C_{0}^{1}(\bQ)$.
Then we have 
\[
\int_{\bQ}\eps\nabla\ue\cdot\varphi\psi\left(\tau_{\frac{\bullet}{\eps}}\omega\right)\d\lebesgueL=-\int_{\bQ}\ue\psi\left(\tau_{\frac{\bullet}{\eps}}\omega\right)\cdot\eps\nablax\varphi\d\lebesgueL\,.
\]
As $\eps\to\infty$ we obtain 
\[
\int_{\bQ}\int_{\Omega}\vel(x,\tilde{\omega})\cdot\varphi(x)\psi(\tilde{\omega})\,\d\P(\tilde{\omega})\,\d x=0
\]
and since this holds for every $\psi\in L_{\sol}^{q}(\Omega)$ and
$\varphi\in C_{0}^{1}(\bQ)$, we obtain that $\vel(x,\omega)\in L^{p}(\bQ;\cV_{\pot}^{p}(\Omega))$.
Furthermore, for a countable dense family $\psi\in W^{1,p}(\Omega)$
and $\varphi\in C_{0}^{1}(\bQ)$ we obtain 
\[
\int_{\bQ}\eps\partial_{i}\ue(x)\varphi(x)\psi\left(\tau_{\frac{x}{\eps}}\omega\right)\,\d x=-\int_{\bQ}\ue(x)\psi\left(\tau_{\frac{x}{\eps}}\omega\right)\cdot\eps\partial_{i}\varphi(x)\d x-\int_{\bQ}\ue(x)\rmD_{i}\psi\left(\tau_{\frac{x}{\eps}}\omega\right)\varphi(x)\,\d x
\]
and in the limit
\[
\int_{\bQ}\int_{\Y}\vel_{i}(x,y)\cdot\varphi(x)\psi(\omega)\,\d\P(\omega)\,\d x=-\int_{\bQ}\int_{\Y}u(x,y)\rmD_{i}\psi\left(\omega\right)\varphi(x)\,\d\P(\omega)\,\d x\,.
\]
This implies $\vel_{i}=\rmD_{i}u$.

\emph{Proof of 4:} Part 3 implies that $\tilde{u}^{\eps}:=\eps^{\gamma-1}$
satisfies $\tilde{u}^{\eps}\tsweakto0$ and $\eps^{\gamma}\nabla\ue=\eps\nabla\tilde{u}^{\eps}\tsweakto\rmD_{\omega}0=0$.
\end{proof}
Important in the context of convergence of gradients is also the following
recovery lemma, obtained in \cite[Section 2.3]{heida2014stochastic}
for the $L^{2}$-case.
\begin{lem}
\label{lem:vanishing-ergodic-potential} Let Assumption \ref{assu:TS-general}
hold. Let $\v\in\cV_{{\rm pot}}^{p}(\Omega)$, $1<p<\infty$ and let
$\bQ$ be a bounded convex domain. For almost every $\omega$ there
exists $C>0$ such that the following holds: For every $\eps>0$ there
exists a unique $V_{\eps}^{\omega}\in W^{1,p}(\bQ)$ with $\nabla V_{\eps}^{\omega}(x)=\v(\tau_{\frac{x}{\eps}}\omega)$,
$\int_{\bQ}V_{\eps}^{\omega}=0$ and $\|V_{\eps}\|_{W^{1,p}(\bQ)}\leq C\|\v\|_{L_{{\rm pot}}^{p}(\Omega)}$
for all $\eps>0$. Furthermore, 
\[
\lim_{\eps\to0}\|V_{\eps}^{\omega}\|_{L^{p}(\bQ)}=0\,.
\]
\end{lem}

\begin{proof}[Sketch of Proof, see \cite{heida2014stochastic}]
 By definition of $L_{{\rm pot}}^{p}(\Omega)$ there exists for almost
every $\omega\in\Omega$ a function $V_{\eps}^{\omega}\in W^{1,p}(\bQ)$
with $\nabla V_{\eps}^{\omega}(x)=\v(\tau_{\frac{x}{\eps}}\omega)$,
$\int_{\bQ}V_{\eps}^{\omega}=0$. By a standard contradiction argument,
there exists a constant $C>0$ such that 
\[
\forall V\in W^{1,p}(\bQ):\quad\norm V_{L^{p}(\bQ)}\leq C\left(\norm{\nabla V}_{L^{p}(\bQ)}+\left|\int_{\bQ}V\right|\right)\,.
\]
The last inequality implies that $V_{\eps}^{\omega}\weakto V$ weakly
in $W^{1,p}(\bQ)$ and $V_{\eps}^{\omega}\to V$ strongly in $L^{p}(\bQ)$.
Furthermore, the Ergodic Theorem \ref{thm:general-lebesgue-ergodic-thm}
yields for every $f\in C(\overline{\bQ})$ 
\[
\int_{\bQ}f\cdot\nabla V_{\eps}^{\omega}=\int_{\bQ}f\cdot\v(\tau_{\frac{x}{\eps}}\omega)\to\int_{\bQ}f\cdot\int\v\,\d\P=\int_{\bQ}f\cdot0=0\,.
\]
Hence $\nabla V=0$ and since $\int_{\bQ}V=0$ it follows $V=0$.
\end{proof}

\subsection{Uniform Extension- and Trace-Properties}

For the rest of this section, we make the following assumptions. Under
the Assumptions \ref{assu:Omega-mu-tau-fundmental} and \ref{assu:TS-general}
and using the notations introduced in Section \ref{subsec:General-Setting}
we introduce $\bP^{\eps}(\omega):=\eps\bP(\omega)$, $\bQ_{1}^{\eps}(\omega):=\bQ\cap\bP^{\eps}(\omega)$
and $\Gamma^{\eps}(\Omega):=\bQ\cap\eps\Gamma(\omega)$.

Following (\ref{eq:ts-measure}) we recall the definition
\[
\mugammaomegaeps(A):=\eps^{n}\hausdorffH^{d-1}\of{\frac{A}{\eps}\cap\Gamma(\omega)}=\eps\hausdorffH^{d-1}\of{A\cap\Gamma^{\eps}(\omega)}
\]

\begin{defn}[Uniform Dirichlet extension property]
\label{def:r-p-Ex-Tr-prop} Let $\bQ$ be a bounded open convex domain
with Lipschitz boundary. We say for $1\leq r\leq p$ that $\bP^{\eps}(\omega)$
has the uniform $\left(r,p\right)$-Dirichlet extension property on
$\bQ$ if for almost every $\omega$ there exists $C_{\omega}>0$
and a linear extension operator 
\[
\cU:\,W_{\loc}^{1,p}(\bP(\omega))\to W_{\loc}^{1,p}(\Rd)
\]
such that 
\[
\cU_{\eps}[u](x):=\cU{\left[u\of{\eps\cdot}\right]}\of{\frac{x}{\eps}}
\]
satisfies the following: For every $u\in W_{0,\partial\bQ}^{1,p}(\bQ_{1}^{\eps}(\omega))$
\[
\norm{\nabla\cU_{\eps}u}_{L^{r}(\bQ)}\leq C_{\omega}\norm{\nabla u}_{L^{p}(\bQ_{1}^{\eps}(\omega))}\,,\quad\norm{\cU_{\eps}u}_{L^{r}(\bQ)}\leq C_{\omega}\norm u_{L^{p}(\bQ_{1}^{\eps}(\omega))}
\]
and 
\begin{equation}
\norm{\cU_{\eps}u}_{W^{1,r}(\Rd\backslash\bQ)}\to0\,.\label{eq:def:r-p-Ex-Tr-prop-2}
\end{equation}
\end{defn}

Theorem \ref{thm:ergodic-Apaths} shows that virtually every random
geometry to which the theory of Sections \ref{sec:Extension-and-Trace-d-M}--\ref{sec:Construction-of-Macroscopic-1}
applies has the $\left(r,p\right)$- extension property on bounded
convex $C^{1}$-domains $\bQ$. In particular, we obtain the following
reformulation of Theorem \ref{thm:Main-Theorem-4}.
\begin{thm}
\label{thm:uniform-Dir-ext-Thm}For $1\leq r<\tilde{s}<s<p\leq\infty$
let $\bP(\omega)$ be almost surely $\left(\delta,M\right)$-regular
(Def. \ref{def:loc-del-M-reg}) and isotropic cone mixing for $\fr>0$
and $f(R)$ (Def. \ref{def:iso-cone-mix}) as well as locally connected
and satisfy $\P\of{\sfS>S_{0}}\leq f_{s}(S_{0})$ such that Assumption
\ref{assu:dirichlet-extension} holds. Then for almost every $\omega$
the set $\bP^{\eps}$ has the uniform $\left(r,p\right)$-Dirichlet
extension property on $\bQ$.
\end{thm}

\begin{proof}
This is almost the statement of Theorem \ref{thm:Main-Theorem-4}
except for (\ref{eq:def:r-p-Ex-Tr-prop-2}). However, for $u\in W_{0,\partial\bQ}^{1,p}(\bP(\omega)\cap n\bQ)$
and $m_{n}:=\left|\Apaths(n\bQ)\right|$ we obtain note that estimates
(\ref{eq:thm:Main-Theorem-4-1})--(\ref{eq:thm:Main-Theorem-4-2})
can be extended to 
{\small \begin{align*}
\frac{1}{m_{n}}\int_{\Apaths(n\bQ)}\left|\cU u\right|^{r} & \leq C(\omega)\left(\frac{1}{m_{n}}\int_{\bP(\omega)\cap\Apaths(n\bQ)}\left|u\right|^{p}\right)^{\frac{r}{p}}=C(\omega)\left(\frac{\left|n\bQ\right|}{m_{n}}\frac{1}{\left|n\bQ\right|}\int_{\bP(\omega)\cap(n\bQ)}\left|u\right|^{p}\right)^{\frac{r}{p}}\,,\\
\frac{1}{m_{n}}\int_{\Apaths(n\bQ)}\left|\cU u\right|^{r} & \leq C(\omega)\left(\frac{1}{m_{n}}\int_{\bP(\omega)\cap\Apaths(n\bQ)}\left|\nabla u\right|^{p}\right)^{\frac{r}{p}}=C(\omega)\left(\frac{\left|n\bQ\right|}{m_{n}}\frac{1}{\left|n\bQ\right|}\int_{\bP(\omega)\cap(n\bQ)}\left|\nabla u\right|^{p}\right)^{\frac{r}{p}}\,,
\end{align*}}
and the statement follows from Theorem \ref{thm:ergodic-Apaths} and
Corollary \ref{cor:support-extension}.
\end{proof}
There exists a weaker notion of extension property, which is for some
applications sufficient.
\begin{defn}[Uniform weak extension property]
\label{def:r-p-Ex-Tr-prop-weak} Let $\bQ$ be a bounded open convex
domain with Lipschitz boundary. We say for $1\leq r\leq p$ that $\bP^{\eps}(\omega)$
has the uniform weak $\left(r,p\right)$-extension property on $\bQ$
if for almost every $\omega$ there exists $C_{\omega}>0$ and a linear
extension operator 
\[
\cU:\,W_{\loc}^{1,p}(\bP(\omega))\to W_{\loc}^{1,p}(\Rd)
\]
such that 
\[
\cU_{\eps}[u](x):=\cU{\left[u\of{\eps\cdot}\right]}\of{\frac{x}{\eps}}
\]
satisfies the following: For every $u\in W_{0,\partial\bQ}^{1,p}(\bQ_{1}^{\eps}(\omega))$
\[
\eps\norm{\nabla\cU_{\eps}u}_{L^{r}(\bQ)}+\norm{\cU_{\eps}u}_{L^{r}(\bQ)}\leq C_{\omega}\of{\eps\norm{\nabla u}_{L^{p}(\Ball{\eps}{\bQ}\cap\bP^{\eps}(\omega))}+\norm u_{L^{p}(\Ball{\eps}{\bQ}\cap\bP^{\eps}(\omega))}}\,.
\]
\end{defn}

\begin{thm}
\label{thm:weak-uniform-ext}For $1\leq r<p_{0}<p_{1}<p<\infty$ let
$\bP(\omega)$ be almost surely $\left(\delta,M\right)$-regular (Def.
\ref{def:loc-del-M-reg}) such that Assumption \ref{assu:local-extension}
holds. Then for almost every $\omega$ the set $\bP^{\eps}$ has the
weak uniform $\left(r,p\right)$-extension property on $\bQ$.
\end{thm}

\begin{proof}
After rescaling, this is the statement of Theorem \ref{thm:Main-Thm-2}.
\end{proof}
Similarly to the extension property, we may introduce a uniform trace
property.
\begin{defn}[Uniform trace property]
Let $\bQ$ be a bounded open convex domain with Lipschitz boundary.
We say for $1\leq r\leq p$ that $\Gamma^{\eps}(\omega)$ has the
uniform $\left(r,p\right)$-trace property on $\bQ$ if for almost
every $\omega$ there exists $C_{\omega}>0$ such that the trace operators
\[
\cT_{\eps}:\,W^{1,p}(\Ball{\eps}{\bQ}\cap\bP^{\eps}(\omega))\to L^{r}(\bQ\cap\Gamma^{\eps})
\]
satisfy the estimate
\[
\norm{\cT_{\eps}u}_{L^{r}(\Gamma^{\eps}\cap\bQ)}\leq C_{\omega}\left(\norm u_{L^{p}(\Ball{\eps}{\bQ}\cap\bP^{\eps}(\omega))}+\eps\norm{\nabla u}_{L^{p}(\Ball{\eps}{\bQ}\cap\bP^{\eps}(\omega))}\right)\,.
\]
\end{defn}

\begin{thm}
Let $\bP(\omega)$ be a stationary ergodic random open set which is
almost surely $\left(\delta,M\right)$-regular (Def. \ref{def:loc-del-M-reg})
such that Assumption \ref{assu:Trace} holds. For $1\leq r<p_{0}<p<\infty$
and $\bQ\subset\Rd$ a bounded domain with Lipschitz boundary. Then
for almost every $\omega$ the set $\bP^{\eps}$ has the uniform $\left(r,p\right)$-trace
property on $\bQ$.
\end{thm}

\begin{proof}
After rescaling, this is the statement of Theorem \ref{thm:Main-Thm-1}.
\end{proof}

\subsection{\label{subsec:Domains-with-holes}Homogenization on Domains with
Holes}

In what follows, we will naturally deal with two-scale limits of functions
defined solely on $\bQ_{1}^{\eps}$. Hence we introduce the following
definition.
\begin{defn}
\label{def:def-two-scale-holes}Let $1<p\leq\infty$ and $\ue\in L^{p}(\bQ_{1}^{\eps}(\omega))$
for all $\eps>0$. We say that $(\ue)$ converges (weakly) in two
scales to $u\in L^{p}(\bQ;L^{p}(\bP))$ and write $\ue\stackrel{2s}{\weakto}u$
if $\sup_{\eps>0}\norm{\ue}_{L^{2}(\bQ_{1}^{\eps}(\omega))}<\infty$
and if for every $\psi\in\Phi_{q}$ and $\varphi\in C(\overline{\bQ})$
there holds with $\phi_{\omega,\eps}(x):=\varphi(x)\psi(\tau_{\frac{x}{\eps}}\omega)$
\[
\lim_{\eps\to0}\int_{\bQ_{1}^{\eps}}\ue\phi_{\omega,\eps}=\int_{\bQ}\int_{\Omega}\chi_{\bP}u\varphi\psi\,\d\P\,\d\lebesgueL\,.
\]
\end{defn}

The latter definition coincides with Definition \ref{def:two-scale-conv}
for $\d\muomega=\chi_{\bP(\omega)}\d\lebesgueL$, which can be verified
using the ergodic theorem. Hence, we find the following lemma:
\begin{lem}
\label{lem:sto-two-scale-limit-holes}\label{lem:sto-ts-conv-on-holes}
Let $1<p\leq\infty$ and $\ue\in L^{p}(\bQ_{1}^{\eps}(\omega))$ be
a sequence such that $\sup_{\eps>0}\norm{\ue}_{L^{p}(\bQ_{1}^{\eps}(\omega))}<\infty$.
Then there exists $u\in L^{p}(\bQ;L^{p}(\bP))$ and a subsequence
$\eps'\to0$ such that $u^{\eps'}\stackrel{2s}{\weakto}u$.

Furthermore, if $\ue\in L^{p}(\bQ)$ is a sequence such that $\sup_{\eps>0}\norm{\ue}_{L^{p}(\bQ)}<\infty$
and $u^{\eps'}\stackrel{2s}{\weakto}u$ along a subsequence $\eps'\to0$
for some $u\in L^{p}(\bQ;L^{p}(\Omega))$, then $u^{\eps'}\chi_{\bQ_{1}^{\eps'}(\omega)}\stackrel{2s}{\weakto}\chi_{\bP}u$.
\end{lem}

\begin{proof}
This follows immediately from Lemma \ref{lem:Existence-ts-lim} extending
$\ue$ by $0$ to $\bQ$ and on noting that $\psi\in\Phi_{q}$ implies
w.l.o.g. $\chi_{\bP}\psi\in\Phi_{q}$.
\end{proof}
\begin{lem}
\label{lem:hole-exp-grad-ts}Let $\bP(\omega)$ be a random open domain
such that $\bP^{\eps}(\omega)$ has the weak uniform $\left(r,p\right)$-extension
property on $\bQ$ for $1<r<p<\infty$. Then for almost every $\omega\in\Omega$
the following holds: If $u^{\eps}\in W^{1,p}(\Ball{_{\eps}}{\bQ}\cap\bP^{\eps}(\omega);\Rd)$
for all $\eps$ with 
\[
\sup_{\eps}\left(\norm{\ue}_{L^{p}(\Ball{_{\eps}}{\bQ}\cap\bP^{\eps}(\omega))}+\eps\|\nabla u^{\eps}\|_{L^{p}(\Ball{_{\eps}}{\bQ}\cap\bP^{\eps}(\omega))}\right)<C
\]
 for $C$ independent from $\eps>0$ then there exists a subsequence
denoted by $u^{\eps'}$ and a function $u\in L^{p}(\bQ;W^{1,r}(\Omega))\cap L^{p}(\bQ\times\bP)$
such that
\begin{equation}
\cU_{\eps'}u^{\eps'}\tsweakto u\quad\mbox{and }\quad\eps\nabla\cU_{\eps'}u^{\eps'}\tsweakto\nabla_{\omega}u\label{eq:lem:hole-exp-grad-ts-1}
\end{equation}
as well as
\begin{equation}
u^{\eps'}\tsweakto u\quad\mbox{and }\quad\eps\nabla u^{\eps'}\tsweakto\chi_{\bP}\nabla_{\omega}u\label{eq:lem:hole-exp-grad-ts-2}
\end{equation}
as $\eps\to0$.
\end{lem}

\begin{proof}
We find 
\begin{align}
 & \sup_{\eps}\left(\norm{\cU_{\eps}\ue}_{L^{r}(\bQ\cap\bP^{\eps}(\omega))}+\eps\|\nabla\cU_{\eps}u^{\eps}\|_{L^{r}(\bQ\cap\bP^{\eps}(\omega))}\right)\nonumber \\
 & \qquad\leq C\sup_{\eps}\left(\norm{\ue}_{L^{p}(\Ball{_{\eps}}{\bQ}\cap\bP^{\eps}(\omega))}+\eps\|\nabla u^{\eps}\|_{L^{p}(\Ball{_{\eps}}{\bQ}\cap\bP^{\eps}(\omega))}\right)\label{eq:lem:hole-exp-grad-ts-help-1}
\end{align}
Theorem \ref{thm:sto-conver-grad} and Definition \ref{def:r-p-Ex-Tr-prop-weak}
imply now for some limit function $u\in L^{r}(\bQ;W^{1,r}(\Omega))$
that (\ref{eq:lem:hole-exp-grad-ts-1}) and (\ref{eq:lem:hole-exp-grad-ts-2})
hold.
\end{proof}
We are now able to provide the:
\begin{proof}[\textbf{\emph{Proof of Theorem \ref{thm:weak-extension-property}}}]
 Theorem \ref{thm:weak-uniform-ext} shows that $\bP^{\eps}(\omega)$
satisfies the uniform weak extension property. Hence, if $\left(u_{k}\right)_{k\in\N}$
is a countable dense subset of $W^{1,p}(\Omega)$, we find a set of
full measure $\tilde{\Omega}\subset\Omega$ such that for every $k\in\N$
and every $\omega\in\tilde{\Omega}$ the realizations $u_{k,\omega}$
are well defined elements of $W_{\loc}^{1,p}(\bP(\omega))$, the extension
operator defined in (\ref{eq:def:cU-Q-2}) is uniformly bounded and
hence $\cU_{\eps}$ defined in Definition \ref{def:r-p-Ex-Tr-prop-weak}
is uniformly bounded, too. We can thus use the two-scale convergence
method as a tool.

Given such $\omega$, we define $\ue(x):=u_{k}\of{\tau_{\frac{x}{\eps}}\omega}$
and by Lemma \ref{lem:hole-exp-grad-ts} we find $\tilde{u}\in L^{p}(\bQ;W^{1,r}(\Omega))\cap L^{p}(\bQ\times\bP)$
such that $\cU_{\eps}\ue\to\tilde{u}_{k}$ and $\eps\nabla\cU_{\eps}\ue\to\nabla_{\omega}\tilde{u}_{k}$.
Furthermore, we find 
\begin{align*}
\norm{\tilde{u}_{k}}_{L^{r}(\bQ\times\Omega)}+\norm{\nabla_{\omega}\tilde{u}_{k}}_{L^{r}(\bQ\times\Omega)} & \leq\liminf_{\eps\to0}\left(\norm{\cU_{\eps}\ue}_{L^{r}(\bQ)}+\eps\norm{\nabla\cU_{\eps}\ue}_{L^{r}(\bQ)}\right)\\
 & \leq C\liminf_{\eps\to0}\left(\norm{\ue}_{L^{p}(\Ball{\eps}{\bQ}\cap\bP^{\eps}(\omega))}+\eps\norm{\nabla\ue}_{L^{p}(\Ball{\eps}{\bQ}\cap\bP^{\eps}(\omega))}\right)\\
 & =C\left(\norm{u_{k}}_{L^{p}(\bQ\times\Omega)}+\norm{\nabla_{\omega}u_{k}}_{L^{p}(\bQ\times\Omega)}\right)\,.
\end{align*}
Since the operator $u_{k}\to\tilde{u}_{k}$ is linear and bounded,
it can be extended to the whole of $W^{1,p}(\bP)$.
\end{proof}

\begin{proof}[\textbf{\emph{Proof of Theorem \ref{thm:Harmonic-P}}}]
 For every $\eps>0$ there exists a unique $\ue$ that solves 
\begin{align*}
-\eps^{2}\nabla\ue+\ue & =0 &  & \text{on }\Ball{\eps}{\bQ}\cap\bP^{\eps}(\Omega)\,,\\
-\eps\nabla\ue\cdot\nu_{\Gamma^{\eps}(\omega)} & =1 &  & \text{on }\Gamma^{\eps}(\omega)\cap\bQ\,,\\
\ue & =0 &  & \text{on }\partial\bQ\,.
\end{align*}
Deriving apriori estimates in the usual way, for some $C>0$ independent
from $\eps$ it holds
\[
\eps\norm{\nabla\ue}_{L^{2}(\Ball{\eps}{\bQ}\cap\bP^{\eps}(\Omega))}+\norm{\ue}_{L^{2}(\Ball{\eps}{\bQ}\cap\bP^{\eps}(\Omega))}\leq C
\]
and thus according to Lemma \ref{lem:hole-exp-grad-ts} we find $u\in L^{p}(\bQ;W^{1,r}(\Omega))\cap L^{p}(\bQ\times\bP)$
such that 
\[
\cU_{\eps'}u^{\eps'}\tsweakto u\quad\mbox{and }\quad\eps\nabla\cU_{\eps'}u^{\eps'}\tsweakto\nabla_{\omega}u
\]
along a subsequence $u^{\eps'}$ which we again denote $\ue$ in the
following, for simplicity. But then for $\phi\in C_{b}^{1}(\Omega)$
and $\psi\in C_{c}^{1}(\bQ)$ it follows
\begin{align*}
\eps\int_{\bQ\cap\Gamma^{\eps}(\omega)}\phi(\tau_{\frac{x}{\eps}}\omega)\psi(x)\,\d\hausdorffH^{d-1}(x) & =-\eps^{2}\int_{\bQ\cap\Gamma^{\eps}(\omega)}\phi(\tau_{\frac{x}{\eps}}\omega)\psi(x)\nabla\ue(x)\cdot\nu_{\Gamma(\omega)}(\tau_{\frac{x}{\eps}}\omega)\,\d\hausdorffH^{d-1}(x)\\
 & =\int_{\bQ_{1}^{\eps}(\omega)}\eps\nabla\ue\cdot\left(\nabla_{\omega}\phi(\tau_{\frac{x}{\eps}}\omega)\psi(x)+\eps\phi(\tau_{\frac{x}{\eps}}\omega)\nabla\psi(x)\right)\,\d x\\
 & \quad+\int_{\bQ_{1}^{\eps}(\omega)}\ue\phi(\tau_{\frac{x}{\eps}}\omega)\psi(x)\,\d x\\
 & \to\int_{\bQ}\int_{\bP}\left(\nabla_{\omega}u\cdot\nabla_{\omega}\phi\psi+u\phi\psi\right)\,.
\end{align*}
Since the left hand side of the above calculation converges to $\int_{\bQ}\int_{\Gamma}\phi\psi\,\d\mugammapalm$
and $\psi$ was arbitrary, we conclude.
\end{proof}

\begin{proof}[\textbf{\emph{Proof of Theorem \ref{thm:strong-extension-property}}}]
Let $\bQ=\Ball 20$ and let $\phi\in C_{c}^{\infty}(\bQ)$ with $\phi|_{\Ball 10}=1$,
$\phi\geq0$. According to Theorem \ref{thm:uniform-Dir-ext-Thm},
$\bP^{\eps}$ has the uniform $(r,p)$-Dirichlet extension property.
The theorem now follows from part 2 of the following Lemma.
\end{proof}
\begin{lem}
\label{lem:ts-extension-conv}Let $\bP(\omega)$ be a random open
domain such that $\bP^{\eps}(\omega)$ has the uniform $\left(r,p\right)$-Dirichlet
extension property on $\bQ$ for $1<r<p<\infty$. Then for almost
every $\omega\in\Omega$ the following holds:
\begin{enumerate}
\item If $u^{\eps}\in W_{0,\partial\bQ}^{1,p}(\bQ\cap\bP^{\eps}(\omega);\Rd)$
for all $\eps$ with $\sup_{\eps}\norm{\ue}_{L^{p}(\bQ_{1}^{\eps}(\omega))}+\|\nabla u^{\eps}\|_{L^{p}(\bQ_{1}^{\eps}(\omega))}<C$
for $C$ independent from $\eps>0$ then there exists a subsequence
denoted by $u^{\eps'}$ and functions $u\in W_{0}^{1,r}(\bQ;\Rd)\cap L^{p}(\bQ)$
and $\v\in L^{r}(\bQ;\cV_{\pot}^{r}(\Omega))$ such that 
\begin{align}
 & u^{\eps'}\tsweakto\chi_{\bP}u\quad\mbox{and }\quad\nabla u^{\eps'}\tsweakto\chi_{\bP}\nabla u+\chi_{\bP}\v\qquad\mbox{as }\eps\to0\,,\label{eq:lem:ts-extension-conv-1}\\
 & \cU_{\eps'}u^{\eps'}\tsweakto u\quad\mbox{and }\quad\nabla\cU_{\eps'}u^{\eps'}\tsweakto\nabla u+\v\qquad\mbox{as }\eps\to0\,.
\end{align}
Furthermore, $\cU_{\eps'}u^{\eps'}\weakto u$ weakly in $W^{1,r}(\bQ))\cap L^{p}(\bQ)$.
\item $\bP$ has the strong $(r,p)$-extension property with $\cU_{\Omega}\phi=\mathrm{ts{-}lim}_{\eps_{\to0}}\cU_{\eps}\phi(\tau_{\frac{x}{\eps}}\omega)$
for $\phi\in W^{1,p}(\bP)$.
\item If $p\geq2$ and the Assumptions of Theorem \ref{thm:Harmonic-P}
are satisfied and $\Gamma^{\eps}(\omega)$ additionally has the uniform
$\left(s,p\right)$-trace property for some $s>1$ then 
\[
\cT_{\eps'}u^{\eps'}\tsweakto u\quad\text{in }L^{s}(\Gamma^{\eps}\cap\bQ;\mugammaomegaeps)\,.
\]
If, even further, $\Gamma^{\eps}(\omega)$ has the uniform $(s,r)$-trace
property with $r$ from Part 1, then 
\begin{equation}
\lim_{\eps\to0}\norm{\cT_{\eps'}u^{\eps'}-\cT_{\eps'}u}_{L^{s}(\Gamma^{\eps'}\cap\bQ;\mugammaomega^{\eps'})}\to0\,.\label{eq:lem:ts-extension-conv-2}
\end{equation}
\end{enumerate}
\end{lem}

\begin{proof}
In what follows, convergences always hold along subsequently chosen
subsequences of $\ue$, which we always relabel by $\ue$.

\emph{Proof of 1:} Let $\frac{1}{r}+\frac{1}{q}=1$. Then Theorem
\ref{thm:sto-conver-grad} and the assumption that (w.l.o.g.) $\chi_{\bP}\Phi_{q}\subset\Phi_{q}$
yields that for some $u\in W^{1,r}(\bQ;\Rd)$ and $v\in L^{r}(\bQ;L_{\pot}^{r}(\Omega))$
\[
\cU_{\eps}\ue\tsweakto u\quad\mbox{and }\quad\nabla\cU_{\eps}\ue\tsweakto\nabla u+v\qquad\mbox{as }\eps\to0\,.
\]
Due to (\ref{eq:def:r-p-Ex-Tr-prop-2}) we find $u\in W_{0}^{1,r}(\bQ;\Rd)$.
This yields (\ref{eq:lem:ts-extension-conv-1}).

Proof of 2: For $u\in W^{1,p}(\bP)$ with $\ue(x):=u\of{\tau_{\frac{x}{\eps}}\omega}$
we find for almost every $\omega$ that $\cU_{\eps}$ from Definition
\ref{def:r-p-Ex-Tr-prop} satisfies 
\begin{align}
\eps\norm{\nabla\cU_{\eps}(\phi\ue)}_{L^{r}(\bQ)} & \leq C\left(\eps\norm{\ue\nabla\phi}_{L^{p}(\bQ\cap\bP^{\eps}(\omega))}+\eps\norm{\phi\nabla\ue}_{L^{p}(\bQ\cap\bP^{\eps}(\omega))}\right)\label{eq:thm:strong-extension-property-help-1}\\
\norm{\cU_{\eps}(\phi\ue)}_{L^{r}(\bQ)} & \leq C\norm{\ue\phi}_{L^{p}(\bQ\cap\bP^{\eps}(\omega))}\nonumber 
\end{align}
As $\eps\to0$, Lemma \ref{lem:hole-exp-grad-ts} yields $\ue\phi\tsweakto\tilde{u}$,
$\nabla\cU_{\eps}(\phi\ue)\tsweakto\rmD_{\omega}\tilde{u}$, where
$\tilde{u}\in L^{p}(\bQ;W^{1,r,p}(\Omega,\bP))$. Moreover, inequality
(\ref{eq:thm:strong-extension-property-help-1}) implies in the limit
that 
\[
\norm{\rmD_{\omega}\tilde{u}}_{L_{\pot}^{r,p}(\Omega,\bP)}\leq C\norm{\rmD_{\omega}u}_{L_{\pot}^{p}(\bP)}\,.
\]
Hence we can set $\cU_{\Omega}\rmD_{\omega}u:=\int_{\bQ}\rmD_{\omega}\tilde{u}$.
By density, this operator extends to $\cV_{\pot}^{p}(\bP)$.

\emph{Proof of 3:} Now let $p\geq2$ and let the Assumptions of Theorem
\ref{thm:Harmonic-P} be satisfied and let $\Gamma^{\eps}(\omega)$
additionally have the uniform $\left(s,p\right)$-trace property for
some $s>1$. If $u_{\Omega}$ is the function from Theorem \ref{thm:Harmonic-P}
we observe for $u_{\Omega}^{\eps}(x):=u_{\Omega}\of{\tau_{\frac{x}{\eps}}\omega}$
for every $\psi\in C_{c}^{\infty}(\bQ)$ and $\phi\in C_{b}^{1}(\Omega)$
with $\phi^{\eps}(x):=\phi\of{\tau_{\frac{x}{\eps}}\omega}$ that
\begin{align*}
\int_{\bQ\cap\Gamma^{\eps}(\omega)}\ue\psi\phi^{\eps}\,\d\mugammaomegaeps & =\eps\int_{\bQ\cap\Gamma^{\eps}(\omega)}\ue\psi\phi^{\eps}\eps\nabla_{\omega}u_{\Omega}^{\eps}\cdot\nu_{\Gamma^{\eps}(\omega)}\,\d\hausdorffH^{d-1}\\
 & =\int_{\bQ\cap\bP^{\eps}(\omega)}\left(\ue\psi\phi^{\eps}u_{\Omega}^{\eps}+\eps\nabla\ue_{\Omega}\cdot\left(\ue\phi^{\eps}\eps\nabla\psi+\psi\phi^{\eps}\eps\nabla\ue+\psi\ue\eps\nabla\phi^{\eps}\right)\right)\\
 & \to\int_{\bQ}\int_{\bP}\left(u\psi\phi u_{\Omega}+\psi u\nabla_{\omega}u_{\Omega}\cdot\nabla_{\omega}\phi\right)\\
 & =\int_{\bQ}\int_{\Gamma}u\psi\,\d\mugammapalm\,.
\end{align*}
Since $\psi$ and $\phi$ were arbitrary and $\nabla_{\omega}(u\psi)=0$
we conclude 
\[
\lim_{\eps\to0}\int_{\bQ\cap\Gamma^{\eps}(\omega)}\ue\psi\phi^{\eps}\,\d\mugammaomegaeps=\int_{\bQ}\int_{\Gamma}u\phi\psi\,.
\]

In order to show (\ref{eq:lem:ts-extension-conv-2}) note that 
\[
\norm{\cT_{\eps}u^{\eps}-\cT_{\eps}u_{\delta}}_{L^{s}(\Gamma^{\eps}\cap\bQ;\mugammaomega^{\eps})}\leq\norm{u^{\eps}-u}_{L^{r}(\Ball{\eps}{\bQ}\cap\bP^{\eps}(\omega))}+\eps\norm{\nabla\left(u^{\eps}-u\right)}_{L^{r}(\Ball{\eps}{\bQ}\cap\bP^{\eps}(\omega))}\,.
\]
Since the first term on the right hand side converges to zero and
$\norm{\nabla\left(u^{\eps}-u\right)}_{L^{r}(\Ball{\eps}{\bQ}\cap\bP^{\eps}(\omega))}$
is bounded, the claim follows.
\end{proof}

\subsection{\label{sec:Homogenization-of-p-Laplace}Homogenization of $p$-Laplace
Equations}
\begin{assumption}
\label{assu:hom}For the rest of this work, let Assumptions \ref{assu:Trace},
\ref{assu:dirichlet-extension} and \ref{assu:Omega-mu-tau-fundmental-strong}
hold for some $1<r<p$ and $p\geq2$. This implies that $\bP$ and
$\Gamma$ satisfy the strong $(r,p)$-extension and the strong $(r,p)$-trace
property, as well as the weak uniform $(r,p)$-extension and the uniform
$(r,p)$-Dirichlet extension property with the uniform $(r,p)$-trace
property. In particular, we can apply all of the above developed theory.
\end{assumption}

In what follows, we will consider the homogenization of the following
functionals:
\[
\cE_{\eps,\omega}(u)=\int_{\bQ^{\eps}(\omega)}\left(\frac{1}{p}\left|\nabla u\right|^{p}+\frac{1}{p}\left|u\right|^{p}-g\,u\right)+\int_{\Gamma^{\eps}(\omega)}F(u(x))\d\mu_{\Gamma(\omega)}^{\eps}(x)\,,
\]
where $F$ is a convex function with $\partial F=f$, $F(\cdot)\geq F_{0}>-\infty$
for some constant $F_{0}\in\R$ and we assume that $\left|\partial F(A)\right|$
is bounded on bounded subsets $A\subset\R$. Note that compared to
(\ref{eq:system-eps-p-Laplace}) we add the term $\left|u\right|^{p}$
in order to reduce technical difficulties. However, we will discuss
how to treat the case of missing $\left|u\right|^{p}$ in Remark \ref{rem:final-remark}.
Minimizers of $\cE_{\eps,\omega}$ satisfy the partial differential
equation system 
\begin{align}
-\diver\left(a\left|\nabla\ue\right|^{p-2}\nabla\ue\right)+\left|u\right|^{p-1} & =g &  & \mbox{on }\bQ_{\tilde{\bP}}^{\eps}(\omega)\,,\nonumber \\
u & =0 &  & \mbox{on }\partial\bQ\,,\label{eq:system-eps-p-Laplace-modified}\\
\left|\nabla\ue\right|^{p-2}\nabla\ue\cdot\nu_{\Gamma^{\eps}(\omega)} & =f(\ue) &  & \mbox{on }\Gamma^{\eps}(\omega)\,.\nonumber 
\end{align}
and we will see that homogenization of the latter system is equivalent
with a two-scale $\Gamma$-convergence of $\cE_{\eps,\omega}$. In
particular, we find the following
\begin{thm}
\label{thm:ts-gamma-conv}Let Assumption \ref{assu:hom} hold. Then,
for almost every $\omega\in\Omega$ and 
\[
\cE(u,\v):=\int_{\bQ}\int_{\bP}\frac{1}{p}\left(\left|\nabla u+\v\right|^{p}+\left|u\right|^{p}\right)-\int_{\bQ}\int_{\bP}g\,u+\int_{\bQ}\int_{\Gamma}F(u)\d\mugammapalm
\]
we find $\cE_{\eps,\omega}\xrightarrow{2s\Gamma}\cE$ in the following
sense

\begin{enumerate}
\item For $\ue\weakto u$ weakly in $L^{p}(\bQ)$, $\ue\in W_{0,\partial\bQ}^{1,p}(\bQ^{\eps}(\omega))$
with $\sup_{\eps}\cE_{\eps,\omega}(\ue)<\infty$, there holds $u\in W_{0}^{1,r}(\bQ)$
and there exists $\v\in L^{r}(\bQ;\cV_{\pot}^{r}(\Omega,\bP))$ such
that $\nabla\ue\tsweakto\chi_{\bP}\cdot\left(\nabla u+\v\right)$
and 
\[
\cE(u,\v)\leq\liminf_{\eps\to0}\cE_{\eps,\omega}(\ue)\,.
\]
\item For each pair $(u,\v)\in W_{0}^{1,r}(\bQ)\times L^{r}(\bQ;\cV_{\pot}^{r}(\Omega))$
with $\cE(u,\v)<+\infty$ there exists a sequence $\ue\in W_{0,\partial\bQ}^{1,p}(\bQ^{\eps}(\omega))$
such that $\ue\weakto\left|\bP\right|u$ weakly in $L^{p}(\bQ)$,
$\cU_{\eps}\ue\weakto u$ weakly in $W^{1,r}(\bQ)$ and $\nabla\ue\tsweakto\chi_{\bP}\cdot\left(\nabla u+\v\right)$
weakly in two scales and 
\[
\cE(u,\v)=\lim_{\eps\to0}\cE_{\eps,\omega}(\ue)\,.
\]
\end{enumerate}
\end{thm}

\begin{proof}
1. Evidently, 
\[
\int_{\bQ^{\eps}(\omega)}\left(\frac{1}{p}\left|\nabla\ue\right|^{p}+\frac{1}{p}\left|\ue\right|^{p}\right)\leq C\cE_{\eps,\omega}(\ue)
\]
for $C$ independent from $\eps$. Hence the statement follows from
Lemmas \ref{lem:ts-extension-conv} and \ref{lem:General-Hom-Convex}
on particularly noting that $\ue\tsweakto u$ in $L^{s}(\Gamma^{\eps}(\omega);\mugammaomegaeps)$.

2. Step a: Let $\left(u_{k}\right)_{k\in\N}\subset C_{b}^{1}(\Omega)$
be the countable dense family in $W^{1,p}(\Omega)$ according to Theorem
\ref{thm:dense-functions-Omega}. Furthermore, let $\left(\phi_{j}\right)_{j\in\N}\subset C_{c}^{\infty}(\bQ)$
be dense in $W_{0}^{1,p}(\bQ)$. Then the span of the functions $\phi_{j}\nabla_{\omega}u_{k}$
is dense in $L^{r}(\bQ;\cV_{\pot}^{r}(\Omega))$. Writing $S=\mathrm{span}\phi_{j}\nabla_{\omega}u_{k}$
we show statement 2. for $(u,\v)\in\left(\phi_{j}\right)_{j\in\N}\times S$.
However, for such $(u,\v)$ we find $V\in\mathrm{span}\phi_{j}u_{k}$
such that $\v=\nabla_{\omega}V$ and $V^{\eps}(x):=V(x,\tau_{\frac{x}{\eps}}\omega)$
is well defined and measurable for every $\omega$. For simplicity
of notation, we assume $V=\phi_{j}u_{k}$

In particular, we have for $\ue=u+\eps V^{\eps}$ that $\ue\tsweakto u$
and $\nabla\ue=\nabla u+\eps\nabla\phi_{j}\,u_{k}\of{\tau_{\frac{x}{\eps}}\omega}+\phi_{j}\nabla_{\omega}u_{k}\of{\tau_{\frac{x}{\eps}}\omega}$
and hence $\ue\weakto u$ weakly in $W^{1,p}(\bQ)$ and $\nabla\ue\tsweakto\nabla u+\phi_{j}\nabla_{\omega}u_{k}$.
Using essential boundedness of $\nabla\phi_{j}\,u_{k}\of{\tau_{\frac{x}{\eps}}\omega}$,
the ergodic theorem now yields
\begin{align*}
\lim_{\eps\to0}\int_{\bQ^{\eps}(\omega)}\left|\nabla\ue\right|^{p} & =\lim_{\eps\to0}\int_{\bQ}\chi_{\bP}\of{\tau_{\frac{x}{\eps}}\omega}\left|\nabla u+\phi_{j}\nabla_{\omega}u_{k}\of{\tau_{\frac{x}{\eps}}\omega}\right|^{p}\\
 & =\int_{\bQ}\int_{\bP}\left|\nabla u+\v\right|^{p}\,.
\end{align*}
Similarly, we show $\int_{\bQ^{\eps}(\omega)}\left|\ue\right|^{p}\to\int_{\bQ}\int_{\bP}\left|u\right|^{p}$
and $\int_{\bQ^{\eps}(\omega)}g\ue\to\int_{\bQ}\int_{\bP}gu$.

Step b: By Lemma \ref{lem:ts-extension-conv} we find $\cT_{\eps}\ue\tsweakto u$.
Unfortunately, this is not enough to pass to the limit in the integral
$\int_{\Gamma^{\eps}(\omega)}F(u(x))\d\mu_{\Gamma(\omega)}^{\eps}(x)$.
However, we can make use of 
\[
F(u)+\partial F(u)\eps V^{\eps}\leq F(u+\eps V^{\eps})\leq F(u)+\partial F(u+\eps V^{\eps})\eps V^{\eps}\,.
\]
Since $\sup_{\eps}\norm{V^{\eps}}_{\infty}+\norm u_{\infty}<\infty$
we find 
\[
\norm{\partial F(u)}_{\infty}+\sup_{\eps}\norm{\partial F(u+\eps V^{\eps})}_{\infty}\leq C<\infty
\]
and hence 
\[
F(u)-\eps C\leq F(u+\eps V^{\eps})\leq F(u)+\eps C\,.
\]
This implies by the ergodic theorem
\[
\int_{\Gamma^{\eps}(\omega)}F(u+\eps V^{\eps})\d\mu_{\Gamma(\omega)}^{\eps}(x)\to\int_{\bQ}\int_{\Gamma}F(u)\d\mugammapalm\,,
\]
and hence 2. for $(u,\v)\in\left(\phi_{j}\right)_{j\in\N}\times S$.

Step c: We pick up an idea of \cite{duong2013wasserstein}, Proposition
6.2. For general $(u,\v)\in W_{0}^{1,r}(\bQ)\times L^{r}(\bQ;\cV_{\pot}^{r}(\Omega))$
with $\cE(u,\v)<+\infty$ let $(u_{n},\v_{n})\in\left(\phi_{j}\right)_{j\in\N}\times S$
with 
\begin{equation}
\norm{(u,\v)-(u_{n},\v_{n})}_{W_{0}^{1,r}(\bQ)\times L^{r}(\bQ;\cV_{\pot}^{r}(\Omega))}\leq\frac{1}{n}\label{eq:thm:ts-gamma-conv-help-1}
\end{equation}
and 
\begin{equation}
\left|\cE(u,\v)-\cE(u_{n},\v_{n})\right|\leq\frac{1}{n}\,.\label{eq:thm:ts-gamma-conv-help-2}
\end{equation}
We achieve this in the following way: First we introduce $M_{F}:=\sup F^{-1}(-\infty,M)$
and cut $u_{M}:=\min\left\{ u,M_{F}\right\} $. Furthermore, we set
$\v_{M}(x,\omega)=\chi_{(-\infty,M_{F})}(u(x))\,\v(x,\omega)$, i.e.
$u_{M}=M_{F}$ implies $\v=0$. Then $u_{M}$ and $\v_{M}$ are still
in the same respective spaces. Furthermore, as $M\to\infty$ we find
$\cE(u_{M},\v_{M})\to\cE(u,\v)$ by the Lebesgue dominated convergence
theorem. Now, by the properties of $F$, we can approach $\left(u_{M},\v_{M}\right)$
by elements $(u_{M,\delta},\v_{M,\delta})\in\left(\phi_{j}\right)_{j\in\N}\times S$.
By the Lebesgue dominated convergence theorem we get convergence in
the $\left|\cdot\right|^{p}$-terms and using the convexity of $F$
and local boundedness of $\partial F$ like in Step b we show that
$\cE(u_{M,\delta},\v_{M,\delta})\to\cE(u_{M},\v_{M})$. Successively
choosing $M$ and $\delta$, we find $(u_{n},\v_{n})\in\left(\phi_{j}\right)_{j\in\N}\times S$
satisfying \ref{eq:thm:ts-gamma-conv-help-1}--\ref{eq:thm:ts-gamma-conv-help-2}..

We set $\eps_{0}(\omega)=1$ and for each $(u_{n},\v_{n})\in\left(\phi_{j}\right)_{j\in\N}\times S$
we find by Steps a and b for almost every $\omega$ some $\eps_{n}(\omega)\leq\frac{1}{2}\eps_{n-1}(\omega)$
such that for $\eps<\eps_{n}(\omega)$ and $\ue_{n,\omega}=u_{n}(x)+\eps V_{n}(x,\tau_{\frac{x}{\eps}}\omega)$
it holds 
\[
\left|\cE_{\eps,\omega}(u_{n,\omega}^{\eps})-\cE(u_{n},\v_{n})\right|\leq\frac{1}{n}\,.
\]
The set $\tilde{\Omega}\subset\Omega$ such that all $\eps_{n}(\omega)$
are well defined has measure $1$. For such $\omega$ we choose $\ue=\ue_{n,\omega}$
if $\eps\in(\eps_{n+1},\eps_{n})$. Then
\[
\left|\cE_{\eps,\omega}(\ue)-\cE(u,\v)\right|\leq\frac{2}{n}\qquad\text{for }\eps<\eps_{n}\,.
\]
which implies the claim.
\end{proof}
\begin{thm}
\label{thm:Final-homogenization-Theorem}Let Assumption \ref{assu:hom}
hold. Then for almost every $\omega$ the following holds: For every
$\eps>0$ let $\ue_{\mathrm{min}}\in W_{0,\partial\bQ}^{1,p}(\bQ^{\eps}(\omega))$
be the unique minimizer of $\cE_{\eps,\omega}$. Then 
\[
\sup_{\eps>0}\norm{\ue_{\mathrm{min}}}_{W_{0,\partial\bQ}^{1,p}(\bQ^{\eps}(\omega))}+\cE_{\eps,\omega}(\ue_{\mathrm{min}})\leq\infty
\]
and for every subsequence such that $\cU_{\eps}\ue_{\mathrm{min}}\weakto u$
weakly in $L^{p}(\bQ)$ and weakly in $W^{1,r}(\bQ)$ with $\v\in L^{r}(\bQ;\cV_{\pot}^{r}(\Omega,\bP))$
such that $\nabla\ue_{\mathrm{min}}\tsweakto\nabla u+\v$. It further
holds $u\in W_{0}^{1,r}(\bQ)$ and $(u,\v)$ is a global minimizer
of $\cE$ in $W_{0}^{1,r}(\bQ)\times\cV_{\pot}^{r}(\Omega)$.
\end{thm}

\begin{rem}
Unfortunately, we are not able to prove uniqueness of homogenized
solution due to a lack of coercivity in the respective case. However,
note that in case Conjecture \ref{conj:decomposition} holds, one
can immediately prove that both $\nabla u\in L^{p}(\bQ)$ and $\v\in L^{p}(\bQ;\cV_{\pot}^{r,p}(\Omega,\bP))$,
which allows to show the uniqueness of the minimizer by a standard
coercivity argument.
\end{rem}

\begin{proof}
In what follows, we denote 
\[
W_{r}:=W_{0}^{1,r}(\bQ),\quad\cV_{r}:=\cV_{\pot}^{r}(\Omega)\,,
\]
and note that every of the following countable steps works for almost
every $\omega$.

\emph{Step 1:} Let $(u,\v)\in W_{\infty}\times\cV_{p}\subset W_{r}\times\cV_{r}$.
Then $\cE(u,\v)<+\infty$ and hence by standard arguments $\cE$ has
a at least one local minimizer $(u_{R},\v_{R})$ on every closed ball
of sufficiently large radius $R$ in $W_{r}\times\cV_{r}$
\[
\Balldimclosed[W_{r}\times\cV_{r}]R0:=\left\{ (u,\v)\in W_{r}\times\cV_{r}:\;\norm u_{W_{r}}+\norm{\v}_{\cV_{r}}\leq R\right\} \,.
\]
By Theorem \ref{thm:ts-gamma-conv}.2 there exists a recovery sequence
$\ue\in W_{0,\partial\bQ}^{1,p}(\bQ^{\eps}(\omega))$ such that $\ue\weakto\left|\bP\right|u_{R}$
weakly in $L^{p}(\bQ)$, $\cU_{\eps}\ue\weakto u_{R}$ weakly in $W^{1,r}(\bQ)$
and $\nabla\ue\tsweakto\chi_{\bP}\cdot\left(\nabla u_{R}+\v_{R}\right)$
weakly in two scales and 
\[
\cE(u_{R},\v_{R})=\lim_{\eps\to0}\cE_{\eps,\omega}(\ue)\,.
\]
\emph{Step 2:} We conclude for the minimizers 
\[
\liminf_{\eps\to0}\norm{\ue_{\mathrm{min}}}_{W_{0,\partial\bQ}^{1,p}(\bQ^{\eps}(\omega))}\leq\liminf_{\eps\to0}\cE_{\eps,\omega}(\ue_{\mathrm{min}})\leq\liminf_{\eps\to0}\cE_{\eps,\omega}(\ue)\leq\cE(u_{R},\v_{R})\,,
\]
which at the same time implies by Theorem \ref{thm:ts-gamma-conv}.1
that $\cU_{\eps}\ue\weakto u$ weakly in $L^{p}(\bQ)$ and $W^{1,r}(\bQ)$
and there exists $\v\in L^{r}(\bQ;\cV_{\pot}^{r}(\Omega,\bP))$ such
that $\nabla\ue\tsweakto\chi_{\bP}\cdot\left(\nabla u+\v\right)$
and 
\begin{align*}
\norm u_{W_{r}}+\norm{\v}_{\cV_{r}} & \leq C\,\cE(u_{R},\v_{R})\,,\\
\cE(u,\v) & \leq\cE(u_{R},\v_{R})\,,
\end{align*}
with $C$ independent from $\left(u_{R},\v_{R}\right)$. Since also
$\norm{\ue}_{W_{0,\partial\bQ}^{1,p}(\bQ^{\eps}(\omega))}\leq\cE(u_{R},\v_{R})$,
we conclude 
\[
\norm{u_{R}}_{W_{r}}+\norm{\v_{R}}_{\cV_{r}}\leq C\,\cE(u_{R},\v_{R})\,,
\]

\emph{Step 3:} Similarly, if $\left(u_{R^{\ast}},\v_{R^{\ast}}\right)$
is a further minimizer on any ball $\Balldimclosed[W_{r}\times\cV_{r}]{R^{\ast}}0$
with $\cE\of{u_{R^{\ast}},\v_{R^{\ast}}}\leq\cE\of{u_{R},\v_{R}}$
we can conclude 
\[
\norm{u_{R^{\ast}}}_{W_{r}}+\norm{\v_{R^{\ast}}}_{\cV_{r}}\leq C\,\cE(u_{R},\v_{R})
\]
from the argument of Step 2 and a suitable recovery sequence.

\emph{Step 4:} Hence, repeating Step 1 among the local minimizers,
there exists a global minimizer $\left(\bar{u},\bar{\v}\right)\in\Balldimclosed[W_{r}\times\cV_{r}]{C\,\cE(u_{R},\v_{R})}0$.

\emph{Step 5:} Repeating the argument of Step 2 we hence find that
every sequence of minimizers of $\cE_{\eps,\omega}$ satisfies the
claim.
\end{proof}
\begin{rem}
\label{rem:final-remark}If the term $\left|u\right|^{p}$ in the
above arguments is dropped, we first need to embed $\cU_{\eps}\ue$
uniformly into $W^{1,s}(\bQ)$. From here, we need $s$ large enough
such that $\bP^{\eps}$ still has the uniform $(r,s)$-trace property.
This will not affect the basic structure of the proofs, however it
makes the presentation more complicated and less readable.
\end{rem}

%\addcontentsline{toc}{section}{Nomenclature}\settowidth{\nomlabelwidth}{$\closedsets_{V}$, $\closedsets^{K}$, $(\closedsets(\Rd),\ttopology F)$}
%\printnomenclature{}

\section*{Nomenclature}
We use the following notations:

$x\sim y$, $x$ and $y$ are neighbors, see Definition \ref{def:graph-neighbor}

$A_{1,k},A_{2,k},A_{3,k}$, see Equation \eqref{eq:A123-k}

$\cA\left(0,\bP,\rho\right) := \left\{ \left(\tilde{x},{-}x_{d}+2\phi(\tilde{x})\right)\,:\;\left(\tilde{x},x_{d}\right)\in\Ball{\rho}0\backslash\bP\right\} $ (Lemma \ref{lem:uniform-extension-lemma})

$\Apaths(y,x)$, the Admissible paths from $y\in\Y\setminus\{x\}$ to $x\in\X_r$, see Definition \ref{def:admissible-path}

$\Ball{r}{x}$ the Ball around $x$ with radius $r$ (Section \ref{sec:Preliminaries})

$\cone_{\nu,\alpha,R}(x)$ the Cone with apix $x$, direction $\nu$, opening angle $\alpha$ and hight $R$  (Section \ref{sec:Preliminaries})

$\conv A$ the convex hull of $A$ (Section \ref{sec:Preliminaries})

Convex averaging sequence, see Definition \ref{def:convex-averaging-sequence}

$(\delta,M)$-regularity, see Definition \ref{def:loc-del-M-reg}

$\tilde\delta$, see Equation \eqref{eq:tilde-delta}

$\E(f|\sI)$ the Expectation of $f$ wrt. the invariant sets, \eqref{eq:invariant-sets-expectation}

$\E_{\mupalm}(f|\sI)$, the  Expectation of $f$ wrt. $\mupalm$ and the invariant sets, \eqref{eq:invariant-sets-expectation-mup}

Ergodic Theorem, see Theorems \ref{thm:Ergodic-Theorem}, \ref{thm:Ergodic-Theorem-ran-meas}

Ergodicity, see Definition \ref{def:erg-mixing}

$\eta$-regular (local), see Definition \ref{def:eta-regular}

$\eta(x)$, see Equation \eqref{eq:eta-distance}

$\closedsets_{V}$, $\closedsets^{K}$, $(\closedsets(\Rd),\ttopology F)$, see Equations \eqref{eq:closedsets-1}, \eqref{eq:closedsets-2}

$G(x)$ the Voronoi cell with center $x$ (Definition \ref{def:Voronoi})

$\G(\bP,\X),\,\G(\bP) $, the Graph constructed from $\bP$, see Definition \ref{def:Y-Partial-Y}

$\I=[0,1)^d$ the torus (Section \ref{sec:Preliminaries})

$\sI$ the Invariant sets, \eqref{eq:invariant-sets}

Isotropic cone mixing, see Definition \ref{def:iso-cone-mix}

$\Length(Y)$, the Length of an admissible path $Y$, see \eqref{eq:def-Length-apath}

$M(p,\delta)$, see Lemma \eqref{eq:M(delta-p)}

$M_{[\eta]},\,M_{[\eta],A}$ ($A$ a set), see Equation \eqref{eq:def-M-set-A}, a quantity on $\partial\bP$

$\tilde{M}_{\eta}(x)$, see Equation \eqref{eq:lem:local-delta-M-construction-estimate-1b}, a quantity on $\Rd$

$\tilde M$, see Equation \eqref{eq:tilde-M}

$M_k,\,M_{\fr,k}$, see $k\in\N$, $\fr>0$ \eqref{eq:notation-M}

$\fm_{[\eta]}{\left(p,\xi\right)}$, see  Lemma \ref{lem:M-eta}

$\fm_{k}:=\fm\of{p_{k},\tilde{\rho}_{k}/4}$, see Section \ref{subsec:5-Preliminaries}

$\fM(\Rd)$, the Measures on $\Rd$ (Section \ref{subsec:Random-measures-and}) 

Matern process, see Example \ref{exa:Matern}--\ref{exa:Matern-Pois}

Mesoscopic regularity, see Definition \ref{def:meso-regularity}

Mixing, see Definition \ref{def:erg-mixing}

$\bP_{r},\bP_{-r}$ Inner and outer hull of $\bP$ with hight $r$ (Section \ref{sec:Preliminaries})
 
Poisson process, see Example \ref{exa:poisson-point-proc}

$\bQ_1,\,\bQ_3$, see \eqref{eq:def-Q1-Q3}

$\rho(p)= \sup_{r<\delta\of p}r\sqrt{4M_{r}\of p^{2}+2}^{-1}$ \eqref{eq:def-rho-of-p}

$\hat\rho(p)=\inf\left\{ \delta\leq\delta(p)\,:\;\sup_{r<\delta}r\sqrt{4M_{r}\of p^{2}+2}^{-1}=\rho\right\} $ \eqref{eq:def-rhohat-of-p}

$\sfR_{0}(x,y)$, see Equation \eqref{eq:lem:radius-admissible-pahts}

$\Rd_1,\,\Rd_3$, see \eqref{eq:def-Rd1-Rd3}

Random closed sets, see Definition \ref{def:RACS}

$\T=[0,1)^d$ the torus (Section \ref{sec:Preliminaries})

$\tau_x$, Dynamical system (Definitions \ref{def:A-dynamical-system}, \ref{def:A-dynamical-system-Zd}) with respect to $x\in\Rd$ or $x\in\Zd$

$\cU$ for local and global extension operators (Lemma \ref{lem:uniform-extension-lemma})

$\X$, $\Y$ Families of points (Section \ref{sec:Preliminaries})

$\X_r(\omega)=\X_r(\bP(\omega))=2r\Zd\cap \bP_{-r}(\omega)$, \eqref{eq:def-X_r}

$\partial\X,\,\hat\X$, see Notation \ref{nota:XY}

$Y_{\fl}$, see Notation \ref{nota:Apaths}

$\Y_{\partial\X}$, see Notation \ref{nota:XY}

$\mathring{\Y},\,\partial\Y,\,\Y$, see Notation \ref{nota:Y}

\bibliographystyle{plain}
\bibliography{ts-methods}

\end{document}